\documentclass[11pt]{amsart}

\usepackage{stmaryrd}
\usepackage{url}
\usepackage{amscd,amsxtra,amssymb,mathrsfs, bbm}

\usepackage{tikz}
\usetikzlibrary{calc, positioning, shapes, fit, matrix, decorations}
\usetikzlibrary{decorations.shapes, decorations.pathreplacing}

\usepackage{xypic}
\usepackage[all]{xy}

\newtheorem{theorem}{Theorem}[section]
\newtheorem{corollary}[theorem]{Corollary}
\newtheorem{lemma}[theorem]{Lemma}
\newtheorem{proposition}[theorem]{Proposition}
\newtheorem{conjecture}[theorem]{Conjecture}

\theoremstyle{definition}
\newtheorem{definition}[theorem]{Definition}
\newtheorem{remark}[theorem]{Remark}
\newtheorem{example}[theorem]{Example}
\newtheorem{case}{Case}

\numberwithin{equation}{section}

\DeclareMathAlphabet{\mathpzc}{OT1}{pzc}{m}{it}

\DeclareMathOperator{\Perf}{\mathsf{Perf}}

\DeclareMathOperator{\SL}{\mathsf{SL}}
\DeclareMathOperator{\rk}{rk}
\DeclareMathOperator{\rad}{\mathsf{rad}}
\DeclareMathOperator{\coker}{\mathsf{coker}}

\renewcommand{\ker}{\mathsf{ker}}

\renewcommand{\dim}{\mathsf{dim}}
\DeclareMathOperator{\tor}{\mathsf{tor}}
\DeclareMathOperator{\Coh}{\mathsf{Coh}}
\DeclareMathOperator{\QCoh}{\mathsf{QCoh}}
\DeclareMathOperator{\VB}{\mathsf{VB}}

\DeclareMathOperator{\Tri}{\mathsf{Tri}}

\DeclareMathOperator{\MM}{\mathsf{M}}
\DeclareMathOperator{\art}{\mathsf{fnlg}}
\DeclareMathOperator{\can}{\mathsf{can}}
\DeclareMathOperator{\Art}{\mathsf{Tor}}
\DeclareMathOperator{\CM}{\mathsf{CM}}

\DeclareMathOperator{\Rep}{\mathsf{Rep}}

\DeclareMathOperator{\Ob}{\mathsf{Ob}}

\DeclareMathOperator{\krdim}{\mathsf{kr.dim}}

\DeclareMathOperator{\syz}{\mathsf{syz}}

\DeclareMathOperator{\depth}{\mathsf{depth}}

\DeclareMathOperator{\Supp}{\mathsf{Supp}}

\DeclareMathOperator{\Hom}{\mathsf{Hom}}

\DeclareMathOperator{\Ext}{\mathsf{Ext}}

\DeclareMathOperator{\GL}{\mathsf{GL}}

\DeclareMathOperator{\Ann}{\mathsf{ann}}
\DeclareMathOperator{\End}{\mathsf{End}}
\DeclareMathOperator{\Mat}{\mathsf{Mat}}
\DeclareMathOperator{\Spec}{\mathsf{Spec}}

\def\scX{\vec{\mathsf{x}}}
\def\scY{\vec{\mathsf{y}}}

\def\scU{\vec{\mathsf{u}}}
\def\scV{\vec{\mathsf{v}}}

\def\val{\mathsf{val}}
\def\bK{\mathbb{K}}

\def\dec{\vartriangleleft}
\newcommand{\bD}{\mathbbm{D}}

\newcommand{\kk}{\mathbbm{k}}

\newcommand{\FF}{\mathbb{F}}
\newcommand{\GG}{\mathbb{G}}
\newcommand{\LL}{\mathbb{L}}
\renewcommand{\mod}{\mathsf{mod}}

\newcommand{\llbrace}{(\!(}
\newcommand{\rrbrace}{)\!)}

\newcommand{\kA}{\mathcal{A}}
\newcommand{\kB}{\mathcal{B}}

\newcommand{\kF}{\mathcal{F}}
\newcommand{\kG}{\mathcal{G}}

\newcommand{\kI}{\mathcal{I}}

\newcommand{\kP}{\mathcal{P}}

\newcommand{\sP}{\mathcal{P}}

\newcommand{\lar}{\longrightarrow}

\newcommand{\gA}{A}
\newcommand{\gm}{\mathfrak{m}}

\newcommand{\gR}{R}
\newcommand{\gD}{D}

\newcommand{\gJ}{J}

\newcommand{\rA}{A}
\newcommand{\rB}{B}
\newcommand{\rC}{C}

\newcommand{\rR}{R}
\newcommand{\rD}{D}
\newcommand{\rS}{S}
\newcommand{\rQ}{Q}
\newcommand{\rK}{\mathbbm{K}}
\newcommand{\rL}{\mathbbm{L}}

\newcommand{\idm}{\mathfrak{m}}
\newcommand{\idn}{\mathfrak{n}}
\newcommand{\idp}{\mathfrak{p}}

\newcommand{\mM}{M}
\newcommand{\mN}{N}
\newcommand{\mK}{K}
\newcommand{\mT}{T}
\newcommand{\mE}{E}
\newcommand{\mF}{F}

\newcommand{\mL}{L}
\newcommand{\mX}{X}
\newcommand{\mY}{Y}

\newcommand{\mV}{V}
\newcommand{\mP}{P}
\newcommand{\mU}{U}
\newcommand{\mQ}{Q}

\newcommand{\sT}{T}

\newcommand{\BB}{\mathbb{B}}
\newcommand{\EE}{\mathbb{E}}
\newcommand{\HH}{\mathbb{H}}

\newcommand{\DD}{\mathbb{D}}
\newcommand{\PP}{\mathbb{P}}

\newcommand{\II}{\mathbb{I}}

\def\dE{\mathfrak E}		\def\dF{\mathfrak F}		
\def\dX{\mathfrak X}		
\def\dJ{\mathfrak J}

		\def\aK{\Bbbk}		
				
\def\sd{\mathsf d}	
	
\def\bnu{\boldsymbol{\nu}}
		\def\od{\odot}	\def\bu{\bullet}
\def\lb{\textup{(}}	\def\rb{\textup{)}}

\def\bnu{\boldsymbol{\nu}}
\def\bmu{\boldsymbol{\mu}}

\def\rad{\mathop\mathrm{rad}}
\def\rep{\mathop\mathsf{Rep}}

\def\el{\mathop\mathsf{El}}
\def\Mat{\mathop\mathrm{Mat}}
\def\End{\mathop\mathrm{End}\nolimits}

\def\st{\mathrm{st}}

\def\8{\infty}			
\def\bop{\bigoplus}	\def\+{\oplus}		\def\xx{\times}
\def\*{\otimes}
\def\str{\stackrel}

\def\dec{\vartriangleleft}	\def\ced{\vartriangleright}
\def\lha{\leftharpoondown}
\def\rha{\rightharpoondown\phantom{I\!}\!}

\def\row#1#2{\left( #1_1 , #1_2 , \dots , #1_{#2} \right)}
\def\lst#1#2{ #1_1 , #1_2 , \dots , #1_{#2}}

\def\iff{if and only if }

\def\kA{\mathcal A} 
\def\kB{\mathcal B} 
 \def\kP{\mathcal P}

\def\kF{\mathcal F} 
\def\kG{\mathcal G} 
 
\def\kI{\mathcal I}

\def\mZ{\mathbb Z}	\def\NN{\mathbb N}

\def\dX{\mathfrak X}

\def\DMO{\DeclareMathOperator}
\DMO{\ob}{Ob}            \DMO{\mor}{Mor}
\DMO{\Ker}{Ker}
\DMO{\id}{Id}

\usepackage{hyperref}

\newcommand{\hdline}[3]{\draw[dashed] (#1-#2-1.south west) -- (#1-#2-#3.south east);}
\newcommand{\vdline}[3]{\draw[dashed] (#1-1-#3.north east) -- (#1-#2-#3.south east);}
\newcommand{\hsline}[3]{\draw[thick] (#1-#2-1.south west) -- (#1-#2-#3.south east);}
\newcommand{\vsline}[3]{\draw[thick] (#1-1-#3.north east) -- (#1-#2-#3.south east);}

\newcommand{\sfrm}[3]{
\node[draw,solid, thick, fit=(#1-1-1)(#1-#2-#3), inner sep=0pt]{};}

\tikzset{
  decorate with/.style={decorate,decoration={shape backgrounds,shape=#1,shape size=1.5mm}},
   deco/.style={decorate with=dart},
   ordi/.style={draw,-stealth,  thick},
   conj/.style={dashed, draw, thick},
   ve/.style={circle, draw, thick, fill=blue!20, inner sep=1pt, outer sep=2pt, minimum size=7pt},
    dot/.style={fill=blue!10,circle,draw, inner sep=1pt, minimum size=5pt},
  dv/.style={star,star points=5,
star point ratio=2, draw, thick, fill=green!20, inner sep=1pt,outer sep=2pt,minimum size=7pt}
}

\tikzset{
    tbl5 nodes/.style={
        rectangle,
        execute at begin node=$,
       execute at end node=$,
       fill=blue!5,
        align=center,
        text depth=0.5ex,
        text height=2ex,
        inner xsep=0pt,
        outer sep=0pt,
           },
    tbl5/.style={
        matrix of nodes,
        row sep=-\pgflinewidth,
        column sep=-\pgflinewidth,
        nodes={
            tbl5 nodes
        },
        execute at empty cell={\node[draw=none]{};}
    }
  }

\title[Cohen--Macaulay modules and matrix problems]{Maximal Cohen--Macaulay
modules over  non--isolated surface singularities and matrix problems}

\author{Igor Burban}
\address{
Universit\"at zu K\"oln,
Mathematisches Institut,
Weyertal 86-90,
D-50931 K\"oln,
Germany
}
\email{burban@math.uni-koeln.de}

\author{Yuriy Drozd}
\address{
 Institute of Mathematics
National Academy of Sciences,
Tereschenkivska str. 3,
01004 Kyiv, Ukraine
}
\email{drozd@imath.kiev.ua}

\dedicatory{Dedicated to Claus Ringel
on the occasion of his  birthday}

\begin{document}

\begin{abstract}
In this article  we develop a new method to deal with maximal
Cohen--Macaulay modules over  non--isolated surface singularities. In particular, we give a
negative answer on an old question  of Schreyer about surface singularities with only countably
many indecomposable maximal Cohen--Macaulay modules. Next, we prove that
the degenerate cusp singularities have tame  Cohen--Macaulay representation type.
Our approach is illustrated on the case of
$\kk\llbracket x,y,z\rrbracket/(xyz)$ as well as several other rings.
This  study of  maximal Cohen--Macaulay modules over non--isolated singularities leads to a new
class  of problems of linear algebra, which we call representations of decorated
bunches of chains.
We prove that these  matrix problems have tame representation type and describe the underlying canonical forms.
\end{abstract}

\maketitle

\tableofcontents

\section{Introduction, motivation  and historical remarks}
\label{sec1}

In this article,  we essentially deal with the following question.
Let $f \in (x, y, z)^2 \subseteq \mathbb{C}\llbracket x, y, z\rrbracket =:
\rS$ be a polynomial.
How to describe all  pairs  $(\varphi, \psi) \in \Mat_{n \times n}(\rS)^{\times 2}$
such that
$
\varphi \cdot \psi = \psi \cdot \varphi = f \cdot \mathbbm{1}_{n \times n}?
$
Such a pair of matrices $(\varphi, \psi)$ is also called \emph{matrix factorization} of
$f$.
One of the earliest examples of this kind, dating   back to Dirac, is the following
 formula for a ``square root'' of the Laplace operator:
\begin{equation*}\label{E:Dirac}
\Delta := \Bigl[
\partial_x^2  + \partial_y^2 + \partial_z^2
\Bigr] \mathbbm{1}_{2 \times 2} =
\left[ \partial_x
\left(
\begin{array}{cc}
1 & 0 \\
0 & -1
\end{array}
\right) + \partial_y
\left(
\begin{array}{cc}
0   &  1 \\
1  & 0
\end{array}
\right) + \partial_z
\left(
\begin{array}{cc}
0   &  i \\
-i  & 0
\end{array}
\right)
\right]^2 =: \nabla^2.
\end{equation*}
Equivalently, the pair $(\varphi, \varphi)$ is a matrix factorization of the polynomial
$f = x^2 + y^2 + z^2$,  where
$
\varphi =
\left(\begin{smallmatrix}
x & y -i z \\
y + iz & -x
\end{smallmatrix}
\right).
$
One can show (see, for example \cite[Chapter 11]{Yoshino}) that up to a certain natural equivalence relation, the pair $(\varphi, \varphi)$ is  the \emph{only} non--trivial matrix factorization of $f$.
A certain version of this fact  was  already known
to Dirac. One of the  results of our  paper is a complete classification
of all matrix factorizations of the polynomial  $f = xyz$, see Subsection \ref{SS:degcusp2}.

According to Eisenbud \cite{Eisenbud}, the problem  of description   of all matrix factorizations of a polynomial $f$ can be rephrased as the question to classify  all \emph{maximal Cohen--Macaulay} modules over
the hypersurface singularity $\rA = \rS/(f)$. The latter problem can (and actually should) be posed
in a much broader context of  local Cohen--Macaulay rings or even in the non-commutative set-up
of orders over local Noetherian rings. This point of view was  promoted by Auslander starting from his
work  \cite{PhilNotes}.  The theory of maximal Cohen--Macaulay modules over
orders (called in this framework \textit{lattices}) dates back to the beginning of the
twentieth  century and has its origin in the theory of integral representations of finite groups, see for example  \cite{CurtisReiner}.

For a \emph{Gorenstein} local ring $(\rA, \idm)$,  Buchweitz observed that the
stable category of maximal Cohen--Macaulay modules $\underline{\CM}(\rA)$ is triangulated and
proved that the functor
\begin{equation}\label{E:Buchweitz}
\underline{\CM}(\rA) \lar D_{\mathsf{sg}}(\rA) := \frac{D^b(\rA-\mod)}{\Perf(\rA)}
\end{equation}
is an equivalence of triangulated categories,
where  $D^b(\rA-\mod)$ is the derived category of Noetherian $\rA$--modules and $\Perf(\rA)$
is its full subcategory of of perfect complexes \cite{Buchweitz}.

In the past decade, there was a significant growth of interest to a study of maximal Cohen--Macaulay modules
and matrix factorizations. At this place, we mention only the  following four  major directions, which were born in this period.
\begin{itemize}
\item Kapustin and Li discovered  a  connection between the theory of matrix factorizations
  with  topological quantum field theories \cite{KapLi}.
\item Khovanov and Rozansky suggested a new approach to construct  invariants of links,
based on matrix
factorizations of certain polynomials \cite{KhovanovRozansky}.
\item Van den Bergh introduced the notion of a non-commutative crepant resolution  of
a normal Gorenstein singularity \cite{VandenBergh}.
It turned out that this theory is closely related with the study of cluster--tilting objects in
 stable
categories of maximal Cohen--Macaulay modules over Gorenstein singularities \cite{BIKR, IyamaReiten,IyamaYoshino}.
\item Finally, Orlov established a close connection between the stable category of \emph{graded} Cohen--Macaulay modules over a \emph{graded} Gorenstein $\kk$--algebra $\rA$ and  the derived category of coherent sheaves
on $\mathsf{Proj}(\rA)$ \cite{Orlov}. This brought a new light on the study of  D--branes in Landau--Ginzburg models
and provided a new powerful tool for the homological mirror symmetry of Kontsevich \cite{Kontsevich}, see for example \cite{Sheridan}. As was proven by Keller, Murfet and Van den Bergh \cite{KellerMurfetVdB}, the stable categories of graded and non-graded maximal Cohen--Macaulay modules over a graded Gorenstein singularity are related by a  triangulated orbit category construction.
\end{itemize}

In this article, we deal with  the representation--theoretic study of maximal Cohen--Macaulay modules over \emph{surface} singularities. Of course, there are close interactions of this traditional area
 with all four new directions, mentioned above.
Moreover, for surface singularities,  the theory of maximal Cohen--Macaulay modules is particularly rich and interesting.
In a certain sense (which can be rigorously formulated) it is parallel to the theory
of vector bundles on projective curves. Following this analogy, the
\emph{normal} surface singularities
correspond to \emph{smooth} projective curves.

 One of the most beautiful applications
of the study   of maximal Cohen--Macaulay modules over surface singularities
is a conceptual  explanation of the \emph{McKay correspondence} for the finite
subgroups of $\SL_2(\mathbb{C})$, see \cite{GSVerdier,ArtinVerdier, Auslander, Yoshino, SurvOnCM, LeusckeWiegand}. One of the conclusions  of this theory states
that for the  \emph{simple} hypersurface singularities
$x^2 + y^{n+1} + z^2, n \ge 1$ (type $A_n$), $x^2 y + y^{n-1} + z^2, n \ge 4$ (type $D_n$),
 $x^3 + y^4 + z^2, x^3 + xy^3 + z^2 $ and
$ x^3 + y^5 + z^2$ (types $E_6$, $E_7$ and $E_8$)
there are only finitely many indecomposable matrix factorizations.

According to Buchweitz, Greuel and Schreyer  \cite{BGS}, two limiting \emph{non-isolated}
hypersurface singularities $A_\infty$ (respectively $D_\infty$) given by the equation
$x^2  + z^2$ (respectively $x^2 y + z^2$), have only countably many indecomposable maximal
Cohen--Macaulay modules.  In other words, $A_\infty$ and $D_\infty$ have \emph{discrete}
 Cohen--Macaulay representation type. Moreover, in \cite{BGS} it was shown that the simple
hypersurface singularities  are {exactly} the hypersurface singularities of {finite} Cohen--Macaulay
representation type. Moreover, if the base field has uncountably many elements, then
   $A_\infty$ and $D_\infty$ are the {only hypersurface} singularities
with countably many indecomposable maximal Cohen--Macaulay modules.

Going in another direction, in works of Kahn \cite{Kahn}, Dieterich  \cite{Dieterich},
 Drozd, Greuel and Kashuba \cite{DGK}
it was shown that the \emph{minimally elliptic} hypersurface singularities
$$
T_{p,q,r}(\lambda) = x^p + y^q + z^r + \lambda xyz,
$$ where
$
\frac{1}{p}+ \frac{1}{q} + \frac{1}{r} \le 1$ and
$\lambda \in \mathbb{C}\setminus \Delta_{(p,q,r)}$
have \emph{tame} Cohen--Macaulay representation type ($\Delta_{(p,q,r)}$ is a certain finite set).
In the case  $\frac{1}{p}+ \frac{1}{q} + \frac{1}{r} = 1$ the singularity $T_{p,q,r}(\lambda)$
is quasi-homogeneous
and called
\emph{simply elliptic}.  For $\frac{1}{p}+ \frac{1}{q} + \frac{1}{r} < 1$  it is a
\emph{cusp singularity},
 in this case one may  without  loss of generality assume $\lambda = 1$. A special interest to the study
 of maximal Cohen--Macaulay modules over this
  class of surface singularities is in particular motivated  by recent development in the homological mirror symmetry: according to Seidel \cite{Seidel} and Efimov \cite{Efimov}, the stable category (of certain equivariant) matrix factorizations of the potential $x^{2g+1} + y^{2g+1} + z^{2g+1} - xyz$ is  equivalent to the Fukaya category of a compact Riemann surface of genus $g \ge 2$.

 In the approach of Kahn \cite{Kahn},
a description of maximal Cohen--Macaulay modules over  simply elliptic singularities
reduces to the study of vector bundles on elliptic curves, whereas in the case
of the cusp singularities \cite{DGK} it boils down to a classification of vector bundles on the
Kodaira cycles of projective lines. In both cases the complete classification
of  indecomposable vector bundles is known: see \cite{Atiyah} for the case
of elliptic curves and \cite{DrozdGreuelBundles, Survey} for the case of Kodaira cycles.
The method  of  Dieterich \cite{Dieterich}   is  based
on the technique of representation theory of finite dimensional algebras and
can be applied
 only to  certain  simply elliptic singularities.
Unfortunately, neither of  these approaches
 leads to a fairly   explicit description of  indecomposable
matrix factorizations. See, however, recent work of Galinat \cite{Galinat} about the
$T_{3,3,3}(\lambda)$ case.

\medskip
\noindent
Our  article grew up from an attempt   to answer the following questions:
\begin{itemize}
\item Let $\rA$ be a non--isolated Cohen--Macaulay surface singularity of discrete
Cohen--Macaulay representation type over an uncountable
algebraically closed field of characteristic zero.
Is it true that $\rA \cong \rB^G$, where $\rB$ is a singularity of type $A_\infty$ or $D_\infty$ and
$G$ is a finite group of automorphisms of $\rB$? This  question was posed   in 1987 by Schreyer  \cite{Schreyer}.

\item Can a {non--isolated} Cohen--Macaulay surface singularity have
{tame}
Cohen--Macaulay representation type?
\end{itemize}
Now we present  the main results obtained in this article.

\medskip
\noindent
\textbf{Result A}. Let $(\rA, \idm)$ be a reduced complete  Cohen--Macaulay surface singularity, which
is not normal (hence non-isolated). We introduce new categories $\Tri(\rA)$ (\emph{category of triples})
and $\Rep(\mathfrak{X}_\rA)$ (\emph{category of elements} of a certain \emph{bimodule} $\mathfrak{X}_\rA$)
and a pair of functors
\begin{equation}\label{E:generalline}
\CM(\rA) \stackrel{\FF}\lar \Tri(\rA) \stackrel{\HH}\lar \Rep(\mathfrak{X}_\rA),
\end{equation}
such that $\FF$ is an equivalence of categories and $\HH$ preserves indecomposability and isomorphism
classes of objects, see Theorem \ref{T:BurbanDrozd} and Proposition \ref{P:passage-to-MP}.
In other words, this construction reduces a description of indecomposable
maximal Cohen--Macaulay $\rA$--modules to a certain problem of linear algebra (matrix problem).

\medskip
\noindent
\textbf{Result B}. The  above categorical construction leads to a new class of tame matrix problems
which we call \emph{representations of a decorated bunch of chains}, see Definition \ref{bunch}. It generalizes
the usual representations of a bunch of chains  \cite{bo1}, which are widely used
in the representation theory of finite dimensional algebras and its applications.

This new class of problems
is actually interesting by itself.
For example, it contains the following generalization of  the classical problem to find Jordan normal form
of a square matrix. Let $(\bD, \idn)$ be a discrete
valuation ring and
 $X \in \Mat_{n \times n}(\bD)$ for some $n \ge 1$. What is the canonical form of $X$
under the transformation rule
\begin{equation}\label{E:FormJordanique}
X \mapsto S_1 X S_2^{-1},
\end{equation}
where  $S_1, S_2 \in \GL_n(\bD)$ are such that
$S_1 \equiv S_2 \, \mod \, \idn$ (decorated conjugation problem)?
 It turns out that $X$ can be transformed  into a direct sum
of canonical forms from Definition \ref{D:decorJordanBlock}, see Theorem \ref{T:JordanKronecker}.

Another problem of this kind is a generalized Kronecker problem, stated as follows.
 Let $\rK$ be the field of fractions
of $\bD$, $X, Y \in \Mat_{m \times n}(\rK)$ be two matrices of the same size. To what form can
we bring the pair $(X, Y)$ under the transformation rule
\begin{equation}\label{E:FormKronecker}
(X, Y) \mapsto \bigl(S_1 X T^{-1}, S_2 Y T^{-1}\bigr),
\end{equation}
where  $T \in \GL_n(\rK)$ and $S_1, S_2 \in \GL_m(\bD)$ are such that
$S_1 \equiv S_2 \, \mod \, \idn$? In this case, the complete list of indecomposable canonical forms
is given in Subsection  \ref{Ex:CanonicalFormsDecKronecker}.

For a general  decorated bunch of chains $\mathfrak{X}$
there are two types of indecomposable objects in $\Rep(\mathfrak{X})$:
\emph{strings} (discrete series) and \emph{bands} (continuous series). We  prove
this result in Theorem \ref{list}, a separate treatment of  the  decorated conjugation problem (\ref{E:FormJordanique}) is also given in Section \ref{S:DecoratedConjugation}.

\medskip
\noindent
\textbf{Result C}. Using this technique, we give a \emph{negative} answer on Schreyer's question.
For example, let $\rR = \kk\llbracket u, v\rrbracket$ and
$$
\underbrace{\rR \times \rR \times \dots \rR}_{n+1 \; \mbox{\scriptsize{times}}} \supset
\rA = \Bigl\{(r_1, r_2, \dots, r_{n+1}) \, \big| \, r_{i}(0, z) = r_{i+1}(z, 0) \; \mbox{for} \;  1 \le i \le n \Bigr\}.
$$
Note that
for $n = 1$,  the ring $\rA$ is just a hypersurface singularity of type $A_\infty$. However, for $n >1$ it
is not isomorphic to $\rA_\infty^G$ or $\rD_\infty^G$, where $G$ is a finite group.

 We prove that
$\rA$ has only countably many indecomposable maximal Cohen--Macaulay modules, see Theorem \ref{T:Schreyer}. Moreover,  for
 an indecomposable maximal Cohen--Macaulay module $\mM$ we  have the
 following equality for its multi-rank:
$$\underline{\rk}(\mM) = (0, \dots, 0, 1, \dots, 1, 0, \dots, 0).$$
  Our approach leads to a wide  class
of non--isolated Cohen--Macaulay surface singularities of discrete Cohen--Macaulay
representation type, see Section \ref{sec11} for details.

\medskip
\noindent
\textbf{Result D}. There is an important class of non--isolated Gorenstein surface singularities
called \emph{degenerate cusps}. They were  introduced by Shepherd--Barron in \cite{ShepherdBarron}.
For example, the  natural limits of $T_{p,q,r}(\lambda)$--singularities, like the ordinary triple
point $T_{\infty\infty\infty} = \kk\llbracket x,y,z\rrbracket/(xyz)$, are degenerate cusps.
We prove that for a degenerate cusp $\rA$, the corresponding bimodule $\mathfrak{X}_A$ is a decorated bunch of chains. For example, the classification of maximal Cohen--Macaulay modules over $\kk\llbracket x,y,z\rrbracket/(xyz)$ reduces to the matrix problem, whose objects are representations of the
$\tilde{A}_5$--quiver
\begin{equation}\label{E:affinequiv}
\begin{xy}
(0,0)*+{\circ}="A";
(7,12)*+{\bullet}="B";
(20,12)*+{\circ}="C";
(27,0)*+{\bullet}="D";
(7,-12)*+{\bullet}="E";
(20,-12)*+{\circ}="F";
{\ar@{->} "B";"A"};
{\ar@{->} "B";"C"};
{\ar@{->} "D";"C"};
{\ar@{->} "D";"F"};
{\ar@{->} "E";"A"};
{\ar@{->} "E";"F"};
\end{xy}
\end{equation}
over the field of Laurent power series $\kk\llbrace t\rrbrace$.
The transformation rules at the sources $\bullet$ are
as for quiver representations, whereas for the targets $\circ$ they are like row transformations in the
decorated Kronecker problem (\ref{E:FormKronecker}), see (\ref{E:MPTinfty}) for the precise definition.

Hence, the degenerate cusps  have \emph{tame} Cohen--Macaulay representation type, see
Theorem \ref{T:cusp-are-tame}.
This fact is  quite
surprising for us from the following reason. The category of Cohen--Macaulay
modules over a simply elliptic singularity $T_{p,q,r}(\lambda)$ is tame of polynomial growth.
The cusp singularities $T_{p,q,r}$ are tame of exponential growth. In the approach of Kahn,
one reduces first the classification problem  to a description
of vector bundles on Kodaira cycles on projective lines. The latter problem
can be reduced to representations
of a  usual bunch of chains \cite{DrozdGreuelBundles, Survey}.
 This class  of tame matrix problems  was believed  to be the most general
among those with  exponential growth. The singularity $T_{\infty\infty\infty}$ is the
natural limit of the entire  family of all $T_{p,q,r}(\lambda)$ singularities. The tameness
of the underlying classification problem would suggest that it has to be
of the type which goes beyond representations of bunches of chains. But no problems
of such type have been known  before in the representation theory of finite dimensional algebras!

Using the periodicity of Kn\"orrer \cite{Knoerrer2}, the tameness of degenerate cusps
implies  that the non--reduced curve
singularities  $\kk\llbracket x,y\rrbracket/(x^2y^2)$ and
$\kk\llbracket x,y\rrbracket/(x^2y^2 - x^p)$, $ p \ge 3$   are
Cohen--Macaulay tame as well (at least if $\mathsf{char}(\kk) = 0$),
see Theorem \ref{T:tame-curve-sing}.

The categorical construction (\ref{E:generalline})
 turns out to be convenient
in  the following situation. Having a non--isolated Cohen--Macaulay surface singularity
$\rA$, it is natural to restrict oneself  to  the category $\CM^{\mathsf{lf}}(\rA)$ of those
maximal Cohen--Macaulay modules which are \emph{locally free on the punctured spectrum} of $\rA$.
It is natural to study this category because
\begin{itemize}
\item The stable category $\underline{\CM}^{\mathsf{lf}}(\rA)$ is $\Hom$--finite, whereas
the ambient  category $\underline{\CM}(\rA)$ is not.
\item If $\rA$ is  Gorenstein then $\underline{\CM}^{\mathsf{lf}}(\rA)$ is a triangulated
subcategory of $\underline{\CM}(\rA)$. Moreover, according to Auslander \cite{PhilNotes},
the shift functor $\Sigma = \Omega^{-1}$ is a Serre functor in $\underline{\CM}^{\mathsf{lf}}(\rA)$ (in other words,
$\underline{\CM}^{\mathsf{lf}}(\rA)$ is a 1--Calabi--Yau category).

\item In the terms of the Buchweitz's equivalence (\ref{E:Buchweitz}), $\underline{\CM}^{\mathsf{lf}}(\rA)$
can be identified with the thick subcategory of $D_{\mathsf{sg}}(\rA)$ generated
by the class of the residue field $\kk = \rA/\idm$, see  \cite{KellerMurfetVdB} or
 \cite{OrlovOcCompletions} for a proof.
\end{itemize}
It turns out that for a degenerate cusp $\rA$, the essential image of $\underline{\CM}^{\mathsf{lf}}(\rA)$ under the composition
$\HH \circ \FF$ is the additive closure of the category of band objects of $\Rep(\mathfrak{X}_\rA)$, see
Theorem \ref{T:cusp-are-tame}.

\medskip
\noindent
\textbf{Result E}.
We  illustrate our method  on several examples.
\begin{itemize}
\item For the singularity $\rA = T_{23\infty} = \kk\llbracket x,y,z\rrbracket/(x^3 + y^2 - xyz)$,
we give an explicit description of  all indecomposable maximal Cohen--Macaulay modules. In fact,
our method reduces their description to the decorated Kronecker problem.
Moreover, we compute generators of all maximal Cohen--Macaulay modules of rank one (written as ideals
in $\rA$) and describe  several families
of matrix factorizations corresponding to them. See Section \ref{sec5} for details.

\item For the singularity $\rA = T_{\infty\infty\infty} =
\kk\llbracket x,y,z\rrbracket/(xyz)$, we give an explicit  description of  all indecomposable objects of
${\CM}^{\mathsf{lf}}(\rA)$. For the rank one objects of ${\CM}^{\mathsf{lf}}(\rA)$, we compute the corresponding matrix factorizations of the polynomial $xyz$, see Subsection \ref{SS:degcusp2} and especially Theorem \ref{T:xyz} and Proposition \ref{P:xyz}.

According  to
Sheridan \cite[Theorem 1.2]{Sheridan} and Abouzaid et al.~\cite[Section 7.3]{AAEKO}, the triangulated category $\underline{\CM}^{\mathsf{lf}}(\rA)$ admits a symplectic mirror description. We hope that our results will bring a new light in the study of  Fukaya categories of Riemann surfaces.

\item Our method equally allows to treat those degenerate cusps, which are \emph{not} hypersurface
singularities. We consider the following two cases:
\begin{itemize}
\item $\rA = T_{\infty\infty\infty\infty} = \kk\llbracket x,y,u,v\rrbracket/(xy, uv)$. This surface singularity is a complete intersection. We describe  all rank one objects of   ${\CM}^{\mathsf{lf}}(\rA)$. For some of them, we compute the corresponding presentation matrices,  see Subsection \ref{SS:degcusp3}.
\item $\rA = \kk\llbracket x,y,z,u,v\rrbracket/(xz, xu, yu, yv, zv)$. It is a Gorenstein surface singularity, which is not a complete intersection. We explicitly describe all rank one objects of   ${\CM}^{\mathsf{lf}}(\rA)$, see Subsection \ref{SS:degcusp4}.
\end{itemize}
\item Finally, we treat the integral Cohen--Macaulay surface singularity $$\rA = \kk\llbracket u, v, w, a, b\rrbracket/(uv - w^2, ab - w^3, aw - bu, bw - av, a^2-uw^2, b^2-vw^2),$$ which is representation tame and not Gorenstein, see Subsection \ref{SS:sing-semichain}. We describe
all maximal Cohen--Macaulay modules of rank one, see Proposition \ref{P:computesemichains},
 and compute some  one--parameter families of indecomposable
objects of ${\CM}^{\mathsf{lf}}(\rA)$ of rank two, see Remark \ref{R:SemichainsContSeries}.
\end{itemize}

\medskip
\noindent
\emph{Acknowledgement}. Parts of this work were done during the authors stay
at   Max Planck Institute f\"ur Mathematik in Bonn and at the
Mathematical Research Institute in Oberwolfach within the ``Research in Pairs'' programme from
March 18 -- March 30, 2012 and July 21 -- August 2, 2013.
The research of the first-named author was  supported by the DFG
projects Bu--1866/1--1 and    Bu--1866/2--1.  The research of the second-named author was
supported by the INTAS grant 06--1000017--9093. The first-named
author would like to thank Duco van Straten for
helpful discussions  and Wassilij Gnedin for pointing out
numerous misprints in the previous version of this paper. Our special thanks go to Lesya Bodnarchuk for supplying us with TikZ pictures illustrating representation theory of decorated bunches of chains.

\section{Generalities on maximal Cohen--Macaulay modules}
\label{sec2}

Let $(\rA,\idm)$ be a Noetherian local ring,
$\kk = \rA/\idm$ its residue field and $d = \krdim(\rA)$ its Krull dimension.
Throughout the paper
$\rA\mathsf{-mod}$ denotes the category of Noetherian (i.e. finitely generated) $\rA$--modules,
whereas
$\rA\mathsf{-Mod}$ stands for the category of all $\rA$--modules,  $\rQ = \rQ(\rA)$ is the
total ring of fractions of $\rA$ and $\kP$ is the set of prime ideals of height $1$.

\begin{definition}\label{D:CMM}
A  Noetherian  $\rA$--module $\mM$ is  called \emph{maximal Cohen--Macaulay} if
$$\Ext_\rA^i(\kk, \mM) = 0 \quad \mbox{for all} \quad 0 \le i < d.$$
\end{definition}

\subsection{Maximal Cohen--Macaulay modules over surface singularities}
In this article we focus on the study of  maximal Cohen--Macaulay modules over  Noetherian rings of Krull dimension two,
also called surface singularities. This case is actually rather  special because of the
following well--known lemma.

\begin{lemma}\label{L:CMdim2} Let $(\rA, \idm)$ be a surface singularity,
 $\mN$ be a maximal Cohen--Macaulay $\rA$--module and $\mM$ a Noetherian $\rA$--module. Then
the $\rA$--module $\Hom_\rA(\mM, \mN)$ is maximal Cohen--Macaulay.
\end{lemma}

\noindent
\emph{Proof}. From  a free presentation $\rA^n \stackrel{\varphi}\to \rA^m \to M \to 0$
of $M$ we obtain an exact sequence:
$$
0 \lar \Hom_\rA(\mM, \mN) \lar\mN^m \stackrel{\varphi^*}\lar \mN^n \lar \coker(\varphi^*) \lar 0.
$$
Since $\depth_\rA(\mN) = 2$, applying the Depth Lemma twice we obtain:
$$\depth_\rA\bigl(\Hom_\rA(\mM, \mN)\bigr) \ge  2.$$ Hence, the module $\Hom_\rA(\mM, \mN)$ is maximal Cohen--Macaulay.
\qed

\vspace{1mm}
\noindent
The following  standard  result is due to  Serre \cite{Serre}, see also \cite[Proposition 3.7]{SurvOnCM}.

\begin{theorem}\label{T:Serre}
Let $(\rA, \idm)$ be a surface singularity. Then we have:
\begin{enumerate}
\item
 The ring $\rA$ is normal (i.e.~it is a domain, which is integrally
closed in its field of fractions) if and only if it is
Cohen--Macaulay and isolated.
\item Assume $A$ to be Cohen--Macaulay and Gorenstein in codimension one (e.g.~$\rA$
is normal) and  $\mM$ be
a Noetherian $\rA$--module. Then $\mM$ is maximal Cohen--Macaulay if and only if
it is reflexive.
\end{enumerate}
\end{theorem}

\medskip
\noindent
The next  result  underlines  some features of maximal Cohen--Macaulay modules
which occur only in the case of  surface singularities. See for example
\cite[Proposition 3.2 and Proposition 3.7]{SurvOnCM}
for the proof.

\begin{theorem}\label{P:Macaulafic}
Let $(\rA, \idm)$ be a reduced Cohen--Macaulay surface singularity with a canonical module
$K$. Then we have:
\begin{enumerate}
\item The canonical embedding functor
$\CM(\rA) \to \rA\mathsf{-mod}$ has a left adjoint functor
$\mM \stackrel{\delta}\mapsto \mM^\dagger := \mM^{\vee\vee} =
\Hom_\rA\bigl(\Hom_\rA(\mM, \mK), \mK\bigr).$
In other words, for an arbitrary Noetherian module $\mM$ and a maximal Cohen--Macaulay module
$\mN$ the map $\delta$ induces an  isomorphism
$
\Hom_\rA(\mM^\dagger, \mN) \cong \Hom_\rA(\mM, \mN).
$
The constructed functor $\dagger$ will be  called \emph{Macaulayfication functor}.
\item
Moreover, for any Noetherian $\rA$--module $\mM$ the following sequence is exact:
$$
0 \lar \tor(\mM) \lar \mM \stackrel{\delta}\lar  M^{\dagger} \lar T \lar 0,
$$
where $\tor(\mM) = \ker(\mM \to \rQ\otimes_\rA \mM)$ is the \emph{torsion part} of $\mM$
and $\mT$ is some $\rA$--module of finite length.
\item Moreover, if $\rA$ is Gorenstein in codimension one, then for any Noetherian $\rA$--module
$\mM$
there exists a natural isomorphism $\mM^\dagger \cong \mM^{**}$, where
$* = \Hom_\rA(\,-\,,\rA)$.
\end{enumerate}
\end{theorem}

\medskip

\noindent
The following lemma provides  a useful tool to compute  the Macaulayfication of a given
Noetherian module.

\begin{lemma}\label{L:fact-on-Macaulaf} In the notations of Theorem \ref{P:Macaulafic},
let $\mM$ be a Noetherian $\rA$--module, which is a submodule of a maximal Cohen--Macaulay  $\rA$--module
$\mX$. Let $x \in \mX \setminus \mM$ be such that $\idm^t x \, \in \mM$ for some $t \ge 1$.
Then $\mM^\dagger \cong \langle \mM, x \rangle^\dagger$, where
$\langle \mM, x \rangle$ is the $\rA$--submodule of $\mX$ generated by $\mM$ and $x$.
\end{lemma}

\begin{proof}
Consider the short exact sequence
$
0 \rightarrow  \mM \stackrel{\imath}\rightarrow \langle \mM, x\rangle \rightarrow  \mT
\rightarrow  0.
$
From the assumptions of Lemma it follows that $\mT$ is a finite length module. In particular,
for any $\idp \in \sP$ the map $\imath_\idp$ is an isomorphism. By the functoriality of Macaulayfication
we conclude that the morphism $\imath^\dagger: \mM^\dagger \to \langle \mM, x\rangle^\dagger$ is
an isomorphism in codimension one. By \cite[Lemma 3.6]{SurvOnCM}, the morphism
$\imath^\dagger$ is an isomorphism.
\end{proof}

Let us additionally assume our Cohen--Macaulay surface singularity $\rA$ to be Henselian and
$\rA \subseteq \rB$ to be a finite ring extension. Then the ring $\rB$ is semi--local.
Moreover,
$\rB \cong  (\rB_1, \idn_1) \times (\rB_2, \idn_2) \times \dots \times (\rB_t, \idn_t)$, where all
$(\rB_i, \idn_i)$ are local. Assume that all  rings  $\rB_i$ are Cohen--Macaulay.

\begin{proposition}
The functor $\rB \boxtimes_\rA - : \CM(\rA) \to \CM(\rB)$ mapping
 a maximal Cohen--Macaulay module
$\mM$ to $\rB \boxtimes_\rA  \mM := (\rB \otimes_\rA \mM)^\dagger$ is left adjoint to the forgetful functor
$\CM(\rB) \to \CM(\rA)$. In other words, for any maximal Cohen--Macaulay $\rA$--module
$\mM$ and a maximal Cohen--Macaulay $\rB$--module $\mN$ we have:
$$
\Hom_\rB(\rB \boxtimes_\rA \mM, \mN) \cong \Hom_\rA(\mM, \mN).
$$
Assume additionally  $\rA$ and $\rB$  to be both   reduced.
Then for any Noetherian
$\rB$--module $\mM$ there exist a natural isomorphism $\mM^{\dagger_\rA} \cong \mM^{\dagger_\rB}$ in the category of  $\rA$--modules.
\end{proposition}

\noindent
For a proof,  see for example \cite[Proposition 3.18]{SurvOnCM}. \qed

\begin{lemma}\label{L:onMacaulafInDimOne}
Let $(\rA, \idm)$ be a reduced Noetherian ring  of Krull dimension one with a canonical
module
$\mK$. Then for any Noetherian $\rA$--module $\mM$ we have a functorial isomorphism:
$
\mM^{\vee\vee} \cong  \mM/\tor(\mM),$ where $\vee = \Hom_\rA(\,-\,,\mK)$.
\end{lemma}

\begin{proof} From  the canonical short exact sequence
$
0 \to \tor(\mM) \to \mM \to \mM/\tor(\mM) \to 0
$
we get the isomorphism $\bigl(\mM/\tor(\mM)\bigr)^\vee \to \mM^\vee$. Since $\mM/\tor(\mM)$ is
a maximal Cohen--Macaulay $\rA$--module and $\vee$ is a dualizing functor, we get two natural
isomorphisms
$$
\mM/\tor(\mM) \stackrel{\cong}\lar \bigl(\mM/\tor(\mM)\bigr)^{\vee\vee}
\stackrel{\cong}\longleftarrow \mM^{\vee\vee},
$$
being a part of the commutative diagram
$$
\xymatrix
{
\mM \ar[rr] \ar[d] & & \mM/\tor(\mM) \ar[d] \\
\mM^{\vee\vee} \ar[rr] & & \bigl(\mM/\tor(\mM)\bigr)^{\vee\vee},
}
$$
in which all  morphisms are the canonical ones.
This yields the claim.
\end{proof}

\begin{corollary}\label{C:compat-of-Macaulayf}
Let $(\rA, \idm)$ be a reduced Cohen--Macaulay surface singularity
with a canonical module $\mK$ and
$\rA \subseteq \rR$ be its normalization. Then for any Noetherian $\rA$--module $\mM$
and any $\idp \in \kP$  we have a natural
isomorphism
$$
(\rR \boxtimes_\rA \mM)_\idp \cong \rR_\idp \otimes_{\rA_\idp} \mM_\idp/\tor(\rR_\idp \otimes_{\rA_\idp} \mM_\idp).
$$
\end{corollary}

\begin{proof}
Note that $\rA_\idp$ is a reduced Cohen--Macaulay ring  of Krull dimension one,
$\mK_\idp$ is the canonical module
of $\rA_\idp$ and  $\rR_\idp$ is the  normalization of  $\rA_\idp$. Hence, this corollary
 is a consequence of Lemma \ref{L:onMacaulafInDimOne}.
\end{proof}

\subsection{On the category $\CM^{\mathsf{lf}}(\rA)$}
Let $(\rA, \idm)$ be  a  Cohen--Macaulay ring of an arbitrary Krull dimension $d$.

\begin{definition}
 A maximal
Cohen--Macaulay $\rA$--module $\mM$ is locally free  on the punctured spectrum of $\rA$
if for any $\idp \in \Spec(\rA) \setminus \{\idm\}$ the localization
$\mM_\idp$ is a free $\rA_\idp$--module.
\end{definition}

\noindent
Of course, any maximal Cohen--Macaulay module fulfills the above property provided $\rA$ is an
\emph{isolated
singularity}. However, for the non--isolated ones  we get  a very nice proper subcategory of $\CM(\rA)$.

\begin{theorem}
Let $\CM^{\mathsf{lf}}(\rA)$ be the category of maximal Cohen--Macaulay modules which are locally free
on the punctured spectrum. Then the following results are true.
\begin{itemize}
\item The stable category $\underline{\CM}^{\mathsf{lf}}(\rA)$ is $\Hom$--finite. This means that
for any objects $\mM$ and $\mN$ of $\underline{\CM}^{\mathsf{lf}}(\rA)$ the $\rA$--module
$\underline{\Hom}_\rA(\mM, \mN)$ has finite length.
\item Moreover, assume that
 $\rA$ is Gorenstein. Then we have:
\begin{itemize}
\item  $\underline{\CM}^{\mathsf{lf}}(\rA)$ is a triangulated subcategory
of $\underline{\CM}(\rA)$.
\item The shift functor $\Sigma = \Omega^{-1}$ is a Serre functor in $\underline{\CM}^{\mathsf{lf}}(\rA)$. This means that
for any objects $\mM$ and $\mN$ of $\underline{\CM}^{\mathsf{lf}}(\rA)$ we have a bifunctorial
isomorphism
\begin{equation*}\label{E:SerreDual}
\underline{\Hom}_\rA(\mM, \mN) \cong \mathbbm{D}\Bigl(\underline{\Hom}_\rA\bigl(\mN, \Sigma(\mM)\bigr)\Bigr),
\end{equation*}
where $\DD$ is the Matlis duality functor.
\item In the terms of Buchweitz's equivalence (\ref{E:Buchweitz}), we have an exact equivalence:
$$
\underline{\CM}^{\mathsf{lf}}(\rA) \lar \mathsf{thick}(\kk) \subset D_{\mathsf{sg}}(\rA),
$$
where $\mathsf{thick}(\kk)$ is the smallest triangulated subcategory of $D_{\mathsf{sg}}(\rA)$
containing the class of the residue field $\kk$ and closed under direct summands in
$D_{\mathsf{sg}}(\rA)$.
\end{itemize}
\end{itemize}
\end{theorem}

\noindent
\emph{Comment on the proof}. For the proof of the first statement, see for example \cite[Proposition 9.4]{SurvOnCM}. The second  result  easily follows from the definition of the triangulated category
structure of  $\underline{\CM}(\rA)$. The third statement dates back to Auslander \cite[Proposition 8.8 in  Chapter 1
and Proposition 1.3 in  Chapter 3]{PhilNotes}, see also \cite[Chapter 3]{Yoshino}.
Finally, the proof of the last result is  essentially contained in the proof of \cite[Proposition A.2]{KellerMurfetVdB} as well as in \cite[Lemma 2.6 and Proposition 2.7]{OrlovOcCompletions}.
\qed

\begin{remark}
Observe that if $d \ge 2$ and $\mM$ is an objects of ${\CM}^{\mathsf{lf}}(\rA)$ then there exists
$n \ge 1$ such that $\mM_\idp \cong \rA_\idp^n$ for any $\idp \in \Spec(\rA)\setminus \{\idm\}$.
However, this is not true if $d = 1$ and $\Spec(\rA)$ has several irreducible components.
\end{remark}

\vspace{1mm}
\noindent
Let $d \ge 2$,  $(\rS, \idn)$ be a regular ring of Krull dimension $d + 1$ and $f \in \idn^2$ be  such that
the hypersurface singularity
$(\rA, \idm) = \rS/(f)$ is  reduced. Let $(\varphi, \psi) \in \Mat_{n \times n}(\rS)^{\times 2}
$
be a matrix factorization of $f$ and  $\bar\varphi, \bar\psi$ be the corresponding images in
 $\Mat_{n \times n}(\rA)$. Let $\mM = \mathsf{cok}(\bar\varphi)$ be the maximal Cohen--Macaulay
$\rA$--module corresponding to $(\varphi, \psi)$. Next, for any $1 \le p \le n$,  let $I_p(\bar\varphi)$
be the $p$--th Fitting ideal of $\mM$, i.e.~the ideal generated by all $p \times p$ minors of $\bar\varphi$.

\begin{lemma}\label{P:CriterionLF} Let $(\varphi, \psi) \in \Mat_{n \times n}(\rS)^{\times 2}
$
be a matrix factorization of $f$ and $\mM = \mathsf{cok}(\bar\varphi)$.
Then  $\mM$ belongs to ${\CM}^{\mathsf{lf}}(\rA)$
if and only if the following is true.
\begin{itemize}
\item There exists $t \ge 1$ and a unit $u \in \rS$ such that $\det(\varphi) = u \cdot f^t$.
\item $\sqrt{I_{n - t}(\bar\varphi)} = \idm$ and $I_{n - t + 1}(\bar\varphi) = 0$.
\end{itemize}
\end{lemma}

\begin{proof} Recall that  $\rQ$ denotes  the total ring of fractions of $\rA$.
 If $\mM$ belongs to ${\CM}^{\mathsf{lf}}(\rA)$
then there exists $t \ge 1$ such that $\rQ \otimes_\rA \mM = \rQ^t$. By \cite[Lemma 2.34]{SurvOnCM},
we get the first condition. The second condition follows from \cite[Lemma 1.4.8]{BrunsHerzog}.
\end{proof}

\section{Main construction}
\label{sec3}

Let
$(\rA, \idm)$ be a reduced complete (or analytic) Cohen--Macaulay ring
of Krull dimension two, which is not an isolated singularity.
 Let  $\rR$  be the  normalization of $\rA$.
 It is well--known that $\rR$ is again complete (resp.~analytic)  and the
ring extension $\rA \subseteq \rR$ is finite, see \cite{GLS} or
\cite{deJongPfister}.  Moreover,  $\rR$ is isomorphic to the  product of a finite number of normal
local rings:
$$
\rR \cong  (\rR_1, \idn_1) \times (\rR_1, \idn_1) \times \dots \times
(\rR_t, \idn_t).
$$
 According to Theorem \ref{T:Serre}, all rings $\rR_i$ are automatically Cohen--Macaulay.

Let $I = \Ann(\rR/\rA)$
be the \emph{conductor ideal}. It is easy to see  that  $I$ is also an ideal in $\rR$. Denote
$\bar\rA = \rA/I$ and $\bar\rR = \rR/I$.

\begin{lemma}\label{L:prop-of-conduct}
In the notations as above we have.
\begin{enumerate}
\item The ideal $I$ is a maximal Cohen--Macaulay module,  both over $\rA$ and over $\rR$.
\item The rings $\bar{\rA}$ and $\bar{\rR}$ are Cohen--Macaulay of Krull dimension one.
\item The inclusion $\bar{\rA} \to \bar{\rR}$ induces the injective homomorphism
of rings of fractions $\rQ(\bar\rA) \to \rQ(\bar\rR)$. Moreover, the
canonical morphism $\bar\rR \otimes_{\bar\rA} \rQ(\bar\rA) \to \rQ(\bar\rR)$ is an
isomorphism.
\end{enumerate}
\end{lemma}

\begin{proof}
First note that $I \cong \Hom_\rA(\rR, \rA)$. Hence,
by Lemma \ref{L:CMdim2}, the ideal   $I$ is maximal Cohen--Macaulay, viewed  as $\rA$--module.
Since the ring extension $\rA \subseteq \rR$ is finite, $I$   is also maximal Cohen--Macaulay
as a module over
$\rR$.

Next, the closed subscheme  $V(I) \subset \Spec(\rA)$ is exactly the non-normal locus
of  $\rA$.
If $\rA$ is normal then $\rA = \rR$ and there is nothing to prove. If $\rA$ is not normal, then
$\krdim\bigl(V(I)\bigr) \ge 1$. Indeed, by Theorem \ref{T:Serre},
an isolated  surface singularity which is not normal, can   not be
Cohen--Macaulay. Since $\rA$ is reduced,  we have:
 $\krdim\bigl(V(I)\bigr) =  1$. In particular,
 $
 \krdim(\bar\rA) = 1 = \krdim(\bar\rR).
 $
 Applying Depth Lemma to the short exact sequences
 $$
 0 \lar I \lar \rA \lar \bar\rA \lar 0 \quad
 \mbox{and} \quad 0 \lar I \lar \rR \lar \bar\rR \lar 0
 $$
 we conclude that $\bar{\rA}$ and $\bar{\rR}$ are both Cohen--Macaulay (but not necessarily reduced).

 Let $\bar{a} \in \bar\rA$ be a regular element. Since $\bar\rR$ is a Cohen--Macaulay
 $\bar\rA$--module, $\bar{a}$ is regular in  $\bar\rR$, too. Hence, we obtain
 a well--defined injective morphism of rings $\rQ(\bar\rA) \to \rQ(\bar\rR)$.

 Finally, consider the canonical ring homomorphism $\gamma: \bar\rR \otimes_{\bar\rA} \rQ(\bar\rA)
 \to \rQ(\bar\rR)$, mapping a simple tensor $\bar{r} \otimes \frac{\displaystyle \bar{a}}{\displaystyle \bar{b}}$
 to $\frac{\displaystyle \bar{r}\bar{a}}{\displaystyle \bar{b}}$. Since any element
 of $\bar\rR \otimes_{\bar\rA} \rQ(\bar\rA)$ has the form $\bar{r} \otimes \frac{\displaystyle \bar{1}}{\displaystyle \bar{b}}$ for some $\bar{r} \in \bar\rR$ and $\bar{b} \in \bar\rA$,
it is easy to see $\gamma$ is injective. Next, consider the canonical ring homomorphism
$\bar{\rR} \to \bar\rR \otimes_{\bar\rA} \rQ(\bar\rA)$.
It is easy to see that $\bar{r} \otimes \bar{1}$ is a non--zero divisor
in $\bar\rR \otimes_{\bar\rA} \rQ(\bar\rA)$ provided $\bar{r} \in \bar{R}$ is regular.
Since $\bar\rR \otimes_{\bar\rA} \rQ(\bar\rA)$
is a finite ring extension of $\rQ(\bar\rA)$, it is artinian.  In particular, any regular element  in this  ring is invertible.
From the universal property of localization we obtain a ring homomorphism
$\rQ(\bar\rR) \to \bar\rR \otimes_{\bar\rA} \rQ(\bar\rA)$, which is inverse to $\gamma$.
  \end{proof}

\begin{lemma}\label{L:canmorphdim2}
Let $\mM$ be a maximal Cohen--Macaulay $\rA$--module. Then we have:

\noindent $\bullet$
  The canonical morphism of $\rQ(\bar\rR)$--modules
$$
\theta_\mM:
\rQ(\bar\rR) \otimes_{\rQ(\bar\rA)} \bigl(\rQ(\bar\rA)
\otimes_{\rA} \mM\bigr)
\stackrel{\cong}\lar \rQ(\bar\rR) \otimes_{\rR} \bigl(\rR  \otimes_{\rA} \mM\bigr)
 \lar
 \rQ(\bar\rR) \otimes_{\rR} \bigl(\rR  \boxtimes_{\rA} \mM\bigr)
$$
is an epimorphism.

\noindent $\bullet$
The canonical morphism of $\rQ(\bar\rA)$--modules
$$\tilde{\theta}_\mM:  \rQ(\bar\rA) \otimes_{\rA} \mM \to
\rQ(\bar\rR) \otimes_{\rA} \mM \stackrel{\theta_\mM}\lar
\rQ(\bar\rR) \otimes_{\rR} \bigl(\rR  \boxtimes_{\rA} \mM\bigr)
$$
is a monomorphism.
\end{lemma}

\begin{proof}
By Theorem \ref{P:Macaulafic}, the
cokernel of the canonical morphism $\rR \otimes_\rA \mM \to \rR \boxtimes_\rA \mM$
has finite length. Hence, it vanishes after tensoring with $\rQ(\bar\rR)$. Thus,  the map
$\theta_\mM$ is surjective. The first statement of lemma is proven.

Denote by
$
\widetilde\mM' := \rR \otimes_\rA  \mM/\mathsf{tor}_\rR(\rR \otimes_\rA \mM)$ and
$\widetilde\mM := \widetilde\mM'^\dagger$.
First note that the canonical morphism
of $\rA$--modules
$ \mM \stackrel{\kappa}\lar  \widetilde\mM', \; m \mapsto [1 \otimes m]
$
is a monomorphism. As a result, the morphism
$I \mM \stackrel{\bar{\kappa}}\lar  I \widetilde\mM'$, which is a restriction
 of $\kappa$, is also injective. Moreover,
$\bar{\kappa}$ is also surjective: for any $a \in I, b \in \rR$ and $m \in \mM$ we have:
$a \cdot [b \otimes m] = [ab \otimes m] = [1 \otimes (ab) \cdot m]$ and $ab \in I$.

Since the module $\widetilde\mM'$ is torsion free, by  Theorem \ref{P:Macaulafic} we have a short exact sequence
$$
0 \lar \widetilde\mM' \stackrel{\xi}\lar \widetilde\mM \lar \mT \lar 0,
$$
where $\mT$ is a module of finite length. It implies that the cokernel
of the induced map $I \widetilde\mM' \stackrel{\bar\xi}\lar I \widetilde\mM$
has finite length as well. Let $\mM \stackrel{\jmath}\lar  \widetilde\mM$ be the composition
of $\kappa$ and $\xi$ and $I\mM \stackrel{\bar\jmath}\lar  I\widetilde\mM$ be the induced map.
Then we have the following commutative diagram with exact rows:
\begin{equation}
\begin{array}{c}
\xymatrix
{ 0 \ar[r] & I \mM  \ar[r] \ar[d]_{\bar{\jmath}} & M \ar[r] \ar[d]_{\jmath} & \bar{A} \otimes_{{A}} \mM
\ar[d]^{\eta} \ar[r] & 0 \\
0 \ar[r] & I \widetilde\mM  \ar[r] & \widetilde\mM \ar[r] & \bar{R} \otimes_{{R}} \widetilde\mM
\ar[r] & 0. }
\end{array}
\end{equation}
Since $\jmath$ is injective and the cokernel  $\bar{\jmath}$ is of finite length, snake lemma
implies that $\ker(\eta)$ has finite length. Since $\rQ(\bar{\rA}) \otimes_{\bar\rA} \,-\,$
is an exact functor, we obtain an  exact sequence
$$
0 \lar \rQ(\bar{\rA}) \otimes_{\bar\rA} \ker(\eta)
\lar  \rQ(\bar{\rA}) \otimes_{\bar\rA} \bar{A} \otimes_{{A}} \mM
\xrightarrow{1 \otimes \eta} \rQ(\bar{\rA}) \otimes_{\bar\rA} \bar{R} \otimes_{{R}} \widetilde\mM.
$$
It remains to take into account that $\rQ(\bar{\rA}) \otimes_{\bar\rA} \ker(\eta) = 0$,
$\rQ(\bar{\rA}) \otimes_{\bar\rA} \bar{R} = \rQ(\bar\rR)$ and $1 \otimes \eta$ coincides
with the morphism $\widetilde{\theta}_\mM$.
\end{proof}

\begin{definition}\label{D:triples}
In the notations of this section,  consider the following
\emph{category of triples} $\Tri(\rA)$. Its objects
are triples $(\widetilde\mM, V, \theta)$, where $\widetilde\mM$
is a maximal Cohen--Macaulay $\rR$--module, $V$ is a Noetherian
$\rQ(\bar\rA)$--module and
$\theta:  \rQ(\bar\rR) \otimes_{\rQ(\bar\rA)} V \to
\rQ(\bar\rR) \otimes_\rR \widetilde\mM$ is an epimorphism
of $\rQ(\bar\rR)$--modules such that the induced morphism
of $\rQ(\bar\rA)$--modules
$$V \lar  \rQ(\bar\rR) \otimes_{\rQ(\bar\rA)} V
\stackrel{\theta}\lar \rQ(\bar\rR) \otimes_\rR \widetilde\mM $$
is a monomorphism. In what follows,  $\theta$ will be frequently called
\emph{gluing map}.

A morphism between two triples $(\widetilde\mM, V, \theta)$
and $(\widetilde\mM', V', \theta')$ is given by a pair $(\varphi, \psi)$, where
$\varphi: \widetilde\mM \to \widetilde\mM'$ is a morphism of $\rR$--modules and
$\psi: V \to V'$ is a morphism of $\rQ(\bar\rA)$--modules such that
the following diagram
$$
\xymatrix
{
\rQ(\bar\rR) \otimes_{\rQ(\bar\rA)} V \ar[rr]^-{\theta} \ar[d]_{\mathbbm{1} \otimes \varphi} & &
\rQ(\bar\rR) \otimes_\rR \widetilde\mM \ar[d]^{\mathbbm{1} \otimes \psi}\\
\rQ(\bar\rR) \otimes_{\rQ(\bar\rA)} V' \ar[rr]^-{\theta'} & &
\rQ(\bar\rR) \otimes_\rR \widetilde\mM'
}
$$
is commutative in the category of $\rQ(\bar\rR)$--modules.
\end{definition}

\begin{remark}
Consider the  pair of functors
\begin{equation}\label{E:pair-of-functr}
\CM(\rR) \xrightarrow{\bar\rR \otimes_\rR \;-} \rQ(\bar\rR)\mathsf{-mod}
\xleftarrow{\rQ(\bar\rR) \otimes_{\rQ(\bar\rA)} \;-} \rQ(\bar\rA)\mathsf{-mod}.
\end{equation}
Then the  category $\Tri(\rA)$ is a full subcategory of the comma--category defined
by  (\ref{E:pair-of-functr}).
\end{remark}

\noindent
The raison d'\^etre for Definition \ref{D:triples} is  the following
theorem.

\begin{theorem}\label{T:BurbanDrozd} In the notations of this section,
the functor
$$
\mathbb{F}: \CM(\rA) \lar \Tri(\rA), \quad
\mM \mapsto \FF(\mM) :=
\bigl(\rR\boxtimes_\rA \mM, \rQ(\bar\rA) \otimes_\rA M,
\theta_\mM\bigr),$$
  is an equivalence of categories.
\end{theorem}

\noindent
Lemma \ref{L:canmorphdim2} assures that the functor $\FF$
is well--defined. The proof of this theorem as well as the construction
of a quasi--inverse functor $\GG$ will be given in the next section.

\medskip
\noindent
 Now we
shall investigate the compatibility of the functor $\FF$ with localizations with respect to
the prime ideals of height $1$.

\begin{proposition}\label{P:prop-of-conduct}
 Let $a(I) := \mathrm{ass}(I) = \bigl\{\idp_1, \idp_2, \dots, \idp_t\bigr\}$
be the associator of the
conductor ideal $I \subseteq \rA$.  Then  we have:
\begin{enumerate}
\item for all $1 \le i \le t$ the ideal $\idp_i$ has height one;
\item Let $\idp \in \kP$. Then $(\rR/\rA)_\idp  = 0$ for all
$\idp \notin a(I)$;
\item Let  $\bar{\idp}_i$ be   the image of $\idp_i$ in the ring $\bar{A}$ for
$1 \le i \le t$.  Then
$$\rQ(\bar\rA) \cong  \bar\rA_{\bar{\idp}_1} \times \dots \times \bar\rA_{\bar{\idp}_t}\;
\mbox{and}\;
 \rQ(\bar\rR) \cong  \bar\rR_{\bar{\idp}_1} \times \dots \times \bar\rR_{\bar{\idp}_t}.$$
 \item Moreover, for any $\idp \in a(I)$ the ring  $\rR_\idp$ is the normalization
 of $\rA_\idp$, $I_\idp$ is the conductor ideal of $\rA_\idp$, $\rQ(\bar\rA)_\idp =
 \bar\rA_{\bar\idp}$ and $\rQ(\bar\rR)_\idp =
 \bar\rR_{\bar\idp}$.
\end{enumerate}
\end{proposition}

\begin{proof}
According to  Lemma \ref{L:prop-of-conduct}, the ring  $\bar\rA$ is a Cohen--Macaulay curve singularity.
Hence, it is equidimensional, what  proves the first statement.

It is well--known that $a(I)$ coincides with
 the set of minimal elements of $\Supp(\bar\rA)$, see \cite{Serre}.
Hence, for any $\idp \in \kP$ we have: $\bar\rA_\idp \ne  0$ if and only if
$\idp \in a(I)$. Since the ring extension $\bar\rA \subseteq \bar\rR$ is finite,
$\bar\rR_\idp \ne  0$ if and only if
$\idp \in a(I)$.
 This proves the second statement.

Next,  $\bar\rA$ is a one--dimensional Cohen--Macaulay ring and the set of its minimal prime
ideals   is $a(0) = \bigl\{\bar{\idp}_1, \dots, \bar{\idp}_t\bigr\}$. Hence, we have:
$\bar\rA_{\bar{\idp}_i} = \rQ(\bar\rA)_{\bar{\idp}_i \rQ(\bar\rA)}$ for all $1 \le i \le t$.
Since  $\bar\rA$ is Cohen--Macaulay, its  total ring of fractions  $\rQ(\bar\rA)$ is artinian.
Moreover,  $\Bigl\{\bar{\idp}_1 \rQ(\bar\rA), \dots, \bar{\idp}_t \rQ(\bar\rA)\Bigr\}$
is the set of maximal ideals of $\rQ(\bar\rA)$.
In particular, the morphism
$$
\rQ(\bar\rA) \lar \rQ(\bar\rA)_{\bar{\idp}_1 \rQ(\bar\rA)} \times \rQ(\bar\rA)_{\bar{\idp}_2 \rQ(\bar\rA)} \times
\dots \times \rQ(\bar\rA)_{\bar{\idp}_t \rQ(\bar\rA)}
\lar
\bar\rA_{\bar{\idp}_1} \times \dots \times \bar\rA_{\bar{\idp}_t}
$$
is an isomorphism. Taking into account Lemma \ref{L:prop-of-conduct}, we obtain an isomorphism
$$
\bar\rR_{\bar{\idp}_1} \times \dots \times \bar\rR_{\bar{\idp}_t}
\lar
\left(\bar\rA_{\bar{\idp}_1} \times \dots \times \bar\rA_{\bar{\idp}_t}\right) \otimes_{\bar\rA}
 \bar\rR \lar
\rQ(\bar\rA) \otimes_{\bar\rA} \bar\rR \lar \rQ(\bar\rR).
$$
This concludes a proof of the third statement.

For any prime ideal $\idp$  the ring $\rR_\idp$ is the normalization of $\rA_\idp$. Next, we have:
$I_\idp = \bigl(\Ann_\rA(\rR/\rA)\bigr)_\idp \cong \Ann_{\rA_\idp}(\rR_\idp/\rA_\idp)$, hence
$I_\idp$ is the conductor ideal of $\rA_{\idp}$. The ring isomorphisms  $\rQ(\bar\rA)_\idp \cong
 \bar\rA_{\bar\idp}$ and $\rQ(\bar\rR)_\idp \cong
 \bar\rR_{\bar\idp}$ follow from the previous part.
\end{proof}

\begin{remark}
For a Cohen--Macaulay curve singularity  $\rC$, there exists the notion
of  the category
of triples $\Tri(\rC)$ parallel to Definition \ref{D:triples}, see
Section  \ref{S:AppendixTriplesDimOne}.
\end{remark}

\begin{proposition}\label{P:Tripl-and-Localiz}
For any prime ideal $\idp \in a(I)$ we have the localization functor
$\LL_\idp: \Tri(\rA) \to \Tri(\rA_\idp)$ mapping a triple
$\mT = (\widetilde\mM, V, \theta)$ to the triple $\mT_\idp = \LL_\idp(\mT) =
(\widetilde\mM_\idp, V_\idp, \theta_\idp)$.  Moreover,
there is the following  diagram of  categories and functors
$$
\xymatrix{
\CM(\rA) \ar[rr]^-{\rA_\idp \otimes_\rA \,-\,} \ar[d]_{\FF^\rA} & & \CM(\rA_\idp)
\ar@{=>}[dll]_\xi \ar[d]^{\FF^{\rA_\idp}} \\
\Tri(\rA) \ar[rr]^-{\LL_\idp} & & \Tri(\rA_\idp)
}
$$
where the natural transformation $\xi: \FF^{\rA_\idp} \circ (\rA_\idp \otimes_\rA \,-\,) \rightarrow  \LL_\idp \circ \FF^\rA$ is an isomorphism.
Moreover, for a triple $\mT = (\widetilde\mM, V, \theta)$ the gluing morphism
$\theta$ is an isomorphism if and only if $\theta_\idp$ is an isomorphism
for all $\idp \in a(I)$.
\end{proposition}

\begin{proof}
Let $\mT = (\widetilde\mM, V, \theta)$ be an object of $\Tri(\rA)$.
By Proposition \ref{P:prop-of-conduct},
for any prime ideal $\idp$
the localization
$I_\idp$ is the conductor ideal of the ring $\rA_\idp$,
$\rQ(\bar\rA) \cong  \bar\rA_{\bar{\idp}_1} \times \dots \times \bar\rA_{\bar{\idp}_t}$ and
 $\rQ(\bar\rR) \cong  \bar\rR_{\bar{\idp}_1} \times \dots \times \bar\rR_{\bar{\idp}_t}$.
 Hence, for any prime ideal $\idp \in a(I)$ we have: $\widetilde\mM_\idp$ is a maximal Cohen--Macaulay
 $\rR_\idp$--module, $V_\idp = V_{\bar\idp}$ is a Noetherian $\bar\rA_{\bar\idp}$--module. We have a commutative diagram
 $$
 \xymatrix
 {
 \bigl(\rQ(\bar\rR)  \otimes_{\rQ(\bar\rA)} V\bigr)_{\bar\idp} \ar[d]_{\cong}
  \ar[rr]^{\theta_{\bar\idp}} & & \bigl(\rQ(\bar\rR)  \otimes_{\rR} \widetilde\mM\bigr)_{\bar\idp} \ar[d]^{\cong}\\
 \bar\rR_{\bar\idp} \otimes_{\bar\rA_{\bar\idp}} V_{\bar\idp} \ar[rr]^{\theta_\idp} & & \bar\rR_{\bar\idp} \otimes_{\rR_\idp} \widetilde\mM_\idp
 }
 $$
where both vertical maps are canonical isomorphisms. In a similar way, we have a commutative diagram
$$
 \xymatrix
 {
 V_{\bar\idp} \ar[d]_{\mathbbm{1}}
  \ar[rr]^-{\tilde\theta_{\bar\idp}} & & \bigl(\rQ(\bar\rR)  \otimes_{\rR} \widetilde\mM\bigr)_{\bar\idp} \ar[d]^{\cong}\\
 V_{\bar\idp} \ar[rr]^-{\tilde\theta_\idp} & & \bar\rR_{\bar\idp} \otimes_{\rR_\idp} \widetilde\mM_\idp
 }
 $$
 and the morphisms $\tilde\theta_\idp$ and $\theta_\idp$ are mapped to each other under the adjunction
 maps.

By Corollary \ref{C:compat-of-Macaulayf},
for any $\idp \in \kP$  and any maximal Cohen--Macaulay $\rA$--module $\mM$ we have an isomorphism
 $(\rR \boxtimes_\rA \mM)^\dagger_\idp \to \rR_\idp \otimes_{\rA_\idp} \mM_\idp/\tor(\rR_\idp \otimes_{\rA_\idp} \mM_\idp)$ which is natural in $\mM$. Moreover, this map coincides
  the localization  $\theta_{\mM_\idp}$ of
 $\theta_\mM$. This shows the claim.
\end{proof}

\noindent
Combining  Theorem \ref{T:BurbanDrozd}
 and Proposition \ref{P:Tripl-and-Localiz}, we obtain
the following result.

\begin{theorem}\label{T:BurbanDrozd-app-on-lfree}
The functor $\FF$ establishes
an equivalence  between $\CM^{\mathsf{lf}}(\rA)$ and
 the full subcategory $\Tri^{\mathsf{lf}}(\rA)$  of $\Tri(\rA)$ consisting
of those triples $(\widetilde\mM, V, \theta)$ for which the gluing morphism
$\theta$ is an isomorphism.
\end{theorem}

\section{Serre quotients and proof of  Main Theorem}
\label{sec4}

The goal of this section is to give a proof of Theorem \ref{T:BurbanDrozd}.
To do that we need  the technique
of Serre quotient categories, studied by Gabriel in his thesis \cite{Gabriel}, see also
\cite{Popescu}.

\begin{definition}
For a Noetherian ring $\rA$ let  $\art(\gA)$ be the category of finite length mo\-dules.
Then $\art(\rA)$ is a thick subcategory, i.e.~it is closed under taking kernels, cokernels
and extensions inside of $\rA\mathsf{-mod}$. The Serre quotient category
$$\MM(\gA) = \gA\mathsf{-mod}/\art(\gA)$$
is defined as follows.

\vspace{1mm}
\noindent
1. The objects  of $\MM(\rA)$ and $\rA\mathsf{-mod}$ are the same.

\vspace{1mm}
\noindent
2. To define morphisms in $\MM(\rA)$,
for any pair of $\rA$--modules $\mM$ and $\mN$ consider the following
partially ordered set of quadruples
$
I_{\mM, \mN} :=
\bigl\{
Q = (\mX, \varphi, \mY, \psi)
\bigr\},
$
where $\mX$ and $\mY$ are $\rA$--modules,  $\mX \stackrel{\varphi}\lar \mM$ is an
injective homomorphism of $\rA$--modules whose cokernel belongs to $\art(\rA)$ and
$\mN  \stackrel{\psi}\lar \mY$ is a surjective  homomorphism of $\rA$--modules whose
kernel belongs to $\art(\rA)$. For a pair of such quadruples $Q =
(\mX, \varphi, \mY, \psi)$ and $Q' = (\mX', \varphi', \mY', \psi')$ we say that
$Q \le Q'$ if any only if there exists morphisms $\mX' \stackrel{\xi}\lar \mX$ and
$\mY \stackrel{\zeta}\lar \mY'$ such that $\varphi' = \varphi \xi$ and
$\psi' = \zeta \psi$.  Then $I_{\mM, \mN}$ is a directed partially ordered set and we define:
$$\Hom_{\MM(\gA)}(M, N) := \varinjlim\limits_{Q \in I_{\mM, \mN}}
\Hom_\rA(\mX, \mY).
$$

\vspace{1mm}
\noindent
3. Note that for any pair of $\rA$--modules $\mM$ and $\mN$ we have a canonical homomorphism
of abelian groups
$p(M, N): \Hom_\rA(\mM, \mN) \lar \varinjlim \Hom_\rA(\mX, \mY) =  \Hom_{\MM(\rA)}(\mM, \mN).$
\end{definition}

\begin{theorem}
The category $\MM(\gA)$ is abelian and the canonical functor
$$\mathbb{P}_{\rA}: \gA\mathsf{-mod} \lar \MM(\gA)$$ is exact. In particular, if $\mM \xrightarrow{\psi} \mN$
is a morphism in $\rA\mathsf{-mod}$ then $\PP_{\rA}(\psi)$ is a monomorphism (resp.~epimorphism)
if and only if the kernel (resp.~cokernel)  of $\psi$ belongs to $\art(\rA)$.

Moreover, $\MM(\gA)$ is equivalent to the localized category $\MM(\rA)^\circ = \gA\mathsf{-mod}[\Sigma^{-1}]$, where the localizing subclass $\Sigma \subset \mathsf{Mor}(\rA)$
 consists of all
 morphisms
in the category $\rA\mathsf{-mod}$, whose kernels and cokernels have finite length.
\end{theorem}

\begin{proof}
The first part of this theorem was shown by Gabriel, see
\cite[Chapitre III]{Gabriel}.
For the second part we refer to
\cite{Popescu}. In particular, for any pair of objects $\mM$ and $\mN$ and a morphism
$\mM \xrightarrow{\psi} \mN$ in the category $\MM(\rA)$ there exists an $\rA$--module
$\mE$ and a pair of morphisms $\mM \xleftarrow{\phi} \mE \xrightarrow{\varphi} \mN$
such that $\ker(\phi)$ and $\coker(\phi)$ belong to  $\art(\rA)$ and
$\psi = \PP_\rA(\varphi) \cdot \PP_\rA(\phi)^{-1}$.
\end{proof}

\noindent
It turns out that the category $\MM(\rA)$ is very natural from the point of view of
 singularity theory. The following theorem summarizes some of its well--known properties.

\begin{theorem}\label{T:prop-of-Serre-quot}
Let $(\gA, \gm)$ be a local Noetherian ring.
\begin{enumerate}
\item\label{it1:prop-of-Serre-quot} If $\rA$ is Cohen--Macaulay of Krull dimension one then the exact functor
$\rQ(\rA) \otimes_\rA \,-\,: \rA\mathsf{-mod} \to \rQ(\rA)\mathsf{-mod}$ induces
an equivalence of categories $\MM(\rA) \to \rQ(\rA)\mathsf{-mod}$;

\item\label{it2:prop-of-Serre-quot} Let
 $X = \Spec(\gA)$ and $x = \{\gm\}$ be the unique closed
point of $X$. For  $U := X \setminus\{x\}$ let   $\imath: U \to X$ be the canonical embedding
and $\Coh_x(X)$ be the category of coherent sheaves on $X$ supported at $x$.
Then the functor $\imath^*$ induces an equivalence of categories $\Coh(X)/\Coh_x(X) \to
\Coh(U)$.  In particular, the categories $\MM(\rA)$ and $\Coh(U)$ are equivalent.

\item\label{it3:prop-of-Serre-quot} Let $\rA$ be of Krull dimension at least two then the canonical functor
$$
\II: \CM(\rA) \lar \rA\mathsf{-mod} \stackrel{\PP_\rA}\lar  \MM(\rA)
$$
is fully faithful. Moreover, if $\rA$ is a normal surface singularity then the category
$\Coh(U)$ is hereditary and
$\CM(\rA)$ is equivalent to the category $\VB(U)$ of locally free coherent sheaves
on $U$.

\item\label{it4:prop-of-Serre-quot}
Let $\rA$ be a reduced Cohen--Macaulay surface singularity
then the Macaulayfication functor $\dagger: \rA\mathsf{-mod} \to \CM(\rA)$ induces a functor
$\MM(\rA) \to \CM(\rA)$ which is left adjoint to the embedding $\II$.
Moreover, for a torsion free $\rA$--module $\mM$ we have a natural isomorphism
$
\mM^\dagger \to   \Gamma(\imath_* \imath^* \widetilde\mM),
$
 where $\widetilde\mM$
is the coherent sheaf on $X$ obtained by sheafifying the module $\mM$.
\end{enumerate}
\end{theorem}

\begin{proof} (\ref{it1:prop-of-Serre-quot}) Let $\rA\mathsf{-Mod}$ be the category of all
$\rA$--modules and $\Art(\rA)$ be its full subcategory consisting of those modules, for  which
any element is $\idm$--torsion. In other words, $\Art(\rA)$ is the category of modules, which
are direct limits of its finite length submodules.

The total ring of fractions $Q(\gA)$ is flat as an $\gA$--module, hence
$\FF= Q(\gA)\otimes_\gA : \gA\mathsf{-Mod} \to  Q(\gA)\mathsf{-Mod}$ is exact.
The forgetful functor $\GG: Q(\gA)\mathsf{-Mod} \to  \gA\mathsf{-Mod}$ is right
adjoint to $\FF$.
Now note that the counit  of the adjunction  $\xi: \FF \GG \to  \mathbbm{1}_{\rQ(\rA)\mathsf{-mod}}$
is an isomorphism of functors. Since $\FF$ is right exact and $\GG$ is exact,
the composition $\FF \GG$ is right exact. Moreover, $\FF \GG$ commutes with arbitrary direct products.
Hence, to prove that $\xi$  is an isomorphism, it is sufficient to show
that the canonical morphism of $\rQ(\rA)$--modules
$$\xi_{\rQ(\rA)} = \mathsf{mult}: \;
 \rQ(\rA) \otimes_\rA \rQ(\rA) \lar \rQ(\rA)$$
 is an isomorphism, which is a basic property of localization.

Since $\rA$ is a Cohen--Macaulay ring of Krull dimension one, the category
$\sT = \ker(\FF)$  is equal to $\Art(\rA)$.
Let $\widehat{\MM}(\rA) = \rA\mathsf{-Mod}/\Art(\rA)$ (one can consult
\cite{Popescu} for the definition of the  Serre quotients categories in the case they are not small).
By \cite[Proposition III.2.4]{Gabriel}
the functor $\FF$ induces an equivalence of categories $\bar\FF: \widehat{\MM}(\rA) \to
\rQ(\rA)\mathsf{-Mod}$.

It is clear that $\Art(\rA)  \cap \gA\mathsf{-mod} = \art(\gA)$, hence basic properties
of Serre quotients  imply that the functor given by the composition
$$
\gA\mathsf{-mod}/\art(\gA) \lar \gA\mathsf{-Mod}/\Art(\rA)
\stackrel{\bar\FF}\lar Q(\gA)\mathsf{-Mod}
$$
is fully faithful. Since $\rQ(\gA) = \bar\FF(\rA)$ and $\bar\FF: \End_{\MM(\rA)}(\rA)  \to
\rQ(\rA)$ is an isomorphism of rings,
the functor $\bar\FF: \MM(\rA) \lar \rQ(\gA)\mathsf{-mod}$ is essentially surjective.

\vspace{2mm}

\noindent
(\ref{it2:prop-of-Serre-quot}) The proof of this statement is similar to the previous one.
The functor $\imath^*: \QCoh(X) \to \QCoh(U)$ has a right adjoint
$\imath_*: \QCoh(U) \to \QCoh(X)$ and  the counit of the adjunction
$\imath^* \imath_* \to  \mathbbm{1}_{\QCoh(U)}$ is an isomorphism.
It is easy to see
that the kernel  of the functor  $\imath^*$  is the category $\QCoh_x(X)$ consisting  of the
quasi-coherent sheaves on $X$  supported at
the closed point $x$. Again, by  \cite[Proposition III.2.4]{Gabriel} the inverse image functor $\imath^*$ induces an equivalence of categories
$\QCoh(X)/\QCoh_x(X) \to \QCoh(U)$. This functor restricts to a fully faithful
functor $\Coh(X)/\Coh_x(X) \to \Coh(U)$. It remains to verify that this functor is essentially surjective.

Let $\kF$ be a coherent  sheaf on $U$, then the direct image sheaf
$\kG := \imath_*\kF$  is quasi-coherent. However, any quasi-coherent sheaf on a Noetherian scheme
can be written  as  a direct limit of an increasing  sequence of coherent subsheaves
$\kG_1 \subseteq \kG_2 \subseteq \dots \subseteq \kG$. Since the functor
$\imath^*$ is exact, we obtain an increasing filtration
$\imath^*\kG_1  \subseteq \imath^*\kG_1  \subseteq \dots \subseteq
\imath^*\kG$. But $\imath^*\kG = \imath^* \imath_*\kF \cong \kF$. Since the scheme $U$ is Noetherian and $\kF$ is coherent,
it implies that $\kF \cong \imath^*\kG_t$ for some  $t \ge 1$. Hence, the functor
$\imath^*: \Coh(X) \to \Coh(U)$ is essentially surjective and the induced functor
$\Coh(X)/\Coh_x(X) \to \Coh(U)$ is an equivalence of categories.

\vspace{2mm}

\noindent
(\ref{it3:prop-of-Serre-quot}) The fact that the functor
$\II: \CM(\rA) \to \MM(\rA)$ is fully faithful, follows for example from
\cite[Lemme III.2.1]{Gabriel}. It is well--known that for a normal surface singularity  $\rA$
 the category
$\Coh(U)$ is hereditary. A proof of  the equivalence between
$\CM(\rA)$ and $\VB(U)$ can be found for instance in \cite[Corollary 3.12]{SurvOnCM}.
Note that if $\rA$ is an algebra over $\mathbb{C}$, the space $U$ is homotopic to the link
of the singularity $\Spec(A)$.

\vspace{2mm}

\noindent
(\ref{it4:prop-of-Serre-quot})
Let $\rA$ be a reduced Cohen--Macaulay surface singularity.
From  \cite[Lemma 3.6]{SurvOnCM}
we obtain that  $\dagger: \rA\mathsf{-mod} \to \CM(\rA)$ induces the functor
$\MM(\rA)^\circ \to \CM(\rA)$, which for  sake of simplicity will be denoted by the same
symbol  $\dagger$. Moreover, for any Noetherian $\rA$--module $\mM$ and a Cohen--Macaulay
$\rA$--module $\mN$ we have isomorphisms
$$
\Hom_{\MM(\rA)}(\mM, \mN) \xleftarrow{\PP_\rA} \Hom_\rA(\mM, \mN) \lar
\Hom_{\CM(\rA)}(\mM^\dagger, \mN),
$$
which are natural in both arguments. For a proof of the isomorphism
$\mM^\dagger \lar  \Gamma(\imath_* \imath^* \widetilde\mM),$  we refer to \cite[Proposition 3.10]{SurvOnCM}.
\end{proof}

\begin{lemma}\label{L:small-det-onSerrQuot}
Let $\rA \subseteq \rB$ be a finite extension of Noetherian rings. Then
the forgetful functor $\mathsf{for}: \rB\mathsf{-mod} \to \rA\mathsf{-mod}$ and the functor
$\rB \otimes_\rA \,-\,: \rA\mathsf{-mod} \to \rB\mathsf{-mod}$ form an adjoint pair
and induce the functors
$$
\mathsf{for}: \MM(\rB) \lar \MM(\rA) \quad \mathrm{and} \quad  \rB \bar\otimes_\rA \,-\,:
\MM(\rA) \lar \MM(\rB)
$$
which are again adjoint. Moreover, for an arbitrary $\rA$--module
$\mX$ and a $\rB$--module $\mY$ the following diagram is commutative:
$$
\xymatrix
{
\Hom_\rA(\mX, \mY) \ar[rr]^-{\mathsf{can}} \ar[d]_{\PP_\rA} & & \Hom_\rB(\rB \otimes_\rA \mX, \mY) \ar[d]^{\PP_\rB}\\
\Hom_{\MM(\rA)}(\mX, \mY) \ar[rr]^-{\mathsf{can}} & & \Hom_{\MM(\rB)}(\rB \bar\otimes_\rA \mX, \mY)
}
$$
where both horizontal maps are canonical isomorphisms given by adjunction.
\end{lemma}

\begin{proof}
Since the ring extension $\rA \subseteq \rB$ is finite, the functor
$\rB \otimes_\rA \,-\,$ maps the category $\art(\rA)$ to $\art(\rB)$. The functors
$\FF = \rB \bar\otimes_\rA \,-\,: \MM(\rA)^\circ \to \MM(\rB)^\circ$ and
$\GG:  \MM(\rB)^\circ \to \MM(\rA)^\circ$ are obtained from the
 adjoint pair of functors $\rB \otimes_\rA \,-\,$ and $\mathsf{for}$
using the
universal property of the localization:
$$
\xymatrix
{
\rA\mathsf{-mod} \ar[rr]^-{\rB \otimes_\rA \,-\,}  \ar[d]_{\PP_\rA^\circ} & &
\rB\mathsf{-mod}  \ar[d]^{\PP_\rB^\circ} \\
\MM(\rA)^\circ \ar@{.>}[rr]^-{\FF} & & \MM(\rB)^\circ
} \qquad
\xymatrix
{
\rB\mathsf{-mod} \ar[rr]^-{\mathsf{for}}  \ar[d]_{\PP_\rB^\circ} & &
\rA\mathsf{-mod}  \ar[d]^{\PP_\rA^\circ} \\
\MM(\rB)^\circ \ar@{.>}[rr]^-{\GG} & & \MM(\rA)^\circ.
}
$$
For an $\rA$--module $\mM$ let $\xi_\mM: \mM \to \rB \otimes_\rA \mM$ be the unit of adjunction. Let $\psi: \mM \to \mN$ be a morphism in $\MM(\rA)^\circ$ represented by
the pair of $\mM \xleftarrow{\phi} \mE \xrightarrow{\varphi} \mN$, where
$\ker(\phi)$ and $\coker(\phi)$ have finite length. Since the diagram
$$
\xymatrix
{
\mM \ar[d]_{\xi_\mM} & \mE \ar[d]_{\xi_\mE} \ar[l]_-{\phi} \ar[r]^-{\varphi} & \mN \ar[d]^{\xi_\mN}\\
\rB\otimes_\rA \mM  & \rB \otimes_\rA \mE  \ar[l]_-{\mathbbm{1}\otimes \phi} \ar[r]^-{\mathbbm{1}\otimes \varphi} & \rB \otimes_\rA \mN
}
$$
is commutative, we get a natural transformation of functors
$\xi: \mathbbm{1}_{\MM(\rA)^\circ} \to \GG \, \FF$. In the similar way, we
construct a natural transformation $\zeta: \FF \, \GG  \to \mathbbm{1}_{\MM(\rB)^\circ}$.
Note that the natural transformations
$$
\FF \xrightarrow{\FF(\xi)} \FF \, \GG \,  \FF \xrightarrow{\zeta \; \FF} \FF \qquad \mbox{and}
\qquad
\GG \xrightarrow{\xi \;  \GG} \GG \, \FF \, \GG \xrightarrow{\GG(\zeta)} \GG
$$
are $\mathbbm{1}_{\FF}$ and $\mathbbm{1}_{\GG}$, respectively.
Hence, $(\FF, \GG)$ is an adjoint pair of functors.
\end{proof}

Now we possess all necessary ingredients to formulate an alternative definition of the
category of triples $\Tri(\rA)$, given in  Definition \ref{D:triples}.
Note that we have a pair of functors
\begin{equation}\label{E:input-for-comma}
\CM(\rR)  \xrightarrow{\bar\rR \bar\otimes_\rR \,-\,} \MM(\bar\rR) \xleftarrow{\bar\rR \bar\otimes_{\bar\rA} \,-\,}
\MM(\bar\rA).
\end{equation}

\begin{definition}\label{D:newtriples}
The category $\Tri'(\rA)$ is the following full subcategory of the comma category
defined by the diagram  (\ref{E:input-for-comma}). Its objects
are triples $(\widetilde\mM, V, \theta)$, where $\widetilde\mM$
is a maximal Cohen--Macaulay $\rR$--module, $V$ an object of
$\MM(\bar\rA)$   and
$\theta:  \bar\rR \bar\otimes_{\bar\rA} V \to
\bar\rR \bar\otimes_\rR \widetilde\mM$ is an epimorphism
in $\MM(\bar\rR)$  such that the adjoint  morphism
in $\MM(\bar\rA)$
$$V \lar  \bar\rR \bar\otimes_{\bar\rA} V
\stackrel{\theta}\lar \bar\rR \bar\otimes_\rR \widetilde\mM $$
is an monomorphism.

A morphism between two triples $(\widetilde\mM, V, \theta)$
and $(\widetilde\mM', V', \theta')$ is given by a pair $(\varphi, \psi)$, where
$\varphi: \widetilde\mM \to \widetilde\mM'$ is a morphism in $\CM(\rR)$  and
$\psi: V \to V'$ is a morphism in $\MM(\bar\rA)$ such that
the following diagram
$$
\xymatrix
{
\bar\rR \bar\otimes_{\bar\rA} V \ar[rr]^-{\theta} \ar[d]_{\mathbbm{1} \bar\otimes \psi} & &
\bar\rR \bar\otimes_\rR \widetilde\mM \ar[d]^{\mathbbm{1} \bar\otimes \varphi}\\
\bar\rR \bar\otimes_{\bar\rA} V' \ar[rr]^-{\theta'} & &
\bar\rR \bar\otimes_\rR \widetilde\mM'
}
$$
is commutative in the category $\MM(\bar\rR)$. \qed
\end{definition}

Recall that for a maximal Cohen--Macaulay module $\mM$ we denote
$\widetilde\mM := \rR \boxtimes_\rA \mM$, whereas $\theta_\mM$ is the canonical morphism
of $\rR$--modules given by the composition
$$
\bar\rR \otimes_{\bar\rA} \bar\rA \otimes_\rA \mM
\stackrel{\cong}\lar \bar\rR \otimes_{\rR} \rR\otimes_\rA \mM \xrightarrow{\mathbbm{1} \otimes \delta}
\bar\rR \otimes_{\rR} (\rR \boxtimes_\rA \mM).
$$
By Theorem  \ref{P:Macaulafic}, the canonical morphism $\rR \otimes_\rA \mM
 \stackrel{\delta}\lar  \rR \boxtimes_\rA \mM$ has
cokernel of finite length, hence $\theta_\mM$ has finite length cokernel as well.
This implies that the morphism
$$
\PP_{\bar\rR}(\theta_\mM): \bar\rR \otimes_{\bar\rA} \bar\rA \otimes_\rA \mM
\lar
\bar\rR \otimes_{\rR} (\rR \boxtimes_\rA \mM)
$$
is an epimorphism in $\MM(\bar\rR)$. Next, we have the following commutative diagram
in the category of $\rA$--modules:
\begin{equation}\label{E:some-not-for-trpls}
\begin{array}{c}
\xymatrix
{ 0 \ar[r] & I \mM  \ar[r] \ar[d]_{\bar{\jmath}} & M \ar[r]^-\pi \ar[d]_{\jmath} & \bar\rA \otimes_\rA
\mM
\ar[d]^{\tilde{\theta}_\mM} \ar[r] & 0 \\
0 \ar[r] & I \widetilde\mM  \ar[r] & \widetilde\mM \ar[r]^-\gamma & \bar\rR \otimes_\rR
\widetilde\mM
\ar[r] & 0, }
\end{array}
\end{equation}
where $\widetilde\mM = \rR \boxtimes_\rA \mM$,
$\jmath: \mM \to \widetilde\mM$ is the canonical morphism and
$\bar{\jmath}$ is its restriction on $I\mM$. The morphism $\jmath$ is injective.
Moreover, for any $\idp \in \kP$ the
morphism
$
\bar{\jmath}_\idp: (I \mM)_\idp \lar (I \widetilde\mM)_\idp
$
is an isomorphism, see the proof of Lemma \ref{L:prep-in-dim-one}.
 Hence, $\coker(\bar{\jmath})$ is an $\rA$--module of finite length.
Snake lemma implies that $\ker(\tilde\theta)$ is a submodule of $\coker(\bar{\jmath})$.
Hence, it has finite length,  too. By Lemma \ref{L:small-det-onSerrQuot}, the morphisms
$\PP_{\bar\rR}(\theta_\mM)$ and $\PP_{\bar\rA}(\tilde{\theta}_\mM)$ are mapped to each
other under the morphisms of adjunction. This yields the following corollary.

\begin{corollary}\label{C:twoCatTri}
We have a functor  $\FF': \CM(\rA) \to \Tri'(\rA)$ assigning to a maximal Cohen--Macaulay
$\rA$--module $\mM$ the triple $\bigl(\rR \boxtimes_\rA \mM, \; \bar\rA \bar\otimes_\rA \mM, \;
\PP_{\bar\rR}(\theta_\mM)\bigr)$.
Moreover, the equivalences of categories $\MM(\bar\rA) \to \rQ(\bar\rA)\mathsf{-mod}$ and
$\MM(\bar\rR) \to \rQ(\bar\rR)\mathsf{-mod}$ constructed in Theorem \ref{T:prop-of-Serre-quot}
induce an equivalence of categories $\EE: \Tri'(\rA) \to \Tri(\rA)$ such that
the functors $\FF$ and $\EE \, \FF'$ are isomorphic.
\end{corollary}

\begin{definition}
Consider the functor $\BB: \Tri'(\rA) \to \MM(\rA)$ defined as follows. For an object
$\mT = (\widetilde\mM, V, \theta)$ of the category $\Tri'(\rA)$ let
$\widehat\mM := \bar\rR \otimes_\rR  \widetilde\mM$ and $\gamma: \widetilde\mM
\to \widehat\mM$ be the canonical morphism of $\rR$--modules. Let
$\widetilde\mM \xrightarrow{\bar\gamma} \widehat\mM$ be the morphism
in $\MM(\rA)$ obtained by applying to $\gamma$ the functor $\PP_{\rR}$ and then
the forgetful functor $\MM(\rR) \to \MM(\rA)$.  Then we set
$$
\mN:= \BB(\mT) = \ker\bigl(\; \widetilde\mM \oplus V \xrightarrow{(\begin{smallmatrix}
 \bar\gamma \; \tilde\theta
 \end{smallmatrix})} \widehat\mM \;\bigr)
$$
and define $\BB$ on morphisms using the universal property of a kernel.
Equivalently, we have a commutative diagram in the category $\MM(\rA)$:
\begin{equation}\label{E:functorG}
\begin{array}{c}
\xymatrix
{ 0 \ar[r] & I \widetilde\mM  \ar[r]^\alpha \ar[d]_{=} & \mN \ar[r]^\pi \ar[d]_\imath   & V
\ar[d]^{\tilde{\theta}} \ar[r] & 0 \\
0 \ar[r] & I \widetilde\mM  \ar[r]^{\bar\beta} & \widetilde\mM \ar[r]^{\bar\gamma} & \widehat\mM
\ar[r] & 0.
}
\end{array}
\end{equation}
\end{definition}

\medskip
\noindent
According to Corollary \ref{C:twoCatTri},  Theorem \ref{T:BurbanDrozd} is equivalent  to the following statement.

\begin{theorem}\label{P:MainProp}
The functor $\GG: \Tri'(\rA) \to \CM(\rA)$ given by the composition
of the functors $\BB: \Tri'(\rA) \to \MM(\rA)$ and $\dagger: \MM(\rA) \to \CM(\rA)$, is quasi-inverse to
$\FF'$.
\end{theorem}

\begin{proof} Before going to the details, let us first explain the logic of our  proof.
\begin{itemize}
\item We construct an isomorphism of functors $
\mathbbm{1}_{\CM(\rA)} \lar \GG \circ \FF'.
$
\item We show  that $\GG$ is faithful.
\item Finally, we prove that any triple $\mT \in \Ob\bigl(\Tri'(\rA)\bigr)$ is isomorphic to
$\FF' \GG (\mT)$.
\end{itemize}
The first two statements imply that  $\FF'$ is  fully faithful. The last one shows that $\FF'$  is
essentially surjective. Hence, $\FF'$ is an equivalence of categories and $\GG$ is its quasi--inverse.

Now, let  $\mM$ be a maximal Cohen--Macaulay $\rA$--module.
In the notations of the commutative diagram (\ref{E:some-not-for-trpls}), we have the following
exact sequence in the category of $\rA$--modules:
\begin{equation}\label{E:defining-quis-inv}
\mM \xrightarrow{
\left(
\begin{smallmatrix}
- \jmath \\
\; \pi
\end{smallmatrix}
\right)}
\widetilde\mM \oplus \overline\mM \xrightarrow{
\left(
\begin{smallmatrix}
\gamma & \tilde\theta
\end{smallmatrix}
\right)} \widehat\mM \lar 0,
\end{equation}
where $\overline\mM = \bar\rA \otimes_{\rA} \mM$.
Since $\PP_{\rA}(\bar\jmath)$ is an isomorphism in $\MM(\rA)$, the
image of the sequence (\ref{E:defining-quis-inv}) under the functor
$\PP_{\rA}$ becomes short exact.
The morphism $\mM
\xrightarrow{
\left(
\begin{smallmatrix}
- \jmath \\
\pi
\end{smallmatrix}
\right)}
\widetilde\mM \oplus \overline\mM$ is natural in the category of $\rA$--modules, thus it is natural
in $\MM(\rA)$ as well. Hence,  we obtain an isomorphism
of functors
$
\mathbbm{1}_{\CM(\rA)} \lar \GG \circ \FF'.
$
This shows  that  $\FF'$ is faithful.

Next, we prove  that $\GG$ is faithful, too.
 Let $\mT = (\widetilde\mM, V, \theta)$ and $\mT' =
(\widetilde\mM', V', \theta')$
be a pair of objects in $\Tri'(\rA)$ and $\mT \xrightarrow{(\varphi, \psi)} \mT'$ be a morphism
in $\Tri'(\rA)$.  Let $\mM = \BB(T), \mM' = \BB(T')$ and $\phi = \BB\bigl((\varphi, \psi)\bigr)$.
Then we have a commutative diagram in the category $\MM(\rA)$:
\begin{equation}\label{E:proof-of-main}
\begin{array}{c}
\xymatrix{
0 \ar[r] & \mM \ar[r]  \ar[d]_{\phi} & \widetilde\mM \oplus V \ar[r]^-{
\left(\begin{smallmatrix} \gamma & \tilde\theta \end{smallmatrix}\right)} \ar[d]_{\left(\begin{smallmatrix} \varphi & 0 \\ 0 & \psi \end{smallmatrix}\right)} & \widehat\mM \ar[r]
\ar[d]^{\widehat\varphi} & 0 \\
0 \ar[r] & \mM' \ar[r]         & \widetilde\mM' \oplus V' \ar[r]^-{\left(\begin{smallmatrix} \gamma' &
\tilde\theta' \end{smallmatrix}\right)}        & \widehat\mM' \ar[r] & 0.
}
\end{array}
\end{equation}
First note that $(\varphi, \psi) = 0$ in $\Tri'(\rA)$ if and only if $\varphi = 0$. Indeed, one direction is obvious. To show the second, let  $\varphi = 0$.  Then $\widehat\varphi = 0$ and $\tilde\theta' \circ \psi = 0$. It remains to note that
$\tilde\theta'$ is a monomorphism.

Next, a morphism $\varphi: \widetilde\mM \to \widetilde\mM'$ is zero in $\CM(\rR)$ if and only if
$\mathbbm{1} \otimes \varphi: \rQ(\rA) \otimes_\rA \widetilde\mM \to \rQ(\rA) \otimes_\rA \widetilde\mM'$
is zero in $\rQ(\rA)-\mod$.
Assume the morphism of triples $(\varphi, \psi): \mT \to \mT'$ is non-zero.
  Apply the functor $\rQ(\rA) \otimes_\rA \,-\,$ on the
diagram (\ref{E:proof-of-main}). It follows that $\mathbbm{1} \otimes \phi \ne 0$, hence
$\GG\bigl((\varphi, \psi)\bigr) \ne 0$ as well. Hence, $\GG$ is faithful. From the isomorphism
of functors $
\mathbbm{1}_{\CM(\rA)} \lar \GG \circ \FF'
$ it follows that $\FF'$ is full.

The difficult part of the proof is to show that $\FF'$ is essentially surjective.
It is sufficient to show
that for an arbitrary triple $\mT = (\widetilde\mM, V, \theta)$ there
exists  an isomorphism $\mT \cong \FF' \GG(\mT)$ in the category $\Tri'(\rA)$.  We split our
arguments into  several logical steps.

\vspace{1mm}
\noindent
\underline{Step 1}. Since $\bar\rA$ is a Cohen--Macaulay ring of Krull dimension one,
the kernel $\tor(V)$ of the canonical map $V \to \rQ(\bar\rA) \otimes_{\bar\rA} V$ is annihilated by some power of the maximal ideal. Hence, the canonical map
$V \xrightarrow{\nu}  V/\tor(V) =: V'$  is an isomorphism in the category $\MM(\bar\rA)$.
We get  the following  isomorphism in the category $\Tri'(\rA)$
$$
(\mathbbm{1}, \nu): (\widetilde\mM, V, \theta) \lar (\widetilde\mM, V', \theta'),
$$
where the morphism $\theta'$ is induced by $\nu$.
Hence, we may without  loss of generality assume that the object $V$  of the category $\MM(\bar\rA)$
is represented by a maximal Cohen--Macaulay $\bar\rA$--module.

\vspace{1mm}
\noindent
\underline{Step 2}. For a maximal Cohen--Macaulay $\rR$--module $\widetilde\mM$ consider the following
commutative diagram in the category of $\rR$--modules:
$$
\xymatrix
{
0 \ar[r] & I \widetilde\mM \ar[r]^-\beta \ar[d]_\delta & \widetilde\mM \ar[r]^-\gamma \ar[d]^= & \widehat\mM \ar[r] \ar[d]^\rho & 0 \\
0 \ar[r] & (I \widetilde\mM)^\dagger \ar[r]^-{\beta^\circ}        & \widetilde\mM
\ar[r]^-{\gamma^\circ}  & \widehat\mM^\circ \ar[r]  & 0,
}
$$
where $I \widetilde\mM \xrightarrow{\delta} (I \widetilde\mM)^\dagger$ is the canonical morphism
determined by the Macaulayfication functor. Hence, $\coker(\delta)$ is an $\rR$--module
of finite length.
Snake lemma yields  that $\rho$ is a surjective
morphism of $\rR$--modules and $\ker(\rho) \cong \coker(\delta)$. In particular,
$\widehat\mM^\circ$ is annihilated by the conductor ideal $I$, hence it is an $\bar\rR$--module.
Depth Lemma implies that $\depth_{\rR}(\widehat\mM^\circ) = \depth_{\bar\rR}(\widehat\mM^\circ) = 1$, hence $\widehat\mM^\circ$
is maximal Cohen--Macaulay over $\bar\rR$.
Moreover, the morphism
$\bar\rho:= \PP_{\bar\rR}(\rho): \widehat\mM \to \widehat\mM^\circ$ is an isomorphism in
$\MM(\bar\rR)$.

\vspace{1mm}
\noindent
\underline{Step 3}. In the notations as above we have the following isomorphism in
the category $\MM(\rA)$:
$$
\ker\bigl(\; \widetilde\mM \oplus V \xrightarrow{(\begin{smallmatrix}
 \bar\gamma \; \tilde\theta
 \end{smallmatrix})} \widehat\mM \;\bigr)  \cong
 \ker\bigl(\; \widetilde\mM \oplus V \xrightarrow{(\begin{smallmatrix}
 \bar\gamma^\circ  \; \tilde\theta^\circ
 \end{smallmatrix})} \widehat\mM^\circ \,\bigr),
$$
where $\tilde\theta^\circ = \bar\rho \; \tilde\theta: V \to \widehat\mM^\circ$.
Since $\widehat\mM^\circ$ is a maximal Cohen--Macaulay $\bar\rR$--module, it is also maximal Cohen--Macaulay
over $\bar\rA$. In  particular, it has no $\bar\rA$--submodules of finite length. From the definition
of the category $\MM(\rA)$ it follows that  $\tilde\theta^\circ$ can be  written
as
$$
\tilde\theta^\circ = \PP_\rA(\underline{\tilde\theta}) \cdot  \PP_\rA(\tau)^{-1}, \quad
 V \stackrel{\tau}\longleftarrow V' \stackrel{\underline{\tilde\theta}}\lar \widehat\mM^\circ,
$$
where
$\tau: V' \to V$ is a monomorphism of $\bar\rA$--modules with cokernel of finite length
and $ \underline{\tilde\theta}$ is
a morphism of $\bar\rA$-modules. Since we have assumed $V$ to be maximal Cohen--Macaulay
over $\bar\rA$, its submodule $V'$ is maximal Cohen--Macaulay over $\bar\rA$ as well. Next,
$\tilde\theta^\circ$ is a monomorphism in $\MM(\bar\rA)$, hence the kernel
of $ \underline{\tilde\theta}$ has finite length. But $\ker(\underline{\tilde\theta})$ is
a submodule of a maximal Cohen--Macaulay
$\bar\rA$--module
$V'$. Hence, $ \underline{\tilde\theta}$ is a monomorphism of $\bar\rA$--modules.

\vspace{1mm}
\noindent
Identifying $V$ and $V'$ in the category $\MM(\bar\rA)$ we may without  loss of generality assume:

\begin{itemize}
\item In the triple $T = (\widetilde\mM, V, \theta)$,  the module $V$ is Cohen--Macaulay over
$\bar\rA$ and the  morphism $\tilde\theta^\circ: V \to \widehat\mM^\circ$ in
$\MM(\bar\rA)$ is the image of an injective morphism of $\bar\rA$--modules under the functor
$\PP_{\bar\rA}$. For  sake of simplicity, we denote the latter morphism
by the same letter $\tilde\theta^\circ$.

\item The object $\mN = \BB(\mT) \in \mathrm{Ob}\bigl(\MM(\rA)\bigr)$ can be obtained by applying
$\PP_{\rA}$ to the middle term of the upper  short exact sequence in the following diagram
in $\rA\mathsf{-mod}$:
\begin{equation}\label{E:some-interm-diagr}
\begin{array}{c}
\xymatrix
{
0 \ar[r] & (I \widetilde\mM)^\dagger \ar[r]^-{\alpha} \ar[d]_=& \mN  \ar[r]^-\pi \ar[d]^\imath & V \ar[r]
\ar[d]^{\tilde\theta^\circ} & 0 \\
0 \ar[r] & (I \widetilde\mM)^\dagger \ar[r]^-{\beta^\circ}        & \widetilde\mM
\ar[r]^-{\gamma^\circ}  & \widehat\mM^\circ \ar[r]  & 0.
}
\end{array}
\end{equation}
\end{itemize}
Since $\tilde\theta^\circ$ is injective in $\bar\rA\mathsf{-mod}$, snake lemma yields
that $\mN$ is a torsion free $\rA$--module.

\vspace{1mm}
\noindent
\underline{Step 4}. In the notations of the commutative diagram (\ref{E:some-interm-diagr}),
consider the canonical morphism $\mN \xrightarrow{\delta} \mN^\dagger$. Then we obtain the following commutative diagram
 in $\rA\mathsf{-mod}$:
$$
\xymatrix
{
0 \ar[r] & (I \widetilde\mM)^\dagger \ar[r]^-{\alpha} \ar[d]_=& \mN  \ar[r]^-\pi \ar[d]^\delta & V \ar[r]
\ar[d]^{\delta'} & 0 \\
0 \ar[r] & (I \widetilde\mM)^\dagger \ar[r]^-{\alpha^\dagger} & \mN^\dagger
  \ar[r]^-{\pi'}  & W \ar[r]  & 0,
}
$$
where $\pi'$ and $\delta'$ are induced morphisms.
Since $\mN$ is a torsion free $\rA$--module, $\delta$ is injective and its cokernel has finite length,
see Theorem  \ref{P:Macaulafic}.
Snake lemma implies that $\coker(\delta')$ has finite length, too.
Moreover, the universal property of Macaulayfication implies there exists an injective
 morphism  of $\rA$--modules $\mN^\dagger \stackrel{\jmath}\lar \widetilde\mM$ such that
$\jmath \; \delta = \imath$. In particular, we have:
$
\jmath \; \alpha^\dagger = \imath \; \alpha = \beta^\circ
$
and the following diagram
$$
\xymatrix
{
0 \ar[r] & (I \widetilde\mM)^\dagger \ar[d]_= \ar[r]^-{\alpha^\dagger} & \mN^\dagger
  \ar[r]^-{\pi'}  \ar[d]^\jmath & W \ar[r] \ar[d]^{\jmath'}  & 0 \\
 0 \ar[r] & (I \widetilde\mM)^\dagger \ar[r]^-{\beta^\circ}        & \widetilde\mM
\ar[r]^-{\gamma^\circ}  & \widehat\mM^\circ \ar[r]  & 0,
}
$$
commutes in $\rA\mathsf{-mod}$,
where $\jmath'$ is the  morphism induced by $\jmath$. Since $\jmath$ is injective,
the morphism $\jmath'$ is injective as well. Hence, the $\rA$--module $W$ is annihilated by
the conductor ideal $I$. Thus, it is a maximal Cohen--Macaulay  $\bar\rA$--module and the morphism
$\PP_{\bar\rA}(\delta'): V \to W$ is an isomorphism in $\MM(\bar\rA)$.

\vspace{1mm}
\noindent
\underline{Step 5}.
 In other words, we
have shown that any object $\mT$ of the category $\Tri'(\rA)$ has a representative
$(\widetilde\mM, V, \theta)$ such that $V$ is maximal Cohen--Macaulay, the morphism
$V \xrightarrow{\tilde\theta^\circ} \widehat\mM^\circ$ belongs to the image of the functor $\PP_\rA$ and
the module $\mN$ given by the diagram (\ref{E:some-interm-diagr}) is maximal Cohen--Macaulay over
$\rA$. By the definition of the functor $\GG$, we have: $\mN \cong \GG(\mT)$. It remains to find
an isomorphism between the triples $\FF'(\mN)$ and $\mT$.

Let $I \widetilde\mM \xrightarrow{\delta} (I \widetilde\mM)^\dagger$ be the canonical morphism
and $\imath': I\mM \to (I \widetilde\mM)^\dagger$ be the composition of the restriction
of $\imath$ on $I\mN$ with $\delta$. Since $\imath$ is injective, it is easy to see
that the following diagram is commutative:
$$
\xymatrix
{
0 \ar[r] & I \mN \ar[r]^{\beta_\mN} \ar[d]_{\imath'}
& \mN \ar[d]_\imath \ar[r]^{\pi_\mN} & \overline\mN \ar[d]^\kappa \ar[r] & 0 \\
0 \ar[r] & (I \widetilde\mM)^\dagger \ar[r]^-{\alpha} & \mN  \ar[r]^-\pi & V \ar[r]
 & 0.
}
$$
By Lemma \ref{L:prep-in-dim-one}, the morphism $\imath'$ is  injective
and its cokernel has finite length. Since $\ker(\kappa)$ is a subobject of $\coker(\imath')$
the morphism $\PP_{\bar\rA}(\kappa): \overline\mN \to V$ is an isomorphism in
$\MM(\bar\rA)$.
Next, the morphism of maximal Cohen--Macaulay $\rA$--modules $\imath: \mN \to \widetilde\mM$ induces
a morphism of maximal Cohen--Macaulay $\rR$--modules
$\tilde\imath: \rR \boxtimes_\rA \mN \to \widetilde\mM$. Theorem \ref{T:BurbanDrozdindimOne} implies
 that $\tilde\imath_\idp: (\rR \boxtimes_\rA \mN)_\idp
 \to \widetilde\mM_\idp$ is an isomorphism of $\rA_\idp$--modules for all prime ideals
 $\idp \in \kP$. Hence, $\tilde\imath$ is an isomorphism in $\CM(\rR)$.

\vspace{1mm}
\noindent
\underline{Step 6}.
It remains to observe  that
$
\bigl(\tilde\imath, \PP_{\bar\rA}(\kappa)\bigr): \FF'(\mN) \to  (\widetilde\mM, V, \theta)
$
is an isomorphism in the category of triples $\Tri'(\rA)$. Since both
morphisms  $\tilde\imath$  and $\PP_{\bar\rA}(\kappa)$ are known to be isomorphisms, it is
sufficient to show that $\bigl(\tilde\imath, \PP_{\bar\rA}(\kappa)\bigr)$ is a morphism
in $\Tri'(\rA)$. In the notations of the commutative diagram (\ref{E:some-interm-diagr}),
this fact  follows from the commutativity of the following diagram in the category
$\rA\mathsf{-mod}$:
$$
\xymatrix
{
\bar\rR \otimes_\rA \overline\mN \ar[rd]_\can \ar[rr]^-{\theta_N} \ar[dd]_{\mathbbm{1} \otimes \kappa} & &
\bar\rR \otimes_\rR  (\rR \boxtimes_\rA \mN) \ar[dd]^{\mathbbm{1} \otimes \tilde\imath}\\
& \bar\rR \otimes_\rR (\rR \otimes_\rA \mN) \ar[ru]_{\mathbbm{1} \otimes \delta} \ar[rd]^\can & \\
\bar\rR\otimes_{\bar\rA} V \ar[r]^{\theta^\circ} & \widehat\mM^\circ &
\bar\rR \otimes_\rR \widetilde\mM, \ar[l]
}
$$
which can be verified by a simple diagram chasing. Theorem  is proven.
\end{proof}

\begin{remark} In their recent monograph \cite[Section 14.2]{LeusckeWiegand}, Leuschke and Wiegand give
a simpler  proof of Theorem
\ref{T:BurbanDrozd} in the special case when $\bar\rA$ and $\rR$ are both regular.
\end{remark}

\noindent
Observe that we have  the following practical rule to reconstruct a maximal Cohen--Macaulay
$\rA$--module $\mM$ from the corresponding triple
$\FF(\mM) \in \mathrm{Ob}\bigl(\Tri(\rA)\bigr)$.

\begin{corollary}\label{C:how-to-comp}
Let $\mT = (\widetilde\mM, V, \theta)$ be an object
 of the category of triples $\Tri(\rA)$. Then there exists a maximal Cohen--Macaulay
  $\bar\rA$--module
 $\mU$, an injective morphism of $\bar\rA$--modules
 $\varphi: \mU \to \bar\rR \otimes_\rR \widetilde\mM$
 and an isomorphism  $\psi: \rQ(\bar\rR) \otimes_{\bar\rA} \mU
 \to \rQ(\bar\rR) \otimes_{\rQ(\bar\rA)} V$ such that
 the following diagram
 $$
 \xymatrix
 {
 \rQ(\bar\rR) \otimes_{\bar\rA} \mU \ar[rr]^{\mathbbm{1} \otimes \varphi}
 \ar[dr]_\psi
 & & \rQ(\bar\rR) \otimes_{\rR} \widetilde\mM \\
  & \rQ(\bar\rR) \otimes_{\rQ(\bar\rA)} V \ar[ur]_\theta
 }
 $$
 is commutative in the category of $\rQ(\bar\rR)$--modules.
 Consider the following commutative diagram with exact rows
 in the category of $\rA$--modules:
 $$
 \xymatrix
 {
 0 \ar[r] & I\widetilde\mM \ar[r] \ar[d]_{=}
 & \mN \ar[r] \ar[d] & \mU \ar[r] \ar[d]^{\varphi} &
 0 \\
0 \ar[r] & I\widetilde\mM \ar[r] & \widetilde\mM \ar[r] &
\bar{\rR}\otimes_\rR \widetilde\mM \ar[r] &
 0.
}
 $$
 Then we have: $\mathbb{G}(\mT) \cong \mN^\dagger$. In particular, the isomorphy class
 of $\mN^\dagger$ does not depend on the choice of $\mU$ and $\varphi$.
 \end{corollary}

\noindent
As we shall see later, in some cases the module $\mN$ obtained by the recipe  from
 Corollary \ref{C:how-to-comp}, turns out to be automatically maximal Cohen--Macaulay.
This can be tested using the following useful result.

\begin{lemma}\label{L:how-to-avoid-Macaulayfication}
In the notations of this section, let $\widetilde\mM$ be a maximal Cohen--Macaulay
$\rR$--module, $\mV$ be a maximal Cohen--Macaulay $\bar\rA$--module and
$\tilde\theta: \mV \to \widehat\mM$ be an injective  morphism of $\bar\rA$--modules.
Consider the $\rA$--module $\mN$ given by the following commutative diagram:
$$
\xymatrix
{ 0 \ar[r] & I \widetilde\mM  \ar[r] \ar[d]_{=} & \mN \ar[r] \ar[d]^\imath   & \mV
\ar[d]^{\tilde{\theta}} \ar[r] & 0 \\
0 \ar[r] & I \widetilde\mM  \ar[r] & \widetilde\mM \ar[r] & \widehat\mM
\ar[r] & 0.
}
$$
Then there is the following short  exact sequence of $\rA$--modules:
\begin{equation}\label{E:cuseful-ex-sequen}
0 \lar \mN \lar \mN^\dagger \lar H^0_{\{\idm\}}\bigl(\coker(\tilde\theta)\bigr) \lar 0.
\end{equation}
In particular, $\mN$ is a maximal Cohen--Macaulay $\rA$--module if and only if
$\coker(\tilde\theta)$ is a maximal Cohen--Macaulay $\bar\rA$--module.
\end{lemma}

\begin{proof}
By the snake lemma, we get the short exact sequence
$$
0 \lar \mN \stackrel{\imath}\lar \widetilde\mM  \lar \coker(\tilde\theta) \lar 0.
$$
Since the module $\widetilde\mM$ is maximal Cohen--Macaulay over $\rA$, we have:
$H^0_{\{\idm\}}\bigl(\coker(\tilde\theta)\bigr) \cong H^1_{\{\idm\}}(\mN)$. On the other hand,
the module $\mN$ is torsion free and in the canonical short exact sequence
$$
0 \lar \mN \stackrel{\delta}\lar  \mN^\dagger \lar \mT \lar 0
$$
the module $\mT$ has finite length. Hence, we have:
$$
\mT \cong H^0_{\{\idm\}}(\mT) \cong H^1_{\{\idm\}}(\mN) \cong H^0_{\{\idm\}}\bigl(\coker(\tilde\theta)\bigr)
$$
yielding the desired short exact sequence (\ref{E:cuseful-ex-sequen}).
\end{proof}

\section{Maximal Cohen--Macaulay modules over $\kk\llbracket x, y, z\rrbracket/(x^2 + y^3
- xyz)$}
\label{sec5}

In this section,  we shall illustrate our method of study of maximal Cohen--Macaulay
modules overnon-isolated surface singularities, based on Theorem \ref{T:BurbanDrozd},
 on the case of the $T_{23\infty}$--singularity
$\rA = \kk\llbracket x, y, z\rrbracket/(x^3 + y^2 - xyz)$. We first have to accomplish   the following computations.

\begin{itemize}
\item Let $\rR$ be the normalization of $\rA$. Then $\rR = \kk\llbracket u, v \rrbracket$, where
$u = \frac{\displaystyle y}{\displaystyle x}$ and  $v =  \frac{\displaystyle xz - y}{\displaystyle x}$.
\item Next,  $I = (x, y)\rA = (uv)\rR$ is the conductor ideal.  Hence
$\bar\rA = \rA/I = \kk\llbracket z\rrbracket$, whereas $\bar\rR = \kk\llbracket u, v\rrbracket/(uv)$.
The map $\bar\rA \to \bar\rR$ sends  $z$ to $u +v$.
\item Let $\bD = \kk\llbracket z\rrbracket$ and $\rK = \kk\llbrace z\rrbrace$. Then we have:
$\rQ(\bar\rA) \cong \kk\llbrace z\rrbrace = \rK$ and $\rQ(\bar\rR) \cong
\kk\llbrace u\rrbrace \times
\kk\llbrace v\rrbrace \cong  \rK \times \rK$.
\end{itemize}
Let $\mT = (\widetilde\mM, V, \theta)$ be an object of $\Tri(\rA)$. Then the following results are true.
\begin{itemize}
 \item Since $\rR$ is regular,  $\widetilde\mM \cong \rR^m$ for some integer $m \ge 1$.
\item Next, $V$ is just a vector space
over the field $\rK$, hence $V \cong \rK^n$ for some  $n \ge 1$.
\item The gluing map  $\theta$ is given by a pair
of matrices of full row rank and
 the same size:
  $
  \theta = (\theta_u, \theta_v)  \in \Mat_{m\times n}(\rK) \times
  \Mat_{m\times n}(\rK).
  $
\end{itemize}
If two triples $\mT = (\widetilde\mM, V, \theta)$ and $\mT' = (\widetilde\mM', V', \theta')$
are isomorphic then $\widetilde\mM \cong \widetilde\mM'$ and $V \cong V'$. Describing the isomorphy
classes of objects in $\Tri(\rA)$,
we may without loss of generality assume that $\widetilde\mM =  \widetilde\mM'$ and $V =  V'$.
The essential information about the isomorphism class of $\mT$ is contained in the gluing data
$\theta$.
The description of isomorphism classes of objects in $\Tri(\rA)$ leads to the
following matrix problem:
\begin{equation}
(\theta_u, \theta_v) \mapsto (S_1^{-1} \theta_u T, S_2^{-1} \theta_v T) = (\theta'_u, \theta'_v),
\end{equation}
where $T \in \GL_n(\rK)$ and $S_1, S_2  \in \GL_m(\bD)$ are such that $S_1(0) = S_2(0)$.
This matrix problem corresponds to the category of representations of a very special decorated
bunch of chains, which will be treated in the full generality in the subsequent section.

\medskip
\noindent
\textbf{Fact}.
The pair
$(\theta_u, \theta_v)$ splits into a direct sum of the following indecomposable blocks,
see Subsection  \ref{Ex:CanonicalFormsDecKronecker} below.

\vspace{1mm}

\noindent
\emph{Continuous series}. Let $l, t \ge 1$ be positive integers,
$\omega = \bigl((m_1, n_1), \dots, (m_t, n_t)\bigr) \in (\mathbb{Z}_+^2)^t$
be a \emph{``non-periodic sequence''} such that
$\min(m_i, n_i)=1$ for all $1 \le i \le t$ and $\lambda \in \kk^*$. Then we have
the corresponding canonical form:

\begin{equation}\label{E:bands-matr-forbands}
\begin{array}{c}
\begin{tikzpicture}
\matrix (first) [tbl5,  name=tbl,
minimum height=25pt,
text width=25pt%
] at (0,0)
{
A_1 & 0           & 0 & \dots  & 0 \\
0           & A_2 &  0 & \dots  & 0 \\
\vdots      & \vdots      &   \ddots&  \ddots & \vdots \\
0           & 0          & \dots   &  A_{t-1}  & 0 \\
0           & 0 & \dots            & 0 & A_{t} \\
};

\sfrm{tbl}{5}{5};
\node[base left=5pt of tbl-3-1.north west, yshift = -18pt]{$\theta_u  =$};
\matrix (second) [tbl5,  name=tbl,
minimum height=25pt,
text width=25pt%
] at (6,0)
{
0           & B_2 &  0 & \dots    & 0   \\
0           &  0          &  B_3 &  0 & 0  \\
\vdots      & \vdots      & \vdots &   \ddots& \vdots \\
0           &  0          &   0 &  \dots         & B_t  \\
C          & 0    & 0        & \dots & 0 \\
};

\sfrm{tbl}{5}{5};
\node[base left=5pt of tbl-3-1.north west, yshift=-18pt]{$\theta_v  =$};
\end{tikzpicture}
\end{array}
\end{equation}

\noindent
where $A_k= z^{m_k} I,$ $B_k= z^{n_k} I$ and $C=z^{n_1}J$ with $I = I_l$ the identity $l\times l$--matrix and $J = J_l(\lambda)$ the Jordan block of size $l \times l$ with the eigenvalue $\lambda$.

\medskip
\noindent
\emph{Discrete series}. Let
$\omega = \bigl(m_0,  (m_1, n_1), \dots, (m_t, n_t), n_{t+1}\bigr)$,
where $m_0 = n_{t+1} = 1$ and $\min(m_i, n_i) = 1 $ for all $1 \le i \le t$.
Consider the matrices   $\theta_u$ and $\theta_v$ of the size $(t+1) \times (t+2)$
given as follows:

\begin{equation}\label{E:strings-matr-forstrings}
\begin{array}{c}
\begin{tikzpicture}
\matrix (first) [tbl5,  name=tbl,
minimum height=25pt,
text width=25pt%
] at (0,0)
{
z^{m_0}      & 0      & 0 & \dots & 0 \\
0      & z^{m_1}      & 0 & \dots & 0 \\
\vdots & \vdots & \ddots  & \ddots & \vdots \\
0      & 0      & \dots & z^{m_t} & 0\\
};

\sfrm{tbl}{4}{5};
\node[base left=5pt of tbl-3-1.north west, yshift=-5pt]{$\theta_u  =$};

\matrix (second) [tbl5,  name=tbl,
minimum height=25pt,
text width=25pt%
] at (6,0)
{
0      & z^{n_1}      & 0 & \dots & 0 \\
0      & 0      & z^{n_2} & \dots & 0 \\
\vdots & \vdots & \vdots   &  \ddots & 0 \\
0      & 0      & 0  &  \dots  &  z^{n_{t+1}}\\
};

\sfrm{tbl}{4}{5};
\node[base left=5pt of tbl-3-1.north west, yshift=-5pt]{$\theta_v  =$};

\end{tikzpicture}
\end{array}
\end{equation}

\noindent
In the case $t = 0$ we set $\theta_u = (1\; 0)$ and $\theta_v = (0\; 1)$.
\qed

\vspace{1mm}
\noindent
This result is a special case of the classification of indecomposable objects in the category
of representations of a decorated bunch of chains, which will be treated in the subsequent sections.
We get the following conclusion.

\begin{itemize}
\item  Let $(\omega, l, \lambda)$ be a band datum as above. Then
the triple $\bigl(\rR^{tl}, \rK^{tl}, (\theta_u, \theta_v))\bigr)$ defines
an indecomposable maximal Cohen--Macaulay module $\mM(\omega, l, \lambda)$,  which is locally free
of rank $tl$ on the punctured spectrum of $\rA$. Moreover, any indecomposable
object of $\CM^{\mathsf{lf}}(\rA)$  is described by a triple
of the above form.

\item Let $\omega$ be a string datum as above.
 Then the triple $\bigl(\rR^{t+1}, \rK^{t+1}, (\theta_u, \theta_v)\bigr)$
defines an indecomposable  maximal
Cohen--Macaulay  $\rA$--module $\mN(\omega)$ of rank $t + 1$, which is not
locally free on the punctured spectrum.
Moreover, any indecomposable
object of $\CM(\rA)$ which does not belong to $\CM^{\mathsf{lf}}(\rA)$,
 is isomorphic to some   $\mN(\omega)$.
\end{itemize}

\vspace{1mm}
\noindent
Our next goal is to describe an algorithm to construct maximal Cohen--Macaulay $\rA$--modules
corresponding to the canonical forms (\ref{E:bands-matr-forbands}) and
(\ref{E:strings-matr-forstrings}). First note the following simple result.

\begin{lemma}\label{L:onmaps-and-matric}
Let $(\theta_u, \theta_v)$ be the canonical form
given either by a band datum
$(\omega, l, \lambda)$ or by a string datum $\omega$.
Let $\theta = \theta_u(u) + \theta_v(v) \in \Hom_{\bar\rA}(\bar\rA^{p}, \bar\rR^{q})$ be the corresponding morphism of $\bar\rA$--modules, where $p = q = tl$ in the case of bands and $p= t+2$, $q = t+1$
 in the case of strings.
 Then the induced morphism
of $\rQ(\bar\rA)$--modules $\mathbbm{1} \otimes \theta :
\rQ(\bar\rA)^{p} \to  \rQ(\bar\rR)^{q}$ is given
by the original matrices $\bigl(\theta_u(u), \theta_v(v)\bigr)$, where we use the canonical isomorphism
$\rQ(\bar\rA) \otimes_{\bar\rA} \bar\rR \to \rQ(\bar\rR)$.
\end{lemma}

\noindent
As a consequence, we get a complete description of the indecomposable maximal Cohen--Macaulay modules
over the ring $\rA = \kk\llbracket x, y, z\rrbracket/(x^3 + y^2 - xyz)$.

\begin{corollary}\label{C:descr-bands-and-strings} Let $(\theta_u, \theta_v)$ be the canonical form
defined  by a band datum
$(\omega, l, \lambda)$ or by a string datum  $\omega$ and $\theta = \theta_u(u) + \theta_v(v)$ and
$p,q$ be as in Lemma \ref{L:onmaps-and-matric}.
 Consider
the matrix
$$
\bar\theta    := \bigl(xI_{q} \, | \,  y I_{q} \, | \,  \theta\bigr)
\in {\Mat}_{q \times (2q + p)}(\rR).
$$
Let $\mL \subseteq \rR^{q}$ be the $\rA$--module generated by the columns of the
matrix $\bar\theta$. Then the maximal Cohen--Macaulay $\rA$--module $\mM :=\mL^\dagger =
\mL^{\vee\vee}$ satisfies:
$$
\FF(\mM) \cong \bigl(\rR^q, \rK^{p}, (\theta_u, \theta_v) \bigr).
$$
In other words, $\mM$ is an indecomposable maximal Cohen--Macaulay $\rA$--module
corresponding to the canonical forms (\ref{E:bands-matr-forbands}) and  (\ref{E:strings-matr-forstrings}).
\end{corollary}

\noindent
Corollary \ref{C:descr-bands-and-strings} leads to  the following result.

\begin{proposition}
For the ring $\rA = \kk\llbracket x, y, z\rrbracket/(x^3  + y^2 - xyz)$ the classification
of  maximal Cohen--Macaulay $\rA$--modules of \textsl{rank one} is the following.
\begin{enumerate}
\item There exists exactly one maximal Cohen--Macaulay module $\mN = \mN\bigl(1(\,,)1\bigr)$ of rank one, which is not locally free on the punctured spectrum. We have the following $\rA$--module isomorphisms: $\mN \cong I \cong \rR$.
\item A rank one object of $\CM^{\mathsf{lf}}(\rA)$  is either regular or  has the following form:
$$
\mM\bigl((1, m), \lambda\bigr) \cong I_{m, \, \lambda} \quad \mbox{\rm and}  \quad
\mM\bigl((m, 1), 1, \lambda\bigr) \cong J_{m, \,  \lambda},
$$
where $\lambda \in \kk^*$ for $m \ge 2$ and  $\lambda \in \kk^*\setminus \{1\}$ for $m =1$,
 $I_{m, \, \lambda} = \bigl\langle x^{m+1}, y x^{m-1} +
\lambda (xz - y)^m\bigr\rangle \subset \rA$ and
$J_{m, \, \lambda} = \bigl\langle x^{m+1}, y^m  +
\lambda x^{m-1} (xz - y)\bigr\rangle \subset \rA$.
\end{enumerate}
\end{proposition}

\begin{proof}
The fact that there exists precisely one
object of $\CM(\rA)$ of rank one, which does not belong to $\CM^{\mathsf{lf}}(\rA)$,
follows from  Corollary \ref{C:descr-bands-and-strings}.
Note that both modules $I$ and $\rR$ share the property to be maximal Cohen--Macaulay of rank one,
being
not locally free on the punctured spectrum.

Let $\theta_u = z^m$, $\theta_v = \lambda z^n$ and $\theta = \theta_u(u) + \theta_v(v)$,
where $\lambda \in \kk^*$ and
$\min(m,n) = 1$. Then unless $\max(m, n)= 1$ and $\lambda = 1$, the cokernel of the
morphism of $\bar\rA$--modules $\theta: \bar\rA \to \bar\rR$ has no finite length submodules.
By Lemma \ref{L:how-to-avoid-Macaulayfication} and Corollary \ref{C:descr-bands-and-strings} we get:
$$
\mM\bigl((m,1), 1, \lambda\bigr)  = \bigl\langle x,y, u^m + \lambda v\bigr\rangle_\rA \subseteq \rR
\quad \mbox{and} \quad \mM\bigl((1,m), 1, \lambda\bigr)  = \bigl\langle x,y, u + \lambda v^m\bigr\rangle_\rA \subseteq \rR.
$$
Next, observe that $u = \frac{\displaystyle y}{\displaystyle x}$ fulfills the equation $u^2-zu+x = 0$. By induction
it is not difficult to show that for any $m \ge 2$ there exist  polynomials
$p_m(X,Z)$ and $q_m(X,Z)$ from $\kk\llbracket X,Z\rrbracket$ such that
the following equality holds in $\rR$:
$
u^m = p_m(x,z) u + q_m(x,z).
$
Using this fact it is not difficult to derive that
$$y \in \bigl\langle x, u + \lambda v^m\bigr\rangle_\rA \cong
\bigl\langle x^{m+1}, x^m y, x^{m-1}y  + \lambda (xz -y)^m\bigr\rangle_\rA.$$
In a similar way,
$y \in \bigl\langle x, u^m + \lambda v\bigr\rangle_\rA \cong
\bigl\langle x^{m+1}, x^m y, y^m   + \lambda x^{m-1} (xz -y)\bigr\rangle_\rA.$
\end{proof}

 It is very instructive to compute the  matrix factorizations
corresponding to some rank one Cohen--Macaulay $\rA$--modules. Note that the conductor ideal
$I$ corresponds to the matrix factorization
$\Bigl(\left(\begin{smallmatrix}  x &  y \\  - y &  x^2 -yz \end{smallmatrix}\right),
\left(\begin{smallmatrix} x & - y \\  y & x^2 -yz \end{smallmatrix}\right)\Bigr)$.

\vspace{1mm}

Consider now the family of modules $\mM\bigl((1,1),1,\lambda\bigr)
\cong \bigl\langle x, \frac{\displaystyle y}{\displaystyle x} + \lambda \frac{\displaystyle xz-y}{\displaystyle x}\bigr\rangle_\rA$, where
$\lambda \in \kk^*\setminus \{1\}$.  The special value
$\lambda = 1$ has to be treated separately:  in this case we have
$\mM\bigl((1,1),1,\lambda\bigr) \cong \rA$.
For $\lambda \ne 1$ we know that  $\mM\bigl((1,1),1,\lambda\bigr) =
\bigl\langle x^2, y + \frac{\lambda}{\lambda-1} xz\bigr\rangle_\rA$. The new moduli
parameter $\mu = \frac{\lambda}{\lambda-1}$ takes its values in
$\mathbb{P}^1\setminus\bigl\{(1:-1)\bigr\} = \bigl(\kk \cup \{\infty\}\bigr)\setminus
\{-1\}$. One can check that $\mM\bigl((1,1),1,\lambda\bigr)$ has  a presentation:
$$
\rA^2
\xrightarrow{
\left(
\begin{smallmatrix}
x + \mu(\mu+1) z^2 & y + \mu xz \\
y - (\mu +1) xz & -x^2
\end{smallmatrix}
\right)
}
\rA^2 \lar \mM\bigl((1,1),1,\lambda\bigr) \lar 0.
$$
Note that the forbidden value $\mu = -1$  corresponds to  the module  $\rR \cong I$,
whereas the value $\mu = \infty$ corresponds  the regular module $\rA$.  In other words,
the ``pragmatic moduli  space'' of the rank one modules $\mM\bigl((1,1),1,\lambda\bigr)$
 can be naturally compactified to the nodal cubic curve $zy^2 = x^3 + x^2 z$, where
the unique singular point corresponds to the unique rank one Cohen--Macaulay $\rA$--module,
which is not locally free on the punctured spectrum. Note that the explicit
expression for the presentation matrices of $\mM\bigl((1,1),1,\lambda\bigr)$ are consistent
with the criteria to be  locally free on the punctured spectrum from Lemma
\ref{P:CriterionLF}.
\vspace{1mm}

Next, let us compute the matrix factorization describing the family $\mM\bigl((2,1),1,\lambda\bigr)$.
By Corollary \ref{C:descr-bands-and-strings} we  have:
$$
\mM\bigl((2,1),1,\lambda\bigr) = \bigl\langle x^3, y^2 + \lambda (xz-y)x\bigr\rangle_\rA =
\bigl\langle x^3, xyz  + \lambda (xz-y)x\bigr\rangle_\rA \cong
\bigl\langle x^2, y(z-\lambda) + \lambda xz\bigr\rangle_\rA.
$$
This family has the following presentation:
$$
\rA^2
\xrightarrow{
\left(
\begin{smallmatrix}
x (z -\lambda)^2 + \lambda z^3 & y (z-\lambda) - \lambda xz\\
y (z-\lambda) - xz^2 & x^2
\end{smallmatrix}
\right)
}
\rA^2 \lar \mM\bigl((2,1),1,\lambda\bigr) \lar 0.
$$

\vspace{1mm}

Our approach can be also applied to describe maximal Cohen--Macaulay modules  of higher rank.
Consider the long exact sequence
$$
0 \lar \rA \lar \mF \lar \rA \lar \kk \lar 0,
$$
corresponding to the generator of the $\rA$--module  $\Ext^2_\rA(\kk, \rA) \cong \kk$. The module
$\mF$ is  called  \emph{fundamental module}.
By a result of Auslander \cite{Auslander}, $\mF$ is maximal Cohen--Macaulay. It plays a central role in the
theory of almost split sequences in the category $\CM(\rA)$. Let us compute a  presentation
of $\mF$ using our method.
First, with some efforts one can  show that $\mF$ corresponds to the triple
$$\Bigl(\rR^2, \rK^2, \bigl(\left(\begin{smallmatrix}  1 & 0 \\ 0 & 1 \end{smallmatrix}\right),
\left(\begin{smallmatrix}  1 & 1 \\ 0 & 1 \end{smallmatrix}\right)\bigr)\Bigr)
\cong
\Bigl(\rR^2, \rK^2, \bigl(\theta_u = \left(\begin{smallmatrix}  u & 0 \\ 0 & u \end{smallmatrix}\right), \theta_v =
\left(\begin{smallmatrix}  v & v \\ 0 & v \end{smallmatrix}\right)\bigr)\Bigr).
$$
Consider the module $\mN$ given by the diagram
$$
\xymatrix
{ 0 \ar[r] & I^2  \ar[r] \ar[d]_{=} & \mN \ar[r] \ar[d]   & \bar\rA^2
\ar[d]^{\tilde{\theta}} \ar[r] & 0 \\
0 \ar[r] & I^2  \ar[r] & \rR^2  \ar[r] & \bar\rR^2
\ar[r] & 0.
}
$$
Then we have:
$$
\mN =
\left\langle
\left(
\begin{array}{c}
x^2 \\
0
\end{array}
\right),
\left(
\begin{array}{c}
xy \\
0
\end{array}
\right),
\left(
\begin{array}{c}
0 \\
x^2
\end{array}
\right),
\left(
\begin{array}{c}
0 \\
xy
\end{array}
\right),
\left(
\begin{array}{c}
xz \\
0
\end{array}
\right),
\left(
\begin{array}{c}
-y \\
xy
\end{array}
\right)
\right\rangle_\rA \subseteq \rA^2
$$
and $\mF \cong \mN^\dagger$. Note that the element $a := \left(\begin{smallmatrix} x \\ 0 \end{smallmatrix}
\right) \in \rA^2$ does
not belong to $\mN$, however  $\idm a \in \mN$. Applying
 Lemma \ref{L:how-to-avoid-Macaulayfication} and
Lemma \ref{L:fact-on-Macaulaf} we conclude that
$$
\mF \cong \mN^\dagger =
\left\langle
\left(
\begin{array}{c}
x \\
0
\end{array}
\right),
\left(
\begin{array}{c}
0 \\
-x^2
\end{array}
\right),
\left(
\begin{array}{c}
-y \\
xy
\end{array}
\right),
\left(
\begin{array}{c}
0 \\
xy
\end{array}
\right)
\right\rangle_\rA \subseteq \rA^2.
$$
Moreover, $\mF$ has the following presentation:
$$
\rA^4
\xrightarrow{
\left(
\begin{smallmatrix}
y & z & x & 0 \\
0 & y & 0 & x \\
yz-x^2 & 0 & y & -z \\
0 & yz-x^2 & 0 & y
\end{smallmatrix}
\right)
}
\rA^4 \lar \mF \lar 0.
$$
One can check that we have an isomorphism $\mF \cong \syz^3(\kk)$, which matches with
a  result  obtained
by  Yoshino and Kawamoto \cite{YoshinoKawamoto}.

\begin{remark}
According to Kahn \cite{Kahn} as well as to
Drozd, Greuel and Kashuba \cite{DGK}, the normal surface  singularity $\rB = \kk\llbracket x, y, z\rrbracket/(x^3 + y^2 + z^p - xyz)$,
 has tame Cohen--Macaulay representation type for $p \ge 6$.
On the other hand, an explicit  description
of indecomposable maximal Cohen--Macaulay $\rB$--modules   still remains  unknown. It would be interesting
to know what  objects of $\CM(\rA)$
can be deformed to objects in $\CM(\rB)$, as well as to describe the corresponding families explicitly.
\end{remark}


\tikzset{
    table nodes/.style={
        rectangle,
        fill=blue!5,
        draw,
        thin, dashed,
        align=center,
        minimum height=40pt,
        text depth=0.5ex,
        text height=2ex,
        inner xsep=0pt,
        outer sep=0pt
    },
    table/.style={
        matrix of nodes,
        row sep=-\pgflinewidth,
        column sep=-\pgflinewidth,
        nodes={
            table nodes
        },
        execute at empty cell={\node[draw=none]{};}
    }
  }

\section{Representations of decorated bunches of chains--I}\label{bc1}

In this section we  introduce a certain type of matrix problems called
``representations of decorated bunches of chains'' and explain the combinatorics of indecomposable objects.

\subsection{Notation}  Let $\bD$ be a discrete valuation ring,
$\idm$ its maximal ideal, $t \in \idm$ an uniformizing element (i.e.~such that $(t) = \idm$), $\kk = \bD/\idm$ the residue field of $\bD$ and
$\rK$  the field of fractions of $\bD$. For an element $a \in \bD$ we denote by $\bar{a}$ its image in $\kk$. Similarly,
for a matrix $W \in \Mat_{m \times n}(\bD)$ we denote by $\bar{W} \in \Mat_{m \times n}(\kk)$ its residue modulo $\idm$.
Finally,
$\bD \times \bD \supset \widetilde\bD = \left\{(a, b) \, | \,  \bar{a} = \bar{b}\right\}$ is  the \emph{dyad} of $\bD$ with itself.

\subsection{Bimodule problems} The language of bimodule problems has been introduced by Drozd
in \cite{DrozdLOMI} as an attempt
to formalize the notion of a matrix problem. See also \cite{CB} and \cite{dr} for further elaborations.

Let $\rR$ be a commutative ring, $\kA$ be an $\rR$--linear category and $\kB$ be an $\kA$--bimodule. The last  means that
for  any pair of objects $A, B$ of $\kA$ we have an $\rR$--module $\kB(A, B)$ and for any further pair of objects
$A', B'$ there are left and right multiplication maps
$$
\kA(B, B') \times \kB(A, B) \times \kA(A', A) \lar \kB(A', B'),
$$
which are $R$--multilinear and associative.
\begin{definition}
The   \emph{bimodule category} $\el(\kA, \kB)$  (sometimes called \emph{category of elements of the $\kA$--bimodule $\kB$}) is defined as follows. Its objects are pairs $(A, W)$, where
$A$ is an object of $\kA$ and $W\in \kB(A, A)$. The morphism spaces in $\el(\kA, \kB)$ are defined as follows:
$$\el(\kA, \kB)\bigl((A, W), (A', W')\bigr) = \bigl\{F \in \kA(A, A') \, | \, F W = W' F \bigr\}.$$
The composition of morphisms in $\el(\kA, \kB)$  is the same as  in $\kA$.
\end{definition}

\begin{remark}
 The category $\el(\kA, \kB)$  is additive and idempotent complete provided $\kA$ is additive and idempotent complete. However, one typically  starts
with a category $\kA$  having the property that the endomorphism algebra of any of its objects is local
(obviously, in this case, $\kA$ can not be additive). Then one takes the \emph{additive closure}
$\kA^{\omega}$ of $\kA$ and extends $\kB$ to an $\kA^{\omega}$--bimodule $\kB^{\omega}$ by additivity. Abusing the notation,
we write  $\el(\kA, \kB)$ having actually the category $\el(\kA^{\omega}, \kB^{\omega})$ in mind.
\end{remark}

\noindent

\begin{example}\label{E:typDecBunch} Assume $\bD = \kk\llbracket t\rrbracket$. We define the  category $\kA$ and bimodule $\kB$  as follows.
\begin{itemize}
\item $\kA$ has three objects: $\Ob(\kA) = \{a, b, c\}$.
\item The non-zero morphism spaces of $\kA$ are:
\begin{itemize}
\item $\kA(a, a) = \rK 1_{a}$, $\kA(b, c) = \langle \nu_1, \nu_2\rangle_{\bD} \cong \bD^2$ and $\kA(c, b) = \langle \rho_1, \rho_2\rangle_{\bD} \cong \bD^2$.
\item
$\kA(b, b) = \kk\llbracket \beta_1, \beta_2\rrbracket/(\beta_1\beta_2)$ and
$\kA(c, c) = \kk\llbracket \gamma_1, \gamma_2\rrbracket/(\gamma_1\gamma_2)$.
\end{itemize}
\item For $\imath = 1, 2$ we have the following relations: $\rho_\imath \nu_\imath = \beta_\imath$ and
 $\nu_\imath \rho_\imath = \gamma_\imath$.
\item The bimodule $\kB$ is defined by the following rules:
\begin{itemize}
\item $\kB(a, b) = \langle \phi_1, \phi_2\rangle_{\rK} \cong \rK^2 \cong \kB(a, c) = \langle \psi_1, \psi_2\rangle_{\rK}$.
    \item For $(x, y) \notin \bigl\{(a, b), (a, c)\bigr\}$ we have: $\kB(x, y) = 0$.
\item The action of $\kA$ on $\kB$ is given by the following rules:
$$
\beta_\imath \circ \phi_\jmath = \delta_{\imath\jmath} t \cdot  \phi_\imath,
\gamma_\imath \circ \psi_\jmath = \delta_{\imath\jmath} t \cdot \psi_\imath,
\nu_\imath \circ \phi_\jmath = \delta_{\imath\jmath}  \psi_\imath,
\rho_\imath \circ \psi_\jmath = \delta_{\imath\jmath}  t \cdot\phi_\imath, \; \imath, \jmath = 1, 2 \; \mbox{and}
$$
$$
(\phi_1, \phi_2, \psi_1, \psi_2) \circ (\kappa  1_{a}) = (\kappa \cdot \phi_1, \kappa \cdot \phi_2, \kappa \cdot  \psi_1, \kappa \cdot \psi_2), \; \kappa \in \rK.
$$
\end{itemize}
\end{itemize}
\noindent
The entire  data can be visualized by the following picture.
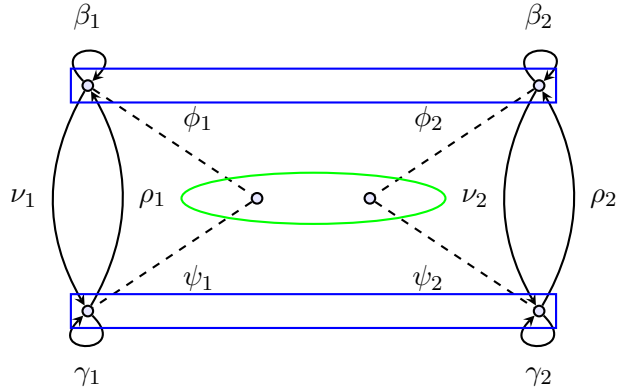
\begin{center}
\begin{tikzpicture}[scale=0.75,
    thick,
    dot/.style={fill=blue!10,circle,draw, inner sep=1pt, minimum size=4pt}]
 \draw (0,2) node[dot] (a1){} (0,-2)node[dot](b1){} (3,0)node[dot](c1){}
       (8,2) node[dot] (a2){} (8,-2)node[dot](b2){} (5,0)node[dot](c2){}
 ;
 \draw[thick, dashed](a1)-- node [above right]{$\phi_1$}(c1)--node [below right]{$\psi_1$}(b1);
 \draw[-stealth,thick] (a1) to[bend right] node[left=0.75mm]{$\nu_1$} (b1);
 \draw[-stealth,thick] (b1) to[bend right] node[right=0.75mm]{$\rho_1$} (a1);

\draw[thick, dashed](a2)-- node [above left]{$\phi_2$}(c2)--node [below left]{$\psi_2$}(b2);
 \draw[-stealth,thick] (a2) to[bend right] node[left=0.75mm]{$\nu_2$} (b2);
 \draw[-stealth,thick] (b2) to[bend right] node[right=0.75mm]{$\rho_2$} (a2);

\tikzset{every loop/.style={min distance=10mm,in=45,out=135,looseness=10}}
\draw[-stealth,thick] (a2)to[loop ] node[above=1.5mm]{$\beta_2$}   ();
\draw[-stealth,thick] (a1)to[loop above] node[above=1.5mm]{$\beta_1$}   ();

\tikzset{every loop/.style={min distance=10mm,in=225,out=315,looseness=10}}
\draw[-stealth,thick] (b1)to[loop below] node[below=1.5mm]{$\gamma_1$}   ();
\draw[-stealth,thick] (b2)to[loop below] node[below=1.5mm]{$\gamma_2$}   ();

\draw[blue] ($(a1)-(3mm,3mm)$ ) rectangle  ($(a2)+(3mm,3mm)$ );
\draw[blue] ($(b1)-(3mm,3mm)$ ) rectangle  ($(b2)+(3mm,3mm)$ );

 \node[draw= green, ellipse,  fit=(c1) (c2)]{};

\end{tikzpicture}
\end{center}
The encircled points represent three objects of the category $\kA$, the solid arrows denote morphisms
in $\kA$ whereas the dotted ones stand for generators of the bimodule $\kB$.

Let   us now derive the matrix problem describing  the isomorphy classes of objects
of  $\el(\kA,\kB)$. For any $x \in \{a, b, c\}$ let $Z_x$ denote  the corresponding object of the category $\kA^\omega$.
Then an object of $\el(\kA,\kB)$ is a pair $(Z, W)$, where
$Z = Z_a^{n} \oplus Z_b^{m} \oplus Z_c^{p}$ and $W$ is  a matrix of the following shape:

\begin{center}
  \begin{tikzpicture}
\matrix  [tbl5,  name=table,
row 1/.style={minimum height=40pt},
row 2/.style={minimum height=20pt},
row 3/.style={minimum height=30pt},
column 1/.style={text width=40pt},
column 2/.style={text width=20pt},
column 3/.style={text width=30pt},
]
{
0  & 0 & 0 \\
 P &  0&0\\
 Q &  0&0\\
};
\sfrm{table}{3}{3};

\node[above=2pt of table-1-1]{$a$};\node[above=2pt of table-1-2]{$b$}; \node[above=2pt of table-1-3]{$c$};

\node[base right=2pt of table-1-3]{$a$};\node[base right=2pt of table-2-3]{$b$};\node[base right=2pt of table-3-3]{$c$};

\node [base left=3pt of table.west]{$W\;=$};

\foreach \x in {1,2,3}{
\foreach \y in {1,2,3}{
\hdline{table}{\x}{\y};
\vdline{table}{\x}{\y};
}}

\end{tikzpicture}
\end{center}

\noindent
Here, $P= \Phi_1\phi_1+\Phi_2\phi_2 $ and $Q= \Psi_1\psi_1+\Psi_2\psi_2, $ where
 $\Phi_\imath \in \Mat_{m \times n}(\rK)$ and $\Psi_\imath \in \Mat_{p \times n}(\rK)$ for $\imath = 1, 2$.

\medskip
\noindent
An isomorphism of $(Z, W) \lar (Z, \tilde{W})$ is given by a matrix

\begin{center}
\begin{tikzpicture}

\matrix (first) [tbl5,  name=table,
row 1/.style={minimum height=40pt}, row 2/.style={minimum height=20pt},
row 3/.style={minimum height=30pt},
column 1/.style={text width=40pt}, column 2/.style={text width=20pt},
column 3/.style={text width=30pt},
]
{
S   & 0 & 0 \\
0  &   X  &  R   \\
   0 & N  & Y     \\
};
\sfrm{table}{3}{3};

\node[above=2pt of table-1-1]{$a$};
\node[above=2pt of table-1-2]{$b$};
\node[above=2pt of table-1-3]{$c$};

\node[base right=2pt of table-1-3]{$a$};
\node[base right=2pt of table-2-3]{$b$};
\node[base right=2pt of table-3-3]{$c$};

\node [base left=3pt of table.west]{$F\;=$};

\foreach \x in {1,2,3}{
\foreach \y in {1,2,3}{
\hdline{table}{\x}{\y};
\vdline{table}{\x}{\y};
}}
\end{tikzpicture}
\end{center}
Here, $S \in \GL_n\bigl(\kA(a, a)\bigr) \cong \GL_n(\rK)$, $X \in \GL_m\bigl(\kA(b, b)\bigr)$ and
$Y \in \GL_p\bigl(\kA(c, c)\bigr)$. Next, we write  $R = R_1 \rho_1 + R_2 \rho_2$ with
$R_1, R_2 \in \Mat_{m \times p}(\bD)$ and $N = N_1 \nu_1 + N_2 \nu_2$ with
$N_1, N_2 \in \Mat_{p \times n}(\bD)$. Let $(X_1, X_2)$ (respectively $(Y_1, Y_2)$) be the image
of $X$ (respectively $Y$) under the group homomorphism $\GL_m\bigl(\kA(b, b)\bigr) \lar
\GL_m(\bD)\times \GL_m(\bD)$ (respectively $\GL_p\bigl(\kA(c, c)\bigr) \lar
\GL_p(\bD)\times \GL_p(\bD)$).

The equality $F W = \tilde{W} F$ leads to the following matrix equalities:
\begin{equation}\label{E:typicalMatrProbl}
\left\{
\begin{array}{ccc}
X_\imath \Phi_\imath + t R_\imath \Psi_\imath & = & \widetilde{\Phi}_\imath S \\
Y_\imath \Psi_\imath +  N_\imath \Phi_\imath & = & \widetilde{\Psi}_\imath S,
\end{array}
\right.
\end{equation}
where $\imath = 1, 2$.
The obtained matrix problem can be visualized by the following picture.

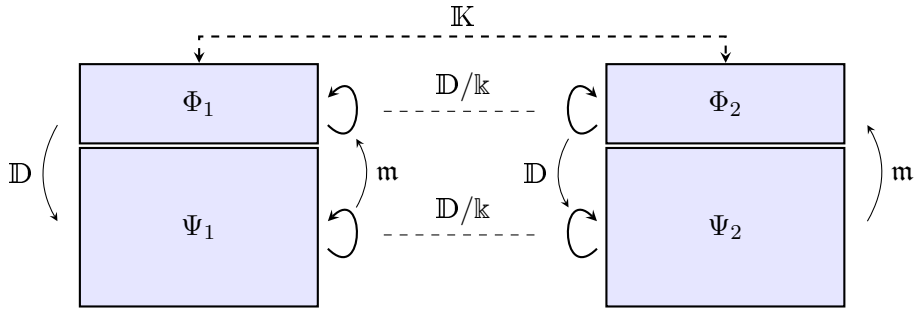
\begin{center}
 \begin{tikzpicture}
  \matrix[nodes={draw, thick, fill=blue!10},
    row sep=1pt, ]  {
  \node[ minimum height=30pt, minimum width=90pt](h1){$\Phi_1$}; &&\\
  \node[ minimum height=60pt, minimum width=90pt](s1){$\Psi_1$};&& \\
  }  ;

  \matrix[nodes={draw,thick, fill=blue!10},
    row sep=1pt, ]  at (7,0) {
  \node[ minimum height=30pt, minimum width=90pt](h2){$\Phi_2$}; &&\\
  \node[ minimum height=60pt, minimum width=90pt](s2){$\Psi_2$};&& \\
  }
  ;

\node (hr1) [base right=-0.25 cm of h1, inner sep=0pt, minimum size=10pt ]{};
\node (sr1) [base right=-0.25 cm of s1, inner sep=0pt, minimum size=10pt ]{};

\node (hl2) [base left=-0.25 of h2,inner sep=0pt, minimum size=10pt ]{};
\node (sl2) [base left=-0.25 cm of s2,inner sep=0pt, minimum size=10pt  ]{};

\node (hl1) [base left=0 cm of h1, inner sep=0pt, minimum size=10pt ]{};
\node (sl1) [base left=0 cm of s1, inner sep=0pt, minimum size=10pt ]{};

\node (hr2) [base right=0 cm of h2,inner sep=0pt, minimum size=10pt  ]{};
\node (sr2) [base right=0 cm of s2,inner sep=0pt, minimum size=10pt  ]{};

\draw[-stealth](hl1) to[bend right]node[midway,left]{$\bD$}  (sl1);
\draw[-stealth](sr2) to[bend right]node[midway,right]{$\idm$} (hr2) ;

\tikzset{every loop/.style={min distance=5mm,in=135,out=225,looseness=5}}
\draw[-stealth,thick] (hl2)to[loop ] node[minimum size=20pt,inner sep=10pt](hh2){}   ();
\draw[-stealth,thick] (sl2)to[loop ] node[minimum size=20pt,inner sep=10pt](ss2){}   ();

\tikzset{every loop/.style={min distance=5mm,in=45,out=315,looseness=5}}
\draw[-stealth,thick] (hr1)to[loop ] node(hh1)[minimum size=20pt,inner sep=10pt]{}   ();
\draw[-stealth,thick] (sr1)to[loop ] node(ss1)[minimum size=20pt,inner sep=10pt]{}   ();

 \draw[dashed](hh1) to node[midway,above]{$\bD/\aK$} (hh2);
 \draw[dashed](ss1) to node[midway,above]{$\bD/\aK$} (ss2);

\draw[-stealth](ss1.north) to[bend right]node[midway,right]{$\idm$} (hh1.south) ;
\draw[-stealth](hh2.south) to[bend right]node[midway,left]{$\bD$}  (ss2.north);

 \coordinate[above=10pt of h2](ha2){};
 \coordinate[above=10pt of  h1](ha1){};

 \draw[stealth-stealth, thick, dashed](h1) -- (ha1)-- (ha2)node[midway,above]{$\rK$} -- (h2);

\end{tikzpicture}
\end{center}
\noindent
The matrix problem (\ref{E:typicalMatrProbl}) can be rephrased   as follows.
\begin{itemize}
\item  We have four matrices $\Phi_1, \Phi_2, \Psi_1$ and $\Psi_2$ over $\rK$. All of them have the same number of columns.
   The matrices $\Phi_1$ and $\Phi_2$ (respectively, $\Psi_1$ and $\Psi_2$) have the same number of rows. We can perform transformations of  columns and rows of  $\Phi_1, \Phi_2, \Psi_1$ and $\Psi_2$,
which are compositions of the following elementary ones.
\item \emph{Simultaneous transformations}.  We can perform  {simultaneous}  elementary transformations
\begin{itemize}
\item   of columns of  $\Phi_1, \Phi_2, \Psi_1$ and $\Psi_2$ with coefficients in the field of fractions $\rK$.
\item   of rows  of  $\Phi_1$ and  $\Phi_2$ (respectively, $\Psi_1$ and $\Psi_2$)  with coefficients in the residue field $\kk$.
\end{itemize}
\item  \emph{Independent transformations}.
\begin{itemize}
\item We can  independently perform  (invertible)  elementary transformations of rows of matrices $\Phi_\imath$ and $\Psi_\imath$, for $\imath = 1, 2$ with coefficients in the maximal ideal $\idm$.
\item For $\imath = 1, 2$,  we can add an arbitrary $\bD$--multiple of any row of $\Phi_\imath$ to any row
of $\Psi_\imath$ and an arbitrary  $\idm$--multiple of any row of $\Psi_\imath$ to any row
of $\Phi_\imath$.
\end{itemize}
\end{itemize}
Note that this is precisely the matrix problem,   describing maximal Cohen--Macaulay modules over the degenerate cusp
$T_{2 4 \infty} = \kk\llbracket u, v, w\rrbracket/(u^2 + v^4 - uvw)$.  \qed
\end{example}

\noindent
Omitting  some details, we state now several other bimodule problems playing  a role in the study of maximal Cohen--Macaulay modules over non--isolated surface singularities.

\begin{example}[\textsl{Decorated conjugation problem}]\label{E:decJordan} Consider the following category $\kA$ and $\kA$--bimodule $\kB$:
\begin{itemize}
\item $\Ob(\kA) = \{a\}$,  $\kA(a, a) = \kk\llbracket \alpha_1, \alpha_2\rrbracket/(\alpha_1 \alpha_2) \cong \widetilde\bD$.
\item $\kB(a, a) = \langle \varphi\rangle_{\rK} \cong \rK$.
The $\kA$--bimodule structure on $\kB$ is given by the following rule: for any $\alpha \in \kA(a, a)$ we have
$\alpha \circ \varphi = \alpha(t, 0)\cdot \varphi$ and $\varphi \circ \alpha = \alpha(0, t) \cdot \varphi$.
\end{itemize}
The underlying matrix problem is the following. We have a square matrix $\Phi \in \Mat_{n\times n}(\rK)$ which can be transformed according to  the following rule:
\begin{equation}
\Phi \mapsto S_1 \Phi S_2^{-1},
\end{equation}
where $S_1, S_2 \in \GL_n(\bD)$ are such that $\bar{S}_1 = \bar{S}_2$.
\end{example}

\begin{example}[\textsl{Decorated Kronecker  problem}]\label{E:decKronecker} The category $\kA$ and $\kA$--bimodule $\kB$ are defined as follows.
\begin{itemize}
\item $\Ob(\kA) = \{a, b\}$,  $\kA(a, a) = \kk\llbracket \alpha_1, \alpha_2\rrbracket/(\alpha_1 \alpha_2) \cong \widetilde\bD$, $\kA(b, b) = \rK 1_{b}$, whereas $\kA(a, b) = 0 = \kA(b, a)$.
\item $\kB(b, a) = \langle \varphi, \psi\rangle_{\rK} \cong \rK^2$, $\kB(a, a) = \kB(b, b) = \kB(a, b) = 0$.
\item The $\kA$--bimodule structure on $\kB$ is given by the following rules.
\begin{itemize}
\item For $\alpha \in \kA(a, a)$ we have:
$\alpha \circ \varphi = \alpha(t, 0) \cdot \varphi$ and $\alpha \circ \psi = \alpha(0, t) \cdot \psi$.
\item For any $\kappa \in \rK$ we have:
$(\varphi, \psi) \circ (\kappa 1_b) = (\kappa \varphi, \kappa \psi)$.
\end{itemize}
\end{itemize}
The underlying matrix problem is the following. We have a pair of matrices
 $\Phi, \Psi  \in \Mat_{m\times n}(\rK)$ which can be transformed by the  rule:
\begin{equation}
(\Phi, \Psi)  \mapsto (S_1 \Phi T^{-1}, S_2 \Phi T^{-1}),
\end{equation}
where $T \in \GL_m(\rK)$ and $S_1, S_2 \in \GL_m(\bD)$ are such that $\bar{S}_1 = \bar{S}_2$. This is precisely the matrix problem arising in the classification of maximal Cohen--Macaulay modules over the degenerate cusp
$T_{23\infty} = \kk\llbracket u, v, w\rrbracket/(u^2 + v^3 - uvw)$, see Section \ref{sec5}.
\end{example}

\begin{example}[\textsl{Decorated chessboard}]\label{E:decChessBoard}
For any $n \ge 1$ consider the set $\Sigma = \Sigma_n = \{1, \dots, n\}$ and a permutation
 $\sigma$ of $\Sigma$. For any $\imath, \jmath \in \Sigma$ introduce
symbols $p_{\imath \jmath}$ and $q_{\imath \jmath}$. In what follows, we shall operate
 with them using  the following rules:
\begin{equation}\label{E:comprules}
p_{\imath \jmath} p_{\jmath l} = p_{\imath l}, q_{\imath \jmath} q_{\jmath l} = q_{\imath l},
p_{\imath \jmath} q_{\jmath l} = 0 \; \mbox{and} \; q_{\imath \jmath} p_{\jmath l} = 0 \;
\mbox{for all} \; \imath, \jmath, l \in \Sigma.
\end{equation}
The category $\kA$ and $\kA$--bimodule $\kB$ are defined as follows.
\begin{itemize}
\item $\Ob(\kA) = \Sigma$. For $1 \le \imath < \jmath \le n$ we pose:
$$
\kA(\imath, \jmath) = \bD p_{\jmath \imath} \oplus \idm q_{\jmath \imath} \; \mbox{and}\,
\kA(\jmath, \imath) = \bD q_{\imath \jmath} \oplus \idm p_{\imath \jmath}.
$$
\item For any $\imath \in \Sigma$ we put
$
\kA(\imath, \imath) = \bD 1_{\imath} \oplus \idm p_{\imath \imath} \oplus \idm q_{\imath \imath}/I_\imath,
$
where $I_\imath$ is the $\bD$--module generated by $t \cdot 1_\imath - t \cdot p_{\imath \imath} - t \cdot q_{\sigma(\imath) \sigma(\imath)}$.
\item The composition of morphisms in $\kA$ is defined by $\bD$--bilinearity and the multiplication rules (\ref{E:comprules}).
    \item For any $\imath, \jmath \in \Sigma$ we put $\kB(\imath, \jmath) = \rK \cdot \phi_{\jmath \imath}$.
    \item The action of $\kA$ on $\kB$ is given by  the following rules:  for any $\imath, \jmath, l \in \Sigma$ we have
\begin{equation*}
p_{\imath \jmath} \phi_{\jmath l} = \phi_{\imath l}, \phi_{\imath \jmath} q_{\jmath l} = \phi_{\imath l},
\phi_{\imath \jmath} p_{\jmath l} = 0 \; \mbox{and} \; q_{\imath \jmath} \phi_{\jmath l} = 0.
\end{equation*}
\end{itemize}
The description of isomorphy classes of objects in $\el(\kA, \kB)$ leads to the following matrix
problem.
Let $\sd_1, \dots, \sd_n \in \mathbb{Z}_{\ge 0}$ and $\sd := \sd_1 + \dots + \sd_n$.
 An object of $\el(\kA, \kB)$ is given by a matrix $W \in \Mat_{\sd \times \sd}(\rK)$,
whose rows  and columns  are divided into $n$ stripes labeled by elements $x_1, \dots, x_n$ (respectively, $y_1, \dots, y_n$)
so that the  $x_\imath$-th horizontal  stripe and $y_{\sigma(\imath)}$-th vertical stripe have width
$\sd_\imath$.  One can transform $W$ by the  rule:
$
W \mapsto S W T^{-1},
$
where $S, T \in \GL_{\sd}(\bD)$  satisfy the following additional constraints. Consider the division of $S$ and $T$ into $n$ horizontal and vertical stripes, the same as for $W$. For
any $\imath, \jmath \in \Sigma$ let $S_{\imath \jmath}$ (respectively, $T_{\imath \jmath}$) be the corresponding block of size
$\sd_\imath \times \sd_\jmath$. Then
\begin{itemize}
\item for any $1 \le \imath < \jmath \le n$ we have: $S_{\imath \jmath} \in \Mat_{\sd_\imath \times \sd_\jmath}(\idm)$ and
$T_{\jmath \imath} \in \Mat_{\sd_\jmath \times \sd_\imath}(\idm)$,
\item for any $\imath \in \Sigma$ we have: $\bar{S}_{\imath \imath} = \bar{T}_{\imath \imath}$.
\end{itemize}
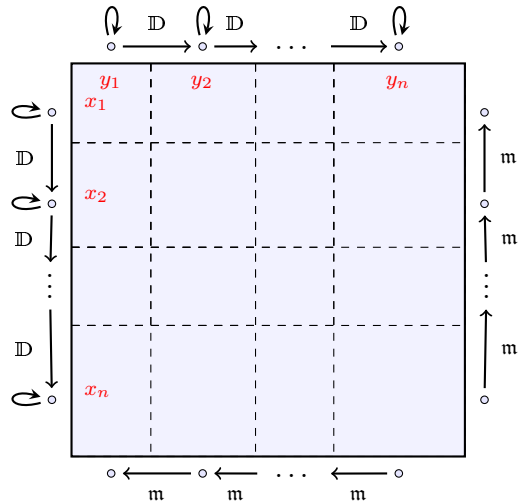
\begin{center}
\begin{tikzpicture}
[ dot/.style={fill=blue!10,circle,draw, inner sep=1pt, minimum size=1pt},
str/.style={inner sep=1pt, minimum size=0pt}
]

\matrix (first) [tbl5,  name=tbl,
row 1/.style={minimum height=30pt},
row 2/.style={minimum height=40pt},
row 3/.style={minimum height=30pt},
row 4/.style={minimum height=50pt},
column 1/.style={text width=30pt},
column 2/.style={text width=40pt},
column 3/.style={text width=30pt},
column 4/.style={text width=50pt},
] at (0,0)
{
~ & ~ & ~ & ~                \\
~ & ~ & ~ & ~                \\~ & ~ & ~ & ~                \\~ & ~ & ~ & ~                \\%
};
\sfrm{tbl}{4}{4};

\foreach \x in {1,..., 4}{\foreach \y in {1,..., 4}{
\hdline{tbl}{\x}{\y}; \vdline{tbl}{\x}{\y};
}

\ifnum\x =3
\node[above=1pt of tbl-1-\x, outer sep=3pt, text depth= -0.5ex, text height=0pt](a\x){$\dots$};
\node[below=1pt of tbl-4-\x, outer sep=3pt, text depth= -0.5ex, text height=0pt](b\x){$\dots$};
\node[left=0pt of tbl-\x-1, outer sep=3pt, text depth= -0.5ex, text height=7pt](l\x){$\vdots$};
\node[right =0pt of tbl-\x-4, outer sep=3pt,text depth= -0.5ex,text height=7pt](r\x){$\vdots$};
 \else
{
\node[above=2pt of tbl-1-\x, dot, outer sep=3pt](a\x){};
\node[below=2pt of tbl-4-\x, outer sep=3pt, dot](b\x){};
\node[base left=3pt of tbl-\x-1, outer sep=3pt, dot](l\x){};
\node[base right =3pt of tbl-\x-4, outer sep=3pt, dot](r\x){};
}
\fi


\ifnum\x=1
\else
 \pgfmathsetmacro\i{\x-1};
 \draw[thick] (a\i)[ ->] to node[above =2pt]{$_\bD$}  (a\x);
 \draw[thick, <-] (b\i) to node[below =2pt]{$_\idm$} (b\x);
  \fi
}

\node[below= 1pt of tbl-1-1.north,red ]{$_{y_1}$};
\node[below= 1pt of tbl-1-2.north,red ]{$_{y_2}$};
\node[below= 1pt of tbl-1-4.north,red ]{$_{y_n}$};
\node[base right=1pt of tbl-1-1.west,red]{$_{x_1}$};
\node[base right=1pt of tbl-2-1.west,red]{$_{x_2}$};
\node[base right=1pt of tbl-4-1.west,red]{$_{x_n}$};

\draw[-stealth,thick] (l1)to[loop left]  node{} ();
\draw[-stealth,thick] (a1)to[loop above]  node{} ();
\draw[-stealth,thick] (l2)to[loop left]  node{} ();
\draw[-stealth,thick] (a2)to[loop above]  node{} ();
\draw[-stealth,thick] (l4)to[loop left]  node{} ();
\draw[-stealth,thick] (a4)to[loop above]  node{} ();

\draw[thick,->] (l1)--node[left=2pt]{$_\bD$} (l2); \draw[thick,->] (l2)--node[left=2pt]{$_\bD$}(l3);
\draw[thick,->] (l3)--node[left=2pt]{$_\bD$}(l4);
\draw[thick,->] (r4)--node[right=2pt]{$_\idm$} (r3); \draw[thick,->] (r3)--node[right=2pt]{$_\idm$}(r2);
\draw[thick,->] (r2)--node[right=2pt]{$_\idm$}(r1);

\end{tikzpicture}
\end{center}
In other words, the matrix problem we obtain is the following.
\begin{itemize}
\item For any $1 \le \imath \le n$ one can perform arbitrary elementary transformations of  rows of the $x_\imath$-th stripe and columns of the $y_{\sigma(\imath)}$-th stripe of $W$ with coefficients in
     $\bD$ such that modulo $\idm$ they are \emph{inverse} to each other.
\item For any $1 \le \imath < \jmath \le n$ one can add any $\bD$--multiple of any row of $x_\imath$-th stripe to any row
of $x_\jmath$-th stripe. Similarly, one can add any $\bD$--multiple of any column of $y_\imath$-th stripe to any  column
of $y_\jmath$-th stripe.
\item One can perform arbitrary elementary transformations of rows and columns of $W$ with coefficients in $\idm$.
\end{itemize}
\end{example}

\begin{remark}
All bimodule problems from this subsection belong to the class of representations of decorated bunches of chains, which will be introduced in the next subsection. Other (more general) examples of bimodule problems, occurring in the representation theory of finite dimensional algebras and their applications, can be found in  \cite{DrozdLOMI, CB, dr} as well as
\cite{bd,  vb, DrozdGreuelBundles}.
\end{remark}

\subsection{Definition of a decorated bunch of chains} We start with   the following combinatorial data.
\begin{itemize}
\item Let $I$ be a set (usually finite or countable).
\item For any $\imath \in I$ we have a pair of totally ordered sets (chains) $\dE_\imath$ and $\dF_\imath$. All these sets are disjoint:  $\dE_\imath \cap \dE_\jmath = \dF_\imath \cap \dF_\jmath = \emptyset$
     for all $\imath \ne \jmath$ and $\dE_\imath \cap \dF_\jmath = \emptyset$ for all
     $\imath, \jmath \in I$.
\item We denote $\dE = \cup_{\imath \in I} \dE_\imath$, $\dF = \cup_{\imath \in I} \dF_\imath$ and $\dX = \dE \cup \dF$. Is this way, $\dX$ becomes a partially ordered set. We use the notation $x < y$ for the partial order in  $\dX$.
    If $x, y \in \dX$ are such that $x \in \dE_\imath$ and $y \in \dF_\imath$ (or vice versa) for some $\imath \in I$ then we write  $x - y$ and say that $x$ and $y$ are ``$-$'' related. Elements of $\dE$ (respectively $\dF$) are called \emph{row elements} (respectively \emph{column elements}).
\item Next, we have a relation $\sim$  on $\dX$ such that for any $x \in \dX$ there exists at most one $x' \in \dX$ such that $x \sim x'$. Here, we only consider irreflexive relations, i.e.~$z \not\sim z$ for any $z \in \dX$. An element $x$ admitting an equivalent element is called \emph{tied}.
\item Finally, we have a suborder $\trianglelefteq$ of $\le$ on $\dX$ which fulfils the following two conditions.
\begin{itemize}
\item If $ x \le y \le z$ in $\dX$ and $x \unlhd z$ then $x \unlhd y$ and $y \unlhd z$.
\item If $x \unlhd x$ (such element is called \emph{decorated}) and $x\sim y$ then $y \unlhd y$.
\end{itemize}
\end{itemize}

\begin{definition}
The entire data $\dX = \bigl(I, \{\dE_\imath\}_{\imath \in I}, \{\dF_\imath\}_{\imath \in I},\unlhd, \sim\bigr)$ is called
\emph{decorated bunch of chains}. In absence of decorated elements,  $\dX$ is a usual bunch of chains in the sense
of  \cite{bo, bo1}, see also \cite{bd, vb, nar}.
\end{definition}

\begin{definition}\label{bunch}
Let $\dX$ be a decorated bunch of chains. Then it defines a category $\kA =\kA(\dX, \bD)$ and an $\kA$--bimodule $\kB = \kB(\dX, \bD)$ in the following way.
\begin{itemize}
\item For any $x < y$ as well as $x \trianglerighteq y$ introduce the symbol $p_{yx}$.
\item Next, for $x, y \in \dX$ we introduce the following $\bD$--module $\kA_{xy}$:
\begin{equation*}
\kA_{xy} = \left\{
\begin{array}{ccl}
\rK p_{yx} & \mbox{if} & x < y \; \mbox{and} \; x \not\vartriangleleft y \\
\bD p_{yx} & \mbox{if} & x \dec y  \\
\idm p_{yx} & \mbox{if} & y \unlhd x  \\
0 & \mbox{otherwise}.
\end{array}
\right.
\end{equation*}
\item Now  we pose $\Ob(\kA) = \tilde\dX : = \dX/\sim$.
\item  The sets of morphisms in $\kA$ are the following $\bD$--modules:
$$
 \kA(a,b)= \left\{
 \begin{array}{lcl}
 \bop_{x \in a, y\in b} \kA_{xy} &\text{if }&  a\ne b,\\
 \rK 1_a\+\bop_{x,y\in a} \kA_{xy} &\text{if } & a= b,\, a \text{ is not decorated}\\
 (\bD 1_a\+\bop_{x,y\in a} \kA_{xy})/t(1_a-{\sum_{x\in a}}p_{xx}) &\text{if } & a= b,\, a \text{ is decorated}.
 \end{array}
\right.
$$
\item   The composition of morphisms  in $\kA$ is defined by the rule $p_{xy}p_{yz}=p_{xz}$ for any $x, y, z \in \dX$, all
other products \lb if defined\rb\ are zero.

\item
For any $\imath \in I$, $x \in \dE_\imath$ and
$y \in \dF_\imath$ introduce the symbol $\phi_{xy}$ and put
\[
 \kB(a,b)=\bop_{\substack{y\in a\cap\dF,\\x\in b\cap\dE,\\ x-y}} \rK \phi_{xy}.
\]

\item  The action of $\kA$ on $\kB$ is given by the rules
$$
\left\{
\begin{array}{ccc}
 p_{xy} \phi_{yz} &= & \phi_{xz},\\
\phi_{xy} p_{yz} &= & \phi_{xz},
\end{array}
\right.
$$
all other products (if defined) are zero.
\end{itemize}
In what follows, we shall use the notation $\Rep(\dX) = \el\bigl(\kA(\dX, \bD), \kB(\dX, \bD)\bigr)$. \qed
\end{definition}

\begin{example}\label{E:typicalEx}
Consider the decorated bunch of chains given by the following data.
\begin{itemize}
\item The index set $I = \{1, 2\}$. For $\imath \in I$ we have: $\dF_\imath = \{a_\imath\}$, $\dE_\imath =
\{b_\imath, c_\imath\}$.
\item We have $b_\imath  \dec c_\imath$ for $\imath \in I$.   In particular, $b_1, b_2, c_1, c_2$ are decorated. The elements $a_1$ and $a_2$ are not decorated.
\item We have the following equivalence relations: $a_1 \sim a_2$, $b_1 \sim b_2$ and $c_1 \sim c_2$.
\end{itemize}
This  decorated bunch of chains $\dX$  can be visualized by the following picture:

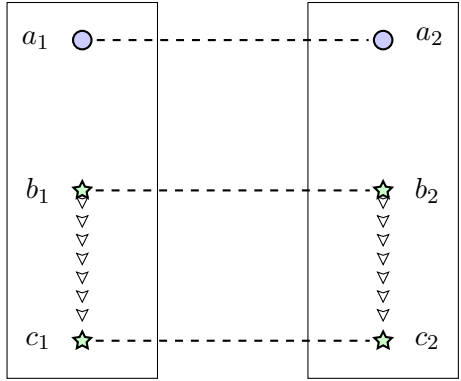
\begin{center}

\begin{tikzpicture}

\node[ve](y1) at (0,0){};
\node[dv](a1) at (0,-2){};
\node[dv](b1) at (0,-4){};

\node[ve](y2) at (4,0){};
\node[dv](a2) at (4,-2){};
\node[dv](b2) at (4,-4){};

  \draw[deco](a1) --(b1);
  \draw[deco](a2) --(b2);

 \node[left = 3pt of y1]{$a_1$};
 \node[ left = 3pt of a1]{$b_1$};
 \node[left = 3pt of b1]{$c_1$};

 \node[base right = 3pt of y2]{$a_2$};
 \node[right = 3pt of a2]{$b_2$};
 \node[right = 3pt of b2]{$c_2$};

 \draw[conj](y1) to (y2);
 \draw[conj](a1) to (a2);
 \draw[conj](b1) to (b2);

 \draw(-1,-4.5)rectangle (1,0.5);
 \draw(3,-4.5)rectangle (5,0.5);
 \end{tikzpicture}
\end{center}
Up to an automorphism, the pair $(\kA, \kB)$ is the one we have considered in Example \ref{E:typDecBunch}. \qed
\end{example}

\begin{example}\label{E:DecorConjugation}
Let $I = \{\ast\}$, $\dE_\ast = \{e\}$, $\dF_\ast = \{f\}$, $e \sim f$ and $e, f$ are both decorated. Then $\Rep(\dX)$ is the bimodule category
described in Example \ref{E:decJordan} (decorated conjugation problem). \qed
\end{example}

\begin{example}\label{E:bunchKronecker}
Let $I = \{1, 2\}$, $\dE_\imath = \{x_\imath\}$ and $\dF_\imath = \{y_\imath\}$ for $\imath \in I$. Let $x_1, x_2$ be decorated and $y_1, y_2$ not decorated. We have: $x_1 \sim x_2$ and $y_1 \sim y_2$. Then the  corresponding bimodule category $\Rep(\dX)$ is the one considered in Example \ref{E:decKronecker} (decorated Kronecker problem). \qed
\end{example}

\begin{example}\label{E:DecorChessBoard}
Let $I = \{\ast\}$, $\dE_\ast = \{x_1 \dec  \dots \dec x_n \dec \dots \}$, $\dF_\ast = \{y_1 \ced \dots \ced y_n \ced \dots\}$, $x_\imath  \sim y_\imath$ for all $\imath \in \mathbb{N}$. Then $\Rep(\dX)$ is the bimodule category
described in Example \ref{E:decChessBoard} (decorated chessboard). \qed
\end{example}

\noindent
We conclude this subsection stating the Krull--Schmidt property of $\Rep(\dX)$ (which is actually true for
much more general class of bimodule problems).

\subsection{Matrix description of the category $\Rep(\dX)$}\label{SS:mpConcrete} Let $\dX = \bigl(I, \{\dE_\imath\}_{\imath \in I}, \{\dF_\imath\}_{\imath \in I},\unlhd, \sim\bigr)$  be a  {decorated bunch of chains}. Then the bimodule category $\Rep(\dX)$
admits the following ``concrete'' description.

\begin{itemize}
\item First, we take a function $\sd: \dX \lar \mathbb{Z}_{\ge 0}, \, x \mapsto \sd_x$, which has  finite support and factors through
the canonical projection $\dX \lar \tilde\dX$ (i.e.~$\sd_x = \sd_y$ if   $x \sim y$).
\item For any $\imath \in I$, $x \in \dE_\imath$ and $y \in \dF_\imath$  we take a matrix
$W_{xy}^{(\imath)} \in \Mat_{\sd_x \times \sd_y}(\rK)$.
\end{itemize}
Then the data $\Bigl(\sd, \bigl\{W_{xy}^{(\imath)}\bigr\}_{\imath \in I,  (x, y) \in \dE_\imath \times \dF_\imath}\Bigr)$ uniquely determine an object $(Z, W)$ of $\Rep(\dX)$:
\begin{itemize}
\item $Z = Z_{a_1}^{\sd_1} \oplus \dots \oplus Z_{a_n}^{\sd_n}$, where $\tilde\dX  \supseteq \{a_1, \dots, a_n\} = \mathsf{supp}(\sd)$. Here, $Z_a$ denotes  the object of $\kA^\omega$ corresponding
    to the element $a \in \tilde\dX$
     and $\sd_l := \sd_{a_l}$ for $1 \le l \le n$.
\item Assume that $x \in a_p\cap \dE_\imath$ and $y \in a_q\cap \dF_\imath$ for $\imath \in I$ and $1 \le p, q \le n$. The
$\bD$--module $\kB(Z, Z)$ has a direct summand $\kB\bigl(Z_{a_p}^{\sd_p}, Z_{a_q}^{\sd_q}\bigr)$  and  $W_{xy}^{(\imath)} p_{xy}$ is the  corresponding entry of the element $W \in \kB(Z, Z)$.
\end{itemize}
In these notations, the \emph{total dimension} of $(Z, W)$ is set to be
$
\dim\bigl((Z, W)\bigr) = \sum\limits_{x \in \dX} \sd_x.$

\begin{proposition}\label{P:morphismAsMatrix}
Let $(Z, W)$ and $(\check{Z}, \check{W})$ be two objects of $\Rep(\dX)$ given by the matrix data
$\Bigl(\sd, \bigl\{W_{xy}^{(\imath)}\bigr\}\Bigr)$ and  $\Bigl(\check\sd, \bigl\{\check{W}_{xy}^{(\imath)}\bigr\}\Bigr)$ respectively.
Then a morphism $(Z, W) \stackrel{h}\lar (\check{Z}, \check{W})$ in $\Rep(\dX)$ is given by a collection  of matrices
$\bigl\{F_{xu}^{(\imath)}\bigr\}_{\imath \in I,  x, u \in \dE_\imath}$, $\bigl\{G_{vy}^{(\imath)}\bigr\}_{\imath \in I,  v, y \in \dF_\imath}$
such that
\begin{itemize}
\item $F_{xu}^{(\imath)} \in \Mat_{\check{\sd}_x \times \sd_u}(A_{xu})$ and $G_{vy}^{(\imath)} \in \Mat_{\check{\sd}_v \times \sd_y}(A_{vy})$, where
    $$
    A_{xu} = \left\{
    \begin{array}{ccl}
    \rK & \text{\rm if} & u \le x \; \text{\rm and}\; u \not\trianglelefteq x \\
    \bD & \text{\rm if} & u \unlhd x \\
    \idm & \text{\rm if} & x \dec u \\
    0 &\text{\rm otherwise} &
    \end{array}
    \right.
    $$
 and
 $$
 B_{vy} = \left\{
    \begin{array}{ccl}
    \rK & \text{\rm if} & y \le v \; \mbox{\rm and}\; y \not\trianglelefteq v \\
    \bD & \text{\rm if} & y \unlhd v \\
    \idm & \text{\rm if} & v \dec y \\
    0 &\text{\rm otherwise}. &
    \end{array}
    \right.
    $$
\item $F_{xx} = F_{x' x'}$ (respectively $F_{xx} = G_{yy}$) if $x \sim x'$ (respectively $x \sim y$) and  $x$ is not decorated and $\bar{F}_{xx} = \bar{F}_{x' x'}$ (respectively $\bar{F}_{xx} = \bar{G}_{yy}$) if $x \sim x'$ (respectively $x \sim y$) and $x$ is  decorated;
\end{itemize}
and such that
for any $\imath \in I$ and $(x, y) \in \dE_\imath \times \dF_\imath$ the following equality is true:
\begin{equation}
\sum\limits_{u} F_{xu}^{(\imath)} W_{uy}^{(\imath)}  =
\sum\limits_{v} \check{W}_{xv}^{(\imath)} G^{(\imath)}_{vy}.
\end{equation}
The matrices $\Bigl(\bigl\{\check{F}_{xu}^{(\imath)}\bigr\}, \bigl\{\check{G}_{vy}^{(\imath)}\bigr\}\Bigr)$ corresponding to the composition $\tilde{h} \circ h$ of  $h = \Bigl(\bigl\{F_{xu}^{(\imath)}\bigr\}$, $\bigl\{G_{vy}^{(\imath)}\bigr\}\Bigr)$ and $\tilde{h} = \Bigl(\bigl\{\tilde{F}_{xu}^{(\imath)}\bigr\}$, $\bigl\{\tilde{G}_{vy}^{(\imath)}\bigr\}\Bigr)$  are  given by the usual matrix product:
$$
\check{F}^{(\imath)}_{xu} = \sum_{c \in \dE_\imath} \tilde{F}_{xc}^{(\imath)} F_{c u}^{(\imath)} \quad
\mbox{\rm and} \quad
\check{G}^{(\imath)}_{vy} = \sum_{c \in \dF_\imath} \tilde{G}_{vc}^{(\imath)} G_{c y}^{(\imath)}.
$$
\end{proposition}

\begin{proof}
It is a straightforward computation, analogous to the one made in Example \ref{E:typDecBunch}.
\end{proof}

\begin{remark}
``Directedness'' of  $\kA$ implies that a morphism  $h = \Bigl(\bigl\{F_{xu}^{(\imath)}\bigr\}$, $\bigl\{G_{vy}^{(\imath)}\bigr\}\Bigr)$ is an isomorphism if and only if all diagonal blocks $F_{xx}^{(\imath)}$ and
$G_{yy}^{(\imath)}$ are invertible.
\end{remark}

\begin{definition}
Let $X = (Z, W)$ and $\tilde{X} = (\tilde{Z}, \tilde{W})$ be two objects of $\Rep(\dX)$. Consider the following $\bD$--module:
$
\kI(X, \tilde{X}) := \Rep(\dX)\bigl(X, \tilde{X}) \cap \rad\bigl(\kA(Z, \tilde{Z})\bigr).
$
Then $\kI$ is an ideal in the category $\Rep(\dX)$. The quotient  category
$\underline{\Rep}(\dX) := \Rep(\dX)/\kI$ is called \emph{stabilized bimodule category}.
\end{definition}

\begin{remark} Since $\kI(X, \tilde{X}) \subseteq \rad\bigl(\Rep(\dX)\bigl(X, \tilde{X})\bigr)$,
the projection  functor
$$\Rep(\dX) \stackrel{\Pi}\lar  \underline{\Rep}(\dX)$$ preserves indecomposability and isomorphy classes of objects.
Let $h, \hat{h} \in \Hom_{\dX}(X, \tilde{X})$ be given by
$h = \Bigl(\bigl\{F_{xu}^{(\imath)}\bigr\}, \bigl\{G_{vy}^{(\imath)}\bigr\}\Bigr)$ and $\hat{h} = \Bigl(\bigl\{\hat{F}_{xu}^{(\imath)}\bigr\}, \bigl\{\hat{G}_{vy}^{(\imath)}\bigr\}\Bigr)$. Then
 $h - \hat{h} \in \kI(X, \tilde{X})$ if and only if $F_{xx}^{(\imath)} \simeq \hat{F}_{xx}^{(\imath)}$
and $G_{yy}^{(\imath)} \simeq \hat{G}_{yy}^{(\imath)}$ for all $\imath \in I$, $x \in \dE_\imath$ and
$y \in \dF_\imath$, where $\simeq$ means equality if $x$ (respectively $y$) is not decorated and equality modulo $\idm$ if $x$ (respectively $y$) is  decorated.
\end{remark}

 \subsection{Strings and Bands}\label{SS:StringsBands} Let $\dX$ be a decorated bunch of chains.
 To present a description of the isomorphism classes of indecomposable objects of  $\rep(\dX)$, we use
the combinatorics of \emph{strings and bands}, just as for ``usual'' (non--decorated)
bunches of chains \cite{bo1}.

\medskip
\noindent
1.~We define an $\dX$\emph{--word}  as a sequence
$w=x_1r_1\dots
x_{l-1}r_{l-1}x_l$, where $x_i\in\dX$, $r_i\in \bigl\{\sim,-\bigr\}$ and the following conditions
hold:
 \begin{itemize}
\item   $x_ir_ix_{i+1}$ in $\dX$ for each $i\in \bigl\{1,2,\dots,l-1\bigr\}$.

\item   $r_i\ne r_{i+1}$ for each $i\in \bigl\{1,2,\dots,l-2\bigr\}$.

\item  If $x_1$ is tied, then $r_1=\sim$, and if $x_l$ is tied, then $r_{l-1}=\sim$.
\end{itemize}
 We call $l$ the \emph{length} of the word $w$ and denote it by $l(w)$. We also denote
$\tau(w)=\bigl\{\,i \mid 1\le i<l,\,r_i=\!-\bigr\}$. The word $w$ is said to be \emph{decorable} if
at least one of the letters $\lst xl$ is decorated.

\medskip
\noindent
2.~A \emph{decoration} of a \textsl{decorable} word $w$ is a function $\rho:\tau(w)\to\mZ$. A (unique) decoration
of a \textsl{non--decorable} word is, by definition, the constant function $\rho:\tau(w)\to \{0\}$.

\medskip
\noindent
3.~Two decorations $\rho,\rho':\tau(w)\to\mZ$ of a decorable word $w$ are said to be \emph{neighbour} if there is
an index $i\in\tau(w)$ and an integer $k$ such that  $x_i\not\sim x_{i+1}$ and
\begin{itemize}
\item  either $x_i$ is not decorated, $\rho'(i)=\rho(i)+k$ and, if $i>2$, also
$\rho'(i-2)=\rho(i-2)-(-1)^{\sigma(x_{i},x_{i-1})}k$, where $\sigma(x,y)=1$ if both $x,y$
are either row or column labels and $\sigma(x,y)=0$
if one of them is a row label and the other is a column label.

\item  or $x_{i+1}$ is not decorated, $\rho'(i)=\rho(i)+k$, and, if $i<n-2$, also
$ \rho'(i+2)=\rho(i+2)-(-1)^{\sigma(x_{i+1},x_{i+2})}k$.
\end{itemize}
 Two decorations $\rho,\rho'$ are said to be \emph{equivalent} if there is a sequence of
decorations $\rho=\rho_1,\rho_2,\dots,\rho_r=\rho'$ such that $\rho_i$ and $\rho_{i+1}$
are neighbour for $1\le i<r$.

\medskip
\noindent
4.~We denote by $w^*$ the \emph{inverse word} to $w$, i.e. the word
$w^*=x_lr_{l-1}x_{l-1}\dots r_2x_2r_1x_1$. If $\rho$ is a decoration of $w$, we define
the decoration $\rho^*$ of $w^*$ setting $\rho^*(i)=\rho(l-i)$.

\medskip
\noindent 5.~An $\dX$--word $w$ of length $l$ is called \emph{cyclic} if $r_1=r_{l-1}=\sim$ and
$x_l-x_1$ in $\dX$.
 For such a cyclic word we set $r_l=\!-$ and define $x_i,r_i$ for all $i\in\mZ$ setting
$x_{i+ql}=x_i,\,r_{i+ql}=r_i$ for any $q\in\mZ$. In particular, $x_0=x_l$ and $r_0=\!-$.
Note that the length of a cyclic word is always even. We also set
$\tau^+(w)=\tau(w)\cup\{l\}$.

\medskip
\noindent
 5.~A \emph{cyclic decoration} of a \textsl{decorable} cyclic word $w$ is a function
$\rho:\tau^+(w)\to\mZ$.  For such a function we
set $\rho(i+ql)=\rho(i)$ for any $q\in\mZ$. A (unique) cyclic decoration of a
\textsl{non--decorable} cyclic word $w$ is, by definition, the constant function
$\rho:\tau^+(w)\to\{0\}$.

\medskip
\noindent
6.~If $w$ is a cyclic word, we define its \emph{$k$--shift} as the word
\begin{equation*}
w^{(k)}=x_{2k+1}r_{2k+1}r_{2k+2}\dots x_{2k-1}r_{2k-1}x_{2k}
\end{equation*} and write
\begin{equation}\label{sigma}
  \sigma(k,w)=\sum_{j=1}^k\sigma(x_{2j-1},x_{2j}).
\end{equation}
 If $\rho$ is a cyclic decoration of $w$, we define the
cyclic decoration $\rho^{(k)}$ of $w^{(k)}$ setting $\rho^{(k)}(i)=\rho(i-2k)$.

\medskip
\noindent
7.~We call a pair $(w,\rho)$, where $w$ is a cyclic word of length $l$ and
$\rho$ is its cyclic decoration, \emph{periodic} if $w^{(k)}=w$ and $\rho^{(k)}=\rho$ for
some $k<l/2$.

\medskip
\noindent
8.~Two cyclic decorations $\rho,\rho':\tau^+(w)\to\mZ$ of a decorated cyclic word are said
to be \emph{neighbour} if there is an index $i\in\tau^+(w)$ and an integer $k$ such that
$x_i\not\sim x_{i+1}$ and
\begin{itemize}
\item  either $x_i$ is not decorated, $$\rho'(i)=\rho(i)+k \; \mbox{and} \;
\rho'(i-2)=\rho(i-2)-(-1)^{\sigma(x_{i},x_{i-1})}k,$$
\item  or $x_{i+1}$ is not decorated, $$\rho'(i)=\rho(i)+k  \; \mbox{and} \;
 \rho'(i+2)=\rho(i+2)-(-1)^{\sigma(x_{i+1},x_{i+2})}k.$$
\end{itemize}
Two cyclic decorations $\rho,\rho'$ are said to be \emph{equivalent} if there is a
sequence of cyclic decorations $\rho=\rho_1,\rho_2,\dots,\rho_r=\rho'$ such that $\rho_i$
and $\rho_{i+1}$ are neighbour for $1\le i<r$.

\medskip
\noindent
9.~A pair $(w, \rho)$, where $w$ is a (cyclic) word and $\rho$ is its (cyclic) decoration, is called
a \emph{decorated (cyclic) word}. Two decorated cyclic words $(w, \rho)$ and $(w', \rho')$ are said to be \emph{equivalent}
if $w = w'$ and the decorations $\rho$ and $\rho'$ are equivalent.

\medskip
\noindent 10.~A \emph{$\sim$subword} of a word $w$ is its subword $v$ of the form $x\sim y$ or
of the form $x$ if $x\nsim y$ for any $y\ne x$. We denote by $|x|$ the  class of
$x$ with respect to $\sim$
and by $[w]$ the set of all $\sim$subwords of $w$. Note that if $w$ is cyclic,
every its $\sim$subword is of the form $x\sim y$. Thus any word has the form
$v_1-v_2-\dots-v_n$, where $\lst vn$ are its $\sim$subwords. If it is cyclic, then
$w^{(k)}=v_{k+1}-v_{k+2}-\dots-v_{k-1}-v_k$. A decoration of the word $w$ is given by the
sequence $\bnu=\row \nu{n-1}$ of its values and written as
$v_1\str{\nu_1}-v_2\str{\nu_2}-\dots\str{\nu_{n-1}}-v_n$. A cyclic decoration is given by the
sequence of its values $\bnu=\row \nu n$ and written as
${}\lha v_1\str{\nu_1}-v_2\str{\nu_2}-\dots
\str{\nu_{n-1}}-v_n\str{\nu_n}{\rha}$. \qed

\medskip
\noindent
 Now we introduce  the following objects of the bimodule category $\Rep(\dX)$.

\begin{definition}[\sf Strings]\label{string}
 Let $w$ be an $\dX$--word, $\rho:\tau(w)\to\mZ$ be its decoration.
 The \emph{string representation} $S(w,\rho) = (Z, S)$ is defined as follows:
\begin{itemize}
\item  $Z=\bop_{v\in[w]} Z_{|v|}$  and $S\in\kB(Z, Z)$.

\item  Suppose that the decorated word $(w, \rho)$ has a subword
$v_i\str{\nu_i}-v_{i+1}$ for some
$v_i,v_{i+1}\in[w]$.
Then $\kB(Z, Z)$ has a direct summand $\kB\bigl(Z_{|v_{i+1}|},Z_{|v_i|}\bigr)
\oplus \kB\bigl(Z_{|v_{i}|}, Z_{|v_{i+1}|}\bigr)$ and we
define the corresponding component of $S$ as
\begin{itemize}
\item $t^{\nu_i} \phi_{v_{i+1} v_i} \in \kB\bigl(Z_{|v_{i}|}, Z_{|v_{i+1}|}\bigr)$ if $v_{i+1} \in \dE$
and $
v_i \in \dF$,
\item
$t^{\nu_i} \phi_{v_{i} v_{i+1}} \in \kB\bigl(Z_{|v_{i+1}|}, Z_{|v_{i}|}\bigr)$ if $v_{i} \in \dE$
and $
v_{i+1} \in \dF$.
\end{itemize}
\item  All other components of $S$ are set to be zero.
\end{itemize}
\end{definition}

\begin{definition}[\sf Bands]\label{band}
 Let $w$ be a cyclic $\dX$--word of length $l=2n$, $\rho:\tau^+(w)\to\mZ$ be its
cyclic decoration, such that the decorated word $(w,\rho)$ is not equivalent to any periodic
one, $m\in\NN$ and $\pi = \pi(\xi)\ne \xi$ be an irreducible polynomial of degree $d$  from $\aK[\xi]$ if $w$
is decorable and over $\rK[\xi]$ if $w$ is not decorable. If $w$ is
not decorable, denote by $F = F(\pi^m) \in \Mat_{dm\xx dm}(\rK)$ the Frobenius block corresponding to the
polynomial $\pi^m$. If $w$ is decorable, denote by $F$ a matrix from
$\Mat_{dm\xx dm}(\bD)$ such that its image in $\Mat_{dm\xx dm}(\aK)$ is the Frobenius block
corresponding to  $\pi^m$. The \emph{band representation}
$B\bigl((w,\rho),m,\pi\bigr) = (Z, B)$ is defined as follows:
\begin{itemize}
\item  $Z=\bop_{v\in[w]} Z_{|v|}^{\oplus dm}$ and  $B\in\kB(Z,Z)$.

\item  Suppose that the decorated cyclic word $(w, \rho)$ has a subword
$v_i\str{\nu_i}-v_{i+1}$, where $v_i,v_{i+1}\in[w],\,1\le i<n$. Then $\kB(Z, Z)$ has a
direct summand $$\kB\bigl(Z_{|v_{i+1}|}^{\oplus dm}, Z_{|v_i|}^{\oplus dm}\bigr) \oplus
\kB\bigl(Z_{|v_{i}|}^{\oplus dm}, Z_{|v_{i+1}|}^{\oplus dm}\bigr)$$ and we define the
corresponding component of $B$ as follows:
\begin{itemize}
\item $
t^{\nu_i} \phi_{v_{i+1} v_i} I \in \kB\bigl(Z_{|v_{i}|}^{\oplus dm}, Z_{|v_{i+1}|}^{\oplus dm}\bigr)$
if $ v_{i} \in \dF$ and $
v_{i+1} \in \dE$,
\item
$
t^{\nu_i} \phi_{v_{i} v_{i+1}} I\in \kB\bigl(Z_{|v_{i+1}|}^{\oplus dm}, Z_{|v_{i}|}^{\oplus dm}\bigr)$ if $ v_{i+1} \in \dF$ and $
v_{i} \in \dE$,
\end{itemize}
 where $I$ is the identity $dm \times dm$ matrix.

\item  The component of $B$ corresponding to the direct summand
$$\kB\bigl(Z_{|v_1|}^{\oplus dm}, Z_{|v_n|}^{\oplus dm}\bigr) \oplus
\kB\bigl(Z_{|v_n|}^{\oplus dm}, Z_{|v_1|}^{\oplus dm}\bigr)
$$ of $\kB(Z, Z)$
 is defined as
 \begin{itemize}
\item $
t^{\nu_n} \phi_{v_{1} v_n} F \in \kB\bigl(Z_{|v_{n}|}^{\oplus dm}, Z_{|v_{1}|}^{\oplus dm}\bigr)$  if $ v_{n} \in \dF$ and $
v_{1} \in \dE$,
\item
$
t^{\nu_n} \phi_{v_{n} v_{1}} F \in \kB\bigl(Z_{|v_{1}|}^{\oplus dm}, Z_{|v_{n}|}^{\oplus dm}\bigr)$ if $ v_{1} \in \dF$ and $
v_{n} \in \dE$.
\end{itemize}

\item  All other components of $B$ are zero.
\end{itemize}
\end{definition}

\begin{example} Consider the decorated bunch of chains $\dX$ introduced in Example \ref{E:typicalEx}.
Let $a \in \tilde\dX$ (respectively $b, c \in \tilde\dX$) be the equivalences class of $a_1, a_2 \in \dX$ (respectively,
$b_1, b_2; c_1, c_2 \in \dX$). Consider the following decorated cyclic word:
$$
(w, \rho) := {}\lha a_1 \sim a_2 \str{l_1}- b_2 \sim b_1 \str{l_2}- a_1 \sim a_2 \str{l_3}- c_2 \sim c_1 \str{l_4}-a_1
\sim a_2 \str{l_5}- b_2 \sim b_1 \str{l_6}\rha
$$
Let $m \in \mathbb{N}$ and $\xi \ne \pi \in \kk[\xi]$ an irreducible polynomial of degree $d$.  Then the band object $B\bigl((w, \rho), m, \pi\bigr)$ is given by the canonical form
\begin{center}
  \begin{tikzpicture}

\matrix (first) [tbl5,text width=40pt, minimum height=40pt,  name=table]
{
0   &  t^{l_2}I  & 0 \\
 t^{l_6} F    &  0 & 0  \\[0.5ex]
   0 & 0 & t^{l_4} I      \\
};
\sfrm{table}{2}{3}; \node[draw, thick, fit=(table-3-1)(table-3-3), inner sep=0pt]{};

\node[above=3pt of table-1-2]{$a_1$};
\draw(table-1-3) to node[midway,base right=23pt]{$b_1$}  (table-2-3);
\node[base right=3pt of table-3-3]{$c_1$};

\draw(table-1-1) to node[midway,base left=30pt]{$\Phi_1\;=$}  (table-2-1);
\node[base left=10pt of table-3-1]{$\Psi_1\;=$};
\foreach \x in {1,2,3}{
\foreach \y in {1,2,3}{
\hdline{table}{\x}{\y};
\vdline{table}{\x}{\y};
}}

\matrix (second) [tbl5,text width=40pt, minimum height=40pt,  name=table]
at(6,0)
{
 t^{l_1}I  & 0&0 \\
 0 & 0 & t^{l_2}I    \\[0.5ex]
   0 &  t^{l_3}I  &0   \\
};
\sfrm{table}{2}{3};
\node[draw, thick, fit=(table-3-1)(table-3-3), inner sep=0pt]{};

\node[above=3pt of table-1-2]{$a_2$};
\draw(table-1-1) to node[midway,base left=23pt]{$b_2$}  (table-2-1);
\node[base left=3pt of table-3-1]{$c_2$};

\draw(table-1-3) to node[midway,base right=30pt, text depth=5pt,]{$=\;\Phi_2$}  (table-2-3);
\node[base right=10pt of table-3-3]{$=\;\Psi_2$};

\foreach \x in {1,2,3}{
\foreach \y in {1,2,3}{
\hdline{table}{\x}{\y};
\vdline{table}{\x}{\y};
}}
\end{tikzpicture}
\end{center}
In the language of bimodule problems, $B\bigl((w, \rho), m, \pi\bigr)$ is given by the pair  $(Z, B)$, where
$$Z = Z_{a}^{\oplus dm} \oplus Z_b^{\oplus dm} \oplus Z_{a}^{\oplus dm} \oplus Z_{c}^{\oplus dm} \oplus Z_{a}^{\oplus dm} \oplus
Z_{b}^{\oplus dm}$$
and $B \in \kB(Z, Z)$ is given by the following matrix

\begin{center}
\begin{tikzpicture}
\matrix (first) [table,text width=40pt,  name=table]
{
0               & 0          &0             &0       & 0                &0\\
$t^{l_2}\phi_2I$& 0       &$t^{l_2}\phi_1I$ &0       &0                 &0\\
0               & 0          &0             &0       & 0                &0\\
0               &0          &$t^{l_3}\psi_2I$& 0       &$t^{l_4}\psi_1I$ &0       \\
0               & 0          &0             &0       & 0                &0\\
$t^{l_6}\phi_1 F$& 0        &0              &0    &$t^{l_5}\phi_2I$      &0\\
};

\node[draw,solid, thick, fit=(table-1-1)(table-6-6), inner sep=0pt]{};

\node[above=3pt of table-1-1]{$a$};
\node[above=3pt of table-1-2]{$b$};
\node[above=3pt of table-1-3]{$a$};
\node[above=3pt of table-1-4]{$c$};
\node[above=3pt of table-1-5]{$a$};
\node[above=3pt of table-1-6]{$b$};


\node[base left=3pt of table-1-1]{$a$};
\node[base left=3pt of table-2-1]{$b$};
\node[base left=3pt of table-3-1]{$a$};
\node[base left=3pt of table-4-1]{$c$};
\node[base left=3pt of table-5-1]{$a$};
\node[base left=3pt of table-6-1]{$b$};

\end{tikzpicture}
\end{center}

\noindent
where $I$ is  the identity matrix of size $dm \times dm$ and $F$ is the Frobenius block of $\pi^m$.
\end{example}


\section{Representations of decorated bunches of chains--II}

\noindent
The goal of this sections is to prove
 the following result.

\begin{theorem}\label{list} Let $\dX$ be a decorated bunch of chains.
Then the  description of indecomposable objects of $\Rep(\dX)$ is the following.
 \begin{itemize}
\item   Every string or band representation  is indecomposable and every indecomposable
object  of $\Rep(\dX)$  is isomorphic to some string or band representation.

\item  Any string representation is not isomorphic to any band representation.

\item  Two string representations $S(w,\rho)$ and $S(w',\rho')$ are isomorphic \iff
either $w=w'$ and $\rho$ and $\rho'$ are equivalent, or $w'=w^*$ and the functions
$\rho^*$ and $\rho'$ are equivalent.

\item
 The isomorphism class of a band representation $B\bigl((w,\rho),m,\pi\bigr)$ with decorable $w$
does not depend on the choice of the matrix $F(\pi^m)$.

\item  Two band representations $B\bigl((w,\rho),m,\pi\bigr)$ and
$B\bigl((w',\rho'),m',\pi'\bigr)$ are isomorphic
\iff either $w'=w^{(k)},\,m'=m,\,\pi'= {\check\pi}_{k,w}$ and $\rho'$ is equivalent to $\rho^{(k)}$
for some $k$, or $w'={w^{(k)}}^*,\,m'=m,\, \pi'= \check{\pi}_{k,w}$ and $\rho'$ is equivalent to
${\rho^{(k)}}^*$ for some $k$, where
\[
 \check{\pi}_{k,w}(\xi)=\begin{cases}
 \pi(\xi) &\text{\em if } \sigma(k,w) \text{ \em is even},\\
 \xi^{\deg(\pi)} \pi(1/\xi) &\text{\em if } \sigma(k,w) \text{ \em is odd},
\end{cases}
\]
see  \eqref{sigma} for the definition of the function  $\sigma(k,w)$ above.
\end{itemize}
\end{theorem}

\noindent
A special case of this result (decorated chessboard problem) is treated  in details in Section \ref{S:DecoratedConjugation}. This problem is notationally easier to handle, being at the same time
an important ingredient of the proof of general result. So an interested reader is advised to look at
the proof of Theorem \ref{T:JordanKronecker} first.

\begin{remark} It is obvious that  any
 $Y \in \Ob\bigl(\Rep(\dX)\bigr)$ admits a  direct sum decomposition
$$
Y \cong Y_1^{\oplus m_1} \oplus \dots \oplus Y_t^{\oplus m_t}
$$
with $Y_1, \dots, Y_t$ indecomposable and pairwise non--isomorphic.
It follows from the proof of Theorem \ref{list} that the endomorphism ring
of an indecomposable object of $\Rep(\dX)$ is local. According to \cite[Chapter I.3.6]{ba},
the category $\Rep(\dX)$ is has Krull-Schmidt property, i.e.~the set $\bigl\{(Y_1, m_1) \dots, (Y_t, m_t)\bigr\}$
is uniquely determined by $Y$ (up to  isomorphisms  of indecomposable summands).
\end{remark}

\subsection{Idea of the proof} In  what follows we use
the notation of Subsection \ref{SS:mpConcrete}.
Let $(Z, W) = \bigl(\sd, \{W^{(\imath)}\}_{uv}\bigr)$ be an object of $\Rep(\dX)$.
Replacing if necessary, $\dX$ by $\dX \cap \mathsf{supp}(\sd)$, we may  without loss
of generality assume $\dX$ is a finite set. We follow the following convention: if
$\dE_\imath = \bigl\{a_1 < \dots < a_r\bigr\}$ and $\dF_\imath = \bigl\{b_1 > \dots > b_s\bigr\}$ then we
write the matrix $W^{(\imath)}$ as follows:

\begin{center}
\begin{tikzpicture}
\matrix (first) [tbl5,  name=table,
row 1/.style={minimum height=20pt},
row 2/.style={minimum height=30pt},
row 3/.style={minimum height=20pt},
row 4/.style={minimum height=40pt},
column 1/.style={text width=30pt},
column 2/.style={text width=20pt},
column 3/.style={text width=40pt},
column 4/.style={text width=20pt},
] at (0,0)
{
~&~& ~&   ~  &        \\
~&~& ~&   ~  &        \\
~&~& ~&   ~  &        \\
~&~& ~&   ~  &        \\
};

\node[draw,solid, very thick, fit=(table-1-1)(table-4-4), inner sep=0pt]{};

\foreach \x in {1,..., 4}{
\foreach \y in {1,..., 4}{
\hsline{table}{\x}{\y};
\vsline{table}{\x}{\y};
}}

\node[above=2pt of table-1-1]{$b_1$};
\node[above=2pt of table-1-2]{$b_2$};
\node[above=2pt of table-1-3]{$\dots$};
\node[above=2pt of table-1-4]{$b_s$};

\node[base left=3pt of table-1-1]{$a_1$};
\node[base left=3pt of table-2-1]{$a_2$};
\node[base left=3pt of table-3-1]{$\vdots$};
\node[base left=3pt of table-4-1]{$a_r$};

\node[base right=3pt of table-1-4](r1){};
\node[base right=3pt of table-4-4](r2){};

\node[below=2pt of table-4-1](b1){};
\node[below=3pt of table-4-4](b2){};

\draw[very thick, -stealth] (r1) to (r2);
\draw[very thick, -stealth] (b1) to (b2);

\end{tikzpicture}
\end{center}
For $\imath \in I$ we say
that $\scX = \left\{x_1,\dots, x_m \right\} \subseteq \dE_\imath$ (respectively,
$\scY = \left\{y_1,\dots, y_m \right\} \subseteq \dF_\imath$) is a \emph{maximal elementary  subchain}
if either $\scX = \{x\}$ and $x$ is not decorated or $x_1 \lhd \dots \lhd x_m$ and $\scX$ is maximal
with respect to this property. In the latter case we say $\scX$ is decorated. Maximal elementary  subchains inherit $<$ ordering from $\dX$, thus
both  sets $\dE_\imath$ and $\dF_\imath$ split into a union of such subchains.
On the set of pairs
$$
 \kB_\imath := \left\{(\scX, \scY)\,\big|\, \scX \;\mbox{and}\;  \scY \;\mbox{are maximal elementary  subchains in}\, \dE_\imath,\; \mbox{respectively in}\; \dF_\imath\right\}
$$
we introduce the following total ordering:
$
(\scU, \scV) < (\scX, \scY)$ if either $\scU < \scX$ or $\scU = \scX$ and $\scV > \scY$. In the next,
we shall use the notation $\bigl(\sd, \{W^{(\imath)}\}_{\scX\scY}\bigr)$ for an object
$(Z, W)$ of  $\Rep(\dX)$. A morphism  $(Z, W) \stackrel{h}\lar (\check{Z}, \check{W})$ in $\Rep(\dX)$
is given by a collection  of matrices
$F = \bigl\{F_{\scU\scX}^{(\imath)}\bigr\}$ and
 $G =\bigl\{G_{\scV\scY}^{(\imath)}\bigr\}$ satisfying
the equality
\begin{equation}\label{E:keyeqnew}
\sum\limits_{\scU \subseteq \dE_\imath} F_{\scX\scU}^{(\imath)} W_{\scU\scY}^{(\imath)}  =
\sum\limits_{\scV \subseteq \dF_\imath} \check{W}_{\scX\scV}^{(\imath)} G^{(\imath)}_{\scV\scY}
\end{equation}
for any $\imath \in I$ and $(\scX, \scY) \in \kB_\imath$, as well as some additional constraints on diagonal blocks described
in Proposition \ref{P:morphismAsMatrix}.
Note that $F_{\scX\scU}^{(\imath)} = 0$ (respectively $G^{(\imath)}_{\scV\scY} = 0$) for all
$\scX < \scU$ (respectively $\scV < \scY$). In particular, equation (\ref{E:keyeqnew})
takes  the form
\begin{equation}\label{E:keyeqkey}
\sum\limits_{\scU < \scX} F_{\scX\scU}^{(\imath)} W_{\scU\scY}^{(\imath)}  +
F_{\scX\scX}^{(\imath)} W_{\scX\scY}^{(\imath)}
 =
\sum\limits_{\scV < \scY } \check{W}_{\scX\scV}^{(\imath)} G^{(\imath)}_{\scV\scY} +
\check{W}_{\scX\scY}^{(\imath)} G^{(\imath)}_{\scX\scY}.
\end{equation}
Next, observe that for  $\scX = \{x_1 \lhd \dots \lhd x_m\} \subseteq \dE_\imath$ the diagonal
block  $F_{\scX\scX}^{(\imath)}$ of $F^{(\imath)}$ splits further into subblocks $F_{x_p x_q}^{(\imath)}$ of size
$\sd_{x_p} \times \sd_{x_q}$ with coefficients in $\bD$ for $1 \le q \le p \le m$ and $\idm$ for $1 \le p<q \le m$. A similar statement holds for the diagonal blocks $G_{\scY\scY}$ of the matrix
$G^{(\imath)}$ for any decorated maximal elementary subchain  $\scY \subseteq \dF_\imath$.

\medskip
\noindent
For $(\scX, \scY) \in \kB_\imath$ consider the following full subcategory
of $\Rep(\dX)$:
\begin{equation}\label{E:StrataInRep(X)}
\mathrm{Ob}\bigl(\Rep^{\le(\scX, \scY)}(\dX)\bigr)  := \left\{(Z, W) \,\Big|\, W_{\scU\scV}^{(\imath)} = 0 \;\mbox{for}\; (\scU, \scV) < (\scX, \scY)\right\}.
\end{equation}
In a similar way, we define
$
\mathrm{Ob}\bigl(\Rep^{<(\scX, \scY)}(\dX)\bigr)  := \left\{(Z, W) \,\big|\, W_{\scU\scV}^{(\imath)} = 0 \;\mbox{for}\; (\scU, \scV) \le (\scX, \scY)\right\}.
$
If $(Z, W)$ and $(\check{Z}, \check{W})$ both belong to  $\Rep^{\le(\scX, \scY)}(\dX)$ and
$(Z, W) \stackrel{(F, G)}\lar (\check{Z}, \check{W})$ is a morphism then (\ref{E:keyeqkey}) implies that
\begin{equation}\label{E:keyeqkuku}
F_{\scX\scX}^{(\imath)} W_{\scX\scY}^{(\imath)}
 =
\check{W}_{\scX\scY}^{(\imath)} G^{(\imath)}_{\scY\scY}.
\end{equation}
Let $\dX_{(\scX, \scY)}$ be the decorated bunch of chains obtained by restriction of $\dX$ on $(\scX, \scY)$. Equality (\ref{E:keyeqkuku}) implies that  we have
the forgetful functor
\begin{equation}\label{E:RestrOnBlock}
\Rep^{\le(\scX, \scY)}(\dX) \lar \Rep(\dX_{(\scX, \scY)}), \quad (Z, W) \mapsto
W_{\scX\scY}^{(\imath)}.
\end{equation}

\medskip
\noindent
Now we are ready to present the main steps
 of the proof of Theorem \ref{list}.
\begin{itemize}
\item
Any object $(Z, W)$  of $\rep(\dX)$ belongs to some
subcategory $\rep^{\le (\scX, \scY)}(\dX)$ such that the component $W_{\scX \scY}^{(\imath)}$ is not zero for some
$\imath \in I$, $\scX \subseteq \dE_\imath$ and $\scY \subseteq \dF_\imath$.
\begin{center}
\begin{tikzpicture}
\matrix (first) [tbl5,  name=table,
row 1/.style={minimum height=20pt},
row 2/.style={minimum height=20pt},
row 3/.style={minimum height=20pt},
row 4/.style={minimum height=20pt},
column 1/.style={text width=20pt},
column 2/.style={text width=20pt},
column 3/.style={text width=20pt},
column 4/.style={text width=20pt},
column 5/.style={text width=20pt},
] at (0,0)
{
0&0& 0&   0  & 0       \\
0&~& ~&   ~  &        \\
0&~& ~&   ~  &        \\
~&~& ~&   ~  &        \\
};

\node[draw, solid, very thick,  fill, color=red!30  , fit=(table-2-2)(table-3-4), inner sep=0pt]{};

\node[draw,solid, very thick, fit=(table-1-1)(table-4-5), inner sep=0pt]{};

\hsline{table}{1}{5};
\hdline{table}{2}{5};
\hsline{table}{3}{5};

\vsline{table}{4}{1};
\vdline{table}{4}{2};
\vdline{table}{4}{3};
\vsline{table}{4}{4};

\node[above=2pt of table-1-3]{$\scY$};
\node[base left=3pt of table-2-1.south west]{$\scX$};
\end{tikzpicture}
\end{center}
For simplicity we assume here that $\scX$ or $\scY$ is not decorated, otherwise the treatment requires additional notations.

\item We bring  the matrix
$W_{\scX\scY}^{(\imath)}$, viewed as object of $\Rep(\dX_{(\scX, \scY)})$, into
a \emph{normal form}. Then we transform the entire object $(Z, W)$ into a \emph{standard form}.
If  $\rep^{\le (\scX, \scY)}_{\st}(\dX)$  is the full subcategory
 of \emph{standard objects} (i.e.~objects in the standard form) then
the embedding $\rep^{\le (\scX, \scY)}_{\st}(\dX) \hookrightarrow \rep^{\le (\scX, \scY)}(\dX)$ is
an equivalence of categories.

\item In some cases (e.g.~if all elements of $\scX$ and $\scY$ are untied, or if $\scX = \{x\}$,
$\scY = \{y\}$ and $x \sim y$), certain  direct summands of $W_{\scX\scY}^{(\imath)}$ viewed as objects
of  $\Rep(\dX_{(\scX, \scY)})$ split up globally as direct summands of $(Z, W)$ in $\rep(\dX)$.
Denoting $\rep^{\le (\scX, \scY), \circ}_{\st}(\dX)$ the full subcategory of $\rep^{\le (\scX, \scY)}_{\st}(\dX)$ consisting of standard objects without such direct summands, we
construct a \emph{new bunch of chains} $\dX^{[\scX,\scY]}$ and a \emph{reduction functor} between the stabilized categories
\begin{equation}\label{E:ReductionCategorielle}
 R^{\scX\scY}: \underline{\rep}^{\le (\scX, \scY), \circ}_{\st}(\dX) \lar \underline{\rep}^{<(\scX, \scY)}(\dX^{[\scX, \scY]}).
\end{equation}
The  new bunch of chains $\dX^{[\scX,\scY]}$ is constructed from $\dX$  using an explicit computation
of the automorphism group of  a ``general'' object  in the category $\rep(\dX_{(\scX,\scY)})$.

\item
 The entire sequence of categories and functors introduced above
 can be  summarized  as follows:
\begin{equation*}
\underline{\Rep}(\dX) \hookleftarrow  \underline{\rep}^{\le (\scX, \scY)}(\dX)  \hookleftarrow  \underline{\rep}^{\le (\scX, \scY), \circ}(\dX) \stackrel{\simeq}\hookleftarrow \underline{\rep}^{\le (\scX, \scY), \circ}_{\st}(\dX)
\stackrel{R^{\scX\scY}}\lar \underline{\Rep}^{<(\scX, \scY)}(\dX^{[\scX,\scY]}).
\end{equation*}
The  reduction functor  $R^{\scX\scY}$  is a \emph{representation equivalence} of categories, i.e.
\begin{itemize}
 \item  $R^{\scX\scY}$ is essentially surjective,
 \item $R^{\scX\scY}\bigl((Z, W)\bigr) \cong R^{\scX\scY}\bigl((Z', W')\bigr)$ if and only if $(Z, W) \cong (Z', W')$.
 \end{itemize}
   These  two properties imply  that $R^{\scX\scY}$
maps
indecomposable objects to indecomposable ones.
Moreover, $R^{\scX\scY}$ reduces the total dimension of objects, what
 allows to use induction arguments.
 \item For any band datum $\bigl((w, \rho), m, \pi\bigr)$ and string datum $(v, \nu)$ in $\dX$ we have isomorphisms:
\begin{equation}\label{E:RuleToRewrite}
R^{\scX\scY}\Bigl(B\bigl((w, \rho), m, \pi\bigr)\Bigr) \cong B\bigl((\hat{w}, \hat\rho), m, \pi) \quad
\mbox{and} \quad
R^{\scX\scY}\bigl(S(v, \nu)\bigr) \cong S(\hat{v}, \hat\nu)
\end{equation}
for appropriate decorated (cyclic) words $(\hat{w}, \hat{\rho})$ and $(\hat{v}, \hat\nu)$. Moreover, the pair
$(w, \rho)$ (respectively  $(v, \nu)$) can be uniquely recovered from
 $(\hat{w}, \hat{\rho})$ (respectively  $(\hat{v}, \hat{\nu})$).
\end{itemize}

\subsection{Reduction Cases} In this subsection we give a proof of Theorem \ref{list}.
To simplify the notation, we  keep the  index $\imath \in I$ fixed, so it is no longer mentioned
when referring  for blocks of the matrix $W^{(\imath)}$ of an object
$(Z, W)$ of the category $\Rep^{\le (\scX, \scY)}(\dX)$ for  $\scX \times \scY \subseteq \dE_\imath \times \dF_\imath$.

\medskip
\noindent
\underline{\textbf{Case 1}}. We start with the case $\scX = \{x\}$ and $\scY = \{y\}$, where  both  $x$ and $y$
are not decorated. Note that $\dX_{(\scX, \scY)}$ is a usual (non--decorated) bunch of chains
over the field $\rK$.

\medskip
\noindent
\underline{Case 1a}. Assume first that $x \not\sim y$. Let $(Z, W)$ be an object of
${\rep}^{\le (\scX, \scY)}(\dX)$.  In the category   $\Rep(\dX_{(\scX, \scY)})$ we have an isomorphism:
$$
W_{\scX\scY} \cong U :=
\left(
\begin{array}{cc}
0 & 0 \\
0 & I
\end{array}
\right).
$$
If both $x$ and $y$ are untied, then any string $S(x-y)$ splits as a direct summand of $(Z, W)$. The complement
of such strings  belongs to the subcategory ${\rep}^{< (\scX, \scY)}(\dX)$. So, from now on we assume
that at least one element of $\{x, y\}$ is tied. A straightforward  computation
shows that
$$
\End_{\dX_{(\scX,\scY)}}(U) = \left\{
\left(
\begin{array}{cc}
A & 0 \\
B & C
\end{array}
\right),
\left(
\begin{array}{cc}
D & F \\
0 & C
\end{array}
\right)
 \right\},
$$
where $A, B, C, D, F$ are arbitrary matrices of appropriate size (determined by $U$) with coefficients in the field $\rK$.
 The new decorated bunch of chains
$\dX^{[\scX, \scY]}$ is defined as follows.
\begin{itemize}
\item If  $x \sim \tilde{x}$ for some $\tilde{x} \in \dX_\jmath, \jmath \in I$ then we add a new element $\tilde{x}_y$ to  $\dX_\jmath$. We have
$a < \tilde{x}_y < b$ in $\dX_\jmath^{[\scX, \scY]}$, whenever $a < \tilde{x} < b$ in $\dX_\jmath$. Moreover, $\tilde{x} < \tilde{x}_y$.
\item Similarly, if $y \sim \tilde{y}$ for some $\tilde{y} \in \dX_\sigma$  then we add to $\dX_\sigma$ a new element  $\tilde{y}_x$.
As above, $\tilde{y}_x$ inherits all order relations from its parent element $\tilde{y}$.
 Moreover,  $\tilde{y} > \tilde{y}_x$.
\item If both $x$ and $y$ are tied then we additionally impose: $\tilde{x}_y \sim  \tilde{y}_x$.
\end{itemize}
The admissible transformations of $(Z, W)$ (i.e.~those automorphisms which preserve the reduced form of $W_{\scX \scY}$) are
shown in the following picture:
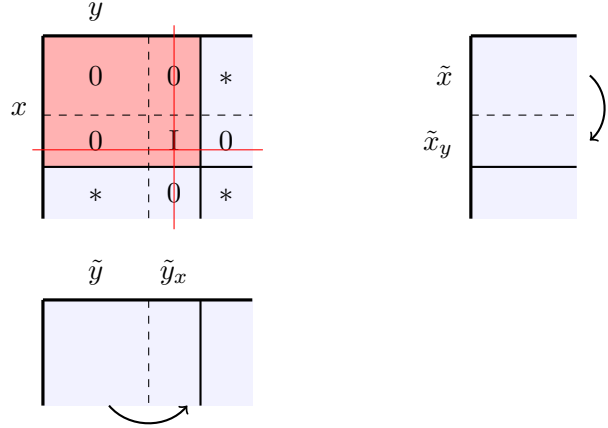
\begin{center}
\begin{tikzpicture}
\matrix (first) [tbl5,  name=table,
row 1/.style={minimum height=30pt},
row 2/.style={minimum height=20pt},
row 3/.style={minimum height=20pt},
column 1/.style={text width=40pt},
column 2/.style={text width=20pt},
column 3/.style={text width=20pt},
] at (0,0)
{
0&0& *\\
0&I&0 \\
*&0&*\\
};

\matrix  [tbl5,  name=rt,
row 1/.style={minimum height=30pt},
row 2/.style={minimum height=20pt},
row 3/.style={minimum height=20pt},
column 1/.style={text width=40pt},
] at (5,0)
{
~\\
~\\
~\\
};

\draw[very thick] (rt-1-1.north west) -- (rt-3-1.south west);
\draw[very thick] (rt-1-1.north west) -- (rt-1-1.north east);
\hdline{rt}{1}{1};
\hsline{rt}{2}{1};
\draw[ thick, ->, bend left=50, ] ([xshift= 5pt] rt-1-1.east) to ( [xshift=5pt] rt-2-1.east);
\node[base left=3pt of rt-1-1]{$\tilde{x}$};
\node[base left=3pt of rt-2-1]{$\tilde{x}_y$};

\matrix  [tbl5,  name=bt,
row 1/.style={minimum height=40pt},
column 1/.style={text width=40pt},
column 2/.style={text width=20pt},
column 3/.style={text width=20pt},
] at (0,-3)
{
~&~&~\\
};

\draw[very thick] (bt-1-1.north west) -- (bt-1-1.south west);
\draw[very thick] (bt-1-1.north west) -- (bt-1-3.north east);
\vdline{bt}{1}{1};
\vsline{bt}{1}{2};
\node[above=2pt of bt-1-1]{$\tilde{y}$};
\node[above=2pt of bt-1-2]{$\tilde{y}_x$};
\draw[ thick, ->, bend right=50, ] ([xshift= 5pt] bt-1-1.south) to ( [xshift=5pt] bt-1-2.south);

\node[draw,  fill, color=red!30  , fit=(table-1-1)(table-2-2), inner sep=0pt]{};

\node at (table-1-1){0};
\node at (table-1-2){0};
\node at (table-2-1){0};
\node at (table-2-2){I};


\hdline{table}{1}{3};
\hsline{table}{2}{3};

\vdline{table}{3}{1};
\vsline{table}{3}{2};

\draw[very thick] (table-1-1.north west) -- (table-1-3.north east);
\draw[very thick] (table-1-1.north west) -- (table-3-1.south west);

\node[above=2pt of table-1-1]{$y$};
\node[base left=2pt of table-1-1.south west]{$x$};

\node[above=4pt of table-1-2](a){};
\node[below=4pt of table-3-2](b){};
\node[base left=4pt of table-2-1](c){};
\node[base right=4pt of table-2-3](d){};

\draw[red] (a) to (b);
\draw[red] (c) to (d);
\end{tikzpicture}
\end{center}
\noindent
A more formal way to explain the reduction procedure is the following.
Crossing--Out Lemma \ref{L:crossingout} implies that  we have a  representation equivalence (\ref{E:ReductionCategorielle}). Note that in this case ${\rep}^{\le (\scX, \scY)}(\dX)
= {\rep}^{\le (\scX, \scY), \circ}(\dX)$.
The correspondence (\ref{E:RuleToRewrite})  is given by the rule
$$
 u \stackrel{\alpha}-\tilde{x}_y \sim \tilde{y}_x \stackrel{\beta}- v \; \mapsto \; u \stackrel{\alpha}- \tilde{x} \sim x \stackrel{0}- y \sim \tilde{y} \stackrel{\beta}-v.
$$
If only one
element (e.g.~$x$) is tied then the translation rule is analogous: $$\tilde{x}_y \stackrel{\gamma}- z \, \mapsto\,
y \stackrel{0}- x \sim \tilde{x}  \stackrel{\gamma}- z.$$

\medskip
\noindent
\underline{Case 1b}. Assume now that $x \sim y$. First observe that in $\Rep(\dX_{(\scX, \scY)})$ we have:
$$
W_{\scX\scY} \cong
\left(
\begin{array}{cc}
F & 0 \\
0 & N
\end{array}
\right),
$$
where $F$ is an invertible matrix and $N$ is a nilpotent matrix.
It is then easy to see that $F$ splits as a direct summand of $(Z, W)$ in $\Rep(\dX)$. Moreover,
$F$ decomposes into a direct sum of bands $B(w, m, \pi)$, where $w = (\lha x\sim y \rha)$, $m \in \mathbb{N}$ and $\xi \ne \pi \in \rK[\xi]$ is
a monic irreducible polynomial (note that $w$ is not decorable).

Next, we consider the full subcategory $\Rep^{\le (\scX, \scY), \circ}(\dX)$ of
$\Rep^{\le (\scX, \scY)}(\dX)$ consisting of those objects for which the block $W_{\scX\scY}$ is nilpotent.
We bring  $W_{\scX\scY}$ into its modified  Jordan normal form, see Lemma \ref{L:FormedeJordan}.
Then we reduce the entire matrix $W^{(\imath)}$ into a standard form by killing with any unit entry $\varpi$  of $W_{\scX\scY}$
all entries of matrices $W_{\scX\scU}$, where $\scY \ne \scU \subseteq \dF_\imath$ (respectively $W_{\scV\scY}$, where
$\scX \ne \scV \subseteq \dE_\imath$), standing with $\varpi$ in  the same row (respectively column).
The  new  decorated bunch of chains $\dX^{[\scX,\scY]}$ is defined as follows.
\begin{itemize}
\item  For any $l \in \NN_{\ge 2}$ we introduce new elements $x_l \in \dE_\imath$ and
$y_l \in \dF_\imath$. It is convenient to pose $x_1 = x$ and $y_1 = y$.
\item  We have:     $x_l \sim y_l$ for all $l \in \NN$ and $a < x_l < b$ (respectively $c < y_l < d$) whenever
$a < x < b$ (respectively $c < y < d$).
\item Finally, we put $\dots < x_3 < x_2 < x_1$ and $\dots > y_3 > y_2 > y_1$.
\end{itemize}
Again, Crossing--Out Lemma \ref{L:crossingout} yields a  representation equivalence (\ref{E:ReductionCategorielle}) and the translation rule (\ref{E:RuleToRewrite}) is given by the formula
$$
u \stackrel{\alpha}- y_l \sim x_l \stackrel{\beta}- v  \; \mapsto \;
u \stackrel{\alpha}-  \underbrace{y \sim x
 \stackrel{0}-   \dots   \stackrel{0}- y \sim x}_{l \; \scriptsize\mbox{times}} \stackrel{\beta}- v.
$$
The reduction procedure can be illustrated by the following picture:
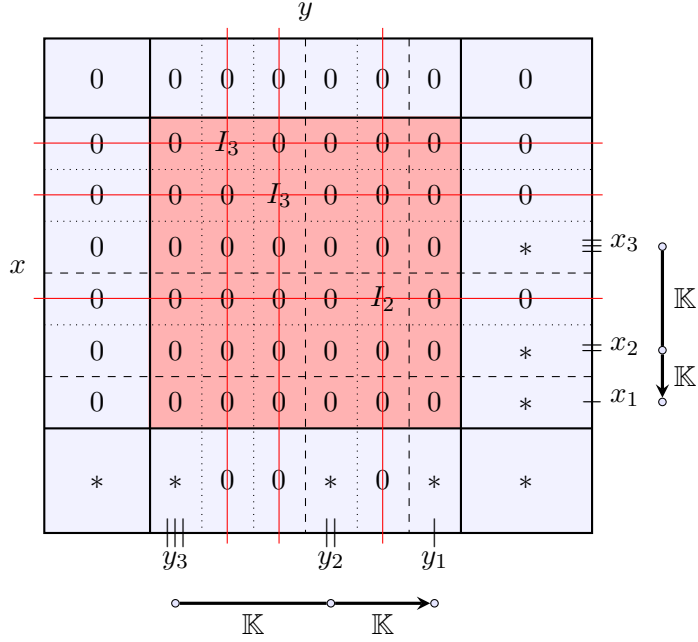
\begin{center}
\begin{tikzpicture}
[ dot/.style={fill=blue!10,circle,draw, inner sep=1pt, minimum size=1pt},
str/.style={inner sep=1pt, minimum size=0pt}
]

\matrix (first) [tbl5,  name=tbl,
row 1/.style={minimum height=30pt},
row 2/.style={minimum height=20pt},
row 3/.style={minimum height=20pt},
row 4/.style={minimum height=20pt},
row 5/.style={minimum height=20pt},
row 6/.style={minimum height=20pt},
row 7/.style={minimum height=20pt},
row 8/.style={minimum height=40pt},
%
column 1/.style={text width=40pt},
column 2/.style={text width=20pt},
column 3/.style={text width=20pt},
column 4/.style={text width=20pt},
column 5/.style={text width=20pt},
column 6/.style={text width=20pt},
column 7/.style={text width=20pt},
column 8/.style={text width=50pt},
] at (0,0)
{
0& 0&0&0&     0&0&0&     0 \\
0& 0&0&0&     0&0&0&     0 \\
0& 0&0&0&     0&0&0&     0 \\
0& 0&0&0&     0&0&0&     * \\
0& 0&0&0&     0&0&0&     0 \\
0& 0&0&0&     0&0&0&     * \\
0& 0&0&0&     0&0&0&     * \\
*& *&0&0&     *&0&*&     * \\
};

\node[draw,  fill, color=red!30  , fit=(tbl-2-2)(tbl-7-7), inner sep=0pt]{};
\foreach \x in {2,...,7}{\foreach \y in {2,...,7}{
\node at (tbl-\y-\x){0};
}}
\node[draw,  fill, color=red!30, inner sep=0pt] at (tbl-2-3){$I_3$};
\node at (tbl-2-3) {$I_3$};
\node[draw,  fill, color=red!30, inner sep=0pt] at (tbl-3-4){$I_3$};
\node at (tbl-3-4) {$I_3$};
\node[draw,  fill, color=red!30, inner sep=0pt] at (tbl-5-6){$I_2$};
\node at (tbl-5-6) {$I_2$};

\sfrm{tbl}{8}{8};

\hsline{tbl}{1}{8};
\hdline{tbl}{4}{8};
\hdline{tbl}{6}{8};
\hsline{tbl}{7}{8};

\draw [dotted](tbl-2-1.south west)-- (tbl-2-8.south east);
\draw [dotted](tbl-5-1.south west)-- (tbl-5-8.south east);
\draw [dotted](tbl-3-1.south west)-- (tbl-3-8.south east);

\vsline{tbl}{8}{1};
\vdline{tbl}{8}{4};
\vdline{tbl}{8}{6};
\vsline{tbl}{8}{7};
\draw [dotted](tbl-1-2.north east)-- (tbl-8-2.south east);
\draw [dotted](tbl-1-3.north east)-- (tbl-8-3.south east);
\draw [dotted](tbl-1-5.north east)-- (tbl-8-5.south east);

\node[base left=3pt of tbl-4-1.south west]{$x$};
\node[above =3pt of tbl-1-4.north east]{$y$};

\node at (tbl-4-8.east){$\equiv$}{};
\node[base right=3pt of tbl-4-8.east]{$x_3$};
\node[dot,base right=25pt of tbl-4-8.east](a3){};

\node at (tbl-6-8.east){$=$}{};
\node[base right=3pt of tbl-6-8.east]{$x_2$};
\node[dot, base right=25pt of tbl-6-8.east](a2){};

\node at (tbl-7-8.east){$-$}{};
\node[base right=3pt of tbl-7-8.east]{$x_1$};
\node[dot, base right=25pt of tbl-7-8.east](a1){};

\draw[ very thick, -stealth ] (a3) to node [right]{$\bK$} (a2) to node [right] {$\bK$}(a1);

\node at (tbl-8-2.south){$|||$}{};
\node[below=3pt of tbl-8-2.south]{$y_3$};
\node[dot,below=25pt of tbl-8-2.south](b3){};

\node at (tbl-8-5.south){$||$}{};
\node[below=3pt of tbl-8-5.south]{$y_2$};
\node[dot, below=25pt of tbl-8-5.south](b2){};

\node at (tbl-8-7.south){$|$}{};
\node[below=3pt of tbl-8-7.south]{$y_1$};
\node[dot, below=25pt of tbl-8-7.south](b1){};

\draw[ very thick, -stealth ] (b3) to node [below]{$\bK$} (b2) to node [below] {$\bK$}(b1);

\draw[red] ([xshift=-4pt] tbl-2-1.west) to ([xshift=4pt] tbl-2-8.east);
\draw[red] ([xshift=-4pt] tbl-3-1.west) to ([xshift=4pt] tbl-3-8.east);
\draw[red] ([xshift=-4pt] tbl-5-1.west) to ([xshift=4pt] tbl-5-8.east);

\draw[red] ([yshift=4pt] tbl-1-3.north) to ([yshift=-4pt] tbl-8-3.south);
\draw[red] ([yshift=4pt] tbl-1-4.north) to ([yshift=-4pt] tbl-8-4.south);
\draw[red] ([yshift=4pt] tbl-1-6.north) to ([yshift=-4pt] tbl-8-6.south);
\end{tikzpicture}
\end{center}

\medskip
\noindent
\underline{\textbf{Case 2}}. Now assume that both maximal elementary subchains
$\scX$ and $\scY$ are decorated. The decorated chessboard problem from Example \ref{E:DecorChessBoard}, treated in details  in Section \ref{S:DecoratedConjugation},  can occur as a special case of such
decorated bunch of chains $\dX_{(\scX,\scY)}$.

\begin{definition}
Let $m, n \in \NN$ and $M \in \Mat_{m \times n}(\rK)$. The \emph{valuation} of $M$ is the largest
integer $\nu = \val(M)$ such that there exists $M_\diamond \in \Mat_{m \times n}(\bD)$ satisfying
$M = t^\nu M_{\diamond}$. In particular,   $\val(M) = \infty$ for $M = 0$.
\end{definition}

\noindent
On the set of pairs
$
\kB_{(\scX,\scY)} = \left\{(u, v) \,\big|\, u \in \scX, v \in \scY\right\}
$
consider the following total ordering: $(u, v) < (x, y)$ if either $u < x$ or $u = x$ and $v > y$. Fix the following notations:

\begin{itemize}
\item $\Rep^{\le (\scX,\scY), \nu}(\dX)$ is the full subcategory of $\Rep^{\le (\scX,\scY)}(\dX)$
consisting of those objects for which $\val(W_{\scX\scY}) \ge \nu$.
\item For $(x, y) \in \kB_{(\scX,\scY)}$ define  $\Rep^{\le (x, y), \nu}(\dX)$ to be
 the full subcategory of $\Rep^{\le (\scX,\scY), \nu}(\dX)$
consisting of those objects for which $\val(W_{uv}) > \val(W_{xy})$ for all $(u, v) < (x, y)$.
\item Finally, $\Rep^{<(x,y), \nu}(\dX)$ is the full subcategory of $\Rep^{\le (x,y), \nu}(\dX)$ consisting of those objects for which $\val(W_{xy}) > \nu$.
\end{itemize}

\noindent
Let $\dX_{(x,y)}$ be the decorated bunch of chains obtained by restricting $\dX$ on $(x, y)$ and
$(Z, W)$ be an object of the category $\Rep^{\le (x, y), \nu}(\dX)$, where $(x, y) \in \scX \times \scY \subseteq \dE_\imath \times \dF_\imath$ and $\nu \in \mathbb{Z}$.

\medskip
\noindent
\underline{Case 2a}. Assume that $x \not\sim y$. First note that in $\Rep(\dX_{(x,y)})$ we have:
$$
W_{xy} = t^\nu \bigl(W_{xy}\bigr)_\diamond \cong t^\nu
\left(
\begin{array}{cc}
t \Psi &   0 \\
0      &  I
\end{array}
\right)
 =: t^\nu U,
$$
where $I$ is the identity matrix of size
$\mathsf{rk}_{\kk}\bigl(\overline{\bigl(W_{xy}\bigr)_{\diamond}}\bigr)$ and $\Psi$ has coefficients in $\bD$.  If both elements
$x$ and $y$ are untied then any  string  $S(x \stackrel{\nu}- y)$ splits up as direct summands
of $(Z, W)$ in $\Rep(\dX)$, allowing to
proceed to the next subcategory $\Rep^{<(x,y), \nu}(\dX)$.
Hence,  we may without loss of generality assume that at least one element of $\{x, y\}$ is tied.
Similarly to Case 1a, we have:
$$
\End_{\dX_{(x,y)}}(U) = \left\{
\left(
\begin{array}{cc}
A & C \\
B & D
\end{array}
\right),
\left(
\begin{array}{cc}
E & F \\
G & D
\end{array}
\right)
 \right\},
$$
where $A, B, C, D, F, G$ are matrices over $\bD$ of appropriate size (determined by $U$)   satisfying the constraints
$A\Psi = \Psi E$, $G = t B\Psi$ and $C = t \Psi F$.

\medskip
\noindent
The new decorated bunch of chains $\dX^{[x,y]}$ is defined as follows.
\begin{itemize}
\item If  $x \sim \tilde{x}$ for some $\tilde{x} \in \dX_\jmath, \jmath \in I$ then we add a new element $\tilde{x}_y$ to  $\dX_\jmath$. We have: $\tilde{x} \lhd \tilde{x}_y$
and $a < \tilde{x}_y < b$ provided $a < \tilde{x} < b$ and $a \lhd \tilde{x}_y \lhd b$ if $a \lhd \tilde{x} \lhd b$.
\item Similarly, if $y \sim  \tilde{y}$ for some $\tilde{y} \in \dX_\sigma$ then we add to $\dX_\sigma$ a new element  $\tilde{y}_x$.
We have $c < \tilde{y}_x < d$ whenever  $c < \tilde{y} < d$ and $c \lhd \tilde{y}_x \lhd d$ if   $c \lhd\tilde{y} \lhd d$. Moreover,  $\tilde{y} \rhd \tilde{y}_x$.
\item If both $x$ and $y$ are tied then we additionally impose: $\tilde{x}_y \sim  \tilde{y}_x$.
\end{itemize}

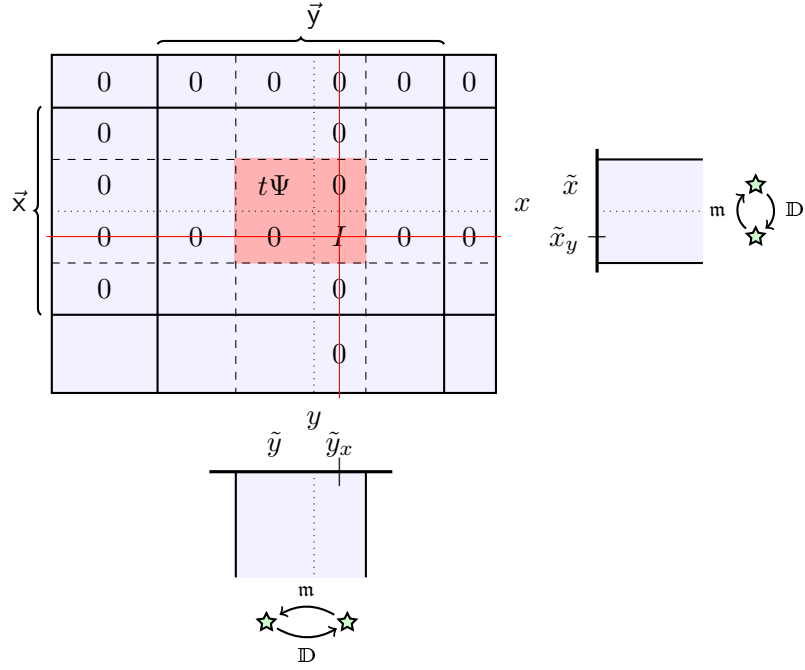
\begin{center}
\begin{tikzpicture}
[ dot/.style={fill=blue!10,circle,draw, inner sep=1pt, minimum size=1pt},
dv/.style={star,star points=5,
star point ratio=2, draw, thick, fill=green!20, inner sep=1pt,outer sep=2pt,minimum size=7pt}
]

\matrix (first) [tbl5,  name=tbl,
row 1/.style={minimum height=20pt},
row 2/.style={minimum height=20pt},
row 3/.style={minimum height=20pt},
row 4/.style={minimum height=20pt},
row 5/.style={minimum height=20pt},
row 6/.style={minimum height=30pt},
column 1/.style={text width=40pt},
column 2/.style={text width=30pt},
column 3/.style={text width=30pt},
column 4/.style={text width=20pt},
column 5/.style={text width=30pt},
column 6/.style={text width=20pt},
] at (0,0)
{
0&  0&0&     0&0&0 \\
0& &&   && \\
0& &&   && \\
0& &&   && \\
0& &&   && \\
& &&   && \\
};

\node[draw,  fill, color=red!30  , fit=(tbl-3-3)(tbl-4-4), inner sep=0pt]{};

\foreach \x in {2,...,6}
{
\node at (tbl-4-\x){0};
\node at (tbl-\x-4){0};
}

\node[draw,  fill, color=red!30, inner sep=0pt] at (tbl-3-3){$t\Psi$};
\node at (tbl-3-3) {$t\Psi$};
\node[draw,  fill, color=red!30, inner sep=0pt] at (tbl-3-4){$0$};
\node at (tbl-3-4) {0};
\node[draw,  fill, color=red!30, inner sep=0pt] at (tbl-4-3){$0$};
\node at (tbl-4-3) {$0$};
\node[draw,  fill, color=red!30, inner sep=0pt] at (tbl-4-4){$I$};
\node at (tbl-4-4) {$I$};

\sfrm{tbl}{6}{6};

\hsline{tbl}{1}{6};
\hdline{tbl}{2}{6};
\hdline{tbl}{4}{6};
\hsline{tbl}{5}{6};

\draw [dotted](tbl-3-1.south west)-- (tbl-3-6.south east);

\vsline{tbl}{6}{1};
\vdline{tbl}{6}{2};
\vdline{tbl}{6}{4};
\vsline{tbl}{6}{5};
\draw [dotted](tbl-1-3.north east)-- (tbl-6-3.south east);

\node[base left=6pt of tbl-3-1.south west]{$\scX$};
\node[above =6pt of tbl-1-3.north east]{$\scY$};

\node[base right=3pt of tbl-3-6.south east]{$x$};
\node[below =3pt of tbl-6-3.south east]{$y$};

\draw[red] ([xshift=-2pt] tbl-4-1.west) to ([xshift=2pt] tbl-4-6.east);
\draw[red] ([yshift=2pt] tbl-1-4.north) to ([yshift=-2pt] tbl-6-4.south);

\matrix  [tbl5,  name=rt,
row 1/.style={minimum height=40pt},
row 2/.style={minimum height=20pt},
row 3/.style={minimum height=20pt},
row 4/.style={minimum height=50pt},
column 1/.style={text width=40pt},
] at (5,0)
{
\\
\\
\\
\\
};

\node[fill, color=white!00, fit=(rt-1-1), inner sep=0pt]{};
\node[fill, color=white!00, fit=(rt-4-1), inner sep=0pt]{};

\draw[very thick] ([yshift=4pt] rt-1-1.south west) -- ([yshift=-4pt] rt-4-1.north west);
\hsline{rt}{1}{1};
\hsline{rt}{3}{1};
\draw [dotted](rt-2-1.south west)-- (rt-2-1.south east);

\node[dv] (a) at ([xshift= 20pt] rt-2-1.east){};
\node[dv] (b) at ([xshift= 20pt] rt-3-1.east){};

\draw[ thick, ->, bend left=50, ]
 (a)
to node[base right]{$_\bD$} (b) ;
\draw[ thick, ->, bend left=50, ] (b) to node[base left]{$_\idm$} (a);

\node[base left=3pt of rt-2-1]{$\tilde{x}$};
\node[base left=3pt of rt-3-1]{$\tilde{x}_y$};
\node at (rt-3-1.west){$-$};

\draw
[decorate,decoration={brace},thick]
([yshift=4pt]tbl-1-1.north east)--([yshift=4pt]tbl-1-5.north east);
\draw
[decorate,decoration={brace},thick]
([xshift=-4pt]tbl-5-1.south west)--([xshift=-4pt]tbl-2-1.north west) ;

\matrix  [tbl5,  name=bt,
row 1/.style={minimum height=40pt},
column 1/.style={text width=70pt},
column 2/.style={text width=30pt},
column 3/.style={text width=20pt},
column 4/.style={text width=50pt},
] at (0,-4)
{
~&~&~&~~\\
};

\node[fill, color=white!00, fit=(bt-1-1), inner sep=0pt]{};
\node[fill, color=white!00, fit=(bt-1-4), inner sep=0pt]{};

\draw[very thick] ([xshift=-10pt]bt-1-1.north east) -- ([xshift=10pt]bt-1-3.north east);
\vsline{bt}{1}{1};
\vsline{bt}{1}{3};
\draw[dotted] (bt-1-2.north east) -- (bt-1-2.south east);

\node[above=2pt of bt-1-2]{$\tilde{y}$};
\node[above=2pt of bt-1-3]{$\tilde{y}_x$};
\node at (bt-1-3.north){$|$};
\node[dv] (d) at ([xshift= -3pt,yshift= -17pt] bt-1-2.south){};
\node[dv] (c) at( [xshift= 3pt,yshift= -17pt] bt-1-3.south){};
\draw[ thick, ->, bend right=30, ] (c) to node[above]{$_\idm$} (d);
\draw[ thick, ->, bend right=30, ] (d) to node[below]{$_\bD$} (c);

\end{tikzpicture}
\end{center}

\noindent
Analogously to (\ref{E:ReductionCategorielle}), we get a  representation equivalence
$$
R^{xy}_\nu: \underline{\rep}^{\le (x, y), \nu}_{\st}(\dX) \lar \underline{\rep}^{< (x, y), \nu}(\dX^{[x,y]}).
$$
The correspondence (\ref{E:RuleToRewrite})
between strings and bands in both categories  is given by
$$
 u \stackrel{\alpha}-\tilde{x}_y \sim \tilde{y}_x \stackrel{\beta}- v \; \mapsto \; u \stackrel{\alpha}- \tilde{x} \sim x \stackrel{\nu}- y \sim \tilde{y} \stackrel{\beta}-v.
$$
The decorations of decorable words are transferred   in a straightforward way. If only one
element (e.g.~$x$) is tied then the translation rule is analogous: $\tilde{x}_y \stackrel{\gamma}- z \, \mapsto\,
y \stackrel{\nu}- x \sim \tilde{x}  \stackrel{\gamma}- z$.

\medskip
\noindent
\underline{Case 2b}. Now assume that $x \sim y$. In the category $\Rep(\dX_{(x,y)})$ we have:
$$
W_{xy} \cong
t^\nu \left(
\begin{array}{cc}
F & 0 \\
0 & N
\end{array}
\right),
$$
 where $F$ is  invertible over $\bD$ and $N$ is nilpotent modulo $\idm$.
As in Case 1b, it is  easy to see that $t^\nu F$ splits as a direct summand of $(Z, W)$ in $\Rep(\dX)$, decomposing  further into a direct sum of bands $B\bigl((w,\rho), m, \pi\bigr)$, where
$(w, \rho) = \bigl(\lha x\sim y \stackrel{\nu}\rha\bigr)$, $m \in \mathbb{N}$ and $\xi \ne \pi \in \kk[\xi]$ is
a monic irreducible polynomial.

Next, consider the full subcategory $\Rep^{\le (x,y), \nu, \circ}(\dX)$ of
$\Rep^{\le (x, y), \nu}(\dX)$ consisting of objects for which the matrix $\overline{(W_{xy})_\diamond}$ is nilpotent. Let
$(Z, W)$ be an object of $\Rep^{\le (x,y), \nu, \circ}(\dX)$.
As in  Lemma \ref{L:FormedeJordan}, we reduce  the block $W_{xy}$ into a  normal form.
Then we bring the entire matrix $W^{(\imath)}$ into a standard form: if $W_{xy}$ contains  an entry $\varpi$ with valuation
$\nu$, then for all $(u, v) \in \dE_\imath \times \dF_\imath \setminus \{(x, y)\}$ we kill with it all elements of all matrices  $W_{xu}$ (respectively $W_{vy}$)  standing in the same row (respectively column) with $\varpi$.

\medskip
\noindent
The  new  decorated bunch of chains $\dX^{[\scX,\scY]}$ is defined as follows.
\begin{itemize}
\item  For any $l \in \NN_{\ge 2}$ we introduce new decorated elements $x_l \in \dE_\imath$ and
$y_l \in \dF_\imath$. It is convenient to write   $x_1 = x$ and $y_1 = y$.
\item  For all $l \in \NN$ we have:     $x_l \sim y_l$
 and $a < x_l < b$ (respectively $c < y_l < d$) whenever
$a < x < b$ (respectively $c < y < d$). Similarly, $a \lhd x_l \lhd b$ (respectively $c \lhd y_l \lhd d$) provided
$a \lhd x \lhd b$ (respectively $c \lhd y \lhd d$).
\item Finally, the ordering between new elements is the following: $\dots \lhd x_3 \lhd x_2 \lhd x_1$ and $\dots \rhd y_3 \rhd y_2 \rhd y_1$.
\end{itemize}

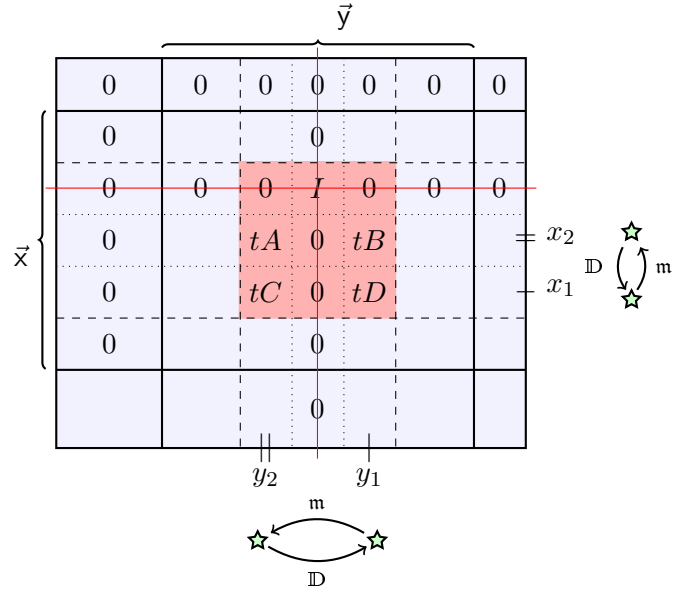
\begin{center}
\begin{tikzpicture}
\matrix (first) [tbl5,  name=tbl,
row 1/.style={minimum height=20pt},
row 2/.style={minimum height=20pt},
row 3/.style={minimum height=20pt},
row 4/.style={minimum height=20pt},
row 5/.style={minimum height=20pt},
row 6/.style={minimum height=20pt},
row 7/.style={minimum height=30pt},
column 1/.style={text width=40pt},
column 2/.style={text width=30pt},
column 3/.style={text width=20pt},
column 4/.style={text width=20pt},
column 5/.style={text width=20pt},
column 6/.style={text width=30pt},
column 7/.style={text width=20pt},
] at (0,0)
{
0&  0&0&     0&0&0 &0\\
0& &&   && & \\
0& &&   && &\\
0& &&   && &\\
0& &&   && &\\
0& &&   && &\\
& &&   && &\\
};

\node[draw,  fill, color=red!30  , fit=(tbl-3-3)(tbl-5-5), inner sep=0pt]{};

\foreach \x in {2,...,7}
{
\node at (tbl-\x-4){0};
}

\node[draw,  fill, color=red!30, inner sep=0pt] at (tbl-3-3){$0$};
\node at (tbl-3-2) {$0$};
\node at (tbl-3-3) {$0$};
\node[draw,  fill, color=red!30, inner sep=0pt] at (tbl-3-4){$I$};
\node at (tbl-3-4) {$I$};
\node[draw,  fill, color=red!30, inner sep=0pt] at (tbl-3-5){$0$};
\node at (tbl-3-5) {$0$};
\node at (tbl-3-6) {$0$};
\node at (tbl-3-7) {$0$};
\node[draw,  fill, color=red!30, inner sep=0pt] at (tbl-5-3){$tC$};
\node at (tbl-5-3) {$tC$};
\node[draw,  fill, color=red!30, inner sep=0pt] at (tbl-5-4){$0$};
\node at (tbl-5-4) {$0$};
\node[draw,  fill, color=red!30, inner sep=0pt] at (tbl-5-5){$tD$};
\node at (tbl-5-5) {$tD$};
\node[draw,  fill, color=red!30, inner sep=0pt] at (tbl-4-3){$tA$};
\node at (tbl-4-3) {$tA$};
\node[draw,  fill, color=red!30, inner sep=0pt] at (tbl-4-4){$0$};
\node at (tbl-4-4) {$0$};
\node[draw,  fill, color=red!30, inner sep=0pt] at (tbl-4-5){$tB$};
\node at (tbl-4-5) {$tB$};

\sfrm{tbl}{7}{7};

\hsline{tbl}{1}{7};
\hdline{tbl}{2}{7};
\draw [dotted](tbl-3-1.south west)-- (tbl-3-7.south east);
\draw [dotted](tbl-4-1.south west)-- (tbl-4-7.south east);
\hdline{tbl}{5}{7};
\hsline{tbl}{6}{7};

\vsline{tbl}{7}{1};
\vdline{tbl}{7}{2};
\draw [dotted](tbl-1-3.north east)-- (tbl-7-3.south east);
\draw [dotted](tbl-1-4.north east)-- (tbl-7-4.south east);
\vdline{tbl}{7}{5};
\vsline{tbl}{7}{6};

\node[base left=7pt of tbl-4-1.south west]{$\scX$};
\node[above =7pt of tbl-1-4.north east]{$\scY$};

\draw[red] ([xshift=-4pt] tbl-3-1.west) to ([xshift=4pt] tbl-3-7.east);
\draw[red] ([yshift=4pt] tbl-1-4.north) to ([yshift=-4pt] tbl-7-4.south);

\node[below=4pt of tbl-7-3.south]{$y_2$};
\node[below=4pt of tbl-7-5.south]{$y_1$};
\node at (tbl-7-3.south)(c){$||$};
\node at (tbl-7-5.south)(d){$|$};
\node[dv] (d) at ([xshift= -3pt,yshift= -35pt]tbl-7-3.south){};
\node[dv] (c) at( [xshift= 3pt,yshift= -35pt] tbl-7-5.south){};
\draw[ thick, ->, bend right=30, ] (c) to node[above]{$_\idm$} (d);
\draw[ thick, ->, bend right=30, ] (d) to node[below]{$_\bD$} (c);

\node[base right =4pt of tbl-4-7.east]{$x_2$};
\node[base right =4pt of tbl-5-7.east]{$x_1$};
\node at (tbl-4-7.east)(c){$=$};
\node at (tbl-5-7.east)(d){$-$};
\node[dv] (b) at ([yshift= 3pt,xshift= 40pt]tbl-4-7.east){};
\node[dv] (a) at( [yshift= -3pt,xshift= 40pt] tbl-5-7.east){};
\draw[ thick, ->, bend right=30, ] (a) to node[base right]{$_\idm$} (b);
\draw[ thick, ->, bend right=30, ] (b) to node[base left=1pt]{$_\bD$} (a);

\draw
[decorate,decoration={brace},thick]
([yshift=4pt]tbl-1-1.north east)--([yshift=4pt]tbl-1-6.north east);
\draw
[decorate,decoration={brace},thick]
([xshift=-4pt]tbl-7-1.north west)--([xshift=-4pt]tbl-1-1.south west) ;

\end{tikzpicture}
\end{center}

\medskip
\noindent
Analogously to (\ref{E:ReductionCategorielle}), we get  a  representation equivalence.
$$
R^{xy}_\nu: \underline{\rep}^{\le (x,y),\nu, \circ}_{\st}(\dX) \lar \underline{\rep}^{< (x,y), \nu}(\dX^{[x,y]}).
$$
The translation rule (\ref{E:RuleToRewrite}) is given by
$$
u \stackrel{\alpha}- y_l \sim x_l \stackrel{\beta}- v  \; \mapsto \;
u \stackrel{\alpha}-  \underbrace{y \sim x
 \stackrel{\nu}-   \dots   \stackrel{\nu}- y \sim x}_{l \; \scriptsize\mbox{times}} \stackrel{\beta}- v.
$$

\medskip
\noindent
\underline{\bf Case 3}. Assume that $\scX = \{x\}$ is not decorated and
$\scY = \{y_1 \rhd y_2 \rhd \dots \rhd y_n\}$ is decorated, where $\scX \times \scY \subseteq \dE_\imath \times \dF_\imath$.  This is the most tricky case in the whole reduction procedure. Let $(Z, W)$ be an object of the category $\Rep^{\le (\scX, \scY)}(\dX)$.

\medskip
\noindent
\underline{Case 3a}. Assume first that $n=2$, $\scY = \{y \rhd z\}$ and $y \not\sim z$. As the first step, observe that  any object of $\Rep(\dX_{(\scX, \scY)})$
is isomorphic to $V$ given by
\begin{center}
\begin{tikzpicture}

\matrix (first) [tbl5,  name=tbl,
row 1/.style={minimum height=20pt},
row 2/.style={minimum height=30pt},
row 3/.style={minimum height=20pt},
column 1/.style={text width=40pt},
column 2/.style={text width=20pt},
column 3/.style={text width=20pt},
column 4/.style={text width=20pt},
] at (0,0)
{
0& 0&0&0\\
0& 0&0&I_2\\
0&I_1&0&0\\
};

\sfrm{tbl}{3}{4};

\hdline{tbl}{1}{4};
\hdline{tbl}{2}{4};

\vdline{tbl}{3}{1};
\vsline{tbl}{3}{2};
\vdline{tbl}{3}{3};

\node[base left=3pt of tbl-2-1.west]{$x$};
\node[above =3pt of tbl-1-1.north]{$y$};
\node[above =3pt of tbl-1-3.north]{$z$};
\end{tikzpicture}
\end{center}
If all elements $x, y, z$ are untied, then we can split up from any object of $\Rep^{\le (\scX, \scY)}(\dX)$ all direct summands isomorphic
to $S(x-y)$ and $S(x-z)$ and proceed to the next subcategory  $\Rep^{< (\scX, \scY)}(\dX)$. So, assume that
at least one element of  $\{x, y, z\}$ is tied.
A direct computation shows that  $\End_{\dX_{(\scX, \scY)}}(V)$ is the $\bD$--module  of all pairs $(F, G)$ of matrices the form
\begin{equation}\label{E:MatricesDeTransform}
\begin{array}{c}
\begin{tikzpicture}

\matrix (first) [tbl5,  name=tbl,
minimum height=20pt,
text width=20pt%
] at (0,0)
{
\bullet & 0 & 0 \\
\bullet & \ast_1 & \odot_3 \\
\bullet & \ast_4 & \ast_2\\
};

\sfrm{tbl}{3}{3};
\node[base left=5pt of tbl-2-1.west, yshift=-5pt]{$F  =$};
\matrix (second) [tbl5,  name=tbl,
minimum height=20pt,
text width=20pt%
] at (5,0)
{
\ast  & \ast & \ast & \ast \\
0     & \ast_2 &  0   & \ast_4 \\
\odot & \odot & \ast & \ast \\
0 & \odot_3 & 0 & \ast_1\\
};

\sfrm{tbl}{4}{4};
\node[base left=5pt of tbl-3-1.north west, yshift=-5pt]{$G  =$};

\hsline{tbl}{2}{4};
\vsline{tbl}{4}{2};
\end{tikzpicture}
\end{array}
\end{equation}

Here $\bu$ means that coefficients of the corresponding block belong to
$\rK$, $*$ stands for  $\bD$ and $\od$ for
$\idm$, equal indices $*_i$ mean that these blocks are equal. The key observation is the following:

\begin{itemize}
\item Both  non--reduced stripes  of the matrix $W^{(\imath)}$ of the vertical blocks $y$ and $z$ can be  transformed
by the matrices of the form
$\left(
\begin{array}{cc}
G_{11} & G_{13} \\
G_{31} & G_{33}
\end{array}\right) = \left(
\begin{array}{cc}
\ast & \ast \\
\odot & \ast
\end{array}\right).
$
\item If $x \sim \tilde{x}$ then the admissible transformations of  $\tilde{x}$--block
are given by matrices
\begin{equation}
\begin{array}{c}
\begin{tikzpicture}

\matrix (first) [tbl5,  name=tbl,
minimum height=20pt,
text width=20pt%
] at (0,0)
{
\bullet & 0 & 0 \\
\bullet & \ast_1 & \odot \\
\bullet & \ast_4 & \ast_2\\
};

\sfrm{tbl}{3}{3};
\node[base left=5pt of tbl-2-1.west, yshift=-5pt]{$\tilde{F} = F =$};
\end{tikzpicture}
\end{array}
\end{equation}

where $F$  is the matrix from   (\ref{E:MatricesDeTransform}).
\item If $y \sim \tilde{y}$ then the admissible transformations of   $\tilde{y}$--block
are given by matrices
$T = \left(
\begin{array}{cc}
T_{11} & T_{12} \\
T_{21} & T_{22}
\end{array}\right),
$
such that $T_{lt} \equiv G_{lt} \,\mod\, \idm$ for all $l, t \in \{1,2\}$. In particular,
$T_{21} \equiv 0 \, \mod\, \idm$, $T_{22} \equiv F_{33} \, \mod\, \idm$ and
$T_{11} \equiv G_{11} \, \mod\, \idm$, i.e.~
$T = \left(
\begin{array}{cc}
\ast & \ast \\
\odot & \ast_{2'}
\end{array}\right),
$ where $\ast_i, \ast_{i'}$ means that the corresponding blocks of $\tilde{F}$ and $T$ are equal modulo $\idm$.
\item If $z \sim \tilde{z}$ then the admissible transformations of  the  $\tilde{z}$--block
are given by matrices
$S = \left(
\begin{array}{cc}
S_{33} & S_{34} \\
S_{43} & S_{44}
\end{array}\right),
$
where $S_{lt} \equiv G_{lt} \,\mod\, \idm$ for all $l, t \in \{3,4\}$.
In particular, $S_{43} \equiv 0 \, \mod\, \idm$, $S_{44} \equiv F_{22} \, \mod\, \idm$ and
$S_{33} \equiv G_{33} \, \mod\, \idm$, i.e.~$S = \left(
\begin{array}{cc}
\ast & \ast \\
\odot & \ast_{1'}
\end{array}\right).
$
\end{itemize}
The new decorated bunch of chains $\dX^{[\scX, \scY]}$ is defined as follows.
\begin{itemize}
\item If $y \sim \tilde{y}$ for some $\tilde{y} \in \dX_\jmath$ then we add a new decorated element $\tilde{y}_x$ to
$\dX_\jmath$. If $z \sim \tilde{z}$ for some $\tilde{z} \in \dX_\kappa$ then we add
a new decorated element $\tilde{z}_x$ to $\dX_\kappa$. Finally, if $x \sim \tilde{x}$ for some $\tilde{z} \in \dX_\tau$ then we add two \emph{decorated}
elements $\tilde{x}_y$ and $\tilde{x}_z$ to $\dX_\tau$ (note that $\tilde{x}$ itself is \emph{not decorated})!
\item The new elements inherit all orderings from their  parent elements and parent chains.
For example, if $a < \tilde{y} < b$ in $\dX_\jmath$ then also $a < \tilde{y}_x < b$ in $\dX^{[\scX, \scY]}_\jmath$ etc.
\item We have: $\tilde{y} \rhd \tilde{y}_x$, $\tilde{z} \rhd \tilde{z}_x$ and $\tilde{x} < \tilde{x}_z \lhd \tilde{x}_y$.
\item Finally, we have impose equivalences $\tilde{x}_y \sim \tilde{y}_x$ and
$\tilde{x}_z \sim \tilde{z}_x$.
\end{itemize}
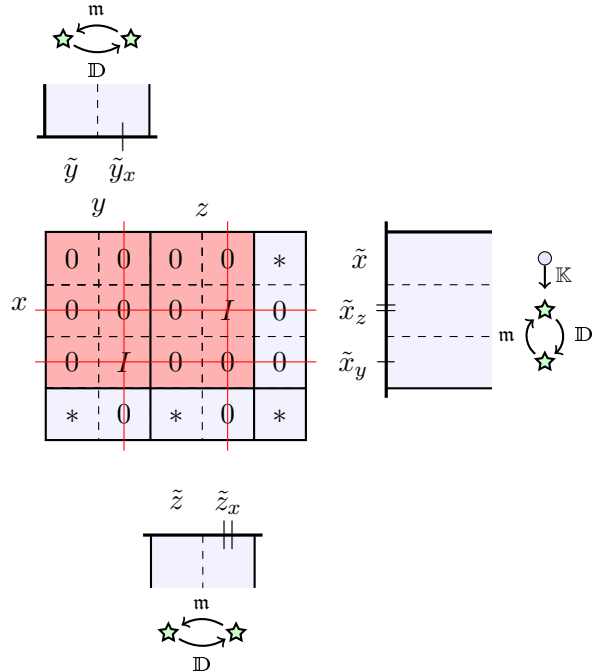
\begin{center}
\begin{tikzpicture}

\matrix (first) [tbl5,  name=tbl,
row 1/.style={minimum height=20pt},
row 2/.style={minimum height=20pt},
row 3/.style={minimum height=20pt},
row 4/.style={minimum height=20pt},
column 1/.style={text width=20pt},
column 2/.style={text width=20pt},
column 3/.style={text width=20pt},
column 4/.style={text width=20pt},
column 5/.style={text width=20pt},
] at (0,0)
{
0&0&0&0& *\\
0&0&0&I& 0\\
0&I&0&0& 0\\
*&0&*&0& *\\
};

\node[fill, color=red!30, fit=(tbl-1-1)(tbl-3-4), inner sep=0pt]{};
\foreach \x in {1,...,3}{
\foreach \y in {1,...,4}{
\node at (tbl-\x-\y){0};
}};
\node[fill=red!30] at (tbl-3-2){$I$};
\node[fill=red!30] at (tbl-2-4){$I$};

\sfrm{tbl}{4}{5};

\foreach \x in {1,...,4}{
\foreach \y in {1,...,5}{
\hdline{tbl}{\x}{\y};
\vdline{tbl}{\x}{\y};
}};

\hsline{tbl}{3}{5};
\vsline{tbl}{4}{2};
\vsline{tbl}{4}{4};

\node[base left=3pt of tbl-2-1.west]{$x$};
\node[above =2pt of tbl-1-1.north east]{$y$};
\node[above =2pt of tbl-1-3.north east]{$z$};

\matrix  [tbl5,  name=rt,
row 1/.style={minimum height=20pt},
row 2/.style={minimum height=20pt},
row 3/.style={minimum height=20pt},
row 4/.style={minimum height=20pt},
column 1/.style={text width=40pt},
] at (3.5,0)
{
\\
\\
\\
\\
};
\node[fill, color=white!00, fit=(rt-4-1), inner sep=0pt]{};

\draw[very thick] ([yshift=4pt] rt-1-1.north west) -- ([yshift=-4pt] rt-4-1.north west);
\draw[very thick] (rt-1-1.north west) -- (rt-1-1.north east);
\hdline{rt}{1}{1};
\hdline{rt}{2}{1};
\hsline{rt}{3}{1};

\node[dot] (t) at ([xshift= 20pt] rt-1-1.east){};
\node[dv] (a) at ([xshift= 20pt] rt-2-1.east){};
\node[dv] (b) at ([xshift= 20pt] rt-3-1.east){};

\draw[thick,->](t) to node[base right]{$_\bK$}(a);
\draw[ thick, ->, bend left=50, ]
 (a)
to node[base right]{$_\bD$} (b) ;
\draw[ thick, ->, bend left=50, ] (b) to node[base left]{$_\idm$} (a);

\node[base left=3pt of rt-1-1]{$\tilde{x}$};
\node[base left=3pt of rt-2-1]{$\tilde{x}_z$};
\node[base left=3pt of rt-3-1]{$\tilde{x}_y$};

\node at (rt-2-1.west){$=$};
\node at (rt-3-1.west){$-$};

\draw[red] ([xshift=-4pt] tbl-2-1.west) to ([xshift=4pt] tbl-2-5.east);
\draw[red] ([xshift=-4pt] tbl-3-1.west) to ([xshift=4pt] tbl-3-5.east);

\draw[red] ([yshift=4pt] tbl-1-2.north) to ([yshift=-4pt] tbl-4-2.south);
\draw[red] ([yshift=4pt] tbl-1-4.north) to ([yshift=-4pt] tbl-4-4.south);

\matrix  [tbl5,  name=bt,
row 1/.style={minimum height=20pt},
column 1/.style={text width=40pt},
column 2/.style={text width=20pt},
column 3/.style={text width=20pt},
column 4/.style={text width=20pt},
] at (0,-3)
{
~&~&~&~~\\
};
\node[fill, color=white!00, fit=(bt-1-1), inner sep=0pt]{};
\node[fill, color=white!00, fit=(bt-1-4), inner sep=0pt]{};

\draw[very thick] ([xshift=-3pt]bt-1-1.north east) -- ([xshift=3pt]bt-1-3.north east);
\vsline{bt}{1}{1};
\vsline{bt}{1}{3};
\draw[dashed] (bt-1-2.north east) -- (bt-1-2.south east);

\node[above=6pt of bt-1-2]{$\tilde{z}$};
\node[above=4pt of bt-1-3]{$\tilde{z}_x$};
\node at (bt-1-3.north){$||$};
\node[dv] (d) at ([xshift= -3pt,yshift= -17pt] bt-1-2.south){};
\node[dv] (c) at( [xshift= 3pt,yshift= -17pt] bt-1-3.south){};
\draw[ thick, ->, bend right=30, ] (c) to node[above]{$_\idm$} (d);
\draw[ thick, ->, bend right=30, ] (d) to node[below]{$_\bD$} (c);

\matrix  [tbl5,  name=abt,
row 1/.style={minimum height=20pt},
column 1/.style={text width=20pt},
column 2/.style={text width=20pt},
column 3/.style={text width=60pt},
] at (0,3)
{
~&~&~\\
};
\node[fill, color=white!00, fit=(abt-1-3), inner sep=0pt]{};

\draw[very thick] ([xshift=-3pt]abt-1-1.south west) -- ([xshift=3pt]abt-1-2.south east);
\vsline{abt}{1}{2};
\draw[dashed] (abt-1-1.north east) -- (abt-1-1.south east);
\draw[very thick] (abt-1-1.north west) -- (abt-1-1.south west);

\node[below=4pt of abt-1-1]{$\tilde{y}$};
\node[below=4pt of abt-1-2]{$\tilde{y}_x$};
\node at (abt-1-2.south){$|$};
\node[dv] (d) at ([xshift= -3pt,yshift= 17pt] abt-1-1.north){};
\node[dv] (c) at( [xshift= 3pt,yshift= 17pt] abt-1-2.north){};
\draw[ thick, ->, bend right=30, ] (c) to node[above]{$_\idm$} (d);
\draw[ thick, ->, bend right=30, ] (d) to node[below]{$_\bD$} (c);

\end{tikzpicture}
\end{center}
The above computations and  Crossing--Out Lemma \ref{L:crossingout} yield a  representation equivalence
(\ref{E:ReductionCategorielle}). Note that in this case $\Rep^{\le (\scX, \scY)}(\dX) = \Rep^{\le (\scX, \scY), \circ}(\dX)$.
The translation rule (\ref{E:RuleToRewrite}) is given by the formulae
$$
\tilde{x}_z \sim \tilde{z}_x \, \mapsto \,
\tilde{x} \sim x \stackrel{0}- y \sim \tilde{y} \quad \mbox{and} \quad
\tilde{y}_x \sim \tilde{x}_y \, \mapsto \,
\tilde{y} \sim y \stackrel{0}- x \sim \tilde{x}.
$$

\medskip
\noindent
\underline{Case 3b}. Now assume  that $n=2$, $\scY = \{y \rhd z\}$ but  this time $y \sim z$.
Note  that  any object of $\Rep(\dX_{(\scX, \scY)})$
is isomorphic to some $U$ given by

\begin{center}
\begin{tikzpicture}

\matrix (first) [tbl5,  name=tbl,
row 1/.style={minimum height=20pt},
row 2/.style={minimum height=20pt},
row 3/.style={minimum height=20pt},
row 4/.style={minimum height=20pt},
row 5/.style={minimum height=20pt},
column 1/.style={text width=20pt},
column 2/.style={text width=20pt},
column 3/.style={text width=20pt},
column 4/.style={text width=20pt},
column 5/.style={text width=20pt},
column 6/.style={text width=20pt},
column 7/.style={text width=20pt},
column 8/.style={text width=20pt},
] at (0,0)
{
0&0&0&0&       0&0&0&0\\
0&0&0&0&       0&0&0& I_2 \\
0&0&0&0&       0& I_3 &0&0\\
0&0&0& I_2 &   0&0&0&0\\
0&0& I_1 &0&   0&0&0&0\\
};

\sfrm{tbl}{5}{8};

\foreach \x in {1,...,5}{
\foreach \y in {1,...,8}{
\hdline{tbl}{\x}{\y};
\vdline{tbl}{\x}{\y};
}};
\vsline{tbl}{5}{4};

\node[base left=3pt of tbl-3-1.west]{$x$};
\node[above =3pt of tbl-1-2.north east]{$y$};
\node[above =3pt of tbl-1-6.north east]{$z$};
\end{tikzpicture}
\end{center}
If all elements $x, y, z$ are untied, then we can split up all direct summands of $(Z, W)$ isomorphic
to the strings $S(x-y\sim z)$, $S(x-z\sim y)$ and $S(x-z\sim y-x)$.  In this case  we just proceed
to the next category $\Rep^{<(\scX, \scY)}(\dX)$. Hence,  we might assume that at least
one element of $\{x, y, z\}$ is tied. A direct computation shows that $\End_{\dX_{(\scX, \scY)}}(U)$ is
the $\bD$--module of all pairs of matrices $(F, G)$ of the form
\begin{equation}\label{E:MatricesDeTransform2}
\begin{array}{c}
\begin{tikzpicture}

\matrix (first) [tbl5,  name=tbl,
minimum height=20pt,
text width=20pt%
] at (0,0)
{
\bullet & 0 & 0 & 0 & 0 \\
\bullet & \ast_1 & \odot_5 & \odot_6 & \odot_7 \\
\bullet & \ast_\beta & \ast_2 & \odot_8 & \odot_9 \\
\bullet & \ast_\gamma & \ast_\delta & \ast_{1'} & \odot_\alpha \\
\bullet & \ast_\varphi & \ast_\psi & \ast_\theta & \ast_3\\
};

\sfrm{tbl}{5}{5};
\node[base left=5pt of tbl-3-1.west, yshift=-5pt]{$F  =$};
\matrix (second) [tbl5,  name=tbl,
minimum height=20pt,
text width=20pt%
] at (6,0)
{
\ast_4 & \ast & \ast & \ast & \ast & \ast & \ast & \ast \\
\odot & \ast_{2'} & \odot & \ast & \ast & \ast & \ast & \ast \\
0 & 0 & \ast_3 & \ast_\theta & 0 & \ast_\psi & 0 & \ast_\varphi \\
0 & 0 & \odot_\alpha & \ast_{1'} & 0 & \ast_\delta & 0 & \ast_\gamma \\
\odot & \odot & \odot & \odot & \ast_{4'} & \ast & \ast & \ast \\
\odot & \odot & \odot_9 & \odot_8 & 0 & \ast_2 & 0 & \ast_\beta \\
\odot & \odot & \odot & \odot & \odot & \odot & \ast_{3'} & \ast \\
\odot & \odot & \odot_7 & \odot_6 & 0 & \odot_5 & 0 & \ast_1 \\
};

\sfrm{tbl}{8}{8};
\node[base left=5pt of tbl-5-1.north west, yshift=-5pt]{$G  =$};

\hsline{tbl}{4}{8};
\vsline{tbl}{8}{4};
\end{tikzpicture}
\end{array}
\end{equation}

We follow here the same conventions as in Case 3a. In particular, $(\ast_{i}, \ast_{i'})$  means that
the corresponding blocks are equal modulo $\idm$. By notational reasons,  the congruences
$G_{tl} \equiv G_{t+4\,  l+4} \,\mod\, \idm$ for  $1 \le t, l \le 4$ are not reflected in (\ref{E:MatricesDeTransform2}).

\begin{itemize}
\item The admissible transformations of the  non-reduced part of $y$--stripe are given by matrices of the form
$
\left(
\begin{array}{cc}
G_{11} & G_{12} \\
G_{21} & G_{22}
\end{array}
\right)
=
\left(
\begin{array}{cc}
\ast_4 & \ast \\
\odot & \ast_{2'}
\end{array}
\right).
$
\item The admissible transformations of the  non-reduced part of $z$--stripe are given by matrices of the form
$
\left(
\begin{array}{cc}
G_{55} & G_{57} \\
G_{75} & G_{77}
\end{array}
\right)
=
\left(
\begin{array}{cc}
\ast_{4'} & \ast \\
\odot & \ast_{3'}
\end{array}
\right).
$
\item If there exists $x \sim \tilde{x}$ then  $\tilde{x}$--stripe is transformed by matrices

\begin{equation}
\begin{array}{c}
\begin{tikzpicture}

\matrix (first) [tbl5,  name=tbl,
minimum height=20pt,
text width=20pt%
] at (0,0)
{
\bullet_0 & 0 & 0 & 0 & 0 \\
\bullet & \ast_1 & \odot & \odot & \odot \\
\bullet & \ast & \ast_2 & \odot & \odot \\
\bullet & \ast & \ast & \ast_{1'} & \odot \\
\bullet & \ast & \ast & \ast & \ast_3\\
};

\sfrm{tbl}{5}{5};
\node[base left=5pt of tbl-3-1.west, yshift=-5pt]{$\tilde{F} = F =$};
\end{tikzpicture}
\end{array}
\end{equation}

whereas the non--reduced part of $x$--stripe is transformed by the matrix $(\bullet_0)$.
\end{itemize}
The new decorated bunch of chains $\dX^{[\scX, \scY]}$ is defined as follows.
\begin{itemize}
\item We add to $\dF_\imath$ two decorated elements $y_z$ and $z_y$ which satisfy
$y \rhd y_z$, $z \rhd z_y$ and inherit from their parent elements $y$ and $z$ all order relations.
\item If $x \sim \tilde{x}$ for some  $\tilde{x}\in \dX_\jmath$ then we add to $\dX_\jmath$ new \emph{decorated} elements
$\tilde{x}_{zy}, \tilde{x}_{yz}, \tilde{x}_y$ and $\tilde{x}_z$. They inherit all order relations
from their parent element $\tilde{x}$ and satisfy $\tilde{x} < \tilde{x}_{zy} \lhd
\tilde{x}_z \lhd \tilde{x}_{yz} \lhd
\tilde{x}_y$.
\item Finally, we impose new equivalence relations
$y_{z} \sim \tilde{x}_z$, $z_{y} \sim \tilde{x}_y$ and $\tilde{x}_{yz} \sim \tilde{x}_{zy}$.
\end{itemize}
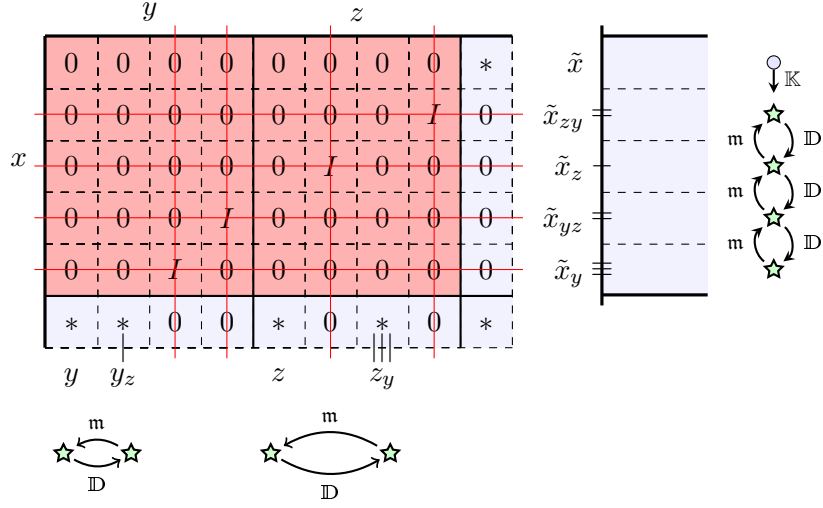
\begin{center}
\begin{tikzpicture}
\matrix (first) [tbl5,  name=tbl,
row 1/.style={minimum height=20pt},
row 2/.style={minimum height=20pt},
row 3/.style={minimum height=20pt},
row 4/.style={minimum height=20pt},
row 5/.style={minimum height=20pt},
row 6/.style={minimum height=20pt},
column 1/.style={text width=20pt},
column 2/.style={text width=20pt},
column 3/.style={text width=20pt},
column 4/.style={text width=20pt},
column 5/.style={text width=20pt},
column 6/.style={text width=20pt},
column 7/.style={text width=20pt},
column 8/.style={text width=20pt},
column 9/.style={text width=20pt},
] at (0,0)
{
~&~&~&~&        ~&~&~&~&     *\\
~&~&~&~&        ~&~&~&~&     0\\
~&~&~&~&        ~&~&~&~&     0\\
~&~&~&~&        ~&~&~&~&     0\\
~&~&~&~&        ~&~&~&~&     0\\
*&*&0&0&        *&0&*&0&     *\\
};

\node[fill, color=red!30, fit=(tbl-1-1)(tbl-5-8), inner sep=0pt]{};
\foreach \x in {1,...,5}{
\foreach \y in {1,...,8}{
\node at (tbl-\x-\y){0};
}};

\node[fill=red!30] at (tbl-2-8){$I$};
\node[fill=red!30] at (tbl-3-6){$I$};
\node[fill=red!30] at (tbl-4-4){$I$};
\node[fill=red!30] at (tbl-5-3){$I$};


\foreach \x in {1,...,6}{
\foreach \y in {1,...,9}{
\hdline{tbl}{\x}{\y};
\vdline{tbl}{\x}{\y};
}};

\hsline{tbl}{5}{9};

\vsline{tbl}{6}{4};
\vsline{tbl}{6}{8};
\draw[very thick, solid](tbl-1-1.north west)-- (tbl-1-9.north east);
\draw[very thick, solid](tbl-1-1.north west)-- (tbl-6-1.south west);

\draw[red] ([xshift=-4pt] tbl-2-1.west) to ([xshift=4pt] tbl-2-9.east);
\draw[red] ([xshift=-4pt] tbl-3-1.west) to ([xshift=4pt] tbl-3-9.east);
\draw[red] ([xshift=-4pt] tbl-4-1.west) to ([xshift=4pt] tbl-4-9.east);
\draw[red] ([xshift=-4pt] tbl-5-1.west) to ([xshift=4pt] tbl-5-9.east);

\draw[red] ([yshift=4pt] tbl-1-3.north) to ([yshift=-4pt] tbl-6-3.south);
\draw[red] ([yshift=4pt] tbl-1-4.north) to ([yshift=-4pt] tbl-6-4.south);
\draw[red] ([yshift=4pt] tbl-1-6.north) to ([yshift=-4pt] tbl-6-6.south);
\draw[red] ([yshift=4pt] tbl-1-8.north) to ([yshift=-4pt] tbl-6-8.south);

\node[base left=3pt of tbl-3-1.west]{$x$};
\node[above =2pt of tbl-1-2.north east]{$y$};
\node[above =2pt of tbl-1-6.north east]{$z$};

\matrix  [tbl5,  name=rt,
row 1/.style={minimum height=20pt},
row 2/.style={minimum height=20pt},
row 3/.style={minimum height=20pt},
row 4/.style={minimum height=20pt},
row 5/.style={minimum height=20pt},
row 6/.style={minimum height=20pt},
column 1/.style={text width=40pt},
] at (5,0)
{
~\\
~\\
~\\
~\\
~\\
~\\
};
\node[fill, color=white!00, fit=(rt-6-1), inner sep=0pt]{};

\draw[very thick] ([yshift=4pt] rt-1-1.north west) -- ([yshift=-4pt] rt-6-1.north west);
\draw[very thick] (rt-1-1.north west) -- (rt-1-1.north east);
\draw[very thick] (rt-6-1.north west) -- (rt-6-1.north east);

\hdline{rt}{1}{1};
\hdline{rt}{2}{1};
\hdline{rt}{3}{1};
\hdline{rt}{4}{1};

\node[dot] (a1) at ([xshift= 25pt] rt-1-1.east){};
\node[dv] (a2) at ([xshift= 25pt] rt-2-1.east){};
\node[dv] (a3) at ([xshift= 25pt] rt-3-1.east){};
\node[dv] (a4) at ([xshift= 25pt] rt-4-1.east){};
\node[dv] (a5) at ([xshift= 25pt] rt-5-1.east){};

\draw[thick,-stealth](a1) to node[base right]{$_\bK$}(a2);

\draw[ thick, -stealth, bend left=50, ] (a2)to node[base right]{$_\bD$} (a3) ;
\draw[ thick, -stealth, bend left=50, ] (a3) to node[base left]{$_\idm$} (a2);

\draw[ thick, -stealth, bend left=50, ] (a3) to node[base right]{$_\bD$} (a4) ;
\draw[ thick, -stealth, bend left=50, ] (a4) to node[base left]{$_\idm$} (a3);

\draw[ thick, -stealth, bend left=50, ] (a4) to node[base right]{$_\bD$} (a5) ;
\draw[ thick, -stealth, bend left=50, ] (a5) to node[base left]{$_\idm$} (a4);

\node[base left=3pt of rt-1-1]{$\tilde{x}$};
\node[base left=3pt of rt-2-1]{$\tilde{x}_{zy}$};
\node[base left=3pt of rt-3-1]{$\tilde{x}_z$};
\node[base left=3pt of rt-4-1]{$\tilde{x}_{yz}$};
\node[base left=3pt of rt-5-1]{$\tilde{x}_y$};

\node at (rt-3-1.west){$-$};
\node at (rt-2-1.west){$=$};
\node at (rt-4-1.west){$=$};
\node at (rt-5-1.west){$\equiv$};

\node[below=4pt of tbl-6-1]{${y}$};
\node[below=4pt of tbl-6-2]{${y}_z$};
\node at (tbl-6-2.south){$|$};

\node[below=4pt of tbl-6-5]{${z}$};
\node[below=4pt of tbl-6-7]{$z_y$};
\node at (tbl-6-7.south){$|||$};

\node[dv] (d) at ([xshift= -3pt,yshift= -40pt] tbl-6-1.south){};
\node[dv] (c) at( [xshift= 3pt,yshift= -40pt] tbl-6-2.south){};
\draw[ thick, ->, bend right=30, ] (c) to node[above]{$_\idm$} (d);
\draw[ thick, ->, bend right=30, ] (d) to node[below]{$_\bD$} (c);

\node[dv] (d2) at ([xshift= -3pt,yshift= -40pt] tbl-6-5.south){};
\node[dv] (c2) at( [xshift= 3pt,yshift= -40pt] tbl-6-7.south){};
\draw[ thick, ->, bend right=30, ] (c2) to node[above]{$_\idm$} (d2);
\draw[ thick, ->, bend right=30, ] (d2) to node[below]{$_\bD$} (c2);

\end{tikzpicture}
\end{center}
Again, we have a representation equivalence (\ref{E:ReductionCategorielle}).
 The translation rule (\ref{E:RuleToRewrite}) is as follows:
$$
\tilde{x}_{yz} \sim \tilde{x}_{zy} \, \mapsto \, \tilde{x} \sim x \stackrel{0}- y \sim z \stackrel{0}- x \sim \tilde{x}, \;
z_y \sim \tilde{x}_y \, \mapsto \, z \sim y \stackrel{0}- x \sim \tilde{x}$$
and $
y_z \sim \tilde{x}_z \, \mapsto \, y \sim z \stackrel{0}- x \sim \tilde{x}.
$

\medskip
\noindent
\underline{Case 3c}. Consider now the general case when
$\scX = \{x\}$ with $x$ not decorated and $\scY = \{y_1 \rhd \dots \rhd y_n\}$ with all
$y_1, \dots, y_n$ decorated.
First one shows  that any indecomposable object of the restricted decorated bunch of chains $\dX_{(\scX, \scY)}$ is isomorphic either to $S(x)$, or to $S(a)$ for some untied  $a \in \scY$, or to
$S(b \sim c)$ where $b, c \in \scY$, or to  $S(x-a)$, $S(x-b\sim c)$, $S(x-b\sim c -x)$, where $a, b, c \in \scY$ are as above.
The new decorated bunch of chains $\dX^{[\scX, \scY]}$ is defined as follows.
\begin{itemize}
\item For any pair $b \sim c$ in $\scY$ we add to $\dF_\imath$ new decorated elements $b_c$ and
$c_b$, which satisfy $b \rhd b_c$ and $c \rhd c_b$ and
inherit all order relations from their parent elements $b$ and $c$.
\item For any pair $a \sim \tilde{a}$  with $a \in \scY$ such that  $\tilde{a}\in \dX_\jmath$ and $\tilde{a} \notin \scY$,
we add to $\dX_\jmath$ a new decorated element $\tilde{a}_x$ which inherits  all order relations
from its parent $\tilde{a}$ and satisfies  $\tilde{a} \rhd \tilde{a}_x$.
\item If there exists $x \sim \tilde{x}$ for some $\tilde{x} \in  \dX_\sigma$ then for any $a \in \scY$ we add to $\dX_\sigma$ a new \emph{decorated} element $\tilde{x}_a$. Moreover, for any pair $b, c \in \scY$ such that $b \sim c$ we add to $\dX_\sigma$ a pair of decorated elements
$\tilde{x}_{bc}$ and $\tilde{x}_{cb}$.  All these new elements inherit all orderings from their parent
$\tilde{x}$.
\item
Assume that  $\tilde{x} \lVert x$ and $\{a \lhd b \lhd f \lhd c \lhd d\} \subseteq \scY$ with  $b \sim c$.
Then we have $$
\tilde{x} < \tilde{x}_d \lhd \tilde{x}_{cb} \lhd \tilde{x}_c \lhd \tilde{x}_f \lhd \tilde{x}_{bd} \lhd \tilde{x}_b \lhd \tilde{x}_a.
$$
\item In the above notations we impose the following equivalence relations.
\begin{itemize}
\item If $b, c \in \scY$ are such that $b \sim c$ then $\tilde{x}_b \sim c_b$, $\tilde{x}_c \sim b_c$ and
$\tilde{x}_{bc} \sim \tilde{x}_{cb}$.
\item If for $a \in \scY$ there exists $\tilde{a} \notin \scY$ such that $a\sim \tilde{a}$ then $\tilde{x}_a \sim \tilde{a}_x$.
\end{itemize}
\end{itemize}
Then we have a representation equivalence (\ref{E:ReductionCategorielle})  with translation rules (\ref{E:RuleToRewrite}) for strings and bands given by
$$
\tilde{x}_b \sim c_b \, \mapsto \,  \tilde{x} \sim x \stackrel{0}- b \sim c, \;
\tilde{x}_c \sim b_c \, \mapsto \,  \tilde{x} \sim x \stackrel{0}- c \sim b, \;
\tilde{x}_{bc} \sim \tilde{x}_{cb} \,  \mapsto \, \tilde{x} \sim x \stackrel{0}- b \sim c \stackrel{0}- x \sim \tilde{x}
$$
and
$\tilde{x}_a \sim \tilde{a}_x \, \mapsto \,  \tilde{x} \sim x \stackrel{0}- a \sim \tilde{a}$.
 The case when $\scX$ is decorated and $\scY$ not decorated is completely analogous. \qed

\subsection{Decorated Kronecker problem}\label{Ex:CanonicalFormsDecKronecker}
\setcounter{case}0
  Consider now the decorated bunch of chains $\dX$ given
     in Example \ref{E:bunchKronecker}, arising in the classification of maximal Cohen--Macaulay modules
     over $T_{23\infty}$. Note that all elements of $\dX$ are tied. This implies  that without loss of generality, we
     can begin any word $w$ (cyclic or not) with a column element. This convention has another advantage as it reduces
     the variety of non--equal but isomorphic canonical forms. According to Theorem \ref{list},
     there are four
 types of indecomposable objects in $\rep(\dX)$, namely:

\begin{case}
  Bands $B\bigl((w,\rho),m,\pi\bigr)$, where
  $$
  (w, \rho) = {}\lha y_2 \sim y_1 \stackrel{\mu_1}- x_1 \sim x_2 \stackrel{\nu_1}- y_2 \sim y_1
  \stackrel{\mu_2}- x_1 \sim x_2 \stackrel{\nu_2}- \dots - y_2 \sim y_1 \stackrel{\mu_n}- x_1 \sim x_2 \str{\nu_n}{\rha},
  $$
  $m$ is any natural number and $\pi \ne \xi$ is any irreducible polynomial.
  We may without loss of generality assume that $\mu_1 = \dots = \mu_n = 1$. In this case, the
  only condition on the decoration is that the sequence of integers $\bnu=\row\nu n$ is non--periodic.
We set  $M_1 = I_{dmn}$, where $d = \deg(\pi)$
 while
 \begin{center}
  \begin{tikzpicture}
\matrix [tbl5,text width=30pt, minimum height=30pt,  name=table] at(6,0)
{
0& t^{\nu_1}I &0& \dots &0 \\
 0&0& t^{\nu_2}I  & \dots &0 \\
  \vdots  &  \vdots  &  \ddots  &  \ddots  &   \vdots   \\
0&0&0& \ddots & t^{\nu_{n-1}}I \\
t^{\nu_{n}} F&0&0&\dots&0\\
};
\sfrm{table}{5}{5}
\node[base left=2pt of table.west]{$M_2 =$};

\end{tikzpicture}
\end{center}

\noindent
 where  $I$ is the identity  $dm\xx dm$ matrix and $F$ is the Frobenius block of $\pi^m$.
\end{case}

\begin{case} The ``degenerate band'' corresponding to the ``forbidden'' polynomial
$\pi = x$ is given by the string
$S(w, \rho)$, where
$$
  (w, \rho) =  y_2 \sim y_1 \stackrel{\mu_1}- x_1 \sim x_2 \stackrel{\nu_1}- y_2 \sim y_1
  \stackrel{\mu_2}- x_1 \sim x_2 - \dots \stackrel{\nu_{n-1}}- y_2 \sim y_1 \stackrel{\mu_n}- x_1 \sim x_2.
  $$
Again, the decoration $\rho$ can be chosen in such a way that
$\mu_1 = \dots = \mu_n = 1$.
 Then the matrix $M_1 = I_{n+1}$, while
\begin{center}
  \begin{tikzpicture}
\matrix [tbl5,text width=22pt, minimum height=22pt,  name=table] at(6,0)
{
0& t^{\nu_1} &0& \dots &0 \\
 0&0& t^{\nu_2} & \dots &0 \\
  \vdots  &  \vdots  &  \ddots  &  \ddots  &   \vdots   \\
0&0&0& \ddots & t^{\nu_n} \\
0&0&0&\dots&0\\
};
\sfrm{table}{5}{5}
\node[base left=2pt of table.west]{$M_2 =$};

\end{tikzpicture}
\end{center}
Of course,
there is also the symmetric object, obtained by permuting indices  $1$ and $2$.
\end{case}

\begin{case} There exists a family of  ``non--square'' indecomposable representations $S(w, \rho)$:
\begin{center}
  \begin{tikzpicture}

\matrix (first) [tbl5,text width=22pt, minimum height=22pt, name=table] at (0,0)
{
1&0& \dots &0&0\\
0&1& \dots &0&0 \\
 \vdots  &  \vdots  & \vdots &  \ddots  &  \vdots   \\
0&0& \dots &1&0\\
};
\sfrm{table}{4}{5}
\node[base left=2pt of table.west]{$M_1=$};
\matrix (second) [tbl5,text width=22pt, minimum height=22pt,  name=table] at(6,0)
{
0& t^{\nu_1} &0& \dots &0 \\
 0&0& t^{\nu_2} & \dots &0 \\
  \vdots  &  \vdots  &  \ddots  &  \vdots  &   \vdots   \\
0&0&0& \dots & t^{\nu_n} \\
};
\sfrm{table}{4}{5}
\node[base right=2pt of table.east]{$= M_2$};

\end{tikzpicture}
\end{center}
given by the decorated word
$$
  (w, \rho) =  y_2 \sim y_1 \stackrel{\mu_1}- x_1 \sim x_2 \stackrel{\nu_1}- y_2 \sim y_1
  \stackrel{\mu_2}- x_1 \sim x_2 \stackrel{\mu_2}- \dots - y_2 \sim y_1
  \stackrel{\mu_n}- x_1 \sim x_2 \stackrel{\nu_n}- y_2 \sim y_1.
  $$
  As before, we have posed  $\mu_1 = \dots = \mu_n = 1$. Moreover, we may additionally (and without loss of generality) assume that $\nu_n = 1$.
\end{case}

\begin{case} Finally, we have   the ``dual object'' to the previous string object:
\begin{center}
  \begin{tikzpicture}
\matrix (first) [tbl5,text width=22pt, minimum height=22pt, name=table]
at (0,0) {
1&0& \dots &0\\
0&1& \dots &0\\
 \vdots  & \vdots &  \ddots  &  \vdots   \\
0&0& \dots &1\\
0&0& \dots &0\\
};
\sfrm{table}{5}{4};
\node[base left=2pt of table.west]{$M_1=$};
\matrix [tbl5,text width=22pt,  minimum height=22pt, name=table]
at(6,0) {
0&0& \dots &0\\
 t^{\nu_1} &0& \dots &0 \\
0& t^{\nu_2} & \dots &0\\
  \vdots  &  \vdots  &  \ddots  &   \vdots    \\
0&0& \dots & t^{\nu_n} \\
};
\sfrm{table}{5}{4}; \node[base right=2pt of table.east]{$= M_2$};
\end{tikzpicture}
\end{center}

\noindent
given by the decorated word
$$
  (w, \rho) =   x_2 \sim x_1 \stackrel{\mu_1}- y_1 \sim y_2
  \stackrel{\nu_1}- x_2 \sim x_1 \stackrel{\mu_2}- \dots \stackrel{\mu_n}- y_1 \sim y_2 \stackrel{\nu_n}- x_2 \sim x_1.
  $$
  As above, we  have posed  $\mu_1 = \dots = \mu_n = 1$.
\end{case}

\noindent
The isomorphism classes of strings in all Cases 2--4 are uniquely determined by the corresponding sequences $\bnu = (\nu_1, \dots, \nu_n)$.
Two bands $B\bigl(\bnu, m, \pi)$ and  $B\bigl(\bnu', m', \pi')$ are isomorphic if and only if $m = m'$, $\pi = \pi'$ and
$\bnu'$ is a rotation of $\bnu$.

\section{Maximal Cohen--Macaulay modules over degenerate cusps--I}\label{sec7}

\noindent
In this section we consider in details three other
important examples  of degenerate cusps:
 $\kk\llbracket x, y, z\rrbracket/(xyz)$,
$\kk\llbracket x, y, z, w\rrbracket/(xy, zw)$ and $\kk\llbracket x, y, z, u, v\rrbracket/(xz, xu, yu, yv, zv)$.

\subsection{Maximal Cohen--Macaulay modules over  $\kk\llbracket x,y,z\rrbracket/(xyz)$}\label{SS:degcusp2}

Consider a $T_{\infty\infty\infty}$--singularity
$\rA = \kk\llbracket x,y,z\rrbracket/(xyz)$. Let  $$\pi: \rA \lar \rR =  \rR_1 \times \rR_2 \times \rR_3 =
\kk\llbracket x_1, y_2\rrbracket \times \kk\llbracket y_1, z_2\rrbracket \times \kk\llbracket z_1, x_2\rrbracket$$ be its normalization, where
$\pi(x) = x_1 + x_2$, $\pi(y) = y_1 + y_2$ and $\pi(z) = z_1 + z_2$. For the conductor ideal  $I =
\Ann_\rA(\rR/\rA)$ we have:
$$
I = \bigl\langle xy, xz, yz\bigl\rangle_\rA =
\bigl\langle x_1 y_2, y_1 z_2, z_1 x_2\bigl\rangle_\rR.
$$
Next, we have:
$
\bar\rA = \rA/I = \kk\llbracket x,y,z\rrbracket/(xy, xz, yz)
$ and
$
\bar\rR = \rR/I = \bar\rR_1 \times \bar\rR_2 \times \bar\rR_3 =
\kk\llbracket x_1, y_2\rrbracket/(x_1 y_2)
 \times \kk\llbracket y_1, z_2\rrbracket/(y_1 z_2)
 \times \kk\llbracket z_1, x_2\rrbracket/(z_1 x_2).
$
It is convenient to introduce the ring $\widetilde\rA = \kk\llbracket x\rrbracket
\times \kk\llbracket y\rrbracket
\times \kk\llbracket z\rrbracket.
$
Note that  the canonical map $\bar\rA \to \bar\rR$ factorizes through $\widetilde\rA$.
It is convenient to visualize an object of the category $\Rep(\dX_A)$ as a representation of the
``decorated'' quiver
(\ref{E:affinequiv}) over the field $\rK = \kk\llbrace t\rrbrace$:
\begin{equation}\label{E:MPTinfty}
\begin{xy}
(0,0)*+{\bK^{m_1}}="A";
(10,17)*+{\bK^{n_x}}="B";
(28,17)*+{\bK^{m_3}}="C";
(38,0)*+{\bK^{n_z}}="D";
(10,-17)*+{\bK^{n_y}}="E";
(28,-17)*+{\bK^{m_2}}="F";
{\ar@{->}_-{\Theta_1^x} "B";"A"};
{\ar@{->}^-{\Theta_3^x} "B";"C"};
{\ar@{->}_-{\Theta_3^z} "D";"C"};
{\ar@{->}^-{\Theta_2^z} "D";"F"};
{\ar@{->}^-{\Theta_1^y} "E";"A"};
{\ar@{->}^-{\Theta_2^y} "E";"F"};
\end{xy}
\end{equation}
The isomorphy classes of objects in $\Rep(\dX_A)$ correspond to the transformation rule
\begin{equation}\label{E:ruleTinfty}
\Theta_\imath^a \mapsto S_{\imath a} \Theta_\imath^a T_a^{-1}, \quad
\imath \in \bigl\{1, 2, 3\bigr\}  \; \mathrm{and} \;  a \in \bigl\{x, y, z\bigr\},
\end{equation}
where  $T_a \in \GL_{n_a}(\rK)$ and
$S_{\imath a} \in \GL_{m_\imath}(\bD)$ are such that   $S_{1 x}(0) =  S_{1 y}(0), S_{2y}(0) =  S_{2 z}(0)$ and $S_{3 x}(0) = S_{3 z}(0)$.
{Our goal} is to describe the indecomposable objects of
$\CM^{\mathsf{lf}}(\rA)$.

\begin{definition}\label{D:bands-over-xyz}
Consider the following \emph{band datum} $(\omega, l, \lambda)$, where:
\begin{itemize}
\item  $\omega = \bigl((a_1, b_1, c_1, d_1, e_1, f_2),
(a_2, b_2, c_2, d_2, e_2, f_3), \dots, (a_t, b_t, c_t, d_t, e_t, f_1)\bigr) \in \mathbb{Z}^{6t}$
for some $t \ge 1$
such that $\min(a_i, f_i) = \min(b_i, c_i) = \min(d_i, e_i) = 1$ for all $1 \le i \le t$.
\item $l \in \mathbb{Z}_{>0}$ and
$\lambda \in \kk^*$.
\end{itemize}
\end{definition}

\noindent
Then we attach to  the data $(\omega, l, \lambda)$ the following matrices:
\begin{equation*}
\begin{array}{c}
\begin{tikzpicture}
\matrix (m1) [tbl5,  name=tbl,
minimum height=20pt,
text width=20pt%
] at (0,0)
{
A_1 & 0              & \dots  & 0 \\
0           & A_2    & \dots  & 0 \\
\vdots      & \vdots & \ddots & \vdots \\
0           & 0      & \dots  & A_{t} \\
};

\sfrm{tbl}{4}{4};
\node[base left=5pt of tbl-3-1.north west, yshift=-5pt]{$\Theta_1^x  =$};

\matrix (T1y) [tbl5,  name=tbl,
minimum height=20pt,
text width=20pt%
] at (5,0)
{
B_1 & 0              & \dots  & 0 \\
0           & B_2    & \dots  & 0 \\
\vdots      & \vdots & \ddots & \vdots \\
0           & 0      & \dots  & B_{t} \\
};

\sfrm{tbl}{4}{4};
\node[base left=5pt of tbl-3-1.north west, yshift=-5pt]{$\Theta_1^y  =$};
\matrix (T2y) [tbl5,  name=tbl,
minimum height=20pt,
text width=20pt%
] at (5,-3.5)
{
C_1 & 0              & \dots  & 0 \\
0           & C_2    & \dots  & 0 \\
\vdots      & \vdots & \ddots & \vdots \\
0           & 0      & \dots  & C_{t} \\
};

\sfrm{tbl}{4}{4};
\node[base left=5pt of tbl-3-1.north west, yshift=-5pt]{$\Theta_2^y  =$};
\matrix (T2y) [tbl5,  name=tbl,
minimum height=20pt,
text width=20pt%
] at (10,-3.5)
{
D_1 & 0              & \dots  & 0 \\
0           & D_2    & \dots  & 0 \\
\vdots      & \vdots & \ddots & \vdots \\
0           & 0      & \dots  & D_{t} \\
};

\sfrm{tbl}{4}{4};
\node[base left=5pt of tbl-3-1.north west, yshift=-5pt]{$\Theta_2^z  =$};
\matrix (T3x) [tbl5,  name=tbl,
minimum height=20pt,
text width=20pt%
] at (0,-7)
{
0           & F_2              & \dots  & 0 \\
\vdots      & \vdots & \ddots & \vdots \\
0           & 0    & \dots  &  F_{t}\\
H           & 0      & \dots  & 0 \\
};

\sfrm{tbl}{4}{4};
\node[base left=5pt of tbl-3-1.north west, yshift=-5pt]{$\Theta_3^x  =$};
\matrix (T3z) [tbl5,  name=tbl,
minimum height=20pt,
text width=20pt%
] at (10,-7)
{
E_1 & 0              & \dots  & 0 \\
0           & E_2    & \dots  & 0 \\
\vdots      & \vdots & \ddots & \vdots \\
0           & 0      & \dots  & E_{t} \\
};

\sfrm{tbl}{4}{4};
\node[base left=5pt of tbl-3-1.north west, yshift=-5pt]{$\Theta_3^z  =$};
\end{tikzpicture}
\end{array}
\end{equation*}

\noindent
where $A_k= t^{a_k} I,$
$B_k= t^{b_k} I,$
$C_k= t^{c_k} I,$
$D_k= t^{d_k} I,$
$F_k= t^{f_k} I,$
$E_k= t^{e_k} I,$
and $H=t^{f_1}J$ with
 $I = I_l$  the identity $l \times l$ matrix and $J = J_l(\lambda)$ the Jordan block
of size $l \times l$ with the eigenvalue $\lambda$.  Denote
$
\Theta^x = \Theta_1^x(x_1) + \Theta_3^x(x_2)$,
$\Theta^y = \Theta_1^y(y_2) + \Theta_2^y(y_1)$ and  $
\Theta^z = \Theta_2^z(z_2) + \Theta_3^z(z_1).$
The indecomposable maximal  Cohen--Macaulay $\rA$--module $\mM = \mM(\omega, l, \lambda)$
attached to the band datum $(\omega, l, \lambda)$ is constructed
by the following recipe:

\begin{itemize}
\item Consider the $\rA$--linear morphism
$\bar\Theta = \bigl[\Theta^x| \Theta^y| \Theta^z\bigr]:
\widetilde\rA^{lt} \rightarrow \bar\rR^{lt}$, where we write
$\widetilde\rA^{lt} = \kk\llbracket x\rrbracket^{lt} \oplus \kk\llbracket y\rrbracket^{lt} \oplus
\kk\llbracket z \rrbracket^{lt}$ and use the isomorphisms $\Hom_{\bar\rA}\bigl(\kk\llbracket x\rrbracket, \bar\rR) \cong  \kk\llbracket x\rrbracket$.
\item Consider the torsion free $\rA$--module
$\mL = \mL(\omega, l, \lambda)$ given by the following commutative diagram with exact rows:
$$
\xymatrix
{
0 \ar[r] & I^{lt} \ar[r] \ar[d]_= & \mL \ar[r] \ar[d] &  \widetilde\rA^{lt} \ar[r] \ar[d]^{\bar\Theta}
& 0 \\
0 \ar[r] & I^{lt} \ar[r] & \rR^{lt} \ar[r] &  \bar\rR^{lt} \ar[r]
& 0.
}
$$
\item From the above description it follows that $\mL$ is the
$\rA$--submodule of $\rR^{lt}$ generated by the columns of the matrix
$\Bigl(x_1 y_2 I \,| \, y_1 z_2 I \,| \,  z_1 x_2 I \, | \,
\Theta^x  \, | \,  \Theta^y  \, | \,  \Theta^z  \Bigr)
\in {\Mat}_{lt \times 6 lt}(\rR).$
\item Finally, we have:
$\mM(\omega, l, \lambda) = \mL(\omega, l,\lambda)^{\vee\vee}$. Moreover, the
following sequence is exact:
$
0 \rightarrow \mL(\omega, l, \lambda) \rightarrow \mM(\omega, l, \lambda) \rightarrow H^0_{\{\idm\}}\bigl(\mathsf{coker}(\bar\Theta)\bigr) \rightarrow  0.
$
\end{itemize}
The ring $\rR = \rR_1 \times \rR_2 \times \rR_3$ is a subring of the total ring
of fractions $\rQ(\rA)$. The units of  $\rR_i$ ($1 \le i \le 3$) are the idempotents
$
e_1 =
\frac{xy}{xy + yz + xz}$,
$e_2 =
\frac{yz}{xy + yz + xz}$ and $
e_3 =
\frac{xz}{xy + yz + xz}.
$
It is easy to see that $e_1 + e_2 + e_3 = 1$ and $e_i e_j = \delta_{ij} e_i$ for
all $1 \le i, j \le 3$.
In these notations we can write the elements $x_1, x_2, y_1, y_2, z_1, z_2$ of the normalization
$\rR$
as elements of the total ring of fractions $\rQ(\rA)$ in the following way:
$
x_1 = e_1 x$, $y_2 = e_1 y$, $y_1 = e_2 y$, $z_2 = e_2 z$, $z_1 = e_3 z$
and $x_2 = e_3 x$.
Next, the element $xy + yz + zx \in \rA$ is not a zero divisor. Since the module
$\mL$ is torsion free, we have:
$
\mL \cong (xy + yz + zx) \cdot \mL \subseteq \rA^{lt}.
$
Consider the  matrices:
\begin{equation*}
\begin{tikzpicture}
\matrix (m1) [tbl5,  name=tbl,
minimum height=20pt,
text width=20pt%
] at (0,0)
{
\widetilde{A}_1 & 0              & \dots  & 0 \\
0           & \widetilde{A}_2    & \dots  & 0 \\
\vdots      & \vdots & \ddots & \vdots \\
0           & 0      & \dots  & \widetilde{A}_{t} \\
};

\sfrm{tbl}{4}{4};
\node[base left=5pt of tbl-3-1.north west, yshift=-5pt]{$\widetilde{\Theta}_1^x  =$};

\matrix (T1y) [tbl5,  name=tbl,
minimum height=20pt,
text width=20pt%
] at (5,0)
{
\widetilde{B}_1 & 0              & \dots  & 0 \\
0           & \widetilde{B}_2    & \dots  & 0 \\
\vdots      & \vdots & \ddots & \vdots \\
0           & 0      & \dots  & \widetilde{B}_{t} \\
};

\sfrm{tbl}{4}{4};
\node[base left=5pt of tbl-3-1.north west, yshift=-5pt]{$\widetilde{\Theta}_1^y  =$};
\matrix (T2y) [tbl5,  name=tbl,
minimum height=20pt,
text width=20pt%
] at (5,-3.5)
{
\widetilde{C}_1 & 0              & \dots  & 0 \\
0           & \widetilde{C}_2    & \dots  & 0 \\
\vdots      & \vdots & \ddots & \vdots \\
0           & 0      & \dots  & \widetilde{C}_{t} \\
};

\sfrm{tbl}{4}{4};
\node[base left=5pt of tbl-3-1.north west, yshift=-5pt]{$\widetilde{\Theta}_2^y  =$};
\matrix (T2y) [tbl5,  name=tbl,
minimum height=20pt,
text width=20pt%
] at (10,-3.5)
{
\widetilde{D}_1 & 0              & \dots  & 0 \\
0           & \widetilde{D}_2    & \dots  & 0 \\
\vdots      & \vdots & \ddots & \vdots \\
0           & 0      & \dots  & \widetilde{D}_{t} \\
};
h
\sfrm{tbl}{4}{4};
\node[base left=5pt of tbl-3-1.north west, yshift=-5pt]{$\widetilde{\Theta}_2^z  =$};
\matrix (T3x) [tbl5,  name=tbl,
minimum height=20pt,
text width=20pt%
] at (0,-7)
{
0           & \widetilde{F}_2              & \dots  & 0 \\
\vdots      & \vdots & \ddots & \vdots \\
0           & 0    & \dots  &  \widetilde{F}_{t}\\
\widetilde{H}           & 0      & \dots  & 0 \\
};

\sfrm{tbl}{4}{4};
\node[base left=5pt of tbl-3-1.north west, yshift=-5pt]{$\widetilde{\Theta}_3^x  =$};
\matrix (T3z) [tbl5,  name=tbl,
minimum height=20pt,
text width=20pt%
] at (10,-7)
{
\widetilde{E}_1 & 0              & \dots  & 0 \\
0           & \widetilde{E}_2    & \dots  & 0 \\
\vdots      & \vdots & \ddots & \vdots \\
0           & 0      & \dots  & \widetilde{E}_{t} \\
};

\sfrm{tbl}{4}{4};
\node[base left=5pt of tbl-3-1.north west, yshift=-5pt]{$\widetilde{\Theta}_3^z  =$};
\end{tikzpicture}
\end{equation*}
where
$\widetilde{A}^k=x^{a_k+1}yI,$
$\widetilde{B}^k=y^{b_k+1}xI,$
$\widetilde{C}^k=y^{c_k+1}zI,$
$\widetilde{D}^k=z^{d_k+1}yI,$
$\widetilde{F}^k=x^{a_k+1}zI,$
$\widetilde{E}^k=z^{a_k+1}xI$
and $\widetilde{H}=x^{f_1+1}zJ$
with  $I = I_l$  the identity $l \times l$ matrix and $J = J_l(\lambda)$  the Jordan block
of size $l \times l$ with the eigenvalue $\lambda$.  Denote
$
\widetilde\Theta^x = \widetilde\Theta_1^x + \widetilde\Theta_3^x$,
$\widetilde\Theta^y = \widetilde\Theta_1^y + \widetilde\Theta_2^y$,
$\widetilde\Theta^z = \widetilde\Theta_2^z + \widetilde\Theta_3^z$
and consider the $\rA$--module $\mL'(\omega, l, \lambda)$ generated by the columns
of the matrix $\bigl((xy)^2I \,| \, (yz)^2 I \,| \,  (xz)^2 I \, | \,
\widetilde\Theta^x  \, | \,  \widetilde\Theta^y  \, | \,  \widetilde\Theta^z  \bigr)
\in {\Mat}_{lt \times 6 lt}(\rA).$
\begin{theorem}\label{T:xyz}
Let $\rA = \kk\llbracket x,y,z\rrbracket/(xyz)$,  $(\omega, l, \lambda)$ be a band datum
as in Definition \ref{D:bands-over-xyz} and $\mL'(\omega, l, \lambda)$ the torsion free $\rA$--module defined above.  Then we have:
\begin{itemize}
\item The module
$\mM(\omega, l, \lambda) := \mL'(\omega, l,\lambda)^{\vee\vee}$ is an indecomposable
maximal  Cohen--Macaulay  module over $\rA$, locally free of rank $lt$ on the punctured spectrum.
\item Any indecomposable object of $\CM^{\mathsf{lf}}(\rA)$
is  isomorphic to some module $\mM(\omega, l, \lambda)$.
\item $\mM(\omega, l, \lambda) \cong \mM(\omega', l', \lambda')$ if and only
if $l = l'$, $\lambda = \lambda'$ and $\omega'$ is obtained from $\omega$ by a cyclic shift.
\end{itemize}
\end{theorem}

\begin{remark}
The indecomposable maximal  Cohen--Macaulay $\rA$--modules which are not locally free
on the punctured spectrum, correspond to the string data. They can be described
along similar lines  as above, but there are more cases one needs to consider.
Therefore we leave it to an interested  reader as an exercise.
\end{remark}

\begin{corollary}\label{C:rank-one-on-xyz}
Let $\mM$ be a rank one object of $\CM^{\mathsf{lf}}(\rA)$.
 Then we have:
$$
\mM \cong \mM(\omega, \lambda) :=
\bigl\langle (xy)^2, (yz)^2, (xz)^2,
x^{m_1+1} y + \lambda x^{m_2+1} z,
y^{n_1+1} z + y^{n_2+1} x,
z^{l_1+1} x + z^{l_2+1} y
\bigr\rangle_\rA^{\vee\vee}
$$
for some $\omega = \bigl((m_1, m_2), (n_1, n_2), (l_1, l_2)\bigr) \in \mathbb{Z}^6$ and
$\lambda \in \kk^*$, where
$$\min(m_1, m_2) = \min(n_1, n_2) = \min(l_1, l_2) = 1.$$ Moreover,
$\mM(\omega, \lambda) \cong \mM(\omega', \lambda')$ if and only if $\omega = \omega'$ and
$\lambda = \lambda'$.
\end{corollary}

\begin{remark}
It is very instructive to consider the case $m_1 = m_2 = n_1 = n_2 = l_1 = l_2 = 1$ and
$\lambda = 1$. First note that
$$
\mL:=
\bigl\langle (xy)^2, (yz)^2, (xz)^2,
x^{2} y + x^{2} z,
y^{2} z + y^{2} x,
z^{2} x + z^{2} y
\bigr\rangle = \bigl\langle
x^{2} y + x^{2} z,
y^{2} z + y^{2} x,
z^{2} x + z^{2} y
\bigr\rangle.
$$
Let $r = xy + yz + xz \in \rA$. Then we have: $\idm \cdot r \subset \mL$. From Lemma \ref{L:fact-on-Macaulaf}
it follows that $\mL^{\vee\vee} = (r) \cong \rA$. Hence,
$
\mM\Bigl(\bigl((1,1), (1, 1), (1, 1)\bigr), 1\Bigr) \cong  \rA,
$
what of course matches with the general theory.
\end{remark}

\noindent
The classification of rank one objects of $\CM^{\mathsf{lf}}(\rA)$  obtained in Corollary \ref{C:rank-one-on-xyz} can be elaborated one step further.

\begin{proposition}\label{P:xyz}
Let $\mM$ be a non--regular rank one object of $\CM^{\mathsf{lf}}(\rA)$.
 Then its minimal number of generators is equal to  two or  three.

\medskip
\noindent
1. Assume  $\mM$ is generated by two elements.  Then there exists
 a bijection $\{u, v, w\} \rightarrow \{x, y, z\}$, a pair of integers
$p, \, q \ge 1$ and $\lambda \in \kk^*$
such that $\mM = \mathsf{coker}(\rA^2 \xrightarrow{\theta} \rA^2)$ with
$\theta = \theta_i\bigl((p,q), \lambda)\bigr)$, $1 \le i \le 3$, where
$$
\theta_1\bigl((p,q), \lambda)\bigr) =
\left(
\begin{array}{cc}
u & 0 \\
v^p + \lambda w^q & vw
\end{array}
\right), \quad \theta_2\bigl((p,q), \lambda)\bigr) =
\left(
\begin{array}{cc}
\lambda u + v^p w^q & w^{q+1} \\
u^{q+1} & vw
\end{array}
\right)
$$
and
$
\theta_3\bigl((p,q), \lambda)\bigr) = \theta_1\bigl((p,q), \lambda)\bigr)^{\mathsf{tr}}.
$

\medskip
\noindent
2. Assume $\mM$ is generated by three elements. Then there exists
 a bijection $\{u, v, w\} \rightarrow \{x, y, z\}$, a triple  of integers
$m, \,  n, \,  l \ge 1$ and $\lambda \in \kk^*$
such that $\mM = \mathsf{coker}(\rA^3 \xrightarrow{\theta} \rA^3)$ with
$\theta = \theta_i\bigl((m, n, l), \lambda\bigr)$, $4 \le i \le 7$, where
$$
\theta_4\bigl((m, n, l), \lambda\bigr) =
\left(
\begin{array}{ccc}
u & w^l & 0 \\
0 & v & u^m \\
\lambda v^n & 0 & w
\end{array}
\right),  \quad
\theta_5\bigl((m, n, l), \lambda\bigr) =
\left(
\begin{array}{ccc}
u & w^l &  \lambda v^n\\
0 & v & u^m \\
0 & 0 & w
\end{array}
\right),
$$
$
\theta_6\bigl((m, n, l), \lambda\bigr) =
\theta_4\bigl((m, n, l), \lambda\bigr)^{\mathsf{tr}}
\quad \mbox{\rm and}
\quad
\theta_7\bigl((m, n, l), \lambda\bigr) =
\theta_5\bigl((m, n, l), \lambda\bigr)^{\mathsf{tr}}.
$
\end{proposition}

\begin{proof} By  Corollary \ref{C:rank-one-on-xyz} we know that
$\mM \cong \mM(\omega, \lambda) = \mL(\omega, \lambda)^{\vee\vee}$
for some $\omega = $  $ \bigl((m_1, m_2), $ $  (n_1, n_2), (l_1, l_2)\bigr) \in \mathbb{Z}^6$ and
$\lambda \in \kk^*$, where
$\min(m_1, m_2) = \min(n_1, n_2) = \min(l_1, l_2) = 1$.

\medskip
\noindent
\underline{Case 1a}. Assume $m_1 = n_1 = l_1 = 1$. Denote for simplicity
$m_2 = m$, $n_2 = n$, $l_2 = l$ and
$$
u_1 = x^2 y + \lambda x^{m+1}z, \;
u_2 = y^2 z + y^{n+1}x
\; \mbox{\rm and} \;
u_3 = z^2 x + z^{l+1} y.
$$
Then we have:
$
(xy)^2 = x u_1, \; (yz)^2 = y u_2$ and $(xz)^2 = z u_3.
$
Hence, $\mL = \langle u_1, u_2, u_3\rangle_\rA$. Moreover,
$\mL$ is already maximal  Cohen--Macaulay and we have an exact sequence
$$
\rA^3
\xrightarrow{
\left(
\begin{smallmatrix}
z & - y^n & 0 \\
0 & x & -z^l \\
-\lambda x^m & 0 & y
\end{smallmatrix}
\right)
}
\rA^3 \lar \mL \lar 0.
$$

\medskip
\noindent
\underline{Case 1b}. Assume $m_2 = n_2 = l_2 = 1$. Again, we denote
$m_1 = m$, $n_1 = n$, $l_1 = l$ and
$$
v_1 = x^{m+1} y + \lambda x^2 z, \;
v_2  = y^{n+1} z + y^{2} x
\; \mbox{\rm and} \;
v_3 = z^{l+1} x + z^2  y.
$$
Then we have:
$
(xy)^2 = z v_1, \; (yz)^2 = y v_3$ and $\lambda  (xz)^2 = z v_1.
$
Hence, $\mL = \langle v_1, v_2, v_3\rangle_\rA$. As in the previous case,
$\mL$ is maximal  Cohen--Macaulay and we have an exact sequence
$$
\rA^3
\xrightarrow{
\left(
\begin{smallmatrix}
\lambda x  &  0 &  - y^n \\
- z^l  & y & 0  \\
0 & -x^m  & z
\end{smallmatrix}
\right)
}
\rA^3 \lar \mL \lar 0.
$$

\medskip
\noindent
\underline{Case 2}. Assume $m_1 = l_1 = n_2 = 1$. Denote
$m_2 = m$, $n_1 = n$ and $l_2 = l$. The  case $n =1$  has been  already treated  in
Case 1a. Hence, without  loss of generality we may assume that $n \ge 2$.
Denote
$$
w_1 = x^2 y + \lambda x^{m+1} z, \;
w_2 = y^{n+1} z + y^2 x
\; \,
\mbox{\rm and}
\; \,
w_3 = z^2 x + z^{l+1} y.
$$
We have: $(xy)^2 = x w_1$ and $(xz)^2 = x w_3$. Hence, $\mL =
\bigl\langle (yz)^2, w_1, w_2, w_3\bigr\rangle_\rA$. Note that the case $m = l = 1$ has been
already considered in Case 1b.

\medskip
\noindent
\underline{Case 2a}. Assume $m = 1$ and $l \ge 2$. Consider the element
$r = xy + y^n z + \lambda(xz + z^l y) \in \rA$ and the
module $\widetilde\mL = \langle \mL, r\rangle_\rA  \subseteq  \rA$.
First note that
$$
x r = w_1, \;
y r = w_2 + \lambda z^{l-2} (yz)^2\;
\mbox{\rm and} \;
z r = y^{n-2} (yz)^2 + \lambda w_3.
$$
By Lemma \ref{L:fact-on-Macaulaf}, the Macaulayfications  of the modules $\mL$ and
$\widetilde\mL$ are isomorphic. Moreover,
$\widetilde\mL = \bigl\langle r, (yz)^2\bigr\rangle_\rA$ is maximal  Cohen--Macaulay and
we have an exact sequence
$$
\rA^2
\xrightarrow{
\left(
\begin{smallmatrix}
x & y^{n-1} + \lambda z^{l-1} \\
0 & - yz
\end{smallmatrix}
\right)
}
\rA^2 \lar \widetilde\mL \lar 0.
$$

\medskip
\noindent
\underline{Case 2b}. Assume $m \ge 2$.  The case $l = 1$ reduces  to the
Case 2a. Hence, we may without  loss of generality assume that $ l \ge 2$.
Consider the element $r = xy + y^n z + \lambda x^m z
\in \rA$ and $\widetilde\mL = \langle \mL, r\rangle_\rA \subseteq \rA$. Then we have:
$
x r = w_1, \;
yr = w_2$ and $
zr = y^{n-2} (yz)^2 + \lambda x^{m-1} w_3.
$
Hence, the Macaulayfications  of $\mL$ and $\widetilde\mL$ are isomorphic.
Next, $\widetilde\mL = \bigl\langle (yz)^2, w_3, r \bigr\rangle_\rA$ is maximal  Cohen--Macaulay
and it has a presentation
$$
\rA^3
\xrightarrow{
\left(
\begin{smallmatrix}
x & - z^{l-1} & - y^{n-2} \\
0 &  y & - \lambda x^{m-1} \\
0 & 0 & z
\end{smallmatrix}
\right)
}
\rA^3 \lar \widetilde\mL \lar 0.
$$
Note that the case $n = 2$ has to be treated separately because in that case
the presentation is not minimal.
Indeed,  $(yz)^2 = z (xy + y^2 z + \lambda x^m z)
= z r$. Hence, $$\widetilde\mL =
\bigl\langle xz^2 + y z^{l+1}, xy + y^2 z + \lambda x^m z \bigr\rangle_\rA \subseteq \rA$$
is generated by two elements. It is easy to see that $\widetilde\mL$ has a presentation
$$
\rA^2
\xrightarrow{
\left(
\begin{smallmatrix}
\lambda y + x^{m-1} z^{l-1} & \; z^l \\
x^m  & \; xz
\end{smallmatrix}
\right)
}
\rA^2 \lar \widetilde\mL \lar 0.
$$
It remains to observe that
the remaining cases reduce to the ones considered above.
\end{proof}

\begin{remark}\label{R:gradability}
Contrary to the case of simple elliptic singularities \cite[Proposition 5.23]{KahnDiss}, not all
indecomposable maximal Cohen--Macaulay modules over $\rA$ are gradable.
For example,
$
\mathsf{coker}
\left(
\begin{array}{cc}
x & y^p + \lambda z^q \\
0 & yz
\end{array}
\right)
$
is not gradable for $(p, q) \ne (2, 2)$ and $\lambda \in \kk^*$,
since  its first Fitting ideal $(x, yz, y^p + \lambda z^q)$ is not graded.
\end{remark}

\begin{remark}\label{R:HMS} In recent works of  Sheridan \cite[Theorem 1.2]{Sheridan} and Abouzaid et al.~\cite[Section 7.3]{AAEKO}
a version of the homological mirror symmetry for the category
$\underline{\CM}^{\mathsf{lf}}(\rA)$  was established.
We hope that our result will contribute to a better understanding of the Fukaya side of this correspondence. In particular, it would be interesting to describe explicitly the symplectic images of rank one matrix factorizations obtained in this subsection.
\end{remark}

\subsection{Maximal  Cohen--Macaulay modules over $\kk\llbracket x,y,u,v\rrbracket/(xy, uv)$}\label{SS:degcusp3}
It seems that the only
concrete examples of families of indecomposable maximal Cohen--Macaulay modules
over surface singularities, which have  been constructed so far,
deal with  the case of  hypersurface singularities.
From this perspective it is particularly interesting  to consider
the case of the  degenerate cusp $\rA = \kk\llbracket x,y,u,v\rrbracket/(xy, uv)$.
Indecomposable maximal Cohen--Macaulay modules over $\rA$ can be described using
essentially the same technique as in  Subsection
\ref{SS:degcusp2}. Since the computations do not contain any new phenomena, we omit them
and only state the final  result.

\begin{proposition}
Denote  $J = \bigl\langle
(xu)^2,  (xv)^2, (yu)^2, (yv)^2\bigr\rangle_\rA$. Let
$\mM$ be a rank one object in  $\CM^{\mathsf{lf}}(\rA)$.  Then
$
\mM \cong \mM(\omega, \lambda) := $
$$
\Bigl\langle J,
 x^{m_1+1} u + \lambda x^{m_2+1} v,
u^{n_1+1} y  + u^{n_2+1} x,
y^{p_1+1} v  +  y^{p_2+1} u,
v^{q_1+1} x  +  v^{q_2+1} y
\Bigr\rangle_\rA^{\vee\vee} \subseteq \rA
$$
for some $\omega = \bigl((m_1, m_2), (n_1, n_2), (p_1, p_2), (q_1, q_2)\bigr) \in
\mathbb{Z}^8$ and
$\lambda \in \kk^*$, where
$$\min(m_1, m_2) = \min(n_1, n_2) = \min(p_1, p_2) = \min(q_1, q_2) = 1.$$ Moreover,
$\mM(\omega, \lambda) \cong \mM(\omega', \lambda')$ if and only if $\omega = \omega'$ and
$\lambda = \lambda'$.
\end{proposition}

\begin{remark}
Let $m_1 = n_1 = p_1 = q_1 = 1$ and
$m_2 = m, n_2 = n, p_2 = p $, $q_2 = q$.
\begin{itemize}
\item If $m = n = p = q = 1$ and $\lambda = 1$ then $\mM(\omega, \lambda) \cong \rA$.
\item Otherwise, we have:
$$ \mM(\omega, \lambda) \cong
\bigl\langle x^{2} u + \lambda x^{m+1} v,
u^{2} y  + u^{n+1} x,
y^{2} v  +  y^{p+1} u,
v^{2} x  +  v^{q+1} y
\bigr\rangle_\rA \subseteq \rA.
$$
\item
Moreover, $\mM(\omega, \lambda)$ has a presentation
$$
\rA^8
\xrightarrow{
\left(
\begin{smallmatrix}
y & 0 & 0 & 0  &  v & u^n & 0 & 0 \\
0 & v & 0 & 0  &  0  & x & y^p & 0 \\
0 & 0 & x & 0  &  0 & 0 & u  & v^q \\
0 & 0 & 0 & u  &  \lambda x^m  & 0 & 0 & y
\end{smallmatrix}
\right)
}
\rA^4 \lar \mM(\omega, \lambda) \lar 0.
$$
\qed
\end{itemize}
\end{remark}

\subsection{Maximal  Cohen--Macaulay modules over $\kk\llbracket x,y,z,u,v\rrbracket/(xz, xu, yu, yv, zv)$}\label{SS:degcusp4} It seems that  even less is known about a  concrete description of maximal Cohen--Macaulay modules over Gorenstein surface singularities which are not complete intersections.
In this subsection we elaborate the classification of rank one objects
of $\CM^{\mathsf{lf}}(\rA)$ for $\rA = \kk\llbracket x,y,z,u,v\rrbracket/(xz, xu, yu, yv, zv)$.
Again, all necessary computations are completely parallel to the ones done in Subsection
\ref{SS:degcusp2}. Hence, we omit them and state the final result.
\begin{proposition}
Denote  $J = \bigl\langle
 (xy)^2, (yz)^2, (zu)^2, (uv)^2, (vx)^2 \bigr\rangle_\rA$. Let
$\mM$ be a rank one object in  $\CM^{\mathsf{lf}}(\rA)$.  Then
$
\mM \cong \mM(\omega, \lambda) := $
$$
\Bigl\langle J,
 v^{m_1+1} x + \lambda v^{m_2+1} u,
 x^{n_1+1} y  + x^{n_2+1} v,
y^{p_1+1} z  +  y^{p_2+1} x,
z^{q_1+1} u  +  z^{q_2+1} y,
u^{t_1+1} v +  u^{t_2+1} z
\Bigr\rangle_\rA^{\vee\vee}
$$
for some $\omega = \bigl((m_1, m_2), (n_1, n_2), (p_1, p_2), (q_1, q_2), (t_1, t_2)\bigr) \in
\mathbb{Z}^{10}$ and
$\lambda \in \kk^*$, where
$$\min(m_1, m_2) = \min(n_1, n_2) = \min(p_1, p_2) = \min(q_1, q_2) = \min(t_1, t_2) = 1.
$$ Moreover,
$\mM(\omega, \lambda) \cong \mM(\omega', \lambda')$ if and only if $\omega = \omega'$ and
$\lambda = \lambda'$.
\end{proposition}

\begin{remark}
Let $m_1 = n_1 = p_1 = q_1 = t_1 = 1$, whereas
$m_2 = m, n_2 = n, p_2 = p $, $q_2 = q$ and  $t_2 = t$.
If $m = n = p = q = t = 1$ and $\lambda = 1$ then $\mM(\omega, \lambda) \cong \rA$.
Otherwise, we have:
$$ \mM(\omega, \lambda) \cong
\bigl\langle
v^{2} x + \lambda v^{m+1} u,
x^{2} y  + x^{n+1} z,
y^{2} z  +  y^{p+1} x,
z^{2} u  +  z^{q+1} y,
u^{2} v  +  u^{t+1} z
\bigr\rangle_\rA \subseteq \rA.
$$
\qed
\end{remark}

\section{Singularities obtained by gluing  cyclic quotient singularities}\label{S:degenarate-cusps}

In this section we recall the definition and some basic properties of an important class
of non--isolated surface singularities called ``degenerate cusps'', see \cite{ShepherdBarron,Stevens2}.

\subsection{Non--isolated surface singularities obtained by gluing  normal rings}

Let $\kk$ be an algebraically closed field and
$(\rR_1, \idn_1),  (\rR_2, \idn_2),  \dots,$ $ (\rR_t, \idn_t)$ be complete local
 normal $\kk$--algebras of Krull dimension
two, where $t\ge 1$. For any $1 \le i \le t$ let
$\pi_i: \rR_i \lar \bar\rR_i \cong  \kk\llbracket \bar{u}_i, \bar{v}_i\rrbracket/(\bar{u}_i \bar{v}_i)$
be a surjective ring homomorphism, $u_i, v_i \in \rR_i$ be some preimages
of $\bar{u}_i$ and $\bar{v}_i$ respectively  and $I_i = \ker(\pi_i)$. Consider the ring
homomorphism
$$\tilde\pi_i: = (\tilde\pi_i^{(1)}, \tilde\pi_i^{(2)}):  \rR_i \lar
\rD := \kk\llbracket {u} \rrbracket \times \kk\llbracket {v}\rrbracket$$
obtained by composing $\pi_i$ with the normalization
map  $\kk\llbracket \bar{u}_i, \bar{v}_i\rrbracket/(\bar{u}_i \bar{v}_i) \to \kk\llbracket {u}\rrbracket \times \kk\llbracket {v}\rrbracket$.

\begin{definition}\label{D:gluing-of-rings}
In the above notations, consider the subring $\rA$ of the ring
$\rR:= \rR_1 \times \rR_2 \times \dots \times \rR_t$ defined as
\begin{equation}
\rA := \Bigl\{(r_1, r_2, \dots, r_t) \in \rR \; \big| \;  \tilde\pi_i^{(2)}(r_i) =
\tilde\pi_{i+1}^{(1)}(r_{i+1}), \;  1 \le i \le t\Bigr\},
\end{equation}
where we identify  $\rR_{t+1}$ with   $\rR_1$ and  $\pi_{t+1}$ with  $\pi_1$.
\end{definition}

\begin{proposition}\label{P:glue-of-rings}
In the above notations we have:
\begin{enumerate}
\item\label{i:propGLofringsIt1} The ring $\rA$ is local and reduced. The ring extension $\rA \subseteq \rR$ is finite.
\item\label{i:propGLofringsIt2} Let
$
I:= \bigl\{(r_1, r_2, \dots, r_t) \in \rR \; \big| \;  \pi_i(r_i) = 0 ,  \;  1 \le i \le t\bigr\}.
$
Then $I = \Hom_\rA(\rR, \rA)$. In other words, $I$  is the conductor ideal.
\item\label{i:propGLofringsIt3} The ring $\rA$ is Noetherian and complete and  $\rR$ is its
normalization.
 In particular, $\rA$ has Krull dimension two.
\item\label{i:propGLofringsIt4}
We have:
 $
 \bar\rA := \rA/I \cong
 \kk\llbracket \bar{w}_1, \bar{w}_2, \dots, \bar{w}_t\rrbracket/(\bar{w}_i \bar{w}_j |
  1 \le i <  j \le t)
 $
  and $\bar\rR = \rR/I = \kk\llbracket \bar{u}_1, \bar{v}_1\rrbracket/(\bar{u}_1 \bar{v}_1)
  \times \dots \times \kk\llbracket \bar{u}_t,  \bar{v}_t\rrbracket/(\bar{u}_t \bar{v}_t)$.
  The canonical morphism $\bar\rA \to \bar\rR$ maps $\bar{w}_i$ to $\bar{v}_i + \bar{u}_{i+1}$
  for all $1 \le i \le t$, where $\bar{u}_{t+1} = \bar{u}_{1}$.
 \item\label{i:propGLofringsIt5} The ring $\rA$ is Cohen-Macaulay.
\end{enumerate}
\end{proposition}

\begin{proof}
(\ref{i:propGLofringsIt1})
Let $r = (r_1, r_2, \dots, r_t) \in \rR$ be an element of $\rA$. Then
$r_1(0) = r_2(0) = \dots = r_t(0) \in \kk$ and $r$ is invertible if and only if $r_1(0) \ne 0$.
This shows that $\rA$ is local. Since $\rA$ is a subring of a reduced ring, it is reduced, too.
To show that the ring extension $\rA \subseteq \rR$ is finite, we consider separately the following
 two cases.

\vspace{1mm}
\noindent
\underline{Case 1}. Let $t \ge 2$ and $e_1, e_2, \dots, e_t$ be the idempotent elements
of $\rR$ corresponding to the units  of the rings $\rR_1, \rR_2, \dots, \rR_t$. Then we have:
$
\rR = \bigl\langle e_1, e_2, \dots, e_t\bigr\rangle_\rA.
$
Indeed, let $r = (r_1, r_2, \dots, r_t) \in \rR$ be an arbitrary element. By the definition, we have:
$r_i = e_i \cdot  r$ and $r = r_1 + r_2 + \dots + r_t$. Hence, it is sufficient to show
that for any $1 \le i \le t$ and any $r_i \in \rR_i$ we have:
$r_i \in \langle e_1, e_2, \dots, e_t\rangle_\rA$.

To prove this, it is sufficient to show the following statement: given
an element $r_i \in \rR_i$ there exist such elements $r_j \in \rR_j$, $j \ne i$ that
$r = (r_1, r_2, \dots, r_t) \in \rA$.  Note that without loss of generality we may assume
$r_i \in \idn_i$. Then $\tilde\pi_i(r_i) \in \rD$ belongs to the radical $\bar\idn$ of the ring $\rD$.
Since the ring homomorphism $\tilde\pi_i: \rR_i \to \rD$ induces a surjective map
$\idn_i \to \bar\idn$, there exist elements $r_{j} \in \idn_j, j \in \{i-1, i+1\}$ such that
$\tilde\pi_i^{(1)}(r_i) = \tilde\pi_{i-1}^{(2)}(r_{i-1})$ and
$\tilde\pi_i^{(2)}(r_i) = \tilde\pi_{i+1}^{(1)}(r_{i+1})$. Proceeding by induction, we get
an element $r = (r_1, r_2, \dots, r_t) \in \rA$ we are looking for.

\vspace{1mm}
\noindent
\underline{Case 2}. Let $t = 1$ and $u = u_1, v= v_1 \in \rR$ be some preimages of the elements
 $\bar{u} = \bar{u}_1$ and $\bar{v} =  \bar{v} \in \bar\rR$ under the map $\bar\pi = \bar{\pi}_1$.
 Then we have: $\rR = \langle u, v\rangle_\rA$.

 Indeed, any element $r \in \rR$ can be written as $r = c + up(u) + vq(v) + r'$ for some
 power series
 $p, q \in \kk\llbracket T \rrbracket$,  a scalar $c \in \kk$  and
 $r' \in I = \ker(\pi) \subseteq \rA$.
 Let $\bar{w} = \bar{u} + \bar{v} \in \bar\rR$ and $w \in \rR$ be its preimage in $\rA \subseteq \rR$.
 Then the element $r'' := r - up(w) - vq(w)$ belongs to $I$, hence to $\rA$, and
  the result follows.

\medskip
\noindent
(\ref{i:propGLofringsIt2})
Since $\pi$ is a ring homomorphism, it is clear that $I$ is an ideal in $\rR$ and $\rA$. We need
to show that  $I = J := \Ann_{\rA}(\rR/\rA) = \bigl\{a \in \rA \, | \, a\rR \subseteq \rA\bigr\}$.
The inclusion $I \subseteq J$ is obvious. Hence, we have to show that
$J \subseteq I$.

Since $J$ is an ideal in $\rR$, we have an isomorphism
of $\rR$--modules
$J = J_1 \oplus J_2 \oplus \dots \oplus J_t$, where $J_i$ is an ideal in $\rR_i$ for all
$1 \le i \le t$.  Again, we distinguish  two cases.

 \vspace{1mm}
\noindent
\underline{Case 1}. Let $t \ge 2$ and $a = (a_1, a_2, \dots, a_t) \in J$, where
$a_i = e_i \cdot a, 1 \le i \le t$. In order to show that $a \in I$ it is sufficient to prove that
$\pi_i(a_i) = 0$ for all $1 \le i \le t$.

By the definition, the element
$
a_i = e_i \cdot a = (0, \dots, 0, a_i, 0, \dots, 0)
$
belongs to $J \subseteq \rA$ for all $1 \le i \le t$.  It implies that $\tilde{\pi}_i^{(k)}(a_i) = 0$ for $k = 1, 2$, thus  $\pi_i(a_i) = 0$ for all $1 \le i \le t$ as wanted.

 \vspace{1mm}
\noindent
\underline{Case 2}. Let $t =  1$. Then any element $a \in \rA$ can be written
in the form
$a = p(w) + a'$, where $p(T) \in \kk\llbracket T\rrbracket$ is a formal power series and
$a' \in I$. If $a$ belongs to $J$ then we have:
$$
\tilde\pi^{(1)}(ua) = u p(u) = \tilde\pi^{(2)}(ua) = 0.
$$
 Hence, $p(T) = 0$ and $a = a' \in I$.

\medskip
\noindent
(\ref{i:propGLofringsIt3})
 Since $\rR$ is a Noetherian
ring and $I$ is an ideal in $\rR$, $I$ is a finitely generated $\rR$--module.
Next, $\rR$ is a finite module over $\rA$, hence there exist elements
$c_1, c_2, \dots, c_l \in I$ such that $I = \bigl\langle c_1, c_2, \dots, c_l\bigr\rangle_\rA$.
For any $1 \le i \le t$ let $w_i =  v_i + u_{i+1} \in \rA$. Then any element
$a \in \rA$ can be written in the form
$
a = \sum\limits_{i = 1}^t p_i(w_i) + c,
$
where $c \in I$ and $p_i(T) \in \kk\llbracket T\rrbracket$, $1 \le i \le t$ are some power series.
Consider the ring homomorphism
$
\rS := \kk\llbracket x_1, \dots, x_t; z_1, \dots, z_l\rrbracket
\stackrel{\varphi}\lar
\rR
$
defined by the rule: $\varphi(x_i) = w_i$ for all  $1 \le i \le t$ and
$\varphi(z_j) = c_j$ for all $1 \le j \le l$. It is clear that the image of
$\varphi$ belongs to $\rA$. In oder to show
$\rA$ is Noetherian and complete, it is sufficient to prove that the ring homomorphism
$\varphi: \rS \to \rA$ is surjective.

For any integer $q \ge 1$, let $\check\rA_q$ be the image of $\rA$ under the
canonical morphism $\rR \to \rR/\idn^q$. To show $\varphi$ is surjective it suffices to prove
that the ring homomorphism $\rS \to  \check\rA_q$ is surjective for all $q \ge 1$.
Any element $a \in \rA$ can be written as
$$
a = \sum\limits_{i=1}^t p_i(w_i) + \sum\limits_{j=1}^l b'_j c_j +
\sum\limits_{j=1}^l b''_j c_j
$$
where $p_1, \dots, p_t \in \kk\llbracket T \rrbracket$,
$b'_1, \dots, b'_l \in \kk$ and $b''_1, \dots, b''_l \in I$. Then
$$a^{(1)} := \sum\limits_{i=1}^t p_i(w_i) + \sum\limits_{j=1}^l b'_j c_j \in \mathrm{Im}(\varphi)
\quad \mathrm{and} \quad
a - a^{(1)} \in  \idn^2.$$
 Writing similar expansions
for $b''_1, \dots, b''_l \in I \subseteq \rA$, we end up with
a sequence of  elements $\bigl\{a^{(n)}\bigr\}_{n\ge 1}$ such that $a^{(n)} \in  \mathrm{Im}(\varphi)$  and $a - a^{(n)} \in \idn^{n+1}$ for
all $n \ge 1$. This shows the surjectivity of $\varphi$.

We have shown that  the ring $\rA$ is Noetherian and the ring extension
$\rA \subseteq \rR$ is finite. Hence, $\rA$ has Krull dimension two.
Moreover, the total rings of fractions of $\rA$ and $\rR$ are equal. Indeed, if
$t \ge 2$ then
$\rR = \bigl\langle e_1, \dots e_t\bigr\rangle_\rA$, where all elements $e_i \in \rR$ satisfy the equation
$e_i^2 = e_i$, hence $e_i \in \rQ(\rA)$. For  $t = 1$ we have the equality
$u^2 - uw + uv = 0$, where $w \in \rA$ and $uv \in I \subseteq \rA$. Hence, $u \in \rQ(\rA)$.
In a similar way, $v \in \rQ(\rA)$.  Thus, in both cases we have
$\rQ(\rA) = \rQ(\rR)$, hence $\rR$ is the normalization of $\rA$.

\medskip
\noindent
(\ref{i:propGLofringsIt4})
Recall that we have:
$
\bar\rR = \rR/I = \kk\llbracket \bar{u}_1, \bar{v}_1\rrbracket/(\bar{u}_1 \bar{v_1}) \times
\dots \times \kk\llbracket \bar{u}_t, \bar{v}_t\rrbracket/(\bar{u}_t \bar{v_t}).
$
By the definition, the ring
$\bar\rA \subseteq \rR$ is generated by power series in the elements
$\bar{w}_i = \bar{v}_{i} + \bar{u}_{i+1}$, where $1 \le i \le t$.

In $\bar\rR$
we have the relations $\bar{w}_i \bar{w}_j = 0$ for all $i \ne j$.
Assume there exists an  additional relation in $\rA$ between
$\bar{w}_1, \dots, \bar{w}_t$. Then it has necessarily the form:
$\bar{w}_i^n = 0$ for some $1 \le i \le t$ and $n \ge 1$. But this implies that we have
the relations $\bar{v}_i^n = 0 = \bar{u}_{i+1}^n$. Contradiction.
 Hence, we have:
 $
 \bar\rA \cong
 \kk\llbracket \bar{w}_1, \bar{w}_2, \dots, \bar{w}_t\rrbracket/(\bar{w}_i \bar{w}_j \; | \;
  1 \le i <  j \le t)
 $
and the canonical imbedding $\bar\rA \to \bar\rR$ is given by the rule
$\bar{w}_i \mapsto \bar{v}_i + \bar{u}_{i+1}$.

\medskip
\noindent
(\ref{i:propGLofringsIt5})
We have a short exact sequence of Noetherian $\rA$--modules
$
0 \rightarrow  \rA \rightarrow  \rR \rightarrow  \rR/\rA \rightarrow  0.
$
Since $\rR$ is a normal surface singularity, it is Cohen-Macaulay. In particular,
$$\depth_\rR(\rR) = 2 = \depth_\rA(\rR).$$
 Next, the $\rA$--module $\rR/\rA$ is annihilated by
$I$, hence it is an $\bar\rA$--module. Moreover, we have an isomorphism
of $\bar\rA$--modules $\rR/\rA \cong \bar\rR/\bar\rA$. Our goal is
to show that
$\depth_{\bar{\rA}}(\bar\rR/\bar\rA) = 1$. It is equivalent to the statement that
the $\bar\rA$--module $\bar{\rR}/\bar{\rA}$ has no finite length submodules.
Let $\bar\idm$ be the maximal ideal of $\bar\rA$. It suffices to show that there
is no element $r \in \bar\rR \setminus \bar\rA$ such that $\bar\idm \cdot r \in \bar\rA$.

Let $p_i, q_i \in \kk\llbracket T\rrbracket$ be power series such that
 $p_i(0) = q_i(0)$ for all  $1 \le i \le t$ and
$r = \Bigl(\bigl(p_1(\bar{u}_1), q_1(\bar{v}_1)\bigr), \dots,
\bigl(p_t(\bar{u}_t), q_t(\bar{v}_t)\bigr)\Bigr)
\in \bar\rR$ be an element satisfying
$\bar{w}_i \cdot r \in \rA$. But this implies that for all $1 \le i \le t$ we
have equalities of the power series $T q_i(T) = T p_{i+1}(T)$. Hence,
$q_i = p_{i+1}$ for all $1 \le i \le t$ and $r \in \bar\rA$ as wanted.
Applying the Depth Lemma we get: $\depth_\rA(\rA) = 2$, hence $\rA$ is Cohen-Macaulay.
\end{proof}

\noindent
\textbf{Summary}.
In the notations of Proposition \ref{P:glue-of-rings} we have:
the ring $\rA$ is Noetherian, local, reduced, complete and Cohen-Macaulay. The ring
$\rR$ is the normalization of $\rA$ and $I$ is the conductor ideal. The canonical
commutative diagram
$$
\xymatrix
{
\rA \ar[r] \ar[d] & \bar\rA \ar[d] \\
\rR \ar[r] & \bar\rR
}
$$
is a pull-back diagram in the category of Noetherian rings.  Moreover, we have:
$$\rQ(\bar\rA) \cong \kk\llbrace \bar{w}_1\rrbrace \times \dots \times  \kk\llbrace \bar{w}_t\rrbrace \;
\mbox{\textrm{and}} \; \rQ(\bar\rR) \cong  \bigl(\kk\llbrace \bar{u}_1\rrbrace  \times \kk\llbrace \bar{v}_1\rrbrace\bigr) \times
\dots  \bigl(\kk\llbrace \bar{u}_t\rrbrace \times \kk\llbrace \bar{v}_t\rrbrace\bigr)$$ and the canonical
ring homomorphism $\rQ(\bar\rA) \to \rQ(\bar\rR)$ sends the element
$\bar{w}_i$ to
$\bar{v}_i + \bar{u}_{i+1}$.

\begin{lemma}
In the notations of Proposition \ref{P:glue-of-rings} we have: for any
prime ideal $\idp \in \kP$ the localization  $\rA_\idp$ is either regular or
$\widehat{\rA}_\idp \cong \kk(\idp)\llbracket u,v\rrbracket/(uv)$, where
$\kk(\idp)$ is the residue field of $\rA_\idp$. In particular, the ring  $\rA$ is Gorenstein
in codimension
one.
\end{lemma}

\begin{proof} Let $\idp \in \kP$. By Proposition \ref{P:prop-of-conduct}, the local
ring $\rA_\idp$ is regular unless $\idp$ belongs to the associator of the conductor
ideal $I$. Let $\bar\idp$ be the image of $\idp$ in $\bar\rA$. Then $\bar\idp$ is
a prime ideal in $\bar\rA$ of height zero, $\rR_\idp$ is the normalization
of $\rA_\idp$, $I_\idp$ is the conductor ideal of $\rA_\idp$ and we have
a pull--back diagram
$$
\xymatrix
{
\rA_\idp \ar[r] \ar[d] & \bar{\rA}_{\bar\idp} \ar[d] \\
\rR_\idp \ar[r] & \bar\rR_{\bar\idp}
}
$$
in the category of Noetherian rings.
Since we know that $
 \bar\rA \cong
 \kk\llbracket \bar{w}_1, \bar{w}_2, \dots, \bar{w}_t\rrbracket/(\bar{w}_i \bar{w}_j \; | \;
  1 \le i <  j \le t),
 $
 we have:
 $\bar\idp = \langle \bar{w}_1, \dots, \bar{w}_{i-1}, \bar{w}_{i+1},
 \dots, \bar{w}_t\rangle$ for some $1 \le i \le t$. We get:
 $
 \bar\rA_{\bar\idp} = \kk\llbrace \bar{w}_i\rrbrace  \quad
 \mbox{and} \quad \bar\rR_{\bar\idp} = \kk\llbrace \bar{v}_i\rrbrace  \times \kk\llbrace u_{i+1}\rrbrace.
 $
 Hence, the conductor ideal of the local one--dimensional ring $\rA_{\idp}$
 is its maximal ideal. Next, $\rA_{\idp}$  contains its residue field $\kk\llbrace t\rrbrace$ and there exists
 a pull--back diagram
 $$
 \xymatrix
 { \rA_\idp \ar[r] \ar[d] & \kk\llbrace t\rrbrace \ar[d] \\
 \rR_{\idp} \ar[r] & \kk\llbrace t\rrbrace \times \kk\llbrace t\rrbrace.
 }
 $$
 Hence, the completion of $\rA_\idp$ is isomorphic to $\kk\llbrace t\rrbrace\llbracket u, v\rrbracket/(uv)$, as wanted.
\end{proof}

\subsection{Generalities about cyclic quotient singularities}
Let $\rS = \kk\llbracket u, v\rrbracket$ and $G = C_{n,m} \subset \GL_2(\kk)$ be a small cyclic subgroup  of order $n$. Then without  loss of generality we may assume that $G$ is
generated by the matrix
$g = \left(
\begin{array}{cc}
\xi & 0 \\
0 & \xi^m
\end{array}
\right)
$, where $\xi$ is a primitive $n$-th root of unity, and $0 \le  m <n$ is  such that
 $\gcd(m,n) = 1$. The group $G$ acts on $\rS$ by the rule $u \stackrel{g}\mapsto \xi u$,
$v \stackrel{g}\mapsto \xi^m v$.
Let  $\rR= \gR(n,m) = \rS^{G}$ be the corresponding ring of invariants and
$\Pi = \Pi(n, m) = \{0, 1, \dots, n-1\}$. For an element $l \in \Pi$ we denote
by $\bar{l}$ the unique element in $\Pi$ such that $l =  \bar{l} m \; \mod \; n$.
  The following result is due to
 Riemenschneider \cite{Riemenschneider}

\begin{theorem}\label{T:Riemenschneider}
In the above notations we have:
\begin{enumerate}
\item $\rR = \kk\llbracket u^n, u^{n-1} v^{\bar{1}}, \dots, u v^{\overline{n-1}}, v^n
\rrbracket \subset \rS = \kk\llbracket u, v \rrbracket$.
\item More precisely, let
 $n/(n-m) =
a_1 - 1/(a_2 - \dots - 1/a_e)
$
be the expansion of $n/(n-m)$ into a continuous fraction,
where $a_i \ge 2$ for all $1 \le i \le e$. Define the positive integers $c_i$ and $d_i$ by
the following recurrent formulae:
$$
c_0 = n, c_1 = n-m, c_{i+1} = a_i c_i - c_{i-1}, \quad
d_0 = 0, d_1 = 1, d_{i+1} = a_i d_i - d_{i-1}.
$$
Then
$
\rR = \kk\llbracket x_0, x_1, \dots, x_e, x_{e+1}\rrbracket/L,
$
where
 $x_i = u^{c_i} v^{d_i} (i = 0,1,\dots, e+1)$ and the
 ideal $L  \subset \kk\llbracket x_0, x_1, \dots, x_e, x_{e+1} \rrbracket$ is generated
by the relations
\begin{equation}\label{E:equat-of-quotient}
x_{i-1}x_{j+1} = x_i x_j \prod\limits_{k = i}^j x_k^{a_k -2}  \qquad  1 \le i \le j \le e.
\end{equation}
\item Let $\gJ$ be the ideal in $\gR$ generated by $(x_1, x_2, \dots, x_e)$.
Then $\gJ = (u^{n-1} v^{\bar{1}}, \dots, u v^{\overline{n-1}})_\rR$ and
$
\bar\rR:= \gR/\gJ = \kk\llbracket x_0, x_{e+1}\rrbracket/(x_0 \, x_{e+1})
\cong \gD =  \kk\llbracket x,y\rrbracket/(xy).
$
\item
The closed subset $V(\gJ) \subset \Spec(\rR)$ is the image of the union of the
coordinate axes
$V(uv) \subset \Spec\bigl(\rS)$ under the map
$\Spec\bigl(\rS) \to  \Spec\bigl(\gR\bigr)$.
\end{enumerate}
\end{theorem}

\noindent In what follows we shall also need the following result of Brieskorn
\cite[Satz 2.10]{Brieskorn} about the exceptional divisor
of a minimal resolution of singularities of $\Spec(\rR)$.

\begin{theorem}\label{T:Brieskorn}
Let $X = \Spec(\rR)$, $o = \{\idm\} \in X$ be its unique closed point  and $\widetilde{X} \stackrel{\pi}\rightarrow X$ be a minimal
resolution of singularities. Then the exceptional divisor $E = \pi^{-1}(o)$ is
a tree of projective lines. More precisely, the dual graph of $E$ is
$$
\xymatrix
{\stackrel{b_1}\bullet \ar@{-}[r] & \stackrel{b_2}\bullet \ar@{-}[r] &  \dots
 \ar@{-}[r] & \stackrel{b_f}\bullet
}
$$
where   $
n/m  = b_1 - 1/(b_2  - \dots - 1/b_f)
$
is the  expansion of $n/m$ into a continuous fraction  such that
and  $b_i \ge 2$ for all $1 \le i \le f$.
\end{theorem}

\subsection{Degenerate cusps and their basic properties}

In this subsection we recall the definition of an important class of non-isolated
surface singularities called \emph{degenerate cusps}.

\begin{definition}
Let $t \ge 1$ and $\underline{w} = \bigl((n_1, m_1), (n_2, m_2),\dots, (n_t, m_t)\bigr) \in (\mathbb{Z}^2)^t$ be a collection
of integers such that $0 \le m_i < n_i$  and $\gcd(m_i, n_i) = 1$ for all $1 \le i \le t$.
Let $\rR_i = \rR(n_i, m_i) =  \kk\llbracket u_i, v_i\rrbracket^{C_{n_i, \, m_i}}  \subseteq
 \kk\llbracket u_i, v_i \rrbracket$ be the corresponding cyclic quotient singularity
(in our convention, $\rR_i = \kk\llbracket u_i, v_i\rrbracket$ if $(n_i, m_i) = (1, 0)$),
$J_i \subseteq \rR_i$ the ideal defined in Theorem \ref{T:Riemenschneider}
and
$
\rA = \widetilde\rA(\underline{w}) \subseteq
\rR_1 \times \rR_2 \times \dots \times \rR_t =:\rR
$
be the ring obtained by the construction of Definition
\ref{D:gluing-of-rings}. Then $\rA$ is called degenerate cusp of \emph{type} $\underline{w}$.
\end{definition}

\begin{lemma}
Given a local complete $\kk$--algebra $\rA$, which is a degenerate cusp.
Then its type $\underline{w}$ is uniquely determined up
to an action of the dihedral group $D_t$ (i.e.~ up to a shift and reflection).
\end{lemma}

\noindent
\emph{Sketch of a proof}.
Let $X = \Spec(\rA)$ and $Y \stackrel{\pi}\rightarrow X$ be its \emph{improvement}, see
\cite{ShepherdBarron} and  \cite{Stevens2,vanStraten}.
Let $o = \{\idm\} \in X$ be the unique closed point of $X$ and $Z = \pi^{-1}(o)$ be
the exceptional divisor. Then $Z$ is a cycle of projective lines. Moreover, $Z$ is a union
of trees of projective lines: $Z = Z_1 \cup Z_2 \cup \dots \cup Z_t$, where each
$Z_i$ is isomorphic to the exceptional divisor $E_i$ of a minimal resolution $\widetilde{X}_i$
of
the  cyclic quotient singularity $\Spec\bigl(\rR(n_i, m_i)\bigr)$. The irreducible components
of each tree $Z_i$ have the same intersection multiplicities in $Y$ as the
intersection multiplicities of the corresponding irreducible components of
$E_i$ in $\widetilde{X}_i$.
These components $E_1, E_2, \dots, E_t$ intersect precisely at those points of $E$, where
the variety $Y$ is not smooth. Thus,  Theorem \ref{T:Brieskorn} allows to reconstruct
the parameters $(n_1, m_1), (n_2, m_2),\dots, (n_t, m_t)$ as well as  the order of gluing
of the corresponding cyclic quotient singularities.
\qed

\medskip
\noindent
The following important result is due to Shepherd--Barron \cite[Lemma 1.1]{ShepherdBarron}.
\begin{theorem}
Let $\rA$ be a degenerate cusp. Then $\rA$ is Gorenstein.
\end{theorem}

\subsection{Irreducible degenerate cusps} In this subsection we write down
equations of irreducible degenerate cusps.
Let $\rR  = \kk\llbracket x_0, x_1, \dots, x_e, x_{e+1}\rrbracket/L$ be a cyclic quotient singularity
(\ref{E:equat-of-quotient}) and $J \subset \rR$ be the ideal defined by $(x_1, \dots,
x_e)$.
Then $\bar\rR := \rR/J = \kk\llbracket x_0, x_{e+1}\rrbracket/x_0 x_{e+1}$ and we define
the ring $\rA$ via the pull-back diagram in the category of commutative rings
\begin{equation}\label{E:pull-back-rings}
\begin{array}{c}
\xymatrix
{
\rA \ar[r] \ar[d] & \bar\rA \ar[d]^{\gamma} \\
\rR \ar[r] & \bar\rR,
}
\end{array}
\end{equation}
where $\bar\rA = \kk\llbracket \bar{z}\rrbracket$ and $\gamma: \bar\rA \to \bar\rR$ maps
$\bar{z}$ to $\bar{x}_0 + \bar{x}_{e+1}$. Our goal is to write explicitly a list of generators and relations of the ring
$\rA$.

\medskip

\noindent
\underline{Case 1}.
Consider first case when the cyclic group is trivial and
 $\rR = \kk\llbracket u,v\rrbracket$. Then the ring $\rA$ is generated by the power
 series in $x = u+v$, $y = uv$ and $z = u^2 v$.
 In the quotient ring $\rQ(\rR)$ we have the equalities
 $$
 u = \frac{z}{y} \quad \mbox{and} \quad v = \frac{xy-z}{y}.
 $$
 The equality $y = uv$ implies the
 relation: $
 y^3 + z^2 - xyz = 0.
 $
We have a ring homomorphism $\pi: \kk\llbracket x,y,z\rrbracket \to \rR$ defined by the formulae
$x \mapsto u+v$, $y \mapsto uv$ and  $z \mapsto u^2 v$, whose image is the ring $\rA$.
Moreover, $y^3 + z^2 -xyz$ belongs to $\ker(\pi)$. If $\ker(\pi)$ has further generators then
$\kk\llbracket x,y,z\rrbracket/\ker(\pi)$ has Krull dimension which is not bigger than one.
Contradiction. Hence, $\rA$ is a hypersurface singularity
\begin{equation}
\rA = \kk\llbracket x,y,z\rrbracket/(y^3 + z^2 -xyz).
\end{equation}

\medskip

\noindent
\underline{Case 2}. Let $e = 1$, $n \ge 2$ and $m = n-1$. Then
$\rR = \kk\llbracket u,v,w\rrbracket/(v^n - uw)$ and $J = (v)$. As in the previous case,
one can
show that $\rA$ is generated by the power series in $x = v, y = uv$ and $z = u+ w$.
In the quotient ring $\rQ(\rR)$ we have the equalities
$$
u = \frac{y}{x} \quad \mbox{and} \quad w = \frac{xz -y}{x}.
$$
Hence, the  relation $uv = v^{n}$ reads as $x^{n+2} + y^2 = xyz$. As in the previous
case, we get:\begin{equation}
\rA = \kk\llbracket x,y,z\rrbracket/(x^{n+2} + y^2 - xyz).
\end{equation}
Since $x^{n+2} + y^2 - xyz = x^{n+2} + (y - \frac{1}{2} xz)^2 +
\frac{1}{4} x^2 z^2$, the equation
of $\rA$ can be rewritten in the form
$
\rA = \kk\llbracket x,y,z\rrbracket/\bigl(x^{2}(x^n + z^2)  + y^2\bigr).
$

\medskip

\noindent
\underline{Case 3}. Let $e = 2$ and
$n/(n-m) = p - 1/q = (pq-1)/q$,
where $p, q \ge 2$.
The ring $\rR = \kk\llbracket x_0, x_1, x_2, x_3\rrbracket/L$ is given by Riemenschneider's relations
$$
x_0 x_2 = x_1^p, \; x_1 x_3 = x_2^q \quad \mbox{and} \quad  x_0 x_3 = x_1^{p-1} x_2^{q-1}.
$$
The subring $\rA \subset \rR$ is generated by the elements $x = x_1$, $y = x_2$ and $z = x_0 + x_3$.
We have the following equality in $\rR$:
$
x_1^{p+1} + x_2^{q+1} = x_1 x_2 (x_0 + x_3).
$
Hence, we get:
\begin{equation}
\rA = \kk\llbracket x,y,z\rrbracket/(x^{p+1} + y^{q+1} - xyz), \; p,q \ge 2.
\end{equation}
Summing up, in  all considered cases we get the singularities
$$T_{pq\infty} = \kk\llbracket x,y,z\rrbracket/(x^p + y^q - xyz),$$
where $p \ge q \ge 2$ and $(p,q) \ne (2,2)$.

\medskip

\noindent
\underline{Case 4}. Let $e = 3$ and
$n/(n-m) = p - 1/(q-1/r)$,
where $p,\;q,\;r \ge 2.$
Then $\rR = \kk\llbracket u,x,y,z,v\rrbracket/L$, where the Riemenschneider's relations
are:
\begin{equation}
u y = x^p, \; xz = y^q, \;  y v = z^r, \; u z = x^{p-1} y^{q-1}, \; xv = y^{q-1} z^{r-1},\;
uv = x^{p-1} y^{q-2} z^{r-1}.
\end{equation}
The ring $\rA$ is generated by the power series in the elements $w = u+v$, $x, y$ and $z$.
Next, we have the following equality in $\rR$:
$
w y^q = w xz = (u +v) xz = x^p y^{q-1} + y^{q-1} z^r,
$
implying the equality $wy = x^p + z^r$. One can show that in this case we have:
$$
\rA = \kk\llbracket x,y,z,w\rrbracket/(xz - y^q, yw - x^p - z^r)$$
for $p,\,q, \,r \ge 2$.
In other words, $\rA$ is a $T_{prq\infty}$--singularity.

\medskip

\noindent
\underline{Case 5}. For  $e \ge 4$ the ring $\rA$ is no longer a complete intersection. For  the
sake of completeness, we write  its presentation via generators and relations as well.
Let $ 0 < m < n$ be coprime integers and
 $$
\frac{n}{\displaystyle n-m} = a_1 - 1/(a_2 - \dots - 1/a_e),
$$
the expansion of $n/(n-m)$ into a continuous fraction,
where  $a_i \ge 2$ for all  $1 \le i \le e$. Then one can show that
$
\rA = \kk\llbracket x_1, x_2, \dots, x_e, z\rrbracket/L,
$
where the ideal $L$ is generated by
\begin{equation}
\left\{
\begin{array}{ccl}
z x_2     & = &  x_1^{a_1} + x_3 \bigl(\prod\limits_{l=3}^e x_l^{a_l-2}\bigr) x_e, \\
z x_{e-1} & = &  x_1 \bigl(\prod\limits_{l=1}^{e-2} x_l^{a_l-2}\bigl)  x_{e-2},  \\
z x_i & = & x_1 \bigl(\prod\limits_{l=1}^{i-1} x_l^{a_l-2}\bigl)  x_{i-1} + x_{i+1}
\bigl(\prod\limits_{l=i+1}^{e} x_l^{a_l-2}\bigl)  x_{e}, \quad 2 < i < e-2, \\
x_{i-1}x_{j+1} & = & x_i \bigl(\prod\limits_{k = i}^j x_k^{a_k -2}\bigl)  x_j,  \quad
2 \le i \le j \le e-1.
\end{array}
\right.
\end{equation}

\subsection{Other cases of degenerate cusps which are complete intersections} In this
subsection we describe all other cases of  degenerate cusps which are
complete intersections.

\medskip

\noindent
\underline{Case 6}. Let $\rR = \kk\llbracket x_1,  x_2\rrbracket \times
\kk\llbracket y_1, y_2\rrbracket$. Then
$\rA$ is generated by the power
 series in the elements $x = (x_1, y_2)$, $y = (x_2, y_1)$ and $z = (x_1 x_2, 0)$.
  The following relation is obviously satisfied: $xyz = z^2$. Hence,   we have:
  \begin{equation}
 \rA = \kk\llbracket x,y,z\rrbracket/(xyz - z^2) \cong \kk\llbracket x,y,z\rrbracket/(x^2 y^2  + z^2).
 \end{equation}

\noindent
\underline{Case 7}. Let $\rR = \kk\llbracket x_1,  x_2\rrbracket \times
\kk\llbracket y_0, y_1, y_2\rrbracket/(y_0 y_2 - y_1^p)$, $p \ge 2$. Then
$\rA$ is generated by the power
 series in the elements $x = (x_1, y_0)$, $y = (x_2, y_2)$ and $z = (0, y_1)$. Moreover, we have:
 $
  x y z = z^{p+1}
 $
 and
 \begin{equation}
 \rA = \kk\llbracket x,y,z\rrbracket/(xyz - z^{p+1}), \; p \ge 2.
 \end{equation}
 In other words, Case 6 and Case 7 yield the class of $T_{p\infty\infty}$--singularities.

 \medskip

\noindent
\underline{Case 8}. Let $\rR = \kk\llbracket x_1, x_2\rrbracket \times \kk\llbracket y_1, y_2\rrbracket \times \kk\llbracket z_1, z_2\rrbracket$. Then the ring $\rA$ is generated
by the power series in the elements
$
x = (x_2, y_1, 0), \; y = (0, y_2, z_1) \quad \mbox{and} \quad
z = (x_1, 0, z_2).
$
They satisfy the relation $xyz = 0$ and we have:
\begin{equation}
\rA = \kk\llbracket x, y, z\rrbracket/(xyz).
\end{equation}
This is a singularity of type $T_{\infty\infty\infty}$.

 \medskip

\noindent
\underline{Case 9}. Let
$\rR = \kk\llbracket x_0, x_1, x_2\rrbracket/(x_0 x_2 - x_1^p)
 \times \kk\llbracket y_0, y_1, y_2\rrbracket/(y_0 y_2 - y_1^q), \; p, q \ge 2.$
 Then the degenerate cusp $\rA$ is generated by the power series in the
 elements
 $$
 x = (x_0, y_2), \;  y = (x_2, y_0), \; z = (x_1, 0) \quad \mbox{and}
 \quad w = (0, y_1).
 $$
 We have the following equalities in $\rR$:
 $
 xy = z^p + w^q \quad \mbox{and} \quad zw = 0.
 $
 One can show that
 \begin{equation}
 \rA = \kk\llbracket x, y, z, w\rrbracket/(z^p + w^q - xy, wz),
 \end{equation}
 i.e.~$\rA$ is a singularity of type $T_{pq\infty\infty}$.
  \medskip

\noindent
\underline{Case 10}. Let
$\rR = \rR_1 \times \rR_2$, where $\rR_1 = \kk\llbracket x_1, x_2\rrbracket$
and $\rR_2 = \kk\llbracket y_0, y_1, y_2, y_3\rrbracket/(y_0 y_2 - y_1^p, y_1 y_3 - y_2^q,
y_0 y_3 - y_1^{p-1} y_2^{q-1})$, where $p, q \ge 2$.   Then $\rA$ is generated by the power series
in the elements
$
x = (x_1, y_3), \; y = (x_2, y_0), \; z = (0, y_1), \; w = (0, y_2)
$
 and we have:
 \begin{equation}
 \rA = \kk\llbracket x, y, z, w\rrbracket/(yw - z^p, xz - w^q)
 \end{equation}
is a singularity of type $T_{p\infty q\infty}$.
  \medskip

\noindent
\underline{Case 11}. Let
$\rR = \kk \llbracket x_1, x_2\rrbracket \times  \kk \llbracket y_1, y_2\rrbracket \times
 \kk \llbracket z_0, z_1, z_2\rrbracket/(z_0 z_2 - z_1^p)$, where $p\ge 2$. Then
 the ring $\rA$ is generated by the power series in the elements
 $$
 x = (x_2, y_1, 0), y = (0, y_2, z_0), z = (x_1, 0, z_2), w = (0, 0, z_1).
 $$
 The following relations are satisfied:
 $
 xw = 0 \quad \mbox{and} \quad yz = w^p
 $
 and we get
 \begin{equation}
 \rA = \kk \llbracket x, y, z, w\rrbracket/(xw, yz - w^p).
 \end{equation}
This is a $T_{p\infty\infty\infty}$--singularity.

  \medskip

\noindent
\underline{Case 12}. Let
$\rR = \kk \llbracket x_1, x_2\rrbracket \times  \kk \llbracket y_1, y_2\rrbracket
 \times \kk \llbracket z_1, z_2\rrbracket \times  \kk \llbracket w_1, w_2\rrbracket$.
 Then $\rA$ is generated by the power series in the elements
 $$
 x = (x_2, y_1, 0, 0), \; y = (0, y_2, z_1, 0), \; z = (0, 0, z_2, w_1)
 \quad \mbox{and} \quad w = (x_1, 0, 0, w_2)
 $$
and we have:
\begin{equation}
\rA = \kk \llbracket x, y, z, w\rrbracket/(xz, yw)
\end{equation}
is a singularity of type $T_{\infty\infty\infty\infty}$.

\section{Maximal Cohen--Macaulay modules over degenerate cusps--II}
\label{sec10}
The goal of this section is to deduce   the  matrix problem describing maximal
Cohen--Macaulay modules over the  degenerate cusp $\rA = \widetilde\rA(\underline{w})$, where
$\underline{w} = \bigl((n_1, m_1), \dots, (n_t, m_t)\bigr)$ with $0 \le m_i < n_i$ and
$\gcd(n_i, m_i) = 1$ for $1 \le i \le t$.  Recall
that the normalization of $\rA$ is $\rR = \rR_1 \times \rR_2 \times \dots \times \rR_t$,
where each component $\rR_i = \rR(n_i, m_i)$, $1 \le i \le t$ is a cyclic quotient singularity
of type $(n_i, m_i)$.
  As a first step, we recall a description
of indecomposable maximal Cohen--Macaulay modules over  cyclic quotient singularities.

\subsection{Maximal Cohen--Macaulay modules on cyclic quotient surface singularities}
Let $\kk$ be an algebraically closed field of characteristic zero,
$\rS = \kk\llbracket u, v\rrbracket$ and $G = C_{n,m} \subset \GL_2(\kk)$ be a  cyclic subgroup group of order $n$
generated by the matrix
$g = \left(
\begin{array}{cc}
\xi & 0 \\
0 & \xi^m
\end{array}
\right)
$, where $\xi$ is a primitive $n$-th root of unity, and $0 \le  m <n$ is  such that
 $\gcd(m, n) = 1$.
Let  $\rR= \gR(n,m) = \rS^{G}$ be the corresponding ring of invariants,
$\Lambda = \rS*G$ the skew group ring and  $\Pi =  \{0, 1, \dots, n-1\}$.
For  $l \in \Pi$ we denote
by $\bar{l}$ the unique element in $\Pi$ such that $l =  \bar{l} m \; \mod \; n$.
 Recall that
$$\rR = \kk\llbracket u^n, u^{n-1} v^{\bar{1}}, \dots, u v^{\overline{n-1}}, v^n
\rrbracket \subset \rS = \kk\llbracket u, v \rrbracket \subset \Lambda = \kk\llbracket u, v \rrbracket
*G
$$
is the center of $\Lambda$, $\rad(\Lambda) = (u, v)*G$ and
$\Lambda/\rad(\Lambda) \cong  \kk[G]$. The following result is due to
Auslander \cite{Auslander}, see also \cite{Wunram, Yoshino}.

\begin{theorem}\label{T:Auslander}
Let $\mathsf{Pro}(\Lambda)$ be the category of finitely generated projective left
$\Lambda$--modules. Then the functor (of taking invariants)
$
\mathsf{Pro}(\Lambda) \lar \CM(\rR),
$
assigning  to a projective module $\mP$ its invariant part
$\mP^G = \bigl\{x \in \mP \,| g \cdot x = x \; \mbox{\rm for all} \;  g \in G \bigr\}$ and to a morphism
$\mP \stackrel{f}\rightarrow \mQ$ its restriction $f\bigl|_{\mP^G}\bigr.$, is an equivalence of categories.
\end{theorem}

\noindent
Let $\mU$ be an irreducible representation of $G$. Then it defines a left $\Lambda$--module
$\mP_\mU := \rS \otimes_\kk U$, where an element $p[g] \in \Lambda$ acts on a simple tensor
$r \otimes v$ by the rule $$p[g] \cdot (r \otimes v) = p g(r) \otimes g(v).$$
It is easy to see that $\mP_\mU$ is projective and
indecomposable and the top of $\mP_\mU$  is
$\mU$ viewed as a $\Lambda$--module.

Since $G$ is a finite cyclic group, all its irreducible representations are one-dimensional.
For any $l \in \Pi$ let $V_l = \kk$ be the representation of $G$ determined
by the condition $g \cdot 1 = \xi^{-l}$. From Theorem \ref{T:Auslander} we obtain the following
description of indecomposable maximal Cohen--Macaulay modules over $\rR$.

\begin{corollary}\label{C:descr-indecCM-over-cycl} There exist precisely $n$ indecomposable maximal Cohen--Macaulay
 $\rR$--modules. For any $l \in \Pi$ the corresponding maximal Cohen--Macaulay $\rR$--module
$I_l$ is
\begin{equation}\label{E:descr-indecCM-over-cycl}
I_l = (S \otimes_\kk V_l)^G \cong
\left\{\sum\limits_{i, j = 0}^\infty a_{ij} u^i v^j \Bigl| i + mj = l \; \mod \; n, a_{ij} \in \kk
\Bigr.\right\} \subset \rS.
\end{equation}
In other words, $I_l = \bigl\langle u^l, u^{l-1} v^{\bar{1}}, \dots, u v^{\overline{l-1}}, v^{\bar{l}}\bigr\rangle_\rR
\subset \rS$.
\end{corollary}

\noindent
Our next goal is to describe the morphisms between the indecomposable maximal Cohen--Macaulay
$\rR$--modules.

\begin{lemma}\label{L:comp-CM-oncycl-sing}
For any $p \in \Pi$,  let $\mP_p = \rS \otimes_\kk V_p$ be the corresponding projective left
$\Lambda$--module. Next, for any $p, q \in \Pi$ we set:
$$
S_{p, q} = \left\{\sum\limits_{i, j = 0}^\infty a_{ij} u^i v^j \Bigl| i + mj  =  q-p \;
\mod \; n, a_{ij} \in \kk\Bigr.\right\} \subset \rS.
$$
Then we have:
\begin{itemize}
\item For any pair $p, q \in \Pi$ we have an isomorphism of $\rR$--modules
$
S_{p, q} \rightarrow  \Hom_{\Lambda}(\mP_p, \mP_q)
$
assigning to a power series  $a \in S_{p,q}$ the map $r \otimes 1 \mapsto ar \otimes 1$.
\item Moreover, these isomorphisms are compatible with compositions of morphisms: for
any triple $p, q, t \in \Pi$ the diagram
$$
\xymatrix{
\rS_{q, t} \times \rS_{p, q} \ar[rr] \ar[d]_{\mathsf{mult}} & & \Hom_{\Lambda}(\mP_q, \mP_t) \times \Hom_{\Lambda}(\mP_p, \mP_q)  \ar[d]^{\circ} \\
\rS_{p, t} \ar[rr] & & \Hom_{\Lambda}(\mP_p, \mP_t)
}
$$
is commutative.
\item In particular, for any $p \in \Pi$ we have an isomorphism of rings
$$
\rR = \rS_{p, p} \lar \End_\Lambda(\mP_p).
$$
\end{itemize}
\end{lemma}

\begin{proof}
We only give a  proof of the first statement of this lemma, since
the remaining two follow from the first one. For any pair $p, q \in \Pi$ we certainly have:
$$
\Hom_{\Lambda}(\mP_p, \mP_q) \subseteq \Hom_{\rS}(\mP_p, \mP_q) =
\Hom_{\rS}(\rS \otimes_\kk V_p, \rS \otimes_\kk V_q) \cong
\Hom_{\rS}(\rS, \rS) \cong \rS.
$$
Hence, any $\Lambda$--linear morphism $\varphi$ from $\mP_p$ to $\mP_q$ is given by the multiplication
with a certain power series $a \in \rS$. In other words,  $\varphi = \varphi_a$, where
$\varphi_a(b \otimes 1) = ab \otimes 1$ for  $b \in \rS$. Since $\varphi_a$ has to be  $\Lambda$--linear,
 we have:
$$
\varphi_a\bigl(g \cdot(b \otimes 1)\bigr) = a g(b) \otimes \xi^{-p} =
g \cdot \varphi_a(b \otimes 1) = g(a) g(b) \otimes \xi^{-q},
$$
where $g$ is a generator of $G$.
It implies  that  $g(a) = \xi^{q-p} \, a$, hence  $a \in \rS_{q-p}$ as stated.
\end{proof}

\begin{corollary}\label{C:endo-proj-mod}
Let $\mP$ be a projective left $\Lambda$--module. Then $\mP \cong
\oplus_{p \in \Pi} \mP_p^{m_p}$ for some uniquely determined multiplicities
$m_p \in \mathbb{Z}_+$.
Moreover, any endomorphism $\varphi \in \End_\Lambda(\mP)$ can be written in
the  matrix form $\varphi = (\varphi_{q, p})$,
where $\varphi_{q, p} \in \Mat_{m_p \times m_q}(\rS_{q, p})$ for all $p, q \in \Pi$. Moreover,
$\varphi$ is an isomorphism if and only if the matrices $\varphi_{p, p}(0) \in \Mat_{m_p \times m_p}(\kk)$
are invertible for all  $p \in \Pi$.
\end{corollary}

\begin{lemma}\label{P:quot-sing-mp}
Let $J = \bigl\langle  u^{n-1} v^{\bar{1}}, \dots, u v^{\overline{n-1}}\bigr\rangle_\rR$ be the ideal
introduced in Theorem \ref{T:Riemenschneider} and $\bar\rR = \rR/J$. Then the following
statements are true.
\begin{itemize}
\item We have a ring isomorphism $\psi = \psi_0: D = \kk\llbracket  x,y\rrbracket/(xy) \to
\bar\rR$ given by the formula $\psi(x) = u^n$ and $\psi(y) = v^n$.
\item For any $p \in \Pi\setminus\{0\}$ there exists  an isomorphism of $D$--modules
$$
\psi_p:
\kk\llbracket x\rrbracket \oplus \kk\llbracket y\rrbracket \lar
\bar\rR \otimes_\rR I_p/\tor_{\bar\rR}(\bar\rR \otimes_\rR I_p)
$$
given by $\psi_p(1_x) = [1 \otimes u^p]$ and $\psi_p(1_y) = [1 \otimes v^{\bar{p}}]$.
\end{itemize}
\end{lemma}

\begin{proof}
The first statement  of this lemma is a part of Theorem \ref{T:Riemenschneider}.
We only give a proof of the second statement.

Since $\rR$ is a domain and $I_p$ is maximal Cohen--Macaulay over $\rR$  of rank one,
the module $\bar\rR \otimes_\rR I_p/\tor_{\bar\rR}(\bar\rR \otimes_\rR I_p)$ is
maximal Cohen-Macaulay of multi--rank $(1, 1)$ over the ring $D$. Hence, it is sufficient to show
that $\psi_p$ is well--defined and  is an epimorphism.
In order to show $\psi_p$ is well-defined, it is sufficient to check that
$
x \cdot \psi_p(1_y) = 0 = y \cdot \psi_p(1_x).
$
The first equality follows from the fact that
$$
x \cdot \psi_p(1_y) = \bigl[1 \otimes u^n v^{\bar{p}}\bigr] =
\bigl[\overline{u^{n-m} v} \otimes u^m v^{\bar{p}-1}\bigr] = 0
$$
in $\bar\rR \otimes_\rR I_p/\tor_{\bar\rR}(\bar\rR \otimes_\rR I_p)$ because
$u^{n-m} v \in J$ and $u^m v^{\bar{p}-1} \in I_p$ (note that  $\bar{p}-1 \ge 0$). The second equality $y \cdot \psi_p(1_x) = 0$
can be proven in the same way.

In oder to show $\psi_p$ is surjective, recall that $I_p = \bigl\langle u^p, u^{p-1} v^{\bar{1}}, \dots, u v^{\overline{p-1}}, v^{\bar{p}}\bigr\rangle_\rR$. Hence, it is sufficient to prove
that for any pair $1 \le i, j <n$ such that $i + mj = p \; \mod \; n$ we have:
$$1 \otimes u^i v^j \in \tor_{\bar\rR}(\bar\rR \otimes_\rR I_p).
$$
To show  this,  it is sufficient to observe that
$$
x \cdot (1 \otimes u^i v^j) = 1 \otimes u^{n+i} v^j = \overline{u^{n-m} v} \otimes u^{i+m} v^{j-1} = 0
$$
in $\bar\rR \otimes_\rR I_p$ because $u^{n-m} v \in J$ and $u^{i+m} v^{j-1} \in I_p$.
In a similar way, we have the equality $y  \cdot (1 \otimes u^i v^j) = 0$ in $\bar\rR \otimes_\rR I_p$.
\end{proof}

\begin{definition}\label{D:poset-for-mp}
On the set $\Pi = \bigl\{0, 1, \dots, n-1\bigr\}$ there are two orderings $\le_x$ and $\le_y$ defined
as follows:
\begin{itemize}
\item $p  \le_x q$ if and only if $p \le q$,  where $p$ and $q$ are regarded as natural numbers.
\item $p  \le_y q$ if and only if $\bar{p} \le \bar{q}$, where $\bar{p}$ and $\bar{q}$ are regarded as natural numbers.
\end{itemize}
\end{definition}

\begin{proposition}\label{P:descr-of-allow-transf}
Let $\mM = \oplus_{p \in \Pi} I_p^{m_p}$ be a maximal Cohen--Macaulay module over  $\rR$.
\begin{itemize}
\item If $m = \sum_{p \in \Pi} m_p$ then we have an isomorphism
\begin{equation}\label{E:choice-of-basis}
\psi_\mM: \bar{\rQ}(\mM) := \rQ(\bar\rR) \otimes_{\bar\rR} \bigl(\bar\rR \otimes_\rR M/ \tor(\bar\rR \otimes_\rR M)\bigr)
\lar \kk((x))^m \oplus \kk((y))^m,
\end{equation}
induced by the isomorphisms $\psi_p$ ($p \in \Pi$) from Lemma  \ref{P:quot-sing-mp}.
\item Let
$\varphi \in \End_\rR(\mM)$ be an automorphism of $\mM$ and $\bar\varphi$ be the  induced automorphism
of $\bar{\rQ}(\mM)$. Taking the  basis of $\bar\rQ(\mM)$ induced  by $\psi_\mM$,
the endomorphism
$\bar\varphi$ can be written as a pair  of matrices
$\bar\varphi^x  = (\bar\varphi^x_{q, p}) \in \GL_m\bigl(\kk\llbracket x\rrbracket\bigr)$ and
$\bar\varphi^y  = (\bar\varphi^x_{q, p}) \in \GL_m\bigl(\kk\llbracket y\rrbracket\bigr)$, where
\begin{itemize}
\item\label{Lab:cond1} $\bar\varphi^x_{q, p} \in \Mat_{m_q \times m_p}\bigl(\kk\llbracket x\rrbracket\bigr)$ if
$p \le_x q$ and $\bar\varphi^x_{q, p} \in \Mat_{m_q \times m_p}\bigl(x\kk\llbracket x\rrbracket\bigr)$
if $p >_x q$.
\item\label{Lab:cond2} $\bar\varphi^y_{q, p} \in \Mat_{m_q \times m_p}\bigl(\kk\llbracket y\rrbracket\bigr)$ if
$p \le_y q$ and $\bar\varphi^y_{q, p} \in \Mat_{m_q \times m_p}\bigl(y\kk\llbracket y\rrbracket
\bigr)$
if $p >_y q$.
\item\label{Lab:cond3} For any $p \in \Pi$ we have: $\bar\varphi^x_{p, p}(0) = \bar\varphi^y_{p, p}(0) \in
\GL_{m_p}(\kk)$.
\end{itemize}
\item Any pair of matrices $(\bar\varphi^x, \bar\varphi^y)  \in
\GL_m\bigl(\kk\llbracket x\rrbracket\bigr) \times \GL_m\bigl(\kk\llbracket y\rrbracket\bigr)$  having a decomposition
into blocks as above and satisfying the above conditions, is induced by an automorphism
$\varphi \in \End_\rR(\mM)$.
\end{itemize}
\end{proposition}

\begin{proof}
It  is a corollary of Theorem \ref{T:Auslander}, Lemma \ref{L:comp-CM-oncycl-sing},
Corollary \ref{C:endo-proj-mod} and
Lemma \ref{P:quot-sing-mp}.
\end{proof}

\subsection{Matrix problem for  degenerate cusps}\label{SS:mat-prob-deg-cusp}
 Let $\rA = \widetilde\rA(\underline{w})$ be the degenerate cusp of type
$\underline{w} =
\bigl((n_1, m_1), (n_2, m_2),
\dots, (n_t, m_t)\bigr)$.
By Theorem \ref{T:BurbanDrozd}, we have an equivalence of categories
$
\CM(\rA) \stackrel{\FF}\lar \Tri(\rA).
$
Hence,
the classification of indecomposable
maximal Cohen--Macaulay modules over $\rA$ reduces to a description  of indecomposable
objects of the
category of triples $\Tri(\rA)$. The latter problem turns out to be more accessible, because
it can be reformulated as a certain matrix problem.  To see this, recall that
\begin{itemize}
\item The normalization $\rR$ of  the ring $\rA$  splits into the product,
$$
\rR = \rR_1 \times \rR_2 \times \dots \times \rR_t,
$$
where $\rR_i = \rR(n_i, m_i)$ is the cyclic quotient singularity of type $(n_i, m_i)$.
\item If $I = \Ann_\rA(\rR/\rA)$ is the conductor ideal and
$\bar\rA = \rA/I$ and $\bar\rR = \rR/I$ then
$$
\bar\rA \cong \kk\llbracket z_1, z_2, \dots, z_t\rrbracket/L,
$$
where the ideal $L$ is generated by the monomials $z_i z_j$, $1 \le i \ne j \le t$ and
$$
\bar\rR = \kk\llbracket x_1, y_1 \rrbracket/(x_1 y_1) \times \kk\llbracket x_2, y_2 \rrbracket/(x_2 y_2)
\times \dots \times \kk\llbracket x_t, y_t \rrbracket/(x_t y_t).
$$
\item Under the canonical morphism $\bar\rA \to \bar\rR$,  the element $z_i$ ($1 \le i \le t$)
 is mapped to
$x_{i} + y_{i-1}$, where $y_{0} =  y_t$.
\item Let $\rK = \kk\llbrace z\rrbrace$. Then  we have:
$$
\rQ(\bar\rA) \cong \rK \times \dots \times \rK \quad \mbox{and} \quad
\rQ(\bar\rR) \cong (\rK \times \rK)  \times \dots \times (\rK \times \rK),
$$
where the product is taken $t$ times.
\end{itemize}

\noindent
Let $T = (\widetilde\mM, V, \theta)$ be an object of the category $\Tri(\rA)$. Then the following is true.

\begin{itemize}
\item The $\rR$--module $\widetilde\mM$ decomposes into a direct sum
\begin{equation}\label{E:Decomp-TildeM}
\widetilde\mM \cong  \bigoplus\limits_{i = 1}^t \bigoplus\limits_{p = 0}^{n_i - 1} I_{i, p}^{d_{i,p}}
\end{equation}
for some uniquely determined multiplicities $d_{i,p} \in \mathbb{Z}_{\ge 0}$, where
$I_{i, p}$ is the rank one maximal Cohen--Macaulay $\rR_i$--module defined by
(\ref{E:descr-indecCM-over-cycl}).

\item  The second term of the triple $T$  is a module $V$ over the semi-simple ring
$\rQ(\bar\rA) \cong  \rK \times \dots \times \rK =
\rK_1 \times \dots \times \rK_t$.
 Hence,
\begin{equation}\label{E:writingV}
V \cong  \rK_1^{l_1} \oplus  \dots \oplus \rK_t^{l_t}
\end{equation}
for some uniquely determined multiplicities $l_1, l_2, \dots, l_t \in \mathbb{Z}_{\ge 0}$.

\item Applying  the isomorphism $\psi_{\widetilde\mM}$ from (\ref{E:choice-of-basis}),
we get:
\begin{equation}\label{E:indeced-decomp}
\rQ(\bar\rR) \otimes_\rR \widetilde\mM \cong
\bigl(\rK \oplus \rK\bigr)^{d_1} \oplus
\bigl(\rK \oplus \rK\bigr)^{d_2} \oplus \dots \oplus
\bigl(\rK \oplus \rK\bigr)^{d_t},
\end{equation}
where $d_i = \sum_{p = 0}^{n_i-1} d_{i, p}$ for every  $1 \le i \le t$.

\item  The morphism $\theta: \rQ(\bar\rR) \otimes_{\rQ(\bar\rA)} V \to
\rQ(\bar\rR) \otimes_\rR \widetilde \mM$ is given by a collection of matrices
$\bigl((\theta_1^x, \theta_1^y), (\theta_2^x, \theta_2^y), \dots, (\theta_t^x, \theta_t^y)\bigr)$,
where $\theta_i^x \in
\Mat_{d_i \times l_i}(\rK)$
and $\theta_i^y \in \Mat_{d_i \times l_{i+1}}(\rK)$.
For any $1 \le i \le t$ these matrices satisfy the following conditions:
\begin{equation}\label{E:gluing-constraints}
\left\{
\begin{tabular}{cc}
  $\theta_i^x, \theta_i^y$ & have  both full row rank. \\
  $\left(\begin{smallmatrix}  \theta_{i-1}^y \\ \hline  \theta_{i}^x \end{smallmatrix}\right)$ &
has  full column rank.
\end{tabular}
\right.
\end{equation}
\end{itemize}

\begin{definition}\label{D:bunch-deg-cusp} Consider the decorated bunch of chains
$\mathfrak{X}_\rA = \mathfrak{X}(\underline{w})$ defined as follows.
\begin{itemize}
\item The index set $I = \{1, \dots, t\} \times \{x, y\}$. We identify two integers
$p, q$ modulo $t+1$, when talking about elements of $I$.
\item For any $(i, u) \in I$ we set:
 $\mathfrak{F}_{(i, u)} = \bigl\{f_{(i, u)}\bigr\}$ and $\mathfrak{E}_{(i, u)} = \bigl\{e_{(i, u)}^{(0)}, \dots, e_{(i, u)}^{(n_i-1)}\bigr\}$.
\item All elements of $\mathfrak{E}_{(i, u)}$ are decorated, whereas the (unique) element of  $\mathfrak{F}_{(i, u)}$ is  not decorated. Moreover, we have the following ordering on the elements of $\mathfrak{E}_{(i, u)}$:
\begin{equation*}
e_{(i, x)}^{(0)} \dec e_{(i, x)}^{(1)} \dec \dots \dec e_{(i, x)}^{(n_i-1)}
\quad \mbox{and} \quad
e_{(i, y)}^{(0)} \dec e_{(i, y)}^{(\bar{1})} \dec \dots \dec e_{(i, y)}^{(\overline{n_i-1})}.
\end{equation*}
Here, for any $1 \le p \le n_i$ we denote by $\bar{p}$ the unique element of
$\bigl\{1, \dots, n_i-1\bigr\}$ such that $p =  \bar{p} m_i \, \mod \,  n_i$.
\item Finally, for any $1 \le i \le t$ and $0 \le p \le n_i - 1$ we have the following equivalence relations:
$$
e_{(i, x)}^{(p)} \sim e_{(i, y)}^{(p)} \quad \mbox{and} \quad
f_{(i, y)} \sim f_{((i+1), x)}.
$$
\end{itemize}
\end{definition}

Let $\mT = (\widetilde\mM, \mV, \theta)$ be an object of $\Tri(\rA)$, where
$\widetilde\mM$ is given
by (\ref{E:Decomp-TildeM}) and  $\mV$  by (\ref{E:writingV}). The isomorphism
(\ref{E:indeced-decomp}) allows to express the gluing morphism $\theta$ via
a collection of matrices $\bigl((\theta_1^x, \theta_1^y), \dots, (\theta_t^x, \theta_t^y)\bigr)$.
Note, that the direct sum decomposition (\ref{E:Decomp-TildeM}) induces a division of these
matrices into horizontal blocks, endowing their rows with certain  ``weights'', indicating their
 origin from a direct summand of $\widetilde\mM$. Concretely,
for any $1 \le i \le t$ and $0 \le p \le n_i-1$,  we have $d_{i,p}$ rows of weight
$e_{(i, x)}^{(p)}$ in  the matrix $\theta_i^x$ (respectively $d_{i,p}$ rows of weight
$e_{(i, y)}^{(p)}$ in   the matrix $\theta_i^y$). The weight of any columns of $\theta_i^x$
(respectively $\theta_i^y$) is $f_{i-1,x}$ (respectively $f_{i,y}$).

\begin{proposition}\label{P:passage-to-MP} The assignment $\mT \mapsto \bigl((\theta_1^x, \theta_1^y), \dots, (\theta_t^x, \theta_t^y)\bigr)$ extends to the  functor
\begin{equation}\label{E:FunctorH}
\Tri(\rA) \stackrel{\HH}\lar \Rep(\mathfrak{X}_\rA),
\end{equation}
satisfying the following two properties:
\begin{itemize}
\item $\mT \in \Ob\bigl(\Tri(\rA)\bigr)$ is indecomposable if and only if $\HH(\mT)$ is indecomposable;
\item
$\mT', \mT'' \in \Ob\bigl(\Tri(\rA)\bigr)$
 are isomorphic if and only if $\HH(\mT')$ and
$\HH(\mT'')$ are isomorphic.
\end{itemize}
\end{proposition}

\begin{proof} It is a consequence of Proposition \ref{P:descr-of-allow-transf}
and Definition \ref{D:bunch-deg-cusp}.
\end{proof}

\noindent
Hence, we obtain the following result, which is one the main achievements of this work.
\begin{theorem}\label{T:cusp-are-tame}
Let $\rA$ be a degenerate cusp. Then  the following is true.
\begin{itemize}
\item
The category $\CM(\rA)$ is representation--tame.
\item
Moreover, the essential image of the category  $\CM^{\mathsf{lf}}(\rA)$ under
$\HH \circ \FF$ is the category $\Rep^{\mathsf{bd}}(\mathfrak{X}_\rA)$, which
is the additive closure of the category of band objects of $\Rep(\mathfrak{X}_\rA)$.
\end{itemize}
\end{theorem}
\begin{proof} According to  Theorem \ref{T:BurbanDrozd},
we have an equivalence of categories  $\CM(\rA) \stackrel{\FF}\lar \Tri(\rA)$. By Proposition
\ref{P:passage-to-MP}, the  functor  $\Tri(\rA) \stackrel{\HH}\lar \Rep(\mathfrak{X}_A)$ preserves  indecomposability and isomorphism
classes of objects. Hence, tameness of $ \Rep(\mathfrak{X}_A)$ implies tameness of
$\CM(\rA)$.

Next, let $\mM$ be an indecomposable object of $\CM(\rA)$ and $\FF(\mM):= \mT = (\widetilde\mM, \mV,\theta)$ the corresponding triple. By Theorem \ref{T:BurbanDrozd-app-on-lfree},
 $\mM$ is locally free on the punctured spectrum if and only
if  $\theta$ is an isomorphism. The classification of indecomposable objects
of $\Rep(\mathfrak{X}_\rA)$ implies that $\theta$ is an isomorphism if and only if
$\HH(\mT)$ is a band object.
\end{proof}

\begin{remark}
Rephrasing Theorem \ref{T:cusp-are-tame} in different terms, we get the following result.
\begin{itemize}
\item The indecomposable maximal Cohen--Macaulay $\rA$--modules, which are locally free on the punctured
spectrum,  correspond to the band objects of $\Rep(\mathfrak{X}_\rA)$.
\item Those  indecomposable maximal Cohen--Macaulay $\rA$--modules, which are not
locally free on the punctured
spectrum,  correspond to the  string objects of $\Rep(\mathfrak{X}_\rA)$,
satisfying the additional  constraint (\ref{E:gluing-constraints}).
\end{itemize}
\end{remark}

\section{Schreyer's question}
\label{sec11}

According to Buchweitz, Greuel and Schreyer  \cite{BGS}, the
hypersurface singularities
$A_\infty = \kk\llbracket x, y, z\rrbracket/(xy)$ and
$D_\infty = \kk\llbracket x, y, z\rrbracket/(x^2 y - z^2$)
have only countably many indecomposable maximal Cohen--Macaulay modules.
We gave a different proof of this result
in \cite[Chapter 5]{SurvOnCM}, removing the assumption
$\mathsf{char}(\kk) \ne 2$ required in \cite{BGS}.
 In 1987, Schreyer posed the following question
\cite[Section 7.2.2]{Schreyer}.

\vspace{1mm}
\noindent
\textbf{Question}. Let $\kk$ be an uncountable algebraically closed field of characteristic zero
and  $\rA$ be a Cohen--Macaulay surface singularity having only countably many
indecomposable maximal Cohen--Macaulay modules.
Is it true that $\rA \cong \rB^G$, where $\rB$ is a singularity of type $A_\infty$ or $D_\infty$ and
$G$ is a finite group of automorphisms of $\rB$?

In this section, we show that the answer on Schreyer's question is negative. In fact, there
exists a wide class of Cohen--Macaulay surface singularities of discrete Cohen--Macaulay
representation type. Note that $\Spec(\rA)$ has at most two irreducible components
for a ring $\rA$ of type $\rB^G$ as above. As we shall see below, this need not be the case for an arbitrary Cohen--Macaulay surface singularity  of discrete Cohen--Macaulay
representation type.

\begin{definition}\label{D:discrete-type} Let $t \ge 2$ and
$\underline{w} = \bigl((n_1, m_1), \dots, (n_t, m_t)\bigr) \in \mathbbm{Z}^{2t}$ be any sequence such
 that  $0 \le m_i < n_i$ and
$\gcd(n_i, m_i) = 1$ for $1 \le i \le t$. Set
$
\rR = \rR_1 \times \dots \times \rR_t$, where $
\rR_i = \rR(n_i, m_i) =  \kk\llbracket u_i, v_i\rrbracket^{C_{n_i, \, m_i}}  \subseteq
 \kk\llbracket u_i, v_i \rrbracket
$ is the corresponding cyclic quotient singularity. As usual, for $(n_i, m_i) = (1, 0)$ we set
$\rR_i = \kk\llbracket u_i, v_i \rrbracket.
$
Let
$J_i \subseteq \rR_i$ the ideal defined in Theorem \ref{T:Riemenschneider},
$\bar\rR_i = \rR_i/J_i \cong \kk\llbracket x_i, y_i\rrbracket/(x_i y_i)$
and $$
\rC   := \kk\llbracket z_0, z_1,  \dots, z_t\rrbracket/(z_i z_j \; | \;  0 \le i < j \le t).$$
Let
$
\rA := \rA(\underline{w}) \subseteq  \rR
$
 be the ring defined through  the following pull--back diagram in the category of commutative rings:
$$
\xymatrix{
\rA \ar[rr] \ar[d] & & \rC  \ar[d]^\gamma \\
\rR \ar[rr]^-\pi & & \bar\rR_1 \times \dots \times \bar\rR_t,
}
$$
where $\pi$ is the canonical projection
and $\gamma$ is given by the following rule:
$\gamma(z_0) = y_1$, $\gamma(z_i) = x_i + y_{i+1}$ for $1 \le i \le t-1$ and
$\gamma(z_{t}) = x_t$. \qed
\end{definition}

\begin{proposition}\label{P:series-of-Atypesingul} The following results are true.

\begin{itemize}
\item  The ring $\rA$ is a complete reduced Cohen-Macaulay surface singularity and
$\rR$ is the normalization of $\rA$.

\item  Let $I = \Ann_\rA(\rR/\rA)$ be the conductor ideal,
$\bar\rA = \rA/I$ and $\bar\rR = \rR/I$.
Then we have the following isomorphisms:
$$
\bar\rA   \cong
\kk\llbracket z_1,  \dots, z_{t-1}\rrbracket/(z_i z_j \; | \;  1  \le i < j \le t-1)
$$
and
$
\bar\rR \cong \kk\llbracket x_1 \rrbracket \times
\kk\llbracket x_2, y_2 \rrbracket/(x_2 y_2)  \times \dots \times
\kk\llbracket x_{t-1}, y_{t-1} \rrbracket/(x_{t-1} y_{t-1}) \times \kk\llbracket y_t \rrbracket.
$
The canonical morphism  $\bar\rA \rightarrow \bar\rR$ sends
 $z_i$ to $x_i + y_{i+1}$ for all
$1 \le i \le t-1$.
\end{itemize}
\end{proposition}

\noindent
The proof of this proposition is the same as of Proposition \ref{P:glue-of-rings}
and is therefore omitted.
\qed

\begin{remark}
By Proposition \ref{P:series-of-Atypesingul}, the
 ring $\rA(\underline{w})$
has the same total ring of fractions as
$\rR = \rR(n_1, m_1) \times \dots \times \rR(n_t, m_t)$. Hence, $\Spec(\rA)$
has $t$ irreducible
components. Note that $\rA(1,0) \cong \kk\llbracket u, v, w\rrbracket/(uv)$ is
the $A_\infty$--singularity.
\end{remark}

\begin{definition}\label{D:bunch-schreyer} Consider the decorated bunch of chains
$\check{\mathfrak{X}} = \check{\mathfrak{X}}(\underline{w})$ defined as follows.
\begin{itemize}
\item The index set $I = \{(1, y)\}\cup \bigl(\{2, \dots, t-1\} \times \{x, y\}\bigr) \cup \{(t, x)\}$.
\item For any $(i, u) \in I$ we put:
 $\mathfrak{F}_{(i, u)} = \bigl\{f_{(i, u)}\bigr\}$ and $\mathfrak{E}_{(i, u)} = \bigl\{e_{(i, u)}^{(0)}, \dots, e_{(i, u)}^{(n_i-1)}\bigr\}$.
\item All elements of $\mathfrak{E}_{(i, u)}$ are decorated, whereas the (unique) element of  $\mathfrak{F}_{(i, u)}$ is  not decorated. Moreover, we have the following ordering on the elements of $\mathfrak{E}_{(i, u)}$:
\begin{equation*}
e_{(i, x)}^{(0)} \dec e_{(i, x)}^{(1)} \dec \dots \dec e_{(i, x)}^{(n_i-1)}
\quad \mbox{and} \quad
e_{(i, y)}^{(0)} \dec e_{(i, y)}^{(\bar{1})} \dec \dots \dec e_{(i, y)}^{(\overline{n_i-1})}.
\end{equation*}
Here, for any $1 \le p \le n_i$ we denote by $\bar{p}$ the unique element of
$\bigl\{1, \dots, n_i-1\bigr\}$ such that $p =  \bar{p} m_i \, \mod \,  n_i$.
\item
We have the following equivalence relations.
\begin{itemize}
\item $e_{(i, x)}^{(p)} \sim e_{(i, y)}^{(p)}$ for any $2 \le i \le t-1$ and $0 \le p \le n_i - 1$.
\item $f_{(i, y)} \sim f_{((i+1), x)}$ for any $1 \le i \le t-1$.
\end{itemize}
\end{itemize}
\end{definition}

\noindent
Making the same choices as those  preceding Proposition \ref{P:passage-to-MP}, we get
the following result.
\begin{proposition} We have a functor $
\Tri(\rA) \stackrel{\HH}\lar \Rep(\check{\mathfrak{X}}),
$
satisfying:
\begin{itemize}
\item $\mT \in \Ob\bigl(\Tri(\rA)\bigr)$ is indecomposable if and only if $\HH(\mT)$ is indecomposable;
\item
$\mT', \mT'' \in \Ob\bigl(\Tri(\rA)\bigr)$
 are isomorphic if and only if $\HH(\mT')$ and
$\HH(\mT'')$ are isomorphic.
\end{itemize}
\end{proposition}

\begin{theorem}\label{T:Schreyer} The ring $\rA = \rA(\underline{w})$ has discrete Cohen--Macaulay representation type.
 Moreover, any indecomposable maximal Cohen--Macaulay $\rA$--module
has  multi--rank of  type $(0, \dots, 0, 1, \dots, 1, 0, \dots, 0)$.
\end{theorem}

\begin{proof}
As we have shown before, the category $\Rep\bigl(\check{\mathfrak{X}}(\underline{w})\bigr)$
has discrete representation type (there are no bands  in this case). Hence, $\CM(\rA)$
has discrete representation type, too.  Moreover,  indecomposable objects of
 $\Tri(\rA)$ are described by the following data  $\bigl((p, q),  \kappa, \omega)\bigr)$, where
\begin{itemize}
\item $(p, q) \in \mathbbm{Z}^2$ are  such that  $1 \le p \le q \le t$.
\item $\kappa = (k_p, \dots, k_q) \in \mathbbm{Z}^{q-p+1}$ with $0 \le k_l \le n_l-1$ for $p \le l \le q$.
\item $\omega = (b_p, \dots, b_q\bigr)$. For $\max(2, p) \le l \le \min(q, t-1)$ we have: $b_l = (a_l, c_l) \in \mathbbm{Z}^2$, while  $b_1 = c_1 \in \mathbbm{Z}$ (if $p = 1$) and $b_t = a_t \in
\mathbbm{Z}$ (if $q = t$). Moreover, we impose that $\min(a_{l+1}, c_l) = 1$ for all $p \le l \le q-1$.
\end{itemize}
If  $T_{p, q}(\kappa, \omega) = (\widetilde\mM, V, \theta) \in
\mathrm{Ob}\bigl(\Tri(\rA)\bigr)$ is the triple corresponding to
$\bigl((p, q),  \kappa, \omega)\bigr)$, then we have:
$
\widetilde\mM = I_{p, k_p}   \oplus
\dots \oplus I_{q, k_q}.
$
Let $\mM= \mM_{p, q}(\kappa, \omega)$ be the maximal Cohen--Macaulay $\rA$--module corresponding to
$T_{p, q}(\kappa, \omega)$.
Then we have:
$
\rQ  \otimes_\rA \mM \cong \rQ_p \oplus \dots \oplus \rQ_q,
$
where $\rQ = \rQ(\rA) = \rQ(\rR) $ is the total ring of fractions of $\rA$ and
$\rQ_l = \rQ(\rR_l)$ is the quotient field of $\rR_l$, $1 \le l \le t$.
This proves the statement about the multi--rank of an indecomposable maximal Cohen--Macaulay
$\rA$--module.
\end{proof}

\begin{example} Since all  proofs in this section are written in a sketchy way,
we describe in details the following example. For $t \ge 1$ and $1 \le i \le t + 1$,
let $\rR_i = \kk\llbracket x_i, y_i\rrbracket$ and
$$
\rR:= \rR_1 \times \dots \times \rR_{t+1} \supset
\rA := \Bigl\{(r_1, r_2, \dots, r_{t+1}) \, \big| \, r_{i}(0, z) = r_{i+1}(z, 0) \; \mbox{for} \;  1 \le i \le t \Bigr\}.
$$
Consider the following elements of $\rR$:   $u = (x_1, 0, 0, \dots, 0)$, $z_1 = (y_1, x_2, 0,\dots, 0), \dots,$ $
z_t = (0, \dots, 0, y_t, x_{t+1})$ and $v = (0, \dots, 0, 0, y_{t+1})$. Then $\rA$ is the ring of formal power series in $u, v, z_1, \dots, z_t$.
Note that $\rA \cong  \rA(\underline{w})$ for $\underline{w} = \bigl(\underbrace{(1,0), \dots, (1, 0)}_{\scriptsize t+1\, \mbox{times}} \bigr)$.
As in Proposition \ref{P:glue-of-rings} one can show that $\rR$ is the normalization of $\rA$.
As usual, let  $I = \mathsf{ann}_\rA(\rR/\rA)$ be the conductor ideal,
 $\bar\rA = \rA/I$ and $\bar\rR = \rR/I$.
Then we have:
$$
I = \bigl\langle u, z_1 z_2, \dots, z_{t-1} z_t, v\bigr\rangle_{\rA},
\quad \bar\rA = \kk\llbracket z_1, \dots, z_t\rrbracket/(z_p z_q \, | \, 1 \le p < q \le t)
$$
and $
\bar\rR = \kk\llbracket y_1\rrbracket \times \kk\llbracket x_2, y_2\rrbracket/(x_2 y_2) \times
\dots
 \times \kk\llbracket x_t, y_t\rrbracket/(x_t y_t) \times \kk\llbracket x_{t+1}\rrbracket.
$
Under the canonical embedding $\bar\rA \rightarrow \rR$, $z_i$ is mapped to $y_i + x_{i+1}$ for
all $1 \le i \le t$. The matrix problem, corresponding to the description of isomorphy classes
of object in $\Rep(\check{\mathfrak{X}}_\rA)$, is the following.
\begin{itemize}
\item
We have $2t$  matrices
$\Bigl(\bigl(\theta_1^{(x)}, \theta_1^{(y)}\bigr), \dots,
\bigl(\theta_t^{(x)}, \theta_t^{(y)}\bigr)\Bigr)$, where
$\theta_i^{(x)} \in \Mat_{m_{i+1} \times n_i}(\rK)$ and  $\theta_i^{(y)} \in \Mat_{m_{i} \times n_i}(\rK)$,
for some   $m_1, \dots, m_{t+1}, n_1, \dots, n_t \in \mathbb{N}$.
\item  These matrices can be transformed using the  rule:
$$
\bigl(\theta_i^{(x)}, \theta_i^{(y)}\bigr) \mapsto
\bigl(S_i^{(x)} \theta_i^{(x)} T_i^{-1}, S_i^{(y)}\theta_1^{(y)} T_i^{-1}\bigr),
$$
where $T_i \in \GL_{n_i}(\rK)$ for $1 \le i \le t$;  whereas $S_i^{(x)} \in \GL_{m_{i+1}}(\bD)$,
$S_i^{(y)} \in \GL_{m_{i}}(\bD)$ are such that  $S_{i+1}^{(x)}(0) = S_i^{(y)}(0)$ for
$1 \le i \le t-1$.
\end{itemize}
Observe that the obtained matrix problem is very close  to
 the problem of classification of indecomposable representations
 of the quiver $$ \stackrel{\rightarrow}{Q} \; := \;
\circ \longleftarrow \bullet \lar \circ  \longleftarrow \bullet \lar \circ\dots \circ \longleftarrow \bullet \lar \circ
$$
of type $A_{2t+1}$
 over the field $\rK$. In fact, there is an obvious map
$$\Ob\bigl(\Rep(\check{\mathfrak{X}})\bigr) \longrightarrow  \Ob\bigl(\Rep(\stackrel{\rightarrow}{Q})\bigr),$$ which is however
 not functorial (for category $\Rep(\check{\mathfrak{X}})$, morphisms  at sources $\bullet$ are as for quiver representations, whereas
 for targets $\circ$ they are given by ``decorated rules''). Nevertheless, in these terms it is convenient to state the final result.
Using just elementary linear algebra one can show that the indecomposable
objects of $\Ob\bigl(\Rep(\check{\mathfrak{X}})\bigr)$ can be written as follows:
$$
\dots \, \lar 0  \longleftarrow \rK \stackrel{z^{e_1}}\lar \rK \stackrel{z^{e_2}}\longleftarrow
\dots \stackrel{z^{e_n}}\longleftarrow \rK \lar 0 \longleftarrow \, \dots,
$$
where the left (respectively right) zero can be both sink and source. Of course, those indecomposable objects of $\Rep(\check{\mathfrak{X}})$ which belong to the image of $\HH$, have to satisfy certain additional constraints, analogous to (\ref{E:gluing-constraints}).

Let us now describe the  indecomposable objects of $\CM^{\mathsf{lf}}(\rA)$.
They are classified by a discrete parameter
$\omega = \bigl((a_1, c_1), \dots, (a_t, c_t)\bigr) \in \mathbbm{Z}^{2t}$, where $\min(a_i, c_i) = 1$ for
$1 \le i \le t$. If $\mM(\omega)$ is the corresponding maximal Cohen--Macaulay $\rA$--module, then
for $\FF\bigl(\mM(\omega)\bigr) =: \mT \cong (\widetilde\mM, \mV, \theta)$
we have:
\begin{itemize}
\item $\widetilde\mM = \rR_1 \oplus \dots \oplus \rR_{t+1}$,
\item $\mV = \rK_1 \oplus \dots \oplus  \rK_t$,
\item $\theta_{i}^{(x)} = (z^{a_i})$ and $\theta_i^{(y)} = (z^{c_i})$ for $1 \le i \le t$.
\end{itemize}
As usual, we have: $\mM\bigl((1,1), \dots, (1, 1)\bigr) \cong \rA$. \qed
\end{example}

\section{Remarks on rings  of discrete and tame
CM--representation type}
\label{sec12}

\subsection{Non--reduced curve singularities}
First note that our results on classification of maximal Cohen--Macaulay modules over surface singularities
imply the following interesting conclusions  for non--reduced curve singularities.

\begin{theorem}\label{T:tame-curve-sing}
Let $\kk$ be an algebraically closed field of characteristic zero,
$f = x^2 y^2$ or $x^2 y^2 + x^p$, $p \ge 3$ and
$\rA = \kk\llbracket x, y\rrbracket/(f)$.
Then the curve singularity $\rA$ has tame Cohen--Macaulay representation type.
\end{theorem}

\begin{proof}
According to  Kn\"orrer \cite{Knoerrer2}, the ring $\rA$ has the same Cohen-Macaulay representation type
as the surface singularity $\rB = \kk\llbracket x, y, z\rrbracket/(f+z^2)$. Hence, it is sufficient
to observe that $\rB$ is a degenerate cusp. Indeed,
$
u^2 + uvw = (u +  \frac{1}{2} vw)^2 - \frac{1}{4} v^2w^2 = z^2 -  x^2 y^2
$
for $z = u + \frac{1}{2} vw$, $x = v$ and $y = \frac{1}{2} w$.
In a similar way,
$
u^2 + v^p + uvw = (u +  \frac{1}{2} vw)^2 + v^p - \frac{1}{4} v^2w^2 = z^2 + x^p - x^2 y^2
$
for $z = u + \frac{1}{2} vw$, $x = v$ and $y = \frac{1}{2} w$.
\end{proof}

\begin{remark}
Note that Kn\"orrer's periodicity theorem \cite{Knoerrer2} only requires that
the characteristic of the base field $\kk$ is different from two. Since the normalization
of  $\rA = \kk\llbracket u, v, w\rrbracket/(u^2 + uvw)$ is a product of two regular rings,
$\rA$  is tame in the case of an arbitrary characteristic. Hence, $\kk\llbracket x, y\rrbracket/(xy)^2$
is tame provided $\mathsf{char}(\kk) \ne 2$. We conjecture that the rings from
Theorem \ref{T:tame-curve-sing} are tame in the case of an arbitrary field $\kk$. This is consistent with
 tameness of  singularities $P_{\infty q} := \kk\llbracket x, y, z\rrbracket/(y^q - z^2, xy)$ for  $q \in \mathbb{N}_{\ge 2} \cup \{\infty\}$ proven in \cite{BG} by different methods.

Recall that by results of Kahn \cite{Kahn}, Dieterich \cite{DieterichInvent},
 Drozd and Greuel \cite{semicont},
the reduced curve singularities $T_{p, q}(\lambda) =
\kk\llbracket x, y\rrbracket/(x^p + y^q + \lambda x^2 y^2)$, where
$\frac{1}{p} + \frac{1}{q} \le \frac{1}{2}$
and $\lambda \in \kk\setminus \{\mbox{\textsl{finite \, set \, of \, values}}\}$, have tame Cohen-Macaulay representation type for an arbitrary
base field $\kk$. See also \cite{BIKR} for an alternative approach to study maximal Cohen--Macaulay modules
over some $T_{p,q}(\lambda)$--singularities via  cluster tilting theory.
\end{remark}

\subsection{Maximal Cohen--Macaulay modules over the ring $\widetilde{D}\bigl((1,0)\bigr)$}\label{SS:sing-semichain}

Let $\kk$ be any  field, $\rR = \kk\llbracket x, y\rrbracket$,
$\bar\rR = \kk\llbracket x, y\rrbracket/(xy)$ and $\pi: \rR \to \bar\rR$  the
canonical  projection. Let $\bar\rA = \kk\llbracket \tilde{x}, \tilde{y}\rrbracket/(\tilde{x}\tilde{y})$,
$\tilde\gamma: \bar\rA \to \bar\rR$ the ring homomorphism given by
the rule $\tilde\gamma(\tilde{x}) = x^2$ and $\tilde\gamma(\tilde{y}) = y^2$ and
$\widetilde{\rD}\bigl(1,0)\bigr):= \rA = \pi^{-1}\bigl(\tilde\gamma(\bar\rA)\bigr)$.
Then we have:
$$\kk\llbracket x, y\rrbracket \supset \rA = \kk\llbracket x^2, y^2, xy, x^2 y, y^2 x\rrbracket \cong \kk\llbracket u, v, w, a, b\rrbracket/J,
$$
where $J = (uv - w^2, ab - w^3, aw - bu, bw - av, a^2 - uw^2, b^2-vw^2)$. Obviously,  $\rA$ has  the following
$\mathbbm{Z}_2$--symmetry: $u \leftrightarrow w, w \leftrightarrow w, a \leftrightarrow b$.
Moreover, the following  results are true.
\begin{itemize}
\item The ring $\rA$ is an integral  Cohen--Macaulay surface singularity.
\item The ring
$\rR$ is the  normalization of $\rA$.
\item We have: $\Ext_\rA^2(\kk, \rA) \cong \kk^2$. In particular,
 $\rA$ is not Gorenstein.
\item Let
$I = \Ann_\rA(\rR/\rA)$ be the conductor ideal. Then we have:
$
I = \langle w, a, b\rangle_\rA = \langle xy \rangle_\rR.
$
In particular, we can identify $\rA/I$ with $\bar\rA$, $\rR/I$ with $\bar\rR$ and the canonical
ring homomorphism $\rA/I \to \rR/I$ with $\gamma$.
\end{itemize}
Let $\rK = \kk\llbrace z\rrbrace$, $\rL = \kk\llbrace z^2\rrbrace$, $\bD = \kk\llbracket z\rrbracket$ and
$\idm  = z \kk\llbracket z\rrbracket$. Then
$\rQ(\bar\rA) \cong  \rL_1 \times \rL_2$ and $\rQ(\bar\rR) \cong \rK_1 \times \rK_2$,
where $\rL_i = \rL$ and $\rK_i = \rK$ for $i = 1,2$.
An  object $\mT$  of the category of triples $\Tri(\rA)$ has the following form:
$\mT = (\widetilde\mM, \mV, \theta)$, where $\widetilde\mM \cong \rR^m$,
$\mV \cong \rL_1^{p_1} \oplus \rL_2^{p_2}$ and the gluing map $\theta$ is given
by  a pair of matrices $(\Theta_1, \Theta_2)$, where
$\Theta_i \in \Mat_{m \times p_i}(\rK)$ for $i = 1, 2$.
Additionally,
\begin{itemize}
\item $\Theta_1$ and $\Theta_2$  have full row rank.
\item If
$
\Theta_i = \Theta'_i + z  \Theta''_i
$
with $\Theta'_i, \Theta''_i  \in \Mat_{m \times p_i}(\rL)$ then $
\left(
\begin{array}{c}
\Theta'_i \\
\hline
\Theta''_i
\end{array}
\right) \in \Mat_{2m \times p_i}(\rL)
$
has full column rank, $i = 1, 2$.
\end{itemize}
The problem of classification
 of isomorphism classes of objects in $\Tri(\rA)$ reduces to the following matrix problem:
\begin{equation}
\bigl(\Theta_1, \Theta_2\bigr) \mapsto \bigl(S_1^{-1} \Theta_1 T_1, S_2^{-1} \Theta_2 T_2\bigr),
\end{equation}
where $S_1, S_2 \in \GL_m(\bD)$ are such that
$S_1(0) = S_2(0)$ and $T_i \in \GL_{p_i}(\rL)$ for $i = 1, 2$.

\begin{proposition}\label{P:computesemichains} Up to  isomorphism, there exist only the following
 maximal Cohen--Macaulay $\rA$--modules of rank one (written as ideals in $\rA$):
\begin{equation}\label{E:formcanoniqueRg1}
\begin{array}{|c|c||c|}
\hline
\Theta_1 & \Theta_2 & \mbox{\textsf{Module}} \\
\hline
(1) & (1) & \rA \\
(z) & (z) & (w^2, a, b) \\
(z) & (1) & (w^2, a, vw) \\
(1)& (z) & (w^2, b, uw) \\
\hline
\hline
(1 \, z) & (1 \, z) & \rR \cong I \\
(1) & (1 \, z) & (w, b) \\
(1 \, z) & (1) & (w, a) \\
(z) & (1 \, z) & (w^2, a, b, vw) \\
(1 \, z) & (z) & (w^2, a, b, uw) \\
\hline
\end{array}
\end{equation}
Moreover, the modules in the upper part of the table are locally free on the punctured spectrum,  whereas
the ones
in the lower part are not. In particular, up to isomorphism
there exist  \textsl{only finitely many} maximal  Cohen--Macaulay $\rA$--modules of rank one.
\end{proposition}

\begin{proof}
It is not difficult to see that in rank one,
there are only those  possibilities
for the canonical forms of $\Theta_1$ and $\Theta_2$, which are presented in the Table
(\ref{E:formcanoniqueRg1}). Recall, that the condition for  $\Theta_1$ and $\Theta_2$ to be  square and invertible is equivalent to the statement  that the corresponding maximal Cohen--Macaulay module is locally free on the punctured spectrum.

Let us consider in details the case $\Theta_1 = (z)$ and $\Theta_2 = (1)$.
First note that this pair is equivalent to $\bar\Theta_1 = (z)$ and $\bar\Theta_2 = (z^2)$.
Let $\tilde\theta: \bar\rA \rightarrow \bar\rR$ be the $\bar\rA$--linear map sending
$1$ to $x + y^2$. Then the induced $\rL \times \rL$--linear map
$$
\rL \times \rL = \rQ(\bar\rA) \otimes_{\bar\rA} \stackrel{\mathbbm{1} \times \tilde\theta}\lar \rQ(\bar\rA) \otimes_{\bar\rA} \bar\rR  \stackrel{\can}\lar \rQ(\bar\rR) = \rK \times \rK
$$
is given by the pair $(\bar\Theta_1, \bar\Theta_2)$. Consider the torsion free $\rA$--module $\mL$ given by the
pull--back diagram
$$
\xymatrix{
0 \ar[r] & I \ar[r] \ar[d]_{\mathsf{id}} & \mL \ar[r] \ar[d] & \bar\rA \ar[r] \ar[d]^{\tilde\theta} & 0 \\
0 \ar[r] & I \ar[r]  & \rR \ar[r]  & \bar\rR \ar[r]  & 0.
}
$$
Let $\mM = \mL^{\vee\vee}$. The reconstruction procedure tells that $\mM$ is the maximal
Cohen--Macaulay $\rA$--module corresponding to the triple
$\Bigl(\rR, \rL_1 \oplus \rL_2, \bigl((z), (1)\bigr)\Bigr)$. We have:
$$
\rR \supset \mL := \Bigl(\bigl(w, a, b\bigr), \frac{a}{w} + v \Bigr)_{\rA}
\xrightarrow{w \cdot} \bigl(w^2, wa, wb, a + vw\bigr)_{rA}.
$$
It is easy to see that $a \in \rR$ has the property that $\idm \cdot a \in \mL$. Hence,
$a$ belongs to the Macaulayfication of $\mL$. It is not difficult to see that
$$
\mM = (a)_{\rA} + \mL = \bigl(w^2, a, vw\bigr)_{\rA} \subset \rA
$$
is maximal Cohen--Macaulay. This proves the result.
All remaining cases can be treated is a similar way. Proposition is proven.
\end{proof}

\begin{remark}\label{R:SemichainsContSeries}
First one--parameter families of indecomposable maximal Cohen--Macaulay
$\rA$--modules arise in rank two. Consider for example the following pair of matrices
\begin{equation}
\Theta_1 = \left(
\begin{array}{cc}
1 & z \\
z^m & 0
\end{array}
\right)
\quad \mbox{and} \quad
\Theta_2 = \left(
\begin{array}{cc}
1 & z \\
\lambda z^n & 0
\end{array}
\right),
\end{equation}
where $m, n \in \mathbbm{Z}_{>0}$ and $\lambda \in \kk^*$. They define an indecomposable
maximal Cohen--Macaulay $\rA$--module $\mM\bigl((n, m), \lambda\bigr)$ of rank two, which is locally
free on the punctured spectrum. Moreover, $\mM\bigl((n, m), \lambda\bigr) \cong \mM\bigl((n', m'), \lambda'\bigr)$
if and only if $n = n', m= m'$ and $\lambda = \lambda'$. Assume that  $n$ and $m$ are even:
$n = 2 \bar{n}$ and $m = 2 \bar{m}$. Consider the following  ideals in $\rA$:
$I_1 = (w^2, a, b)$ and $I_2 = (w, a, b)$. As  in the proof of Proposition \ref{P:computesemichains}
one can show that
\begin{equation}
\mM\bigl((n, m), \lambda) \cong
\left\langle
\left(
\begin{array}{c}
a_1 \\
a_2
\end{array}
\right) +
\left(
\begin{array}{c}
w  \\
u^{\bar{m}} + \lambda v^{\bar{n}}
\end{array}
\right) \, \Big| \, a_p \in I_p, p = 1, 2
\right\rangle \subset \rA^2.
\end{equation}

\end{remark}
\subsection{Other surface singularities of discrete and tame CM--representation type}
According to Buchweitz, Greuel and Schreyer, the singularity $D_\infty = \kk\llbracket x, y,
z\rrbracket/(x^2 y - z^2)$ has discrete Cohen--Macaulay representation type, see also
\cite[Theorem 5.7]{SurvOnCM}. This singularity does not belong to the class of surface
singularities introduced in Definition \ref{D:discrete-type}. This certainly means that the list
of surface singularities of discrete Cohen--Macaulay representation type given in  Definition \ref{D:discrete-type} is not exhaustive.

\begin{definition}\label{D:generaliz-of-Dtype}
Let $t \ge 1$ and $\underline{w} = \bigl((n_1, m_1), (n_2, m_2),\dots, (n_t, m_t)\bigr) \in (\mathbb{Z}^2)^t$ be a collection
of integers such that $0 \le m_i < n_i$,   and $\gcd(m_i, n_i) = 1$ for all $1 \le i \le t$.
Let $\rR_i = \rR(n_i, m_i) =  \kk\llbracket u_i, v_i\rrbracket^{C_{n_i, \, m_i}}  \subseteq
 \kk\llbracket u_i, v_i \rrbracket$ be the corresponding cyclic quotient singularity,
$J_i \subseteq \rR_i$ the ideal defined in Theorem \ref{T:Riemenschneider},
$\bar\rR_i = \rR_i/J_i \cong \kk\llbracket x_i, y_i\rrbracket/(x_i y_i)$
and $\pi: \rR \to \bar\rR_1 \times \dots \bar\rR_t$ to be the canonical projection.
Consider the ring  \begin{equation*}
\rC   := \kk\llbracket u, z_1,  \dots, z_{t-1}, v\rrbracket/\bigl(z_i z_j, 1 \le i < j \le
t-1;  u z_i, v z_i;  1 \le i \le t-1  \bigr).
\end{equation*}
\begin{itemize}
\item
Let $\gamma: \rC \to \bar\rR_1 \times \dots \times \bar\rR_t$ be the ring homomorphism
given by the rule:
$\gamma(u) = x_1^2$, $\gamma(z_i) = y_i + x_{i+1}$ for $1 \le i \le t-1$ and
$\gamma(z_{t}) = y_t$.
\item In a similar way,  let $\tilde\gamma:
\rC \to \bar\rR_1 \times \dots \times \bar\rR_t$ be
given by the rule:
$\gamma(u) = x_1^2$, $\gamma(z_i) = y_i + x_{i+1}$ for $1 \le i \le t-1$ and
$\gamma(z_{t}) = y_t^2$.
\item We set
$\rD(\underline{w})  = \pi^{-1}\bigl(\gamma(\rC)\bigr)$.
\item In a similar way, we define
$\widetilde{\rD}(\underline{w}) := \pi^{-1}\bigl(\tilde{\gamma}(\rC)\bigr)$.
\end{itemize}
\end{definition}

\noindent
Note that $\rD(1,0) \cong \kk\llbracket x, y, z\rrbracket/(x^2 y - z^2)$ is a hypersurface
singularity of type $D_\infty$, whereas
 $\widetilde{\rD}(1,0)$ was considered  in Subsection \ref{SS:sing-semichain}.

\begin{remark}
Let $\rA$ be of type $\rD(\underline{w})$ (respectively $\widetilde{\rD}(\underline{w})$).
The description of indecomposable objects of the category of triples $\Tri(\rA)$ leads
to some new matrix problem, somewhat analogous to ``representations of a bunch of
\textsl{semi--chains}'' in the sense of  \cite{bo1}.
It can be shown that this matrix problem has discrete representation type for  $\rA = \rD(\underline{w})$
and tame representation type for $\rA = \widetilde{\rD}(\underline{w})$. Details will be treated elsewhere.
\end{remark}

\subsection{On deformations of  certain non--isolated surface singularities}
In this subsection we state a
 conjecture  inspired by our study of maximal Cohen--Macaulay modules over surface singularities.
It can be formulated  in pure deformation--theoretic terms.

\begin{conjecture}
Let $\kk$ be an algebraically closed field of characteristic zero and
$X \stackrel{\pi}\longrightarrow B$ be a flat morphism of Noetherian schemes over $\kk$ of relative dimension two. For a closed point $b \in B$ denote by $X_b$ the scheme--theoretic fiber $\pi^{-1}(b)$.
Let $b_0 \in B$ be a closed point and $X_0 = X_{b_0} = \Spec(\rA)$ be the corresponding fiber.
Assume that for all closed points $b \in B\setminus \{b_0\}$  the scheme $X_b$ is normal.
Let $\underline{w} = \bigl((n_1, m_1), \dots, (n_t, m_t)\bigr)$ with $0 \le m_i < n_i$ and
$\gcd(n_i, m_i) = 1$ for $1 \le i \le t$.
\begin{itemize}
\item Assume that $\rA$ has a singularity   of type $\rA(\underline{w})$
or  $\rD(\underline{w})$. Then there exists on open neighborhood
$B'$ of $b_0$ such that for all $b \in B'\setminus \{b_0\}$  the surface $X_b$ has only quotient
singularities.

\item Assume $\rA$ is a degenerate cusp (i.e.~it is of type $\widetilde{\rA}(\underline{w})$).
Then there exists on open neighborhood
$B'$ of $b_0$ such that for all $b \in B'\setminus \{b_0\}$  the scheme  $X_b$ has only simple, simple elliptic  or cusp singularities.

\item Assume $\rA$ is of type
  $\widetilde{\rD}(\underline{w})$. Then there exists an open neighborhood
$B'$ of $b_0$ such that for all $b \in B'\setminus \{b_0\}$  the scheme  $X_b$ has only quotient
or log--canonical singularities.
\end{itemize}
\end{conjecture}
The evidence for this conjecture is the following. By results of Auslander \cite{Auslander}
and Esnault \cite{Esnault}
it is known that the quotient surface singularities are the only surface singularities
of finite Cohen--Macaulay representation type.
In particular, the representation-finite Gorenstein surface singularities are precisely
the simple hypersurface singularities.

By results of Kahn \cite{Kahn} and
Drozd, Greuel and Kashuba \cite{DGK} it is known that log--canonical surface singularities
have tame Cohen--Macaulay representation type. Moreover, conjecturally
these are the only tame normal surface singularities.   The semi--continuity conjecture
(known to be true in the case of   reduced curve singularities \cite{Knoerrer1,DrozdGreuel})
states  that the representation type can be only improved by a flat local deformation:
Cohen--Macaulay finite singularities   deform to   Cohen--Macaulay finite singularities,
Cohen-Macaulay discrete singularities  deform to  Cohen--Macaulay finite or discrete
singularities  and Cohen--Macaulay tame  singularities  can not deform to Cohen--Macaulay wild
singularities. This philosophy is confirmed by a result of
Esnault and Viehweg  stating  that the class of quotient surface singularities is  closed under
deformations \cite{EsnaultViehweg}.

\section{Appendix A: Category of triples in dimension one}\label{S:AppendixTriplesDimOne}

The goal of this section is to provide full details of a construction, which allows to reduce
 a description  of Cohen--Macaulay modules over a local Cohen--Macaulay ring $(\rA, \gm)$ of Krull
dimension \emph{one} to a matrix problem. It seems that for the first time  a construction of this kind
appeared in the  work of Drozd and Roiter \cite{DroRoi} and Jacobinski \cite{Jacob}.
Similar constructions  appeared in the works of Ringel and Roggenkamp \cite{RingelRoggenkamp},
 Green and Reiner \cite{GreenReiner}, Wiegand \cite{Wiegand}, Dieterich
\cite{DieterichInvent} and recent monograph of Leuschke and Wiegand \cite{LeusckeWiegand}.

Since this
construction
plays a key role in our approach to maximal Cohen--Macaulay modules over non--isolated surface singularities
and its presentation in all above references is  essentially  different from the one given below,
we have decided
to include  its detailed exposition in this appendix.

Let $(\rA, \gm)$ be a local Cohen--Macaulay  ring of Krull dimension one (not necessarily reduced),
$\rA \subseteq \rR$ be a finite ring extension such that
$\rR \subset \rQ(\rA)$, where $\rQ(\rA)$ is the total ring of fractions of
$\rA$. Note that $\rR$ is automatically Cohen--Macaulay.
Let  $I = \Ann_\rA(\rR/\rA)$ be the corresponding conductor ideal.
Typically, $\rA$ is supposed  to be reduced and
$\rR$ is  the normalization of $\rA$. In that case, assuming
the completion $\widehat\rA$ to be   reduced, the ring extension
$\rA \subseteq \rR$ is automatically finite, see \cite[Chapitre 9, AC IX.33]{Bourbaki}.

\begin{lemma}\label{L:banal-stat}
In the notations as above, $I$ is a sub-ideal of the Jacobson's radical of $\rR$.
Moreover, the rings  $\bar{\rA} = \rA/I$ and
$\bar{\rR} = \rR/I$ have finite length.
\end{lemma}

\begin{proof}
Let $\idn_1, \dots, \idn_t$ be the set of the maximal ideals of $\rR$.
Since the ring extension $\rA \subseteq \rR$ is finite, for any
$1 \le i \le t$ we have: $\idn_i \cap \rA = \idm$. Since $I$ is a proper ideal
in $\rA$, it is contained in $\idm$. Hence,
$I$ is contained in $\idn = \idn_1 \cap \idn_2 \cap \dots \cap \idn_t$, too.

Let $\idp$ be a prime ideal in $\rA$ of height zero. Then
$I_\idp = \Ann_{\rA_\idp}(\rR_\idp/\rA_\idp) = \rA_\idp$, thus
$\bar\rA_\idp = 0$. It implies that the associator of the $\rA$--module  $\bar\rA$ is
$\{\idm\}$, hence $\bar\rA$ has finite length. Since the extension $\bar\rA \subseteq \bar\rR$ is finite,
$\bar\rR$ has finite length, too.
\end{proof}

\begin{lemma}\label{L:prep-in-dim-one}
For a maximal Cohen-Macaulay $\rA$--module $\mM$ we denote by
$$
\widetilde\mM := \rR \otimes_\rA  \mM/\Gamma_{\{\idn\}}(\rR \otimes_\rA \mM), \;
\overline\mM := \bar\rA \otimes_\rA M \quad \mbox{and} \quad \widehat\mM := \bar\rR \otimes_{\rR} \widetilde\mM.$$ Then the following results are true.
\begin{enumerate}
\item
$\widetilde\mM$ is a maximal Cohen--Macaulay module over $\rR$;
\item  the canonical morphism
of $\rA$--modules
$ \mM \stackrel{\jmath_\mM}\lar  \widetilde\mM, \; m \mapsto [1 \otimes m]
$
is injective;
\item the canonical morphism
$ \overline\mM  \stackrel{\tilde{\theta}_\mM}\lar  \widehat\mM, \; \bar{a} \otimes m \mapsto \bar{a} \otimes \jmath_\mM(m)$ is injective
and the induced morphism $\bar\rR \otimes_{\bar\rA} \overline\mM \stackrel{{\theta}_\mM}\lar \widehat\mM,  \;
\bar{r} \otimes \bar{m} \mapsto \bar{r}\cdot \tilde\theta_\mM(\bar{m})$ is
surjective.
\end{enumerate}
\end{lemma}

\begin{proof} Since $\widetilde\mM$
has no $\idn$--torsion submodules,
 it
is maximal Cohen--Macaulay over $\rR$. Next,  $\mM$ is a
torsion free $\rA$--module, hence  $\ker(\jmath_\mM)$ is also torsion free. However, the
morphism
$$
\rQ(\rA) \otimes_\rA \mM \xrightarrow{1 \otimes \jmath_\mM} \rQ(\rR) \otimes_\rR \widetilde\mM
$$
is an isomorphism, hence $\ker(\jmath_\mM) = 0$. As a result, the morphism
$I \mM \stackrel{\bar{\jmath}_\mM}\lar  I \widetilde\mM$, which is a restriction
 of $\jmath_\mM$, is also injective. Moreover,
$\bar{\jmath}_\mM$ is also surjective: for any $a \in I, b \in \rR$ and $m \in \mM$ we have:
$a \cdot [b \otimes m] = [ab \otimes m] = [1 \otimes (ab) \cdot m]$ and $ab \in I$.

\medskip
\noindent
Next, we have the following commutative diagram with exact rows:
\begin{equation}
\begin{array}{c}
\xymatrix
{ 0 \ar[r] & I \mM  \ar[r] \ar[d]_{\bar{\jmath}_\mM} & M \ar[r] \ar[d]_{\jmath_\mM} & \overline{\mM}
\ar[d]^{\tilde{\theta}_\mM} \ar[r] & 0 \\
0 \ar[r] & I \widetilde\mM  \ar[r] & \widetilde\mM \ar[r] & \widehat\mM
\ar[r] & 0. }
\end{array}
\end{equation}
Since $\jmath_\mM$ is injective and $\bar{\jmath}_\mM$ is an isomorphism, by the snake lemma
$\tilde{\theta}_\mM$ is a monomorphism.
Finally, note that $\theta_\mM$ coincides with  the composition of canonical morphisms:
$$
\bar\rR \otimes_{\bar\rA} \bar{\rA} \otimes_\rA \mM
\lar \bar\rR \otimes_{\rR} {\rR} \otimes_\rA \mM  \lar
\bar\rR \otimes_{\rR} \bigl(\rR \otimes_\rA \mM/\mathsf{tor}_\rR(\rR \otimes_\rA \mM)\bigr),
$$
where the first morphism is an isomorphism and the second one is an epimorphism.
\end{proof}

\begin{definition}\label{D:triples-on-curves}
Consider the following
\emph{category of triples} $\Tri(\rA)$. Its objects
are the triples $(\widetilde\mM, V, \theta)$, where $\widetilde\mM$
is a maximal Cohen--Macaulay $\rR$--module, $V$ is a Noetherian
$\bar\rA$--module and
$\theta:  \bar\rR \otimes_{\bar\rA} V \to
\bar\rR \otimes_\rR \widetilde\mM$ is an epimorphism
of $\bar\rR$--modules such that the induced morphism
of $\bar\rA$--modules
$\tilde\theta: V \to  \bar\rR \otimes_{\bar\rA} V
\stackrel{\theta}\lar \bar\rR \otimes_\rR \widetilde\mM $
is an monomorphism.
A morphism between two triples $(\widetilde\mM, V, \theta)$
and $(\widetilde\mM', V', \theta')$ is given by a pair $(\psi, \varphi)$, where
$\psi: \widetilde\mM \to \widetilde\mM'$ is a morphism of $\rR$--modules and
$\varphi: V \to V'$ is a morphism of $\bar\rA$--modules such that
the following diagram of $\rA$--modules
$$
\xymatrix
{
\bar\rR \otimes_{\bar\rA} V \ar[rr] \ar[d]_{\mathbbm{1} \otimes \varphi} & &
\bar\rR \otimes_\rR \widetilde\mM \ar[d]^{\mathbbm{1} \otimes \psi}\\
\bar\rR \otimes_{\bar\rA} V' \ar[rr] & &
\bar\rR \otimes_\rR \widetilde\mM'
}
$$
is commutative.
\end{definition}

\begin{remark}
Note that the morphisms $\theta$ and $\tilde{\theta}$ correspond  to each other
under the canonical isomorphisms
$
\Hom_{\bar\rR}(\bar\rR \otimes_{\bar\rA} V, \bar\rR \otimes_\rR \widetilde\mM) \cong \Hom_{\bar\rA}(V, \bar\rR \otimes_\rR \widetilde\mM).
$
\end{remark}

\noindent
Definition  \ref{D:triples-on-curves} is motivated  by the following
theorem.

\begin{theorem}\label{T:BurbanDrozdindimOne}
The functor
$
\mathbb{F}: \CM(\rA) \lar  \Tri(\rA)
$
mapping a maximal Cohen--Macaulay module $\mM$ to the triple
$\bigl(\widetilde\mM, \overline{\mM},
\theta_\mM\bigr)$,  is an equivalence of categories.
Next, let $\mN$ be  a maximal Cohen--Macaulay $\rA$--module corresponding to a triple
$(\widetilde\mN, V, \theta)$. Then $\mN$ is free if and only if
$\widetilde\mN$ and $V$  are free
and
$\theta$ is an isomorphism.
\end{theorem}

\begin{proof}
We have to construct a functor $\GG: \Tri(\rA) \lar \CM(\rA)$, which is quasi--inverse to
$\FF$.  For
a triple $\mT = (\widetilde\mM, V, \theta)$ consider the canonical morphism
$\gamma := \gamma_{\widetilde\mM}:  \widetilde\mM \lar \widehat\mM := \bar\rR \otimes_\rR \widetilde\mM$
and define $\mN := \GG(\mT)$ by taking a kernel of the following morphism
in $\rA\mathsf{-mod}$:
$$
0 \lar \mN \xrightarrow
{\left(
\begin{smallmatrix}
- \imath \\
\pi
\end{smallmatrix}
\right)
} \widetilde\mM \oplus V \xrightarrow{(\gamma \; \widetilde\theta)} \widehat\mM \lar 0.
$$
Equivalently, we have a commutative diagram in the category of $\rA$--modules
\begin{equation}\label{E:FunctorG}
\begin{array}{c}
\xymatrix
{ 0 \ar[r] & I \widetilde\mM  \ar[r]^\alpha \ar[d]_{=} & \mN \ar[r]^\pi \ar[d]_\imath   & V
\ar[d]^-{\tilde{\theta}} \ar[r] & 0 \\
0 \ar[r] & I \widetilde\mM  \ar[r]^{\beta_{\widetilde{\mM}}} & \widetilde\mM
\ar[r]^{\gamma_{\widetilde{\mM}}} & \widehat\mM
\ar[r] & 0.
}
\end{array}
\end{equation}
In other words, $\mN = \gamma^{-1}\bigl(\mathrm{Im}(\tilde\theta)\bigr)$.

Since $\tilde\theta$ is a monomorphism, the snake
lemma implies that $\imath$ is a monomorphism, too. Moreover, $\widetilde\mM$ is torsion free
viewed as $\rA$--module, hence $\mN$ is torsion free as well.
From the definition of morphisms  in the category $\Tri(\rA)$ and the universal property
of a kernel it follows that
the correspondence $ \Tri(\rA) \in \mT \mapsto \mN \in \CM(\rA)$ uniquely extends on the
morphisms in $\Tri(\rA)$. Hence,
$\GG$ is a well-defined functor.

Let $\mM$ be a Noetherian $\rA$--module and $\pi_\mM: \mM \lar \overline\mM :=  \bar\rA \otimes_\rA \mM$ be the canonical morphism. Then we
have the short exact sequence
$$
0 \lar \mM \xrightarrow{
\left(
\begin{smallmatrix}
- \jmath_\mM \\
\pi_\mM
\end{smallmatrix}
\right)
} \widetilde\mM \oplus \overline\mM
\xrightarrow{(\gamma_{\widetilde\mM} \; \widetilde{\theta}_\mM)} \widehat\mM \lar 0,
$$
yielding  an isomorphism of functors $\xi: \mathbbm{1}_{\CM(\rA)} \lar \GG \circ \FF$.
In particular, this implies that $\FF$ is faithful.

Next, we show that  $\GG$ is faithful.
Let $\mT = (\widetilde\mM, V, \theta)$ and $\mT' =
(\widetilde\mM', V', \theta')$
be a pair of objects in $\Tri(\rA)$ and $(\varphi, \psi): \mT \to \mT'$ be a morphism
in $\Tri(\rA)$.  Let $\mN = \GG(T), \mN' = \GG(T')$ and $\phi = \GG\bigl((\varphi, \psi)\bigr)$.
Then we have a commutative diagram in the category of $\rA$--modules:
\begin{equation}\label{E:proof-of-maindimone}
\begin{array}{c}
\xymatrix{
0 \ar[r] & \mN \ar[r]  \ar[d]_{\phi} & \widetilde\mM \oplus V \ar[r]^-{
\left(\begin{smallmatrix} \gamma & \tilde\theta \end{smallmatrix}\right)} \ar[d]_{\left(\begin{smallmatrix} \psi & 0 \\ 0 & \varphi \end{smallmatrix}\right)} & \widehat\mM \ar[r]
\ar[d]^{\widehat\psi} & 0 \\
0 \ar[r] & \mN' \ar[r]         & \widetilde\mM' \oplus V' \ar[r]^-{\left(\begin{smallmatrix} \gamma' &
\tilde\theta' \end{smallmatrix}\right)}        & \widehat\mM' \ar[r] & 0.
}
\end{array}
\end{equation}
First note that $(\varphi, \psi) = 0$ in $\Tri(\rA)$ if and only if $\psi = 0$. Indeed, one direction is obvious. To show the second, assume  $\psi = 0$.  Then $\widehat{\psi} = 0$ and $\tilde\theta' \circ \varphi = 0$.
Since
$\tilde\theta'$ is a monomorphism, we have: $\varphi = 0$.

Next, a morphism of Cohen--Macaulay $\rR$--modules
$\psi: \widetilde\mM \to \widetilde\mM'$ is zero if and only if
$\mathbbm{1} \otimes \psi: \rQ(\rA) \otimes_\rA \widetilde\mM \to \rQ(\rA) \otimes_\rA \widetilde\mM'$
is zero in $\rQ(\rA)-\mod$.
Suppose  the morphism of triples $(\varphi, \psi): \mT \to \mT'$ is non--zero.
  Apply the functor $\rQ(\rA) \otimes_\rA \,-\,$ on the
diagram (\ref{E:proof-of-maindimone}). It follows that $\mathbbm{1} \otimes \phi \ne 0$, hence
$\GG\bigl((\varphi, \psi)\bigr) \ne 0$ as well. Hence, $\GG$ is faithful.

Since we have constructed an isomorphism of functors
 $\xi: \mathbbm{1}_{\CM(\rA)} \lar \GG \circ \FF$ and $\GG$ is faithful, it implies that
$\FF$ is also full.
 Hence, to prove that $\FF$ is an equivalence of categories,  it remains to check
that $\FF$ is essentially surjective. For this,  it is sufficient  to show  that for any object
$\mT  = (\widetilde\mM, V, \theta) \in \Tri(\rA)$ we have: $\mT \cong \FF \GG (\mT)$.

Let $\mN := \GG(\mT)$, so that we have the diagram (\ref{E:FunctorG}). In these notations, the
morphism $\imath: \mN \to \widetilde\mM$ restricts to the morphism
$\bar\imath: I\mN \to I\widetilde\mM$ such that the following diagram is commutative:
$$
\xymatrix
{
I\mN \ar[r]^\varepsilon \ar[d]_{\bar\imath} & \mN \ar[d]^\imath\\
I\widetilde\mM \ar[r]^{\beta_{\widetilde{M}}} & \widetilde \mM,
}
$$
where $\varepsilon$ is the canonical inclusion. From the equality
$\imath \alpha \bar\imath = \beta_{\widetilde{M}} \, \bar\imath  = \imath  \varepsilon$ and the fact that $\imath$ is a monomorphism, we conclude that $\alpha  \bar\imath = \varepsilon$. In particular, we obtain
the following commutative diagram with exact rows:
$$
\xymatrix
{
0 \ar[r] & I \widetilde\mN \ar[r]^{\beta_{\widetilde\mN}}  & \widetilde\mN \ar[r]^{\gamma_{\tilde\mN}} & \widehat\mN \ar[r] & 0 \\
0 \ar[r] & I \mN \ar[r]^{\varepsilon}
 \ar[u]^{\bar{\jmath}_\mN}  \ar[d]_{\bar{\imath}} & \mN \ar[r]^{\pi_\mN}
\ar[u]_{\jmath_{\mN}} \ar[d]^{=}
 & \bar\mN \ar[r] \ar[u]_{\tilde\theta_{\mN}} \ar[d]^\varphi & 0 \\
0 \ar[r] & I \widetilde\mM \ar[r]^{\alpha} \ar[d]_{=} & \mN \ar[r]^\pi \ar[d]^{\imath} & V \ar[r]
\ar[d]^{\tilde\theta} & 0 \\
0 \ar[r] & I \widetilde\mM \ar[r]^{\beta_{\widetilde\mM}}
 & \widetilde\mM \ar[r]^{\gamma_{\widetilde\mM}}  & \widehat\mM \ar[r] & 0.
}
$$
First observe that by the snake lemma, the morphism $\varphi$ is an epimorphism. Next,
there exists a unique  morphism of \emph{$\rR$--modules}  $\psi: \widetilde\mN \to
\widetilde\mM$ such that $\psi \circ \jmath_\mN = \imath$ in the category $\rA\mathsf{-mod}$.
This follows from the natural isomorphisms
$$
\Hom_\rA(\mN, \widetilde\mM) \cong \Hom_\rR(\rR \otimes_\rA \mN, \widetilde\mM)
\cong \Hom_\rR(\widetilde\mN, \widetilde\mM).
$$
Since the morphisms $\imath$ and $\jmath_\mN$ are rational isomorphisms (i.e. they become
isomorphisms after applying the functor $\rQ(\rA) \otimes_\rA \,-\,$) the map $\psi$ is a
rational isomorphism, too.  Hence, $\psi$ is a monomorphism.

Consider the morphism $\widehat\psi: \widehat\mN \to \widehat\mM$ induced by $\psi$. It is not difficult to see
that the following diagram is commutative:
\begin{equation}\label{D:diagfortriples}
\begin{tabular}{c}
\xymatrix
{
\bar{R} \otimes_{\bar{A}} \bar\mN \ar[rr]^-{\theta_\mN} \ar[d]_{\mathbbm{1} \otimes \varphi} & & \widehat\mN \ar[d]^{\widehat\psi} \\
\bar{R} \otimes_{\bar\rA} V \ar[rr]^-\theta & & \widehat\mM.
}
\end{tabular}
\end{equation}
Since the morphisms $\theta$ and $\varphi$ are epimorphisms, $\widehat\psi$ is
an epimorphism, too. Moreover, by Lemma \ref{L:banal-stat}, $I$ is a subideal of the Jacobson's radical of $\rR$.
Hence, by
Nakayama's lemma, $\psi$ is an epimorphism, hence an isomorphism.

Note, that the  map $I \widetilde\mN \stackrel{\bar\psi}\lar I \widetilde\mM$, which is a restriction
of $\psi$,  is
again an isomorphism. Since $\bar\imath = \bar\psi \circ \bar{\jmath}_\mN$ and $\bar{\jmath}_\mN$
is an isomorphism, $\bar\imath$ is an isomorphism, too. Hence, $\varphi$ is an isomorphism
as well.
The commutativity of the diagram (\ref{D:diagfortriples}) implies that
  we get the following  isomorphism
$
(\psi, \varphi):
(\widetilde\mN, \bar\mN, \theta_\mN) \to (\widetilde\mM, V, \theta)
$
in the category $\Tri(\rA)$.
 This
concludes the proof of the fact that $\FF$ and $\GG$ are quasi--inverse equivalences of categories.

It remains to characterize the triples corresponding to the images of free modules.
One direction is clear: if $\mM = \rA^n$ for some $n \ge 1$ then
$\FF(\mM) \cong (\rR^n, \bar{\rA}^n, \theta)$, where $\theta$ is given by the
identity matrix in $\Mat_{n \times n}(\bar\rR)$.

On the other hand, let
$\mT = (\rR^n, \bar\rA^n, \theta)$ be a triple,  such that  the morphism $\theta \in \Mat_{n \times n}(\bar\rR)$
is an isomorphism. First note the induced morphism $\widetilde\theta: \bar\rA^n \to \bar\rR^n$
is automatically injective.
Next, Nakayama's lemma implies that the canonical morphism $\GL_n(\rR) \to \GL_n(\bar\rR)$
is an epimorphism.  This means that $\theta$ can be lifted to an  isomorphism
$\psi: \rR^n \to \rR^n$ and we get an isomorphism of triples
$
(\psi, \mathbbm{1}): \bigl(\rR^n, \bar\rA^n, \mathbbm{1}\bigr) \to
\bigl(\rR^n, \bar\rA^n, \theta\bigr).
$
Hence, the triple $\mT$ belongs to the image of the functor $\FF$. Since $\FF$ and $\GG$ are quasi--inverse
equivalences of categories, this concludes the proof of the theorem.
\end{proof}

\begin{remark}
There also exists a global version of Theorem \ref{T:BurbanDrozdindimOne} describing vector bundles and torsion free sheaves on a singular curve $X$ in terms of vector bundles on the normalization
$\widetilde{X}$ and some gluing data, see \cite[Proposition 42]{DrozdGreuelBundles} as well as \cite[Theorem 1.3 and Theorem 3.2]{Thesis}. See also \cite[Section 3]{Survey}, \cite[Chapter 3]{vb} and \cite[Section 5.1]{BK4} for further elaborations as well as \cite[Theorem 4.2]{Duke} for a  generalization of this
construction on the (bounded from above) derived category of coherent sheaves of a singular curve.
\end{remark}

\section{Appendix B: Decorated conjugation problem}\label{S:DecoratedConjugation}
Starting with a work  \cite{DroRoi}, the technique of matrix problems began to play an important  role in various classification results (proofs of representation finiteness and tameness) in commutative algebra and algebraic geometry. For instance, proofs of the following results are essentially  based
on the representation theory of bunches of \linebreak (semi-)chains  \cite{bo}.

\begin{itemize}
\item  The  proof of tameness of the category of indecomposable torsion free  modules over curve singularities
$P_{pq} := \kk\llbracket x, y, z\rrbracket/(x^p + y^q - z^2, xy)$, where  $p, q \in \mathbb{N}_{\ge 2} \cup \{\infty\}$ and  $(p, q) \ne (2, 2)$, see \cite{semicont, BG}.
\item A description of indecomposable vector bundles and torsion free sheaves on Kodaira cycles of projective lines
\cite{DrozdGreuelBundles, Survey, vb} as well as of indecomposable coherent sheaves on them \cite{Duke}.
\item The  proof of tameness of the category of indecomposable torsion free  modules over the cusp curve singularities
$T_{pq}:= \kk\llbracket x, y\rrbracket/(x^p + y^q + x^2 y^2)$, where $\frac{1}{p} + \frac{1}{q} < \frac{1}{2}$.
\end{itemize}
On the other hand, it seems that the technique of  proofs of \cite{bo} is not  really known apart of  the community of people
working in the domain of representation theory of finite dimensional algebras. On our mind, it  requires further elaborations.
In this section, we want to illustrate some central ideas of \cite{nar,bo}  on  the following
 example, which  generalizes  the classical conjugation problem of square matrices (Jordan--Frobenius normal form).  In what follows,  we shall use the same notation as in Section \ref{bc1}.

\begin{definition}
We say  that two matrices $W, \tilde{W} \in \Mat_{n \times n}(\rK)$ are equivalent if there
exists a pair of matrices $S, T \in \GL_n(\bD)$ satisfying the condition
 $\bar{S} = \bar{T}$, such that
 $\tilde{W} = S W T^{-1}$. In what follows we shall write $W \simeq \tilde{W}$, whereas
 $\sim$ will denote the usual conjugacy equivalence over $\kk$. A description of the canonical
 form of a square matrix over $\rK$ with respect to $\simeq$ equivalence relation is called \emph{decorated conjugation problem}.
\end{definition}

\noindent
Being a rather special example of the category of representations of a decorated bunch of chains,
the decorated conjugation problem incorporates many key features of the general case treated in  Theorem \ref{list}.

\begin{definition}\label{D:decorJordanBlock}
Consider now  the following canonical forms.

\medskip
\noindent
1.~For any tuple $\bmu = (\mu_1, \dots, \mu_n) \in \mathbb{Z}^n$ consider the $(n+1) \times (n+1)$ matrix
\begin{center}
  \begin{tikzpicture}
\matrix [tbl5,text width=22pt, minimum height=22pt,  name=table] at(6,0)
{
0& t^{\mu_1} &0& \dots &0 \\
 0&0& t^{\mu_2} & \dots &0 \\
  \vdots  &  \vdots  &  \ddots  &  \ddots  &   \vdots   \\
0&0&0& \ddots & t^{\mu_n} \\
0&0&0&\dots&0\\
};
\sfrm{table}{5}{5}
\node[base left=2pt of table.west]{$S(\bmu) =$};
\end{tikzpicture}
\end{center}

\medskip
\noindent
2.~Similarly, let $\bnu = (\nu_1, \dots, \nu_n) \in \mathbb{Z}^n$ be a non-periodic sequence,
$m \in \mathbb{N}$ and $\xi \ne \pi(\xi) \in \kk[\xi]$ be an irreducible monic polynomial of degree $d$. Let $F \in \Mat_{md \times md}(\bD)$ be such that
$\bar{F}  = F(\pi^m) \in \Mat_{md \times md}(\kk)$, where $F(\pi^m)$  is the Frobenius block corresponding to the polynomial
$\pi^m$. Consider the following $dmn \times dmn$ matrix
\begin{center}
  \begin{tikzpicture}
\matrix [tbl5,text width=30pt, minimum height=30pt,  name=table] at(6,0)
{
0& t^{\nu_1}I &0& \dots &0 \\
 0&0& t^{\nu_2}I  & \dots &0 \\
  \vdots  &  \vdots  &  \ddots  &  \ddots  &   \vdots   \\
0&0&0& \ddots & t^{\nu_{n-1}}I \\
t^{\nu_{n}}F &0&0&\dots&0\\
};
\sfrm{table}{5}{5}
\node[base left=2pt of table.west]{$B(\bnu, m, \pi) =$};

\end{tikzpicture}
\end{center}
\noindent
where $I$ is the identity matrix of size $dm \times dm$. \qed
\end{definition}

\medskip
\noindent
The main goal of this section is to give a complete proof of the following result.

\begin{theorem}\label{T:JordanKronecker}
Let $W \in \Mat_{n \times n}(\rK)$. Then we have:
 \begin{equation}\label{E:decompDecJord}
W \simeq   \mathsf{diag}\bigl(B(\bnu_1, m_1, \pi_1), \dots, B(\bnu_t, m_t, \nu_t),  S(\bmu_1)
 \dots, S(\bmu_p)\bigr),
 \end{equation}
for certain data $(\bnu_1, m_1, \pi_1), \dots, (\bnu_t, m_t, \pi_t), \bmu_1, \dots, \bmu_p$ as above.
Moreover, the blocks $B(\bnu, m, \pi)$ and $S(\bmu)$ are indecomposable (they do not split further)
and  the following uniqueness results are true.
\begin{itemize}
\item The equivalence class of $B(\bnu, m, \pi)$ does not depend on the choice of a lift
$F$ of the Frobenius block $F(\pi^m)$. In fact, it only  depends  on the
conjugacy class of $\bar{F}$ over $\kk$. In particular, if $\kk$ is algebraically closed and
$\lambda \in \kk^*$ then the Frobenius block
$F\bigl((\xi-\lambda)^m\bigr)$ can be replaced by the Jordan block $J_m(\lambda)$.
\item $B(\bnu, m, \pi) \not\simeq S(\bmu)$ for any  data $(\bnu, m, \pi)$ and $\bmu$.
\item $S(\bmu) \simeq S(\bmu')$ if and only if $\bmu' = \bmu$.
\item For $\bnu = (\nu_1, \nu_2, \dots, \nu_n)$ set $\bnu^{(1)} = (\nu_n, \nu_1, \dots, {\nu}_{n-1})$ Then we have:
\begin{equation*}
B(\bnu^{(1)}, m, \pi) \simeq B(\bnu, m, \pi)
\end{equation*}
\item Moreover, $B(\bnu, m, \pi) \simeq B(\tilde\bnu, \tilde{m}, \tilde\pi)$ if and only if
$(\pi, m)  = (\tilde\pi, \tilde{m})$ and $\tilde\bnu = \bnu^{(l)}$ for some $l \in \mathbb{Z}$.
\item The decomposition (\ref{E:decompDecJord}) is unique up to automorphisms of direct summands  and
permutation of blocks (Krull--Schmidt property).
\end{itemize}
\end{theorem}

\begin{remark} Note the following statements about the above  canonical forms \linebreak
$B(\bnu, m, \pi)$ and $S(\bmu)$.
\begin{itemize}
\item The matrix $S(\bmu)$ can be viewed as $B(\bnu, 1, \pi)$ where $\pi = \xi$ is the  ``forbidden''
polynomial  and $\bnu = (\mu_1, \dots, \mu_n, \nu)$ for $\bmu = (\mu_1, \dots, \mu_n)$ and some $\nu \in \mathbb{Z}$. However, this identification is not natural from the point of view of  generalizations of the decorated conjugation problem
(like representations of  decorated bunches of chains).
\item If the word $\bnu$ is periodic then the corresponding matrix $B(\bnu, m, \pi)$ is decomposable
(at least, if $\mathsf{char}(\kk) = 0$).
\item Theorem \ref{T:JordanKronecker} is  a special case of Theorem \ref{list}. In the notation
of Example \ref{E:DecorConjugation},  the canonical form $B(\bnu, m, \pi)$ corresponds to the band object
$B\bigl((w, \rho), m, \pi\bigr)$, where
$$
  (w, \rho) = {}\lha f\sim e  \stackrel{\nu_1}- f \sim e
  \stackrel{\nu_2}-   \dots   \stackrel{\nu_{n-1}}- f \sim e  \str{\nu_n}{\rha}.
  $$
Similarly, the canonical form $S(\bmu)$ corresponds to the string object $S(w, \rho)$, where
$$
  (w, \rho) = f\sim e  \stackrel{\mu_1}- f \sim e
  \stackrel{\mu_2}-  \dots   \stackrel{\mu_{n}}- f \sim e.
  $$
  In the notation of Theorem \ref{list}, we can always start our $\dX$--word (cyclic or not) with $f \in \dF$, reducing the number of non--equal but equivalent canonical forms.
\end{itemize}
\end{remark}

\subsection{Some preparatory results from linear algebra} In this subsection we collect some elementary results from linear algebra, playing a key role in the proof of Theorem \ref{T:JordanKronecker}. All proofs are sometimes lengthy, but always straightforward, so we leave them for an interested reader.

\begin{lemma}
Let $Y \in \Mat_{m \times n}(\bD)$ be such that $\rk_{\kk}(\bar{Y}) = r$. Then there exist
$S \in \GL_{m}(\bD)$ and $T \in \GL_n(\bD)$ such that
\begin{equation*}
SYT^{-1} =
\left(
\begin{array}{cc}
I & 0 \\
0 & t Z
\end{array}
\right),
\end{equation*}
where $I = I_{r}$ is the identity matrix of size $r \times r$ and $Z \in \Mat_{(m-r) \times (n-r)}(\bD)$.
\end{lemma}

\begin{lemma}\label{L:firstreduct}
Let $Y \in \Mat_{m \times m}(\bD)$. Then the following results are true.
\begin{itemize}
\item Assume that $\bar{Y} \in \Mat_{m \times m}(\kk)$ is invertible. Then
$
Y \simeq \mathsf{diag}(Y_1, \dots, Y_n),
$
where $\bar{Y}_l = F(\pi_l^{r_l})$ for some irreducible polynomials $\xi \ne \pi_l \in \kk[\xi]$ and
$r_l \in \mathbb{N}$, $1 \le l \le n$.
\item Assume that
$$
\bar{Y} \sim \left(
\begin{array}{cc}
\bar{Z} & 0 \\
0 & \bar{N}
\end{array}
\right),
$$
where $\bar{Z}$ is a square  invertible matrix  and $\bar{N}$ a nilpotent matrix. Then
$$
Y \simeq \left(
\begin{array}{cc}
Y^{\ast} & 0 \\
0 &  Y^{\circ}
\end{array}
\right),
$$
where $Y^{\ast}$ and $Y^{\circ}$ are
such that  $\bar{Y}^{\ast} = \bar{Z}$ and $\bar{Y}^{\circ} = \bar{N}$.
\end{itemize}
\end{lemma}

\noindent
Lemma \ref{L:firstreduct} shows that the  description of the canonical form
of  $X \in \Mat_{m \times m}(\bD)$ with respect to the $\simeq$ equivalence relation reduces to the case when the matrix $\bar{X}$ is
nilpotent.

\medskip
\noindent
Next, recall  the following version of the Jordan normal form of a nilpotent matrix.

\begin{lemma}\label{L:FormedeJordan}
Let $N \in \Mat_{m \times m}(\kk)$ be a nilpotent matrix. Then we have:

\begin{center}
  \begin{tikzpicture}
[ dot/.style={fill=blue!10,circle,draw, inner sep=1pt, minimum size=1pt},
str/.style={inner sep=1pt, minimum size=0pt}
]

\matrix  [tbl5,  name=tbl,
minimum height=20pt,
text width=20pt%
] at (0,0)
{
\ddots & & \vdots &&  		~& ~		& \vdots                \\
~ 		& 0 & I_3 &0&		 0&0 		&0                \\
\dots 	&0& 0 & I_3 &		 0&0 		&0                \\
~   	&0& 0 & 0   &		 0&0 		&0                \\
~   	&0& 0 & 0   &		 0&I_2 		&0                \\
~ 	    &0& 0 & 0   &		 0&0 		&0                \\
\dots 	&0& 0 & 0   &		 0&0 		&0                \\
};

\node[fit=(tbl-1-5)(tbl-1-6), inner sep=0pt,text depth=3pt]{$\vdots$};
\node[fit=(tbl-5-1)(tbl-6-1), inner sep=0pt]{$\dots$};

\sfrm{tbl}{7}{7};

\foreach \x in {1,4,6}{
\draw[thick] ([xshift=-10pt] tbl-\x-2.south west) to (tbl-\x-7.south east);
\draw[thick] ([yshift=10pt] tbl-2-\x.north east) to (tbl-7-\x.south east);
}

\node[base left=5pt of tbl-4-1.west] {$N \sim N(\underline{m}) :=$};

\end{tikzpicture}
\end{center}

\noindent
where the sequence $\underline{m} = (m_n, \dots, m_1) \in \mathbb{Z}^n_{\ge 0}$ is uniquely determined by $N$, $m = n m_n + \dots + 2 m_2 + m_1$
and $I_l$ is the identity matrix of size $m_l \times m_l$ for all $1 \le l \le n$.
\end{lemma}

\noindent
From Lemma \ref{L:FormedeJordan} one can easily deduce the following result.

\begin{proposition}\label{P:reductiNilpCase}
Let $W \in \Mat_{m \times m}(\bD)$ be such that $\bar{W}$ is nilpotent. Then

\begin{center}
  \begin{tikzpicture}

\matrix (first) [tbl5,  name=tbl,
minimum height=20pt,
text width=20pt%
] at (0,0)
{
\ddots & & \vdots &&  		~& ~		& \vdots                \\
~ 		& 0 & I_3 &0&		 0&0 		&0                \\
\dots 	&0& 0 & I_3 &		 0&0 		&0                \\
~   	&0& 0 & 0   &		 0&0 		&0                \\
~   	&0& 0 & 0   &		 0&I_2 		&0                \\
~ 	    &0& 0 & 0   &		 0&0 		&0                \\
\dots 	&0& 0 & 0   &		 0&0 		&0                \\
};

\node[fit=(tbl-1-5)(tbl-1-6), inner sep=0pt,text depth=3pt]{$\vdots$};
\node[fit=(tbl-5-1)(tbl-6-1), inner sep=0pt]{$\dots$};

\sfrm{tbl}{7}{7};

\foreach \x in {1,4,6}{
\draw[ thick] ([xshift=-10pt] tbl-\x-2.south west) to (tbl-\x-7.south east);
\draw[thick] ([yshift=10pt] tbl-2-\x.north east) to (tbl-7-\x.south east);
}
\node[base left=5pt of tbl-4-1.west ]{$W \simeq W_0 + t Z =$};

\matrix (first) [tbl5,  name=tbl,
minimum height=20pt,
text width=20pt%
] at (6,0)
{
\ddots & & \vdots &&  		~& ~		& \vdots                \\
~ 		& 0 & 0 &0&		 0&0 		&0                \\
\dots 	&0& 0 & 0 &		 0&0 		&0                \\
~   	   &Z_{33}& 0 &0   &Z_{32} 		&0 		&Z_{31}                 \\
~   	&0& 0 & 0   &		 0&0 		&0                \\
~ 	    &Z_{23}& 0 & 0   &Z_{22} 	&0 		&Z_{21}                \\
\dots 	&Z_{13}& 0 & 0   &Z_{12}		&0 		&Z_{11}                \\
};

\node[fit=(tbl-1-5)(tbl-1-6), inner sep=0pt,text depth=3pt]{$\vdots$};
\node[fit=(tbl-5-1)(tbl-6-1), inner sep=0pt]{$\dots$};
\sfrm{tbl}{7}{7};

\foreach \x in {1,4,6}{
\draw[thick] ([xshift=-10pt] tbl-\x-2.south west) to (tbl-\x-7.south east);
\draw[thick] ([yshift=10pt] tbl-2-\x.north east) to (tbl-7-\x.south east);
}
\node[base left=3pt of tbl-4-1]{$+ t$};
\end{tikzpicture}
\end{center}

\noindent
where the block division of $Z$ is the same as of $W_0 = N(\underline{m})$.
\end{proposition}

\noindent
Finally, we shall need the following elementary but quite  useful result.
Let $R$ be any ring and $X \in \Mat_{m \times n}(R)$. For any $1 \le l \le m$ and $1 \le t \le n$
we denote by $X^{\sharp(l, t)}$ the $(m-1) \times (n-1)$--matrix obtained from $X$ by crossing out
its $l$-th row and $t$-th column.

\begin{lemma}[\textsl{Crossing--Out Lemma}]\label{L:crossingout}
Let $m, \check{m}, n, \check{n} \in \mathbb{N}$ and $1 \le \imath \le m$, $1 \le \check{\imath} \le \check{m}$,
$1 \le \jmath \le n$ and $1 \le \check\jmath \le \check{n}$. Next, let
$W \in \Mat_{m \times n}(R)$ (respectively,  $\check{W} \in \Mat_{\check{m} \times \check{n}}(R)$)
be such that all elements of the $\imath$-th row and $\jmath$-th column of $W$ (respectively,
of the $\check\imath$-th row and $\check\jmath$-th column of $\check{W}$) but $w_{\imath\jmath}$ (respectively $w_{\check{\imath}\check{\jmath}}$)
are zero. Let
$S \in \Mat_{\check{m} \times m}(R)$ and $T \in \Mat_{n \times \check{n}}(R)$ be such that
$
S W = \check{W} T.
$
Then we have the following equality:
\begin{equation}
S^{\sharp (\check{\imath}, \imath)} W^{\sharp ({\imath}, \jmath)} =
\check{W}^{\sharp (\check{\imath}, \check\jmath)} T^{\sharp (\check{\jmath}, \jmath)}
\end{equation}
\end{lemma}

\subsection{Reduction to the decorated chessboard problem}
The further strategy is the following: we shall apply only those transformations of $W$ which
do not ``spoil'' the canonical form of $W_0$ and preserve the block structure of $Z$ in the decomposition
from Proposition \ref{P:reductiNilpCase}. The following statement plays a key role in this reduction procedure.

\begin{proposition}\label{P:morpjJordForm}
Let $\underline{m} = (m_n, \dots, m_1), \underline{\tilde{m}} = (\tilde{m}_n, \dots, \tilde{m}_1) \in \mathbb{Z}^{n}_{\ge 0}$
be two collections of non-negative integers, $m = m_1 + 2 m_2 + \dots + n m_n$  and $\tilde{m} =
\tilde{m}_1 + 2\tilde{m}_2 +\dots + n \tilde{m}_n$. Let
\begin{equation}
T(\underline{m}, \underline{\tilde{m}}) :=\Bigl\{A \in {\Mat}_{\tilde{m} \times m}(\kk) \, \big| \,
A N({\underline{m}})
= N(\underline{\tilde{m}}) A\Bigr\}.
\end{equation}
Then this vector space has the following explicit description:

\begin{center}
\begin{tikzpicture}
\matrix (first) [tbl5,  name=tbl,
minimum height=20pt,
text width=20pt%
] at (0,0)
{
\ddots & & \dots &&  		~& ~		& \dots                \\
~ 		& A_{33}^1 &A_{33}^2 &A_{33}^3&		A_{32}^1&A_{32}^2 		&A_{31}^1                \\
\vdots 	&0& A_{33}^1 & A_{33}^2 &		     0&A_{32}^1  		&0                \\
~   	    &0& 0 & A_{33}^1 &  		 0&0 		&0                \\
~   	    &0& A_{23}^1 & A_{23}^2   &		 A_{22}^1&A_{22}^2 		&A_{21}^1                \\
~ 	    &0& 0 & A_{33}^1   &		 0&A_{22}^1 		&0                \\
\vdots 	&0& 0 & A_{13}^1   &		 0&A_{12}^1 		&A_{11}^1                \\
};

\node[fit=(tbl-1-5)(tbl-1-6), inner sep=0pt,text depth=3pt]{$\dots$};
\node[fit=(tbl-5-1)(tbl-6-1), inner sep=0pt]{$\vdots$};

\sfrm{tbl}{7}{7};
\foreach \x in {1,4,6}{
\draw[thick] ([xshift=-10pt] tbl-\x-2.south west) to (tbl-\x-7.south east);
\draw[thick] ([yshift=10pt] tbl-2-\x.north east) to (tbl-7-\x.south east);
}
\node[base left=5pt of tbl-4-1.west, yshift=-3pt]{$A =$};
\node[base right=15pt of tbl-4-7.east,yshift=-3pt] {  $A_{kl}^{(p)} \in {\Mat}_{\tilde{m}_k \times m_l}(\kk)$};
\node[base left=35pt of tbl-4-1.west,yshift=-3pt]{$T(\underline{m}, \underline{\tilde{m}}) =$};

\draw[thick]
      (  [xshift=10pt] tbl-1-7.north east) --  ( [xshift=10pt]tbl-7-7.south east) ;

\draw[thick,decorate,decoration={ brace, amplitude=4pt}]
      (  [xshift=125pt] tbl-1-7.north east) --  ( [xshift=125pt]tbl-7-7.south east) ;

\draw[thick,decorate,decoration={ brace, amplitude=4pt}]
      ( [xshift=-30pt]tbl-7-1.south west) --  (  [xshift=-30pt] tbl-1-1.north west) ;

\end{tikzpicture}
\end{center}
\end{proposition}

\begin{proof}
Straightforward computation, see for example \cite[Section VIII.2]{Gantmacher}.
\end{proof}

\medskip
\noindent
For $\underline{p} = (p_1, \dots, p_n)$ and $p = p_1 + \dots + p_n$ consider the following parabolic
subalgebras ${P}^\pm(\underline{p})$ of the matrix algebra  $\Mat_{p \times p}(\kk)$:

\begin{center}
  \begin{tikzpicture}
  \matrix (first) [tbl5,  name=tbl,
minimum height=20pt,
text width=20pt%
] at (0,0)
{
B_{11}^+ 		& B_{12}^+ &\dots 		&B_{1n}^+           \\
0				   & B_{22}^+  &\dots		 &		 B_{2n}^+ 	       \\
\vdots &  \vdots	           & \ddots    &\vdots            \\
0				& 0	               &\dots		  	&		 B_{nn}^+     \\
};

\sfrm{tbl}{4}{4};

\node[base left=45pt of tbl-2-1.south west, yshift=-3pt]{$P^+(p)=$};
\node[base left=5pt of tbl-2-1.south west, yshift=-3pt]{$B^+=$};

\draw[thick,decorate,decoration={brace, amplitude= 4pt}]
( [xshift=-40pt]tbl-4-1.south west)  --  (  [xshift=-40pt] tbl-1-1.north west)  ;

\draw [thick]  ([xshift=10pt] tbl-1-4.north east) --  ( [xshift=10pt]tbl-4-4.south east) ;

\node [base right =10pt of tbl-2-4.south east, yshift=-3pt]{$B_{lk}\in \Mat_{p_l\times p_k}(\kk)$};
\draw[thick,decorate,decoration={ brace, amplitude=4pt}]
      (  [xshift=105pt] tbl-1-4.north east) --  ( [xshift=105pt]tbl-4-4.south east) ;

\end{tikzpicture}
\end{center}

\noindent
and $P^-(\underline{p})$ is the ``transpose'' of ${P}^+(\underline{p})$. Let
$$
D(\underline{p}) = \left\{(B^+, B^-) \in {P}^+(\underline{p}) \times {P}^-(\underline{p}) \left| B^{+}_{ll} = B^{-}_{ll}, 1\le l \le n\right.\right\}
$$
be the ``dyad'' of $P^{+}(\underline{p})$ and $P^{-}(\underline{p})$.

\begin{corollary}\label{E:RestrendomJordForm} Let $\underline{m} = (m_n, \dots, m_1) \in \mathbb{Z}^n_{\ge 0}$ and
$E(\underline{m}) := T(\underline{m}, \underline{m})$. Then the  map
\begin{equation}
R: \; E(\underline{m})  \lar D(\underline{m}), \;
A \mapsto (B^+, B^-),
\end{equation}
where $B^\pm$ are the following matrices

\begin{center}
  \begin{tikzpicture}
\matrix (first) [tbl5,  name=tbl,
minimum height=20pt,
text width=20pt%
] at (0,0)
{
\ddots &  \vdots &  \vdots	& \vdots                \\
\dots 	&A_{33}^1 		& 0 				 &0           \\
\dots 	&A_{23}^1		& A_{22}^1  		 &	0 	                \\
\dots 	&A_{13}^1		& A_{12}^1		 &A_{11}^1 		                \\
};

\sfrm{tbl}{4}{4};
\node[base left=5pt of tbl-2-1.south west, yshift=-5pt]{$B^-=$};

\matrix (first) [tbl5,  name=tbl,
minimum height=20pt,
text width=20pt%
] at (5,0)
{
\ddots &  \vdots &  \vdots	& \vdots                \\
\dots 	&A_{33}^1 		& A_{32}^1 		&A_{31}^1           \\
\dots 	&0				& A_{22}^1  		 &		 A_{21}^1 	                \\
\dots 	&0				& 0			  	&		 A_{11}^1 		                \\
};

\sfrm{tbl}{4}{4};
\node[base left=5pt of tbl-2-1.south west, yshift=-5pt]{$B^+=$};

\end{tikzpicture}
\end{center}

\noindent
is a surjective algebra homomorphism. Moreover, the induced algebra homomorphism
$$
\bar{R}: E(\underline{m})/\rad(E(\underline{m})) \lar
D(\underline{m})/\rad(D(\underline{m}))
$$
is an isomorphism.
\end{corollary}

\begin{proof}
The fact that $R$ is an algebra homomorphism follows from the following observation.
 Let $A \in E(\underline{m})$ be written as in Proposition \ref{P:morpjJordForm}.
 Then apart of the blocks from the ``diagonal mega--blocks'', either of two blocks of $A$, mirror to each other along the main diagonal, is zero.
Proofs of the remaining statements are straightforward.
\end{proof}

\begin{definition}\label{D:valuation}
For a matrix $ Y  \in \Mat_{m \times n}(\rK)$,  define its \emph{valuation}
$\mathsf{val}(Y)$ as the biggest
$\nu \in \mathbb{Z}$ such that $Y = t^\nu Y_{\diamond}$ for some ${Y}_{\diamond} \in \Mat_{m \times n}(\bD)$.
In particular, the  valuation of the zero matrix is  infinity.
\end{definition}

\noindent
To proceed with the further reduction, it is convenient to use the formalism of decorated bunches of chains, introduced in Section \ref{bc1}.
Let $\dJ$ (respectively $\dX$) be the decorated bunch of chains introduced in Example \ref{E:DecorConjugation} (decorated conjugation problem),
respectively in Example \ref{E:DecorChessBoard} (decorated chessboard problem), where the permutation
$\sigma$ is trivial.

\begin{definition}
For any $\nu \in \mathbb{Z}$ consider  the following full subcategory of $\Rep(\dX)$:
$$
\Rep\bigl(\dX(\nu)\bigr):= \Bigl\{W \in \Ob\bigl(\Rep(\dX)\bigr)\, \big| \, \mathsf{val}(W)\ge \nu\Bigr\}.
$$
\end{definition}

\noindent
Next, consider the  full subcategory of $\Rep(\dJ)$:
$$\Rep^{0}\bigl(\dJ(0)\bigr) :=
\Bigl\{W \in \Ob\bigl(\Rep(\dJ(0))\bigr)\, \big| \, \bar{W} \; \mbox{\rm is nilpotent}\Bigr\}.
$$  Finally,
let $\Rep^{0}_{\st}\bigl(\dJ(0)\bigr)$ be the full subcategory of $\Rep^{0}\bigl(\dJ(0)\bigr)$ consisting of standard objects,
i.e.~of matrices $W$ as in  Proposition \ref{P:reductiNilpCase}. For such  $W$   we set

\begin{center}
\begin{tikzpicture}
\matrix (first) [tbl5,  name=tbl,
minimum height=20pt,
text width=20pt%
] at (0,0)
{
\ddots &  \vdots &  \vdots	& \vdots                \\
\dots 	&Z_{33} 		& Z_{32} 		&Z_{31}           \\
\dots 	&Z_{23}				& Z_{22}  		 &		 Z_{21}	                \\
\dots 	&Z_{13}				& Z_{12}			 &		 Z_{11} 		                \\
};

\sfrm{tbl}{4}{4};
\node[base left=5pt of tbl-3-1.north west, yshift=-5pt]{$Z_{\mathsf{red}}:=$};
\end{tikzpicture}
\end{center}

\noindent
The next  result follows from  Lemma \ref{L:crossingout}, Proposition \ref{P:morpjJordForm} and Corollary
\ref{E:RestrendomJordForm}.

\begin{proposition}\label{P:secondReduct}
The map
$
R: \; \underline{\Rep}^{0}_{\st}\bigl(\dJ(0)\bigr) \lar \underline{\Rep}\bigl(\dX(0)\bigr), \,
W \mapsto  Z_{\mathsf{red}}$
extends to a $\bD$--linear  functor sending a morphism $(S, T)$ to $(S^{\sharp}, T^{\sharp\sharp})$. Here,
$S^{\sharp}$ (respectively  $T^{{\sharp}\sharp}$)  is the  matrix obtained from $S$ (respectively  $T$) by crossing out appropriate zero columns and rows according to the decomposition of $W_0$ in the presentation
as in Proposition \ref{P:reductiNilpCase}.
Moreover, the functor $R$  is a \emph{representation equivalence}, i.e.
\begin{itemize}
\item $R$ is essentially surjective.
\item For  $W, W' \in \Ob\Bigl(\underline{\Rep}^{0}_{\st}\bigl(\dJ(0)\bigr)\Bigr)$  we have:
$R(W) \cong R(W')$ if and only if $W \cong W'$.
\end{itemize}
\end{proposition}

\begin{remark} Note that the last two properties  imply that
$W \in \Ob\bigl(\underline{\Rep}^{0}_{\st}(\dJ)\bigr)$ is indecomposable if and only if $R(W)$ is indecomposable.
\end{remark}

\noindent
Proposition \ref{P:secondReduct} shows that the proof of Theorem \ref{T:JordanKronecker} reduces to a classification of ``indecomposable decorated chessboards''. At first glance it looks like a digression because  the new
matrix problem seems to be  more general (and  complicated) than the original one.
However,  it turns out to be an illusion. In the next subsection we shall see  that
the decorated  chessboard problem  is  ``self--reproducible'' in an appropriate  sense.

\subsection{Reduction procedure for the decorated chessboard problem}\label{SS:decbunchchains}
As in the previous subsection, let $\dX$ be the decorated bunch of chains from Example \ref{E:decChessBoard} with trivial permutation parameter $\sigma$.
Recall  from Subsection \ref{SS:mpConcrete} that an object of $\Rep(\dX)$ is given by the  data $W = \bigl(\sd, \{W_{pq}\}_{p, q \in \mathbb{N}}\bigr)$, where

\begin{itemize}
\item $\sd: \mathbb{N} \lar \mathbb{Z}_{\ge 0}, p \mapsto \sd_p$ is a function with finite support.
\item For any $p, q \in \mathbb{N}$,  $ W_{pq}$ is a matrix from
$\Mat_{\sd_p \times \sd_q}(\rK)$.
\end{itemize}
A morphism $W = \bigl(\sd,  \{W_{pq}\}\bigr) \lar
\check{W} = \bigl(\check\sd,  \{\check{W}_{pq}\}\bigr)$ is given
by a collection of matrices $(F_{\bullet}, G_{\bullet}) = \bigl(\{F_{uv}\}, \{G_{uv}\}\bigr)$ such that
$F_{uv}, G_{uv} \in \Mat_{\check{\sd}_u \times \sd_v}(\bD)$ for  $u, v \in \mathbb{N}$;
$\bar{F}_{uu} = \bar{G}_{uu}$ for all $u \in \mathbb{N}$ and
$F_{uv} \in \Mat_{\check{\sd}_u \times \sd_v}(\idm)$ for $u < v$ and
$G_{uv} \in \Mat_{\check{\sd}_u \times \sd_v}(\idm)$ for $u > v$,
satisfying  the following constraint:
\begin{equation}\label{E:triangularity}
\sum\limits_{l < p} F_{pl} W_{lq} + F_{pp} W_{pq} +
\sum\limits_{l > p} F_{pl} W_{lq} =
\sum\limits_{t > q} \check{W}_{pt} G_{tq} + \check{W}_{pq} G_{qq} +
\sum\limits_{t < q} \check{W}_{pt} G_{tq}.
\end{equation}
\begin{definition}
Consider the following ordering on $\mathbb{N} \times \mathbb{N}$:
$(p, q) < (p', p')$ if $p < p'$ or $p = p'$ and $q > q'$. Let  $W = \bigl(\sd, \{W_{\imath \jmath}\}\bigr)$
be an object of $\Rep(\dX)$. We say that the block $W_{pq}$  is minimal if
$\mathsf{val}(W_{p q}) \le \mathsf{val}(W_{p' q'})$ for all
$(p', q') \in \mathbb{N} \times \mathbb{N}$ and $\mathsf{val}(W_{p q}) < \mathsf{val}(W_{p' q'})$ for
$(p', q') <  (p, q)$.
\end{definition}
\begin{definition}
For any $(p, q) \in \mathbb{N} \times \mathbb{N}$,  let
$\Rep^{\le (p, q)}\bigl(\dX(\nu)\bigr)$ be the full subcategory of $\Rep\bigl(\dX(\nu)\bigr)$ consisting of those objects
$W$ for which the block $W_{pq}$ is  minimal.
\end{definition}

\noindent
\textbf{Convention}.  Let $Y \in \Mat_{m \times n}(\rK)$
be such that $\mathsf{val}(Y) \ge \nu$. In what follows we shall use the notation
$Y = t^\nu Y_{\diamond}$ for an appropriate  $Y_\diamond  \in \Mat_{m \times n}(\bD)$.


\medskip
\noindent
\textbf{Case 1}.  Suppose that $p \ne q$. Then there exist $F_{pp} \in \GL_{\sd_p}(\bD)$,
$G_{qq} \in \GL_{\sd_q}(\bD)$ such that
$$
(\tilde{W}_{pq})_\diamond := F_{pp} \bigl(W_{pq})_\diamond G_{qq}^{-1} =
\left(
\begin{array}{cc}
I & 0 \\
0 & t \Psi
\end{array}
\right)
$$
for some matrix $\Psi$ with coefficients over $\bD$.
 Moreover, the entire system of matrices $W = \bigl\{W_{\imath\jmath}\bigr\}$ can be transformed into the following ``standard'' form: $W = t^\nu W_{\diamond}$, where $W_{\diamond}$ is the following matrix:

 \begin{center}
   \begin{tikzpicture}

\matrix (first) [tbl5,  name=tbl,
minimum height=20pt,
text width=20pt,%
row 3/.style={minimum height=30pt},
row 5/.style={minimum height=40pt},
row 6/.style={minimum height=30pt},
column 3/.style={text width=30pt},
column 5/.style={text width=50pt},
column 6/.style={text width=30pt}
] at (0,0)
{
* 		& B_1		&B_2			&0 			&*				&*                \\
0		& 0 			& 0 			&I		 	&0 				&0                \\
*		& *		    &*		 	&0	 		&t\Psi          &*      \\
C_1		& D_1^1		&D_1^2		&0	 		&E_1            &G_1      \\
C_2		& D_2^1		&D_2^2		&0	 		&E_2            &G_2      \\
*		&F_1 		&F_2		    &0		 	&*		 		&*                \\
};

\sfrm{tbl}{6}{6};

\hdline{tbl}{4}{6};
\draw[dotted] (tbl-2-1.south west) -- (tbl-2-6.south east);
\hsline{tbl}{1}{6};\hsline{tbl}{3}{6};\hsline{tbl}{5}{6};

\vdline{tbl}{6}{2};
\draw[dotted] (tbl-1-4.north east) -- (tbl-6-4.south east);
\vsline{tbl}{6}{1};\vsline{tbl}{6}{3};\vsline{tbl}{6}{5};

\node[above=1pt of tbl-1-3.north]{$y_p$};
\node[above=1pt of tbl-1-5.north west, xshift=5pt]{$y_q$};

\node[below=1pt of tbl-6-2.south,red]{$y_p^1$};
\node[below=1pt of tbl-6-3.south, red]{$y_p^2$};

\node[base left=2pt of tbl-2-1.south west,yshift=-5pt]{$x_p$};
\node[base left=2pt of tbl-4-1.south west,yshift=-8pt]{$x_q$};

\node[base right=2pt of tbl-4-6.east,red]{$x_q^1$};
\node[base right=2pt of tbl-5-6.east,red]{$x_q^2$};
\end{tikzpicture}
 \end{center}

Let $\Rep^{\le (p, q)}_{\st}\bigl(\dX(\nu)\bigr)$ be the full subcategory of $\Rep^{\le (p, q)}\bigl(\dX(\nu)\bigr)$
consisting of objects in  the standard  form.
Clearly,
the embedding $\Rep^{\le (p, q)}_{\st}\bigl(\dX(\nu)\bigr) \hookrightarrow \Rep^{\le (p, q)}\bigl(\dX(\nu)\bigr) $ is an equivalence
of categories.

\begin{definition}
Consider a new decorated chessboard $\dX^{(p, q)} = \check\dX$ defined  as follows:
\begin{itemize}
\item $x_q \mapsto \bigl\{x_q^{(1)}, x_q^{(2)}\bigr\}$ and $y_p \mapsto \bigl\{y_p^{(1)}, y_p^{(2)}\bigr\}$,
\item  The new order on $\check\dE$ and $\check\dF$ is defined as follows:
$$
x_l \dec x_{q}^{(1)} \dec x_{q}^{(2)} \dec x_t
\;
\mbox{\rm and} \;
y_r \ced y_{p}^{(1)} \ced y_{p}^{(2)} \ced y_s \;\mbox{
for all} \; l < q < t \; \mbox{and} \; r < p < s.
$$
\item The new equivalence relation (defining an appropriate bijection between the column and row labels from Example \ref{E:decChessBoard}) is the following:
$$
x_p \sim y_p^{(2)}, y_q \sim x_q^{(2)} \quad \mbox{\rm and} \quad
x_q^{(1)} \sim y_p^{(1)}.
$$
The remaining equivalence relations in $\check\dX$ are the same as in $\dX$.
\end{itemize}
Let $W$ be an object of $\Rep^{\le (p, q)}_{\st}\bigl(\dX(\nu)\bigr))$ such that
$\overline{(W_{pq})_{\diamond}} \ne 0$.  Consider the object $W^{\sharp}$
of $\Rep(\check\dX)$
obtained from $W$ by the following operations.
\begin{itemize}
\item We cross out all rows and columns of $W$, containing the entry $1$  in the block ${W}_{pq}$.
\item Next,  we give new weights $x_q^{(1)}, x_q^{(2)}$ to the horizontal stripe $x_q$ and $y_p^{(1)}, y_p^{(2)}$ to the vertical  stripe $y_p$ induced by  the block division  of the matrix ${W}_{pq}$.
\end{itemize}
If $\overline{(W_{pq})_{\diamond}} =  0$ then we pose $W^\sharp = W$.
\end{definition}

\begin{example} For the object $W$ given in the example above we have
$W^\sharp = t^\nu W^\sharp_{\diamond}$, where $W^\sharp_{\diamond}$ is the following matrix:

\begin{center}
  \begin{tikzpicture}

\matrix [tbl5,  name=tbl,
minimum height=20pt,
text width=20pt,%
row 2/.style={minimum height=30pt},
row 4/.style={minimum height=40pt},
row 5/.style={minimum height=30pt},
column 3/.style={text width=30pt},
column 4/.style={text width=50pt},
column 5/.style={text width=30pt}
] at (7,0)
{
* 		& B_1		&B_2			 			&*				&*                \\
*		& *		    &*		 		 		&t\Psi          &*      \\
C_1		& D_1^1		&D_1^2			 		&E_1            &G_1      \\
C_2		& D_2^1		&D_2^2			 		&E_2            &G_2      \\
*		&F_1 		&F_2		    		 	&*		 		&*                \\
};

\sfrm{tbl}{5}{5};

\hsline{tbl}{1}{5};\hsline{tbl}{2}{5};\hsline{tbl}{3}{5};\hsline{tbl}{4}{5};
\vsline{tbl}{5}{1};\vsline{tbl}{5}{2};\vsline{tbl}{5}{3};\vsline{tbl}{5}{4};

\node[above=1pt of tbl-1-2.north]{$y_p^1$};
\node[above=1pt of tbl-1-3.north]{$y_p^2$};
\node[above=1pt of tbl-1-4.north]{$y_q$};
\node[base left=2pt of tbl-2-1.west]{$x_p$};
\node[base left=2pt of tbl-3-1.west]{$x_q^1$};
\node[base left=2pt of tbl-4-1.west]{$x_q^2$};

\end{tikzpicture}
\end{center}
\end{example}

\begin{proposition}\label{P:reductinPnotQ} The assignment
$
R^{pq}:\; \underline{\Rep}^{\le (p, q)}_{\st}\bigl(\dX(\nu)\bigr) \lar \underline{\Rep}^{<((p, q), \nu)}\bigl(\check\dX(\nu)\bigr),\;  W \mapsto W^\sharp
$
is a representation equivalence of categories, where $\Rep^{<((p, q), \nu)}\bigl(\check\dX(\nu)\bigr)$
is the full subcategory of $\Rep^{\le (p, q)}\bigl(\check\dX(\nu)\bigr)$ consisting of those objects $U$  for which
$\mathsf{val}(U_{pq}) > \nu$.
\begin{proof}
Let $(F_\bullet, G_\bullet): W \lar \tilde{W}$ be a morphism in $\Rep^{\le (p, q)}_{\st}(\dX)$.
Equality (\ref{E:triangularity}) implies:
$$
F_{pp}
\left(
\begin{array}{cc}
I & 0 \\
0 & t \Psi
\end{array}
\right)
=
\left(
\begin{array}{cc}
I' & 0 \\
0 & t \Psi'
\end{array}
\right) G_{qq}\; \mod \; \idm.
$$
Thus we have:
$
\bar{G}_{pp} = \bar{F}_{pp} =
\left(
\begin{array}{cc}
X & Y \\
0 & T
\end{array}
\right) \; \mod \; \idm$ and
$
\bar{F}_{qq} = \bar{G}_{qq} =
\left(
\begin{array}{cc}
X & 0 \\
Z & S
\end{array}
\right) \; \mod \; \idm
$
for appropriate matrices $X, Y, Z, T$ and $S$ over $\kk$.

By Crossing--Out Lemma \ref{L:crossingout},  the equality $F_\bullet W = \tilde{W} G_\bullet$
yields the equality
$
F_\bullet^{\sharp} W^{\sharp} = \tilde{W}^{\sharp} G_\bullet^{\sharp\sharp},
$
where $F_\bullet^{\sharp}$ (respectively $G_\bullet^{\sharp\sharp}$) is obtained from $F_\bullet$ (respectively $G_\bullet$) by crossing--out appropriate columns and rows.
Hence, we obtain a well--defined map
$$
R^{pq}:  \Hom_{\dX}(W, \tilde{W}) \lar \Hom_{\check\dX}(W^\sharp, \tilde{W}^\sharp),
\quad (F_\bullet, G_\bullet) \mapsto (F_\bullet^\sharp, G_\bullet^{\sharp\sharp}).
$$
Note that
$
R^{pq}: \; \underline{\Rep}^{\le (p, q)}_{\st}\bigl(\dX(\nu)\bigr) \lar \underline{\Rep}^{<((p, q), \nu)}\bigl(\check\dX(\nu)\bigr)
$
is indeed   a functor.

\medskip
\noindent
To show that $R^{pq}$ is a representation equivalence, it is sufficient to prove the following statement.
Let $\check{f}: W^\sharp \lar \tilde{W}^\sharp$ be an \emph{isomorphism} in $\underline{\Rep}^{<((p, q), \nu)}\bigl(\check\dX(\nu)\bigr)$.
Then there exists an isomorphism $f: W \lar \tilde{W}$ in $\underline{\Rep}^{\le (p, q)}_{\st}\bigl(\dX(\nu)\bigr)$ such that
$R^{pq}(f) = \check{f}$.
Since any isomorphism in $\underline{\Rep}(\check\dX)$ can be written as a composition of elementary transformations,
it suffices to prove liftability of an elementary transformation. However,  this property
 follows from a direct case-by-case
verification.
\end{proof}
\end{proposition}

\medskip
\noindent
\textbf{Case 2}.  Suppose that $p = q$. According to  Lemma \ref{L:firstreduct}, we have:
$
(W_{pp})_\diamond \simeq
\left(
\begin{array}{cc}
Z & 0 \\
0 & N
\end{array}
\right),
$
where  the matrix $\bar{Z}$ is  invertible and $\bar{N}$ is nilpotent. It is then clear, that
$W \simeq t^\nu V$, where $V$ is the following matrix:

\begin{center}
  \begin{tikzpicture}
\matrix (first) [tbl5,  name=tbl,
minimum height=20pt,
text width=20pt,%
row 2/.style={minimum height=30pt},
column 2/.style={text width=30pt},
] at (0,0)
{
~		&0		&~		&~\\
0		&Z		&0		&0\\
~		&0		&N		&~\\
~		&0		&~		&~\\
};

\sfrm{tbl}{4}{4};
\draw[dotted] (tbl-1-2.east) -- (tbl-4-2.east);
\draw[dotted] (tbl-2-1.south) -- (tbl-2-4.south);
\draw(tbl-1-1.south) -- (tbl-1-4.south); \draw(tbl-3-1.south) -- (tbl-3-4.south);
\draw(tbl-1-1.east) -- (tbl-4-1.east);\draw(tbl-1-3.east) -- (tbl-4-3.east);

\node[above=1pt of tbl-1-2.north east]{$y_p$};
\node[base left=2pt of tbl-2-1.south west]{$x_p$};
\end{tikzpicture}
\end{center}

\noindent
In particular, $t^\nu Z$ splits as a direct summand of $W$.
Thus, we may restrict ourselves on the full subcategory
$\Rep^{\le (p, p), \circ}\bigl(\dX(\nu)\bigr)$ of $\Rep^{\le (p, p)}\bigl(\dX(\nu)\bigr)$
consisting of those objects for which $\overline{(W_{pp})}_\diamond$  is a nilpotent matrix.
According to Proposition \ref{P:reductiNilpCase},
$
(W_{pp})_\diamond
 \simeq N + t Z,
$
where $N$ is the normal form  of $\overline{(W_{pp})}_\diamond$ as in Lemma \ref{L:FormedeJordan}. Then
$W \simeq t^\nu V$, where $V$ is the following matrix:

\begin{center}
  \begin{tikzpicture}

\matrix (first) [tbl5,  name=tbl,
minimum height=20pt,
text width=20pt,%
]
{
~		&~			&~		&0			&0					&~		&0			&~		&~	\\
~		&\ddots		&~		&\vdots		&\vdots				&~		&\vdots		&~		&~	\\
0		&\dots		&0		&I_3			&0					&0		&0			&0		&0	\\
0		&\dots		&0		&0			&I_3					&0		&0			&0		&0	\\
~		&~			&t Z_{33}		&0		&0					&t Z_{32}	&0			&t Z_{31}	&~	\\
0		&\dots		&0		&0			&0					&0		&I_2			&0		&0	\\
~		&~			&t Z_{23}		&0		&0					& t Z_{22}	&0			& t Z_{21}	&~	\\
~		&~			&t Z_{13}		&0		&0					&t Z_{12}	&0			&t Z_{11}	&~	\\
~		&~			&~			&0		&0					&~		&0			&~		&~ \\
};

\sfrm{tbl}{9}{9};

\draw(tbl-2-2.east) -- (tbl-9-2.north east);\draw(tbl-2-5.east) -- (tbl-9-5.north east);\draw(tbl-2-7.east) -- (tbl-9-7.north east);
\draw(tbl-2-2.south) -- (tbl-2-9.south west); \draw(tbl-5-2.south) -- (tbl-5-9.south west);\draw(tbl-7-2.south) -- (tbl-7-9.south west);

\draw[red](tbl-1-1.south) -- (tbl-1-9.south); \draw[red](tbl-8-1.south) -- (tbl-8-9.south);
\draw[red](tbl-1-1.east) -- (tbl-9-1.east);\draw[red](tbl-1-8.east) -- (tbl-9-8.east);

\node[above=1pt of tbl-1-5.north]{$y_p$};
\node[base left=2pt of tbl-5-1.south west]{$x_p$};

\end{tikzpicture}
\end{center}

\noindent
Let $\Rep^{\le (p, p), \circ}_{\st}\bigl(\dX(\nu)\bigr)$ be the full subcategory of $\Rep^{\le (p, p), \circ}\bigl(\dX(\nu)\bigr)$
consisting of objects having the  above standard form. Clearly,
the inclusion $\Rep^{\le (p, p), \circ}_{\st}\bigl(\dX(\nu)\bigr) \hookrightarrow \Rep^{\le (p, p), \circ}\bigl(\dX(\nu)\bigr)$
is an equivalence of categories.

\begin{definition}
Consider a new decorated chessboard $\dX^{(p,p)} = \check\dX$ defined as follows.
\begin{itemize}
\item The sets $\dE_\ast$ and $\dF_\ast$ are obtained by replacing
$$
x_p \mapsto \{ \dots \dec x_p^{(l)} \dec \dots \dec x_p^{(2)} \dec x_p^{(1)}\} \quad
\mbox{\rm
and} \quad y_p \mapsto \{ \dots \ced y_p^{(l)} \ced \dots \ced y_p^{(2)} \ced y_p^{(1)}\}.$$
Abusing the notation, we shall also write $x_p^{(1)} = x_p$ and $y_p^{(1)} = y_p$.
\item The new ordering in $\check\dX$ is defined as follows:
$$
x_m \dec x_p^{(l)} \dec x_n \quad
\mbox{\rm and} \quad
y_m \ced y_p^{(l)} \ced y_n
$$
for any $m < p < n \in \mathbb{N}$ and any $l \in \mathbb{N}$.
\item We have the equivalence relations $x_p^{(l)} \sim y_p^{(l)}$ for any
$l \in \mathbb{N}$ as well as all remaining old equivalences
$x_n \sim y_n$ for $n \ne p$.
\end{itemize}
Again, for any object $W$ of $\Rep^{\le (p,p), \circ}_{\st}\bigl(\dX(\nu)\bigr)$ let
$W^\sharp$ be the object of $\Rep(\check\dX)$ obtained from $W$ by the following operations.
\begin{itemize}
\item We cross out those rows of weight $x_p$ and columns of weight $y_p$ of the matrix $W$, which contain the entry $1$ in the normal form $N$.
\item
The survived  rows of stripe $x_p$ and columns of stripe $y_p$ get new labels $x_p^{(l)}$ and
$y_p^{(l)}$ according to the block division of the normal form of $N$ (i.e.~the upper label  is $l$ if the corresponding column/row corresponds to the nilpotent Jordan block $J_l(0)$).
\end{itemize}
\end{definition}

\begin{proposition}\label{P:reductinPisQ}
The map $R^{pp}: \, \underline{\Rep}^{\le (p,p), \circ}_{\st}\bigl(\dX(\nu)\bigr) \lar \underline{\Rep}^{((p, p), \nu)}\bigl(\check\dX(\nu)\bigr)$ is functorial, where $\underline{\Rep}^{((p, p), \nu)}\bigl(\check\dX(\nu)\bigr)$
is the full subcategory of $\underline{\Rep}\bigl(\check\dX(\nu)\bigr)$ consisting of those objects $V$ for which
$\mathsf{val}(V_{pp}) > \nu$.
Moreover, $R^{pp}$ is a representation equivalence of categories.
\end{proposition}

\begin{proof}
\noindent
All major  steps are basically the same  as in Proposition \ref{P:reductinPnotQ}. First,
we use the equality (\ref{E:triangularity}), Proposition \ref{P:morpjJordForm} and
Crossing--Out Lemma \ref{L:crossingout} to construct a  map
$$
R^{pp}:\; \Hom_{\dX}(W, \tilde{W}) \lar \Hom_{\check{\dX}}(W^\sharp, \tilde{W}^\sharp)
$$
for any pair of objects $W, \tilde{W}$ of $\Rep^{\le (p,p), \circ}_{\st}\bigl(\dX(\nu)\bigr)$.
The functoriality of $R^{pp}$ on the level of  stabilized bimodule categories is clear.
Finally, the claim that  for  any isomorphism $\check{f}: W^\sharp \lar \tilde{W}^\sharp$ in
$\underline{\Rep}^{((p, p), \nu)}\bigl(\check\dX(\nu)\bigr)$  there exists an isomorphism $f: W \lar \tilde{W}^\sharp$
in $\Rep^{\le (p,p), \circ}_{\st}\bigl(\dX(\nu)\bigr)$ such that $R^{pp}(f) = \check{f}$,
is essentially a consequence of Corollary \ref{E:RestrendomJordForm}.
\end{proof}

\subsection{Indecomposable representations of a decorated chessboard} Let $\dX$ be a bunch of chains from Example
\ref{E:decChessBoard}. In this  part, we describe indecomposable objects of $\Rep(\dX)$ (assume for simplicity of notation that the permutation $\sigma$ is trivial). Although the final answer can
be stated in completely elementary terms, for the proof it is convenient to use the formalism
 of strings and bands from  Subsection \ref{SS:StringsBands}.

Observe that any element of $\dX$ is tied. Hence, without loss of generality we may start any  $\dX$--word (cyclic or not)
with an element of $\dF$. Concretely, a decorated word defining a string representation has the form
\begin{equation}\label{E:chessString}
(w, \rho)= \;  y_{\jmath_1} \sim x_{\jmath_1} \stackrel{\mu_1}- y_{\jmath_2} \sim x_{\jmath_2} \stackrel{\mu_2}- \dots
 \stackrel{\mu_n}- y_{\jmath_{n+1}} \sim x_{\jmath_{n+1}}
\end{equation}
for any $n \in \mathbb{Z}_{\ge 0}$, $\jmath_1, \dots, \jmath_{n+1} \in \mathbb{N}$ and $(\mu_1, \dots, \mu_n) \in \mathbb{Z}^n$. Similarly, a cyclic decorated word defining a band representation has the form
\begin{equation}\label{E:chessBand}
(w, \rho)= \;   \lha y_{\imath_1} \sim x_{\imath_1} \stackrel{\nu_1}- y_{\imath_2} \sim x_{\imath_2} \stackrel{\nu_2}- \dots
 \stackrel{\nu_{n-1}}- y_{\imath_{n}} \sim x_{\imath_{n}} \stackrel{\nu_n}\rha
\end{equation}
where  $n \in \mathbb{N}$ and $\bigl((\imath_1, \nu_1), \dots, (\imath_{n}, \nu_n)\bigr)  \in
(\mathbb{N} \times \mathbb{Z})^n$ is a non--periodic sequence.
For a decorated word $(w, \rho)$ (respectively  for  a non--periodic cyclic decorated word
$(w, \rho)$, $m\in \mathbb{N}$ and an irreducible polynomial $\xi \ne \pi \in \kk[\xi]$) consider the objects $S(w, \rho)$
(respectively $B\bigl((w, \rho), m, \pi\bigr)$) given by the same matrix as in Definition \ref{D:decorJordanBlock} but
with an additional labeling of rows and columns with weights $x_{\jmath_1}, \dots, x_{\jmath_{n+1}}$ and
$y_{\jmath_1}, \dots, y_{\jmath_{n+1}}$ (respectively, with $x_{\imath_1}, \dots, x_{\imath_{n}}$ and
$y_{\imath_1}, \dots, y_{\imath_{n}}$).

\begin{example} Let $(w, \rho) = \lha y_1 \sim x_1 \stackrel{3}- y_1 \sim x_1 \stackrel{-2}- y_2 \sim x_2 \stackrel{0}-
y_1 \sim x_1 \stackrel{3}- y_2 \sim x_2 \stackrel{1}\rha$,
$m \in \mathbb{N}$, $\lambda \in \kk^*$ and $\pi = \xi-\lambda$.
Then we get the following canonical forms:
$B\bigl((w, \rho), m, \pi) =$

\begin{center}
\begin{tikzpicture}
\matrix (first) [tbl5,  name=tbl, minimum height=20pt,  text width=20pt,%
] at(0,0)
{
0		&t^3I			&0			&0		&0 		\\
0		&0				&t^{-2}I		&0		&0 		\\
0		&0				&0			&I		&0 		\\
0		&0 				&0			&0		&t^3I 	\\
tJ		&0 				&0			&0		&0 	\\
};
\sfrm{tbl}{5}{5};
\foreach \i in {1,2,4}
{
\node[above=1pt of tbl-1-\i.north]{$y_{1}$};
\node[base left=2pt of tbl-\i-1]{$x_1$};
}

\foreach \i in {3,5}
{
\node[above=1pt of tbl-1-\i.north]{$y_{2}$};
\node[base left=2pt of tbl-\i-1]{$x_2$};
}

\node[base right=7pt of tbl-3-5]{$\simeq$};
\matrix (first) [tbl5,  name=tbl,
minimum height=20pt,
text width=20pt,%
] at(5,0)
{
0		&t^3I			&0		&0			&0		\\
0		&0				&0		&t^{-2}I		&0		\\
0		&0				&0		&0			&t^3I 	\\
0		&0 				&I		&0			&0		\\
tJ		&0 				&0		&0			&0		\\
};

\sfrm{tbl}{5}{5}; \hsline{tbl}{3}{5}; \vsline{tbl}{5}{3};

\node[above=1pt of tbl-1-2.north]{$y_1$}; \node[above=1pt of tbl-1-4.north east]{$y_2$};
\node[base left=2pt of tbl-2-1]{$x_1$}; \node[base left=2pt of tbl-4-1.south west]{$x_2$};
\end{tikzpicture}
\end{center}

\noindent
where as usual, $I= I_m$ is the identity $m \times m$ matrix and $J = J_m(\lambda)$ is the Jordan block of size
$m \times m$ with eigenvalue $\lambda$.
\end{example}

\begin{lemma}\label{L:LocalAlgebra}
Let $p, m \in \mathbb{N}$, $\nu \in \mathbb{Z}$ and $\xi \ne \pi \in \kk[\xi]$ and irreducible polynomial.
Then the band object $W = B\bigl((w, \rho), m, \pi)$ with  $(w, \rho) = \lha y_p \sim x_p \stackrel{\nu}\rha$ is indecomposable.
Similarly, the string object $S(w, \rho)$ with
$(w, \rho) =
\underbrace{y_p \sim x_p \stackrel{\nu}- y_p \sim x_p \stackrel{\nu}- \dots \stackrel{\nu}- y_p \sim x_p}_{ m \;
\scriptsize{\mbox{\rm times}}}
$ is indecomposable as well.
\end{lemma}

\begin{proof} Let $W$ be an object as above.
It is sufficient to show that the endomorphism algebra $\Lambda := \End_{\dX}(W)$ is local. To prove this, observe that
$t \Lambda \subseteq \rad(\Lambda)$ and $\Lambda/t \Lambda = \kk[\xi]/(\pi^m)$ (where $\pi = \xi$ is the string case).
\end{proof}

\begin{lemma}\label{L:RotateBand}
For any band datum $\bigl((w, \rho), m, \pi)$ and $l \in \mathbb{Z}$ we  have an isomorphism
$$
B\bigl((w, \rho), m, \pi\bigr) \cong B\bigl((w^{(l)}, \rho^{(l)}), m, \pi\bigr),
$$
where $(w^{(l)}, \rho^{(l)})$ is the rotation of $(w, \rho)$ by $l$ positions.  In particular, we may always achieve that the
decoration of the exceptional edge takes the maximal value among all $\rho(\imath)$ for  $\imath \in \tau(w^+)$.
\end{lemma}

\begin{proof}
It is a straightforward linear algebra argument.
\end{proof}

\noindent
Now we have all ingredients to state and prove the main theorem of this section.
\begin{theorem}\label{T:ChessBoard} Let $\dX$ be a decorated chessboard  as above.
\begin{enumerate}
\item The string objects $S(w, \rho)$ and band objects $B\bigl((w, \rho), m, \pi\bigr)$ are indecomposable.
\item Moreover, any indecomposable object of $\Rep(\dX)$ is isomorphic to some string or band object.
\item We have: $B\bigl((w, \rho), m, \pi\bigr) \not\cong S(\check{w}, \check\rho)$ for any band datum $((w, \rho), m, \pi\bigr)$ and
string datum $(\check{w}, \check\rho)$.
\item Fixing the conventions (\ref{E:chessString}) and (\ref{E:chessBand}) (any $\dX$--word starts with a column element) we also get:
    \begin{enumerate}
    \item $S(w, \rho) \cong  S(w', \rho')$ if and only if $(w, \rho) = (w', \rho')$.
    \item $B\bigl((w, \rho), m, \pi\bigr) \cong  B\bigl((\tilde{w}, \tilde\rho), \tilde{m}, \tilde{\pi}\bigr)$ if and only if
    $(\pi, m) = (\tilde\pi, \tilde{m})$ and $(w, \rho)$ is a rotation of $(\tilde{w}, \tilde\rho)$.
    \end{enumerate}
\item A decomposition of an object of $\Rep(\dX)$ into a direct sum of indecomposable objects is unique up to a permutation
and automorphisms of indecomposable direct summands.
\end{enumerate}
\end{theorem}

\begin{proof} The main  ingredients are  provided by Propositions \ref{P:reductinPnotQ} and \ref{P:reductinPisQ}.

\medskip
\noindent
(1).~Let $W$ be either band or string object. According to Lemma \ref{L:LocalAlgebra}, the band object
$B\bigl((w, \rho), m, \pi\bigr)$ with $(w, \rho) = \lha y_p \sim x_p \stackrel{\nu}\rha$ is indecomposable for
any $p \in \mathbb{N}$, $\nu \in \mathbb{Z}$. In all other cases, $W$ belongs to some subcategory
$\Rep^{\le (p, q)} \bigl(\dX(\nu)\bigr)$ (with $p \ne q$) or $\underline{\Rep}^{\le (p,p), \circ} \bigl(\dX(\nu)\bigr)$, for
some $\nu \in \mathbb{Z}$. We assume that $\overline{(W_{pq})_{\diamond}} \ne 0$.
Recall that we have constructed representation equivalences
\begin{itemize}
\item $R = R^{pq}:\; \Rep^{\le (p, q)} \bigl(\dX(\nu)\bigr) \lar \Rep^{((p, q), \nu)}\bigl(\check\dX(\nu)\bigr)$ and
\item $R = R^{pp}: \, \underline{\Rep}^{\le (p,p), \circ} \bigl(\dX(\nu)\bigr) \lar \underline{\Rep}^{((p, p), \nu)}\bigl(\check\dX(\nu)\bigr)$ .
\end{itemize}
Then the following key  property is true:
\begin{equation}\label{E:InductionKeyStep}
R\Bigl(B\bigl((w, \rho), m, \pi)\Bigr) \cong B\bigl((\check{w}, \check\rho), m, \pi)\quad
\mbox{\rm and} \quad
R\bigl(S(w, \rho)\bigr) \cong S(\check{w}, \check\rho),
\end{equation}
where $(w, \rho)$ and  $(\check{w}, \check\rho)$  are related by the following rules.
\begin{itemize}
\item In the case $p \ne q$  replace, $x_q^{(2)}$ by $x_q$, $y_p^{(2)}$ by $y_p$ and
 every fragment $u \stackrel{\alpha}- y_p^{(1)} \sim x_q^{(1)} \stackrel{\beta}- v$ in $(\check{w}, \check\rho)$ by $u \stackrel{\alpha}- y_p \sim x_p \stackrel{\nu}- y_q \sim x_q \stackrel{\beta}- v$ to get  $(w, \rho)$.
\item In the case $p = q$, replace
 every fragment $u \stackrel{\alpha}- y_p^{(l)} \sim x_p^{(l)}\stackrel{\beta}-v$   in $(\check{w}, \check\rho)$ by  \linebreak $u \stackrel{\alpha}-  \underbrace{y_p \sim x_p
 \stackrel{\nu}-   \dots   \stackrel{\nu}- y_p \sim x_p}_{l \; \scriptsize\mbox{times}} \stackrel{\beta}-v$ to get $(w, \rho)$.
\end{itemize}
In the case of bands, we use  Lemma \ref{L:RotateBand} to move the Frobenius block $F$ from the matrix  $W_{pp}$ (if necessary). Now, the indecomposability of bands and strings follow  from Lemma \ref{L:LocalAlgebra} and formula
(\ref{E:InductionKeyStep}) by induction on the size of the matrix.

\smallskip
\noindent
(2).~Let $W \in \Rep(\dX)$ be an indecomposable object, $\nu = \mathsf{val}(W)$ and $(p, q) \in \mathbb{N}^2$ are such
that the block $W_{pq}$ is minimal.
Assume first that $p = q$ and $\overline{(W_{pq})_{\diamond}}$ is not nilpotent. Then
$W$ must contain some band $B\bigl((w, \rho), m, \pi\bigr)$ as a direct summand, where $(w, \rho) = \lha y_p \sim x_p \stackrel{\nu}\rha$. According to  Lemma \ref{L:LocalAlgebra}, the latter object is indecomposable, hence $W$ coincides with it.
Otherwise, consider the reduction functor $R^{pq}$. The result follows from the formulae (\ref{E:InductionKeyStep}) by induction.

\smallskip
\noindent
(3) and (4). These parts follow  from the formulae (\ref{E:InductionKeyStep}) and the fact that the reduction functor $R$ is a representation equivalence.

\smallskip
\noindent
(5) The Krull--Schmidt property follows from the observation that the endomorphism algebra of an indecomposable object of $\Rep(\dX)$ is local,  see  \cite[Chapter I.3.6]{ba}.
\end{proof}

\noindent
Finally,  it remains to observe that  Theorem \ref{T:JordanKronecker} is a special case of Theorem \ref{T:ChessBoard}.


\begin{thebibliography}{99}



\bibitem{AAEKO}
M.~Abouzaid, D.~Auroux, A.~Efimov, L.~Katzarkov and D.~Orlov, \emph{Homological mirror symmetry for punctured spheres},  J. Amer. Math. Soc. \textbf{26} (2013), no. \textbf{4}, 1051--1083.


\bibitem{ArtinVerdier}
M.~Artin and J.-L.~Verdier,
\emph{Reflexive modules over rational double points},
{Math. Ann.} \textbf{270} (1985), no. \textbf{1}, 79--82.


\bibitem{Atiyah}
M.~Atiyah, \emph{Vector bundles over an elliptic curve},
  Proc.\/ Lond.\/ Math.\/ Soc.\/ (3)  \textbf{7}  (1957)  414--452.



\bibitem{Auslander} M.~Auslander,
\emph{Rational singularities and almost split sequences},
{Trans. Amer. Math. Soc.} \textbf{293} (1986), no. \textbf{2}, 511--531.

\bibitem{PhilNotes}
M.~Auslander, \emph{Functors and morphisms determined by objects},
In Representation Theory of Algebras, Lecture Notes in Pure Appl. Math. \textbf{37},
(1978),  1--244.


\bibitem{Baciu} C.~Baciu,
 \emph{Maximal Cohen--Macaulay modules and stable vector bundles},   Computational
commutative and non-commutative algebraic geometry,  65--73, NATO Sci. Ser. III Comput.
Syst. Sci.,\textbf{196}   (2005).

\bibitem{ba}
H.~Bass,
\emph{Algebraic K-theory},  W. A. Benjamin, New York,
1968.



\bibitem{Survey} L.~Bodnarchuk, I.~Burban,  Yu.~Drozd and  G.-M.~Greuel,
  \emph{Vector bundles and torsion free sheaves on degenerations of elliptic
  curves},
Global aspects of complex geometry,  83--128, Springer, Berlin, 2006.


\bibitem{bo}
 V.~M.~Bondarenko,
 \emph{Bunches of semichained sets and their representations}, Preprint 88.60 of Institute
of Mathematics, Kiev, 1988.

\bibitem{bo1}
V.~M.~Bondarenko, \emph{
Representations of bundles of semichained sets and their applications},
{Algebra i Analiz}  \textbf{3}  (1991),  no. \textbf{5}, 38--61.


\bibitem{Bourbaki}
  N.~Bourbaki, \emph{\'El\'ements de math\'ematique. Alg\`ebre commutative. Chapitres 8 et 9},
Springer (2006).


 \bibitem{Brieskorn} E.~Brieskorn, \emph{Rationale Singularit\"aten komplexer
 Fl\"achen}, {Invent. Math.}  \textbf{4}  (1968),  336--358.


\bibitem{BrunsHerzog}
W.~Bruns and J.~Herzog,  \emph{Cohen--Macaulay Rings},
 Cambridge Studies in Advanced Mathematics 39,
 Cambridge Univ. Press, 1993.

\bibitem{Buchweitz} R.-O.~Buchweitz, \emph{Maximal Cohen--Macaulay modules and Tate--Cohomology over
Gorenstein rings}, Preprint 1987.


\bibitem{BGS}
R.~Buchweitz, G.-M.~Greuel and F.-O.~Schreyer,
\emph{
Cohen--Macaulay modules on hypersurface singularities  II},
{Invent. Math.} \textbf{88} (1987), no. \textbf{1}, 165--182.

\bibitem{Thesis} I.~Burban, \textit{Abgeleitete Kategorien und Matrixprobleme},
PhD Thesis, Kaiserslautern 2003, available at \url{https://kluedo.ub.uni-kl.de/files/1434/phd.pdf}.



\bibitem{bd}
 I.~Burban and Yu.~Drozd, \emph{
 Derived categories of nodal algebras}, J. Algebra 272 (2004) 46--94.

 \bibitem{Duke}
I.~Burban and  Yu.~ Drozd, \emph{Coherent sheaves on
rational curves with simple double points and transversal intersections},
  Duke Math. J.  \textbf{121}  (2004),  no. \textbf{2}, 189--229.

\bibitem{SurvOnCM} I.~Burban and  Yu.~Drozd,
\emph{Maximal Cohen--Macaulay modules over surface singularities},
  Trends in representation theory of algebras and related topics,  101--166, EMS Ser. Congr. Rep., Eur. Math. Soc.,  Z\"urich  (2008).



\bibitem{BG} I.~Burban and W.~Gnedin, \emph{Cohen--Macaulay modules over some non-reduced curve singularities}, arXiv:1301.3305.




\bibitem{BIKR}
 I.~Burban, O.~Iyama, B.~Keller and I.~Reiten,
\emph{Cluster tilting for one-dimensional hypersurface singularities},
 Adv. Math.  \textbf{217}  (2008),  no. \textbf{6}, 2443--2484.

\bibitem{BK4}
I.~Burban and B.~Kreu\ss{}ler,
\emph{Vector bundles on degenerations of elliptic curves and Yang--Baxter equations}, Memoirs of the AMS, vol. \textbf{220}, no. \textbf{1035} (2012).

\bibitem{CB}
W.~Crawley-Boevey,
\emph{Matrix problems and Drozd's theorem},  Topics in algebra,
Banach Center Publ.~\textbf{26}, Part \textbf{1}, 199--222 (1990).


\bibitem{CurtisReiner}
C.~Curtis and I.~Reiner, \emph{Methods of representation theory. Vol. I. With applications to finite groups and orders},    Wiley--Interscience Publication (1990).

\bibitem{deJongPfister} T.~de Jong and G.~Pfister,
\emph{Local Analytic Geometry},  Advanced Lectures in Mathematics, Vieweg  2000.


\bibitem{Dieterich}
E.~Dieterich, \emph{Tame orders},
In {Topics in Algebra, Part 1} (Warsaw, 1988), Banach Center Publ. \textbf{26}, Part 1, PWN, Warsaw, 1990,  233--261.

\bibitem{DieterichInvent}
E.~Dieterich,
\emph{Lattices over curve singularities with large conductor},
 Invent. Math. \textbf{114} (1993), no. \textbf{2}, 399--433.

\bibitem{DrozdLOMI}
Yu.~Drozd,
\emph{Matrix problems, and categories of matrices},
Investigations on the theory of representations,
Zap.~Nauchn.~Semin.~Leningr.~Otd.~Mat.~Inst.~Steklova \textbf{28}, 144--153 (1972).

\bibitem{dr}
Yu.~Drozd, \emph{
 Reduction algorithm and representations of boxes and algebras},
 Comtes Rendue Math. Acad. Sci. Canada, 23 (2001), 97--125.

\bibitem{DrozdSurvey}
Yu.~Drozd,
\emph{
Vector bundles and Cohen--Macaulay modules},
In {Representations of Finite Dimensional Algebras and Related Topics in
Lie Theory and Geometry},
Fields Inst. Commun., \textbf{40}, Amer. Math. Soc., Providence, RI, (2004), 189--222.

\bibitem{vb}
 Yu.~Drozd, \emph{Vector bundles over projective curves}, IMPA, Rio de Janeiro, 2008.





\bibitem{DrozdGreuel}
Yu.~Drozd and  G.-M.~Greuel,
\emph{
Tame-wild dichotomy for Cohen--Macaulay modules},
Math. Ann. \textbf{294} (1992), no. \textbf{3}, 387--394.




\bibitem{semicont}
Yu.~Drozd and G.-M.~Greuel,
\emph{
Cohen--Macaulay module type},
{Compositio Math.}  \textbf{89}  (1993),  no. 3, 315--338.


\bibitem{DrozdGreuelBundles}
Yu.~Drozd and  G.-M.~Greuel, \emph{Tame and wild projective curves and
classification of vector bundles},   J. Algebra  \textbf{246}  (2001),  no.
\textbf{1}, 1--54.





\bibitem{DGK}
Yu.~Drozd, G.-M.~Greuel and I.~Kashuba,
\emph{
On Cohen--Macaulay modules on surface singularities},
{Mosc. Math. J.} \textbf{3} (2003), no. 2, 397--418.


\bibitem{DroRoi} Yu.~Drozd and A.~Roiter,
\emph{Commutative rings with a finite number of indecomposable integral representations},
{Izv. Akad. Nauk SSSR Ser. Mat.}  \textbf{31}  (1967),  783--798.


\bibitem{Efimov}
A.~Efimov, \emph{Homological mirror symmetry for curves of higher genus},
 Adv. Math. \textbf{230} (2012), no. \textbf{2}, 493--530.

\bibitem{Eisenbud} D.~Eisenbud, \emph{Homological algebra on a
complete intersection, with an application to group representations},
 {Trans. Amer. Math. Soc.}  \textbf{260}  (1980), no. \textbf{1}, 35--64.



\bibitem{Esnault}
H.~Esnault,
\emph{
Reflexive modules on quotient surface singularities},
{J. Reine Angew. Math.} \textbf{362} (1985), 63--71.

\bibitem{EsnaultViehweg}
H.~Esnault and E.~Viehweg, \emph{Two--dimensional
quotient singularities deform to quotient singularities},
  {Math. Ann.}  \textbf{271}  (1985),  no. \textbf{3}, 439--449.








\bibitem{Gabriel}
P.~Gabriel, \emph{Des cat\'egories  ab\'eliennes},
{Bull. Soc. Math. France} \textbf{90} (1962), 323--448.

\bibitem{Galinat} L.~Galinat, \emph{Orlov's Equivalence and Maximal Cohen--Macaulay Modules over the
Cone of an Elliptic Curve},  arXiv:1302.1383, to appear in Math. Nachrichten.


\bibitem{Gantmacher}
F.~Gantmacher, \emph{The theory of matrices},  vol.~\textbf{1},
Providence, AMS Chelsea Publishing  (1998).

\bibitem{GSVerdier}
G.~Gonzalez--Sprinberg and J.-L.~Verdier, \emph{Construction g\'eom\'etrique de la correspondance de McKay},
Ann. Sci. \'Ecole Norm. Sup. (4) \textbf{16} (1983), no. \textbf{3}, 409--449 (1984).


\bibitem{GreenReiner}
E.~Green and  I.~Reiner, \emph{
Integral representations and diagrams},
Michigan Math. J. \textbf{25} (1978), no. \textbf{1}, 53--84.


\bibitem{GreuelKnoerrer}
G.-M.~Greuel and  H.~Kn\"orrer,
\emph{
Einfache Kurvensingularit\"aten und torsionsfreie Moduln},
Math. Ann. \textbf{270} (1985), no. \textbf{3}, 417--425.

\bibitem{GLS} G.-M.~Greuel, C.~Lossen and E.~Shustin,
\emph{Introduction to Singularities and Deformations},
 Springer Monographs in Mathematics. Springer, Berlin, 2007.


\bibitem{Jacob} H.~Jacobinski,
\emph{Sur les ordres commutatifs avec un nombre fini de r\'eseaux ind\'ecomposables},
{Acta Math.} \textbf{118} (1967), 1--31.




\bibitem{Herzog1}
J.~Herzog, \emph{Ringe mit nur endlich vielen Isomorphieklassen von maximalen, unzerlegbaren Cohen--Macaulay--Moduln},
{Math. Ann.}  \textbf{233}  (1978), no. \textbf{1}, 21--34.

\bibitem{IyamaReiten}
O.~Iyama and  I.~Reiten, \emph{Fomin--Zelevinsky mutation and tilting modules over Calabi--Yau algebras},
 Amer. J. Math. \textbf{130} (2008), no. \textbf{4}, 1087--1149.

\bibitem{IyamaYoshino}
O.~Iyama and Yu.~Yoshino,  \emph{Mutation in triangulated categories and rigid Cohen--Macaulay modules},
 Invent. Math. \textbf{172} (2008), no. \textbf{1}, 117--168.






\bibitem{KahnDiss}
 C.~Kahn, \emph{Reflexive Moduln auf einfach--elliptischen Fl\"achensingularit\"aten}, Dissertation,
   Bonner Mathematische Schriften  (1988).


\bibitem{Kahn}
C.~Kahn, \emph{Reflexive modules on minimally elliptic singularities},
{Math. Ann.}  \textbf{285}  (1989),  no. \textbf{1}, 141--160.




 \bibitem{KapLi} A.~Kapustin and Yi Li, \emph{Topological
correlators in Landau--Ginzburg models with boundaries},
 {Adv. Theor. Math. Phys.}  \textbf{7}  (2003),  no. \textbf{4}, 727--749.


\bibitem{KellerMurfetVdB}
B.~Keller, D.~Murfet and M.~Van den Bergh,  \emph{On two examples by Iyama and Yoshino},
 Compos. Math. \textbf{147} (2011), no. \textbf{2}, 591--612.

\bibitem{KhovanovRozansky}
M.~Khovanov and  L.~Rozansky, \emph{Matrix factorizations and link homology},
 Fund. Math. \textbf{199} (2008), no. \textbf{1}, 1--91.

\bibitem{Knoerrer1}
H.~Kn\"orrer, \emph{Torsionsfreie Moduln bei Deformation von Kurvensingularit\"aten},
Singularities, Representation of Algebras  and Vector Bundles,
Lecture  Notes Math. \textbf{1273} (1987), 150-155.


\bibitem{Knoerrer2}
H.~Kn\"orrer, \emph{Cohen--Macaulay modules on hypersurface singularities I},
Invent. Math \textbf{ 88}  (1987),  no. \textbf{1}, 153--164.

\bibitem{Kontsevich}
M.~Kontsevich, \emph{Homological algebra of mirror symmetry} Proceedings of the International Congress of Mathematicians, vol. 1 (Z\"urich, 1994), 120--139, Birkh\"auser (1995).


\bibitem{LeusckeWiegand} G.~Leuschke and R.~Wiegand, \emph{Cohen--Macaulay representations},
AMS Mathematical Surveys and Monographs \textbf{181} (2012).

\bibitem{nar}
L.~Nazarova and  A.~Roiter, \emph{On a problem of I.~M.~Gelfand},
Functional analysis and its applications {\bf 7}, (1973) no. \textbf{4},  54--69.





\bibitem{Orlov}
D.~Orlov, \emph{
Derived categories of coherent sheaves and triangulated categories of singularities},
in Algebra, arithmetic, and geometry. In honor of Y. I. Manin on the occasion of his 70th birthday,
 Vol. II,  Birkh\"auser, Progress in Mathematics \textbf{270}, 503--531 (2009).

\bibitem{OrlovOcCompletions}
D.~Orlov, \emph{
Formal completions and idempotent completions of triangulated categories of singularities},
 Adv. Math. \textbf{226} (2011), no. \textbf{1}, 206--217.

\bibitem{Popescu} N.~Popescu, \emph{Abelian categories with applications to rings and
modules}, London Mathematical Society Monographs, no. \textbf{3},
 Academic Press  (1973).




\bibitem{Riemenschneider}
O.~Riemenschneider, \emph{
Zweidimensionale Quotientensingularitaeten: Gleichungen und Syzygien},
{Arch. Math.} \textbf{37} (1981), 406--417.

\bibitem{RingelRoggenkamp}
C.-M.~Ringel and  K.~Roggenkamp, \emph{Diagrammatic methods in the representation theory of orders},
  J. Algebra  \textbf{60}  (1979), no. \textbf{1}, 11--42.


\bibitem{Schreyer}
F.-O.~Schreyer,  \emph{Finite and countable CM--representation type},
In {Singularities, Representation of Algebras, and Vector Bundles},
 Lecture  Notes Math. \textbf{1273} (1987), 9--34.

\bibitem{Seidel}
P.~Seidel, \emph{Homological mirror symmetry for the genus two curve},
 J. Algebraic Geom. \textbf{20} (2011), no. \textbf{4}, 727–769.


\bibitem{Serre} J.~-P.~Serre, \emph{Local Algebra}, Springer Monographs in Mathematics. Springer-Verlag,
Berlin (2000).

\bibitem{Sheridan} N.~Sheridan, \emph{On the homological mirror symmetry conjecture for pairs of pants},
 J. Differential Geom. \textbf{89} (2011), no. \textbf{2}, 271--367.



\bibitem{ShepherdBarron}
N.~Shepherd--Barron, \textit{Degenerations with numerically effective canonical divisor},
The birational geometry of degenerations, 33--84,
Progr. Math.  \textbf{29}, Birkh\"auser  (1983).

\bibitem{Stevens}
J.~Stevens,
\emph{
Degenerations of elliptic curves and equations for cusp singularities},
Math. Ann. \textbf{311} (1998), no. \textbf{2}, 199--222.

\bibitem{Stevens2}
J.~Stevens,
\emph{Improvements of nonisolated surface singularities},
  J. London Math. Soc. (2)  \textbf{39}  (1989),  no. \textbf{1}, 129--144.




\bibitem{vanStraten} D.~van Straten, \emph{Weakly  normal surface singularities and their improvements},
Dissertation, Leiden 1987.

\bibitem{VandenBergh}
M.~Van den Bergh, \emph{Three--dimensional flops and noncommutative rings},
 Duke Math. J. \textbf{122} (2004), no.~\textbf{3}, 423--455.

\bibitem{Wiegand}
R.~Wiegand, \emph{Noetherian rings of bounded representation type},
Commutative algebra,  497--516, Math. Sci. Res. Inst. Publ. \textbf{15},  1989.

\bibitem{Wunram}
J.~Wunram, \emph{
Reflexive modules on cyclic quotient surface singularities},
In {Singularities, Representation of Algebras and Vector Bundles},
 Lecture  Notes Math. \textbf{1273} (1987),
  221--231.




\bibitem{Yoshino}
Y.~Yoshino,
\textit{Cohen--Macaulay Modules over Cohen--Macaulay Rings},
 London Mathematical Society Lecture Note Series \textbf{146},
 Cambridge Univ. Press, 1990.

\bibitem{YoshinoKawamoto}
Y.~Yoshino and T.~Kawamoto,
\emph{The fundamental module of a normal local domain of dimension 2},
{Trans. Amer. Math. Soc.} \textbf{309} (1988), no. \textbf{1}, 425--431.







\end{thebibliography}
\end{document}